\documentclass[a4paper, reqno,10pt]{amsart}

\usepackage{amsthm,amsfonts,amssymb,amsmath,amsxtra}
\usepackage[all]{xy}
\SelectTips{cm}{}
\usepackage{tikz}
\usepackage{verbatim}

\usepackage[margin=1.25in]{geometry}
\usepackage{mathrsfs}

\RequirePackage{xspace}
\RequirePackage{etoolbox}
\RequirePackage{varwidth}
\RequirePackage{enumitem}
\RequirePackage{tensor}
\RequirePackage{mathtools}
\RequirePackage{longtable}
\RequirePackage{multirow}
\RequirePackage{makecell}
\RequirePackage{bigints}
\RequirePackage{xcolor}

\usepackage[backref=page]{hyperref} %

\newcommand{\jiao}{\stackrel{\BL}{\cap}}
\newcommand{\Ltimes}{\stackrel{\BL}{\otimes}}

\newcommand{\sD}{\ensuremath{\mathscr{D}}\xspace}
\newcommand{\sE}{\ensuremath{\mathscr{E}}\xspace}

\newcommand{\sV}{\ensuremath{\mathscr{V}}\xspace}

\newcommand{\fka}{\ensuremath{\mathfrak{a}}\xspace}

\newcommand{\fkc}{\ensuremath{\mathfrak{c}}\xspace}
\newcommand{\fkd}{\ensuremath{\mathfrak{d}}\xspace}

\newcommand{\fkg}{\ensuremath{\mathfrak{g}}\xspace}

\newcommand{\fkk}{\ensuremath{\mathfrak{k}}\xspace}

\newcommand{\fkp}{\ensuremath{\mathfrak{p}}\xspace}
\newcommand{\fkq}{\ensuremath{\mathfrak{q}}\xspace}
\newcommand{\fkr}{\ensuremath{\mathfrak{r}}\xspace}
\newcommand{\fks}{\ensuremath{\mathfrak{s}}\xspace}

\newcommand{\fku}{\ensuremath{\mathfrak{u}}\xspace}

\newcommand{\BA}{\ensuremath{\mathbb{A}}\xspace}

\newcommand{\BC}{\ensuremath{\mathbb{C}}\xspace}

\newcommand{\BE}{\ensuremath{\mathbb{E}}\xspace}

\newcommand{\BG}{\ensuremath{\mathbb{G}}\xspace}

\newcommand{\BI}{\ensuremath{\mathbb{I}}\xspace}
\newcommand{\BJ}{\ensuremath{\mathbb{J}}\xspace}

\newcommand{\BL}{\ensuremath{\mathbb{L}}\xspace}

\newcommand{\BQ}{\ensuremath{\mathbb{Q}}\xspace}
\newcommand{\BR}{\ensuremath{\mathbb{R}}\xspace}

\newcommand{\BV}{\ensuremath{\mathbb{V}}\xspace}

\newcommand{\BX}{\ensuremath{\mathbb{X}}\xspace}

\newcommand{\BZ}{\ensuremath{\mathbb{Z}}\xspace}

\newcommand{\bA}{\ensuremath{\mathbf{A}}\xspace}
\newcommand{\bB}{\ensuremath{\mathbf{B}}\xspace}

\newcommand{\bH}{\ensuremath{\mathbf{H}}\xspace}

\newcommand{\bK}{\ensuremath{\mathbf{K}}\xspace}

\newcommand{\bP}{\ensuremath{\mathbf{P}}\xspace}

\newcommand{\bV}{\ensuremath{\mathbf{V}}\xspace}

\newcommand{\srs}{\ensuremath{\mathrm{srs}}\xspace}

\newcommand{\Ram}{\ensuremath{\mathrm{Ram}}\xspace}

\newcommand{\CA}{\ensuremath{\mathcal{A}}\xspace}
\newcommand{\CB}{\ensuremath{\mathcal{B}}\xspace}

\newcommand{\CD}{\ensuremath{\mathcal{D}}\xspace}
\newcommand{\CE}{\ensuremath{\mathcal{E}}\xspace}
\newcommand{\CF}{\ensuremath{\mathcal{F}}\xspace}
\newcommand{\CG}{\ensuremath{\mathcal{G}}\xspace}
\newcommand{\CH}{\ensuremath{\mathcal{H}}\xspace}
\newcommand{\CI}{\ensuremath{\mathcal{I}}\xspace}
\newcommand{\CJ}{\ensuremath{\mathcal{J}}\xspace}
\newcommand{\CK}{\ensuremath{\mathcal{K}}\xspace}

\newcommand{\CM}{\ensuremath{\mathcal{M}}\xspace}
\newcommand{\CN}{\ensuremath{\mathcal{N}}\xspace}
\newcommand{\CO}{\ensuremath{\mathcal{O}}\xspace}

\newcommand{\CS}{\ensuremath{\mathcal{S}}\xspace}

\newcommand{\CU}{\ensuremath{\mathcal{U}}\xspace}
\newcommand{\CV}{\ensuremath{\mathcal{V}}\xspace}

\newcommand{\CZ}{\ensuremath{\mathcal{Z}}\xspace}

\newcommand{\CCM}{\ensuremath{\mathcal{C\!M}}\xspace}

\newcommand{\RM}{\ensuremath{\mathrm{M}}\xspace}

\newcommand{\RR}{\ensuremath{\mathrm{R}}\xspace}

\newcommand{\RT}{\ensuremath{\mathrm{T}}\xspace}
\newcommand{\RU}{\ensuremath{\mathrm{U}}\xspace}
\newcommand{\RV}{\ensuremath{\mathrm{V}}\xspace}

\newcommand{\nat}{{\natural}}

\newcommand{\Ei}{\mathrm{Ei}}

\DeclareMathOperator{\Aut}{Aut}

\newcommand{\Ch}{{\mathrm{Ch}}}
\DeclareMathOperator{\charac}{char}

\newcommand{\corr}{\mathrm{corr}}

\newcommand{\del}{\operatorname{\partial Orb}}
\newcommand{\delJ}{\partial\BJ}
\DeclareMathOperator{\diag}{diag}

\renewcommand{\div}{{\mathrm{div}}}

\DeclareMathOperator{\End}{End}

\DeclareMathOperator{\Gal}{Gal}

\newcommand{\GL}{\mathrm{GL}}

\newcommand{\GU}{\mathrm{GU}}

\DeclareMathOperator{\Hom}{Hom}

\newcommand{\id}{\ensuremath{\mathrm{id}}\xspace}
\let\Im\relax

\DeclareMathOperator{\Im}{Im}

\newcommand{\Ind}{{\mathrm{Ind}}}
\DeclareMathOperator{\Int}{\ensuremath{\mathrm{Int}}\xspace}

\DeclareMathOperator{\length}{length}
\DeclareMathOperator{\Lie}{Lie}

\newcommand{\M}{\mathrm{M}}

\DeclareMathOperator{\Nm}{Nm}

\DeclareMathOperator{\Orb}{Orb}

\DeclareMathOperator{\Ros}{Ros}

\newcommand{\LN}{\,^\BL\!\CN}
\newcommand{\LCM}{\,^\BL\CCM}

\renewcommand{\Re}{{\mathrm{Re}}}
\newcommand{\red}{\ensuremath{\mathrm{red}}\xspace}

\DeclareMathOperator{\Res}{Res}
\newcommand{\rs}{\ensuremath{\mathrm{rs}}\xspace}

\newcommand{\Sh}{\mathrm{Sh}}
\newcommand{\sig}{{\mathrm{sig}}}

\newcommand{\SL}{{\mathrm{SL}}}

\DeclareMathOperator{\Spec}{Spec}
\DeclareMathOperator{\Spf}{Spf}

\newcommand{\SO}{{\mathrm{SO}}}

\newcommand{\ssm}{\smallsetminus}

\DeclareMathOperator{\supp}{supp}

\DeclareMathOperator{\tr}{tr}

\newcommand{\Hk}{\mathrm{Hk}}

\newcommand{\U}{\mathrm{U}}

\DeclareMathOperator{\vol}{vol}

\newcommand{\wt}{\widetilde}
\newcommand{\wh}{\widehat}

\newcommand{\pair}[1]{\langle {#1} \rangle}

\newcommand{\ov}{\overline}
\newcommand{\incl}{\hookrightarrow}
\newcommand{\lra}{\longrightarrow}
\newcommand{\imp}{\Longrightarrow}

\newcommand{\bs}{\backslash}
\newcommand{\la}{\langle}
\newcommand{\ra}{\rangle}
\newcommand{\lv}{\lvert}
\newcommand{\rv}{\rvert}
\newcommand{\LNSch}{\ensuremath{(\mathrm{LNSch})}\xspace}

\newtheorem{theorem}{Theorem}
\newtheorem{proposition}[theorem]{Proposition}
\newtheorem{lemma}[theorem]{Lemma}

\newtheorem{conjecture}[theorem]{Conjecture}
\newtheorem{`conjecture'}[theorem]{``Conjecture''}
\newtheorem{corollary}[theorem]{Corollary}

\theoremstyle{definition}
\newtheorem{definition}[theorem]{Definition}

\newtheorem{remark}[theorem]{Remark}

\newenvironment{altenumerate}
   {\begin{list}
      {(\theenumi) }
      {\usecounter{enumi}
       \setlength{\labelwidth}{0pt}
       \setlength{\labelsep}{0pt}
       \setlength{\leftmargin}{0pt}
       \setlength{\itemsep}{\the\smallskipamount}
       \renewcommand{\theenumi}{\roman{enumi}}
      }}
   {\end{list}}

\newenvironment{altitemize}
   {\begin{list}
      {$\bullet$}
      {\setlength{\labelwidth}{0pt}
	   \setlength{\itemindent}{5pt}
       \setlength{\labelsep}{5pt}
       \setlength{\leftmargin}{0pt}
       \setlength{\itemsep}{\the\smallskipamount}
      }}
   {\end{list}}

\numberwithin{equation}{section}
\numberwithin{theorem}{section}

\newcommand{\sform}{\ensuremath{(\text{~,~})}\xspace}

\setitemize[0]{leftmargin=*,itemsep=\the\smallskipamount}
\setenumerate[0]{leftmargin=*,itemsep=\the\smallskipamount}

\renewcommand{\to}{%
   \ifbool{@display}{\longrightarrow}{\rightarrow}%
   }
\let\shortmapsto\mapsto
\renewcommand{\mapsto}{%
   \ifbool{@display}{\longmapsto}{\shortmapsto}%
   }
\newcommand{\hooklongrightarrow}{\mathrel{\mkern 0.5mu\lhook\mkern -3.5mu\relbar\mkern -3mu \rightarrow }}
\newcommand{\inj}{%
   \ifbool{@display}{\hooklongrightarrow}{\hookrightarrow}
   }
\newcommand{\isoarrow}{%
   \ifbool{@display}{\overset{\sim}{\longrightarrow}}{\xrightarrow\sim}%
   }
\newlength{\olen}
\newlength{\ulen}
\newlength{\xlen}
\newcommand{\xra}[2][]{%
   \ifbool{@display}%
      {\settowidth{\olen}{$\overset{#2}{\longrightarrow}$}%
       \settowidth{\ulen}{$\underset{#1}{\longrightarrow}$}%
       \settowidth{\xlen}{$\xrightarrow[#1]{#2}$}%
       \ifdimgreater{\olen}{\xlen}%
          {\underset{#1}{\overset{#2}{\longrightarrow}}}%
          {\ifdimgreater{\ulen}{\xlen}%
             {\underset{#1}{\overset{#2}{\longrightarrow}}}
             {\xrightarrow[#1]{#2}}}}%
      {\xrightarrow[#1]{#2}}
   }
\makeatother
\newcommand{\xyra}[2][]{%
   \settowidth{\xlen}{$\xrightarrow[#1]{#2}$}%
   \ifbool{@display}%
      {\settowidth{\olen}{$\overset{#2}{\longrightarrow}$}%
       \settowidth{\ulen}{$\underset{#1}{\longrightarrow}$}%
       \ifdimgreater{\olen}{\xlen}%
          {\mathrel{\xymatrix@M=.12ex@C=3.2ex{\ar[r]^-{#2}_-{#1} &}}}%
          {\ifdimgreater{\ulen}{\xlen}%
             {\mathrel{\xymatrix@M=.12ex@C=3.2ex{\ar[r]^-{#2}_-{#1} &}}}
             {\mathrel{\xymatrix@M=.12ex@C=\the\xlen{\ar[r]^-{#2}_-{#1} &}}}}}%
      {\mathrel{\xymatrix@M=.12ex@C=\the\xlen{\ar[r]^-{#2}_-{#1} &}}}%
   }
\makeatletter
\newcommand{\xla}[2][]{%
   \ifbool{@display}%
      {\settowidth{\olen}{$\overset{#2}{\longleftarrow}$}%
       \settowidth{\ulen}{$\underset{#1}{\longleftarrow}$}%
       \settowidth{\xlen}{$\xleftarrow[#1]{#2}$}%
       \ifdimgreater{\olen}{\xlen}%
          {\underset{#1}{\overset{#2}{\longleftarrow}}}%
          {\ifdimgreater{\ulen}{\xlen}%
             {\underset{#1}{\overset{#2}{\longleftarrow}}}
             {\xleftarrow[#1]{#2}}}}%
      {\xleftarrow[#1]{#2}}
   }
\renewcommand{\lra}{%
   \ifbool{@display}{\longleftrightarrow}{\leftrightarrow}%
   }
\newcommand{\undertilde}{\raisebox{0.4ex}{\smash[t]{$\scriptstyle\sim$}}}

\begin{document}
\thanks{Research of W. Zhang is partially supported by the NSF grant DMS $\#$1838118 and \#1901642.\\\indent MSC: 11F27, 11F67, 11G40, 14C25, 14G35.}

\title{Weil representation and Arithmetic Fundamental Lemma
}

\author{W. Zhang}
\address{Massachusetts Institute of Technology, Department of Mathematics, 77 Massachusetts Avenue, Cambridge, MA 02139, USA}
\email{weizhang@mit.edu}

\date{\today}

\begin{abstract}
We study a partially linearized version of the relative trace formula for the arithmetic Gan--Gross--Prasad conjecture for the unitary group $\U(V)$.  The linear factor in this relative trace formula provides an $\SL_2$-symmetry which allows  us to prove by induction the arithmetic fundamental lemma over $\BQ_p$ when $p$ is odd and $p\geq \dim V$.
\end{abstract}

\maketitle

\tableofcontents

\section{Introduction}\label{s:intro}

The theorem of Gross and Zagier \cite{GZ} relates the N\'eron--Tate heights of Heegner points on modular curves to the central derivative of certain $L$-functions. The arithmetic Gan--Gross--Prasad conjecture \cite{GGP,Z12,RSZ3} is a generalization of this theorem to higher-dimensional Shimura varieties. This conjecture is inspired by the (usual) Gan--Gross--Prasad conjecture relating period integrals on classical groups to special values of Rankin--Selberg tensor product $L$-functions.  In \cite{JR} Jacquet and Rallis proposed a relative trace formula (RTF) approach to this last conjecture in the case of unitary groups and there have been much progress along this direction in the past years. Inspired by their approach, in \cite{Z12} the author proposed a relative trace formula approach to the arithmetic Gan--Gross--Prasad conjecture.  This approach reduces the problem to certain local statements, notably  the arithmetic fundamental lemma (AFL) conjecture formulated by the author in \cite{Z12}, and the arithmetic transfer (AT) conjecture formulated by Rapoport, Smithling, and the author \cite{RSZ1,RSZ2}. The AFL and AT conjectures relate the special values of the derivative of orbital integrals to arithmetic intersection numbers on a Rapoport--Zink formal moduli space (RZ space) of $p$-divisible groups,
\[
\del\bigl(\gamma, \mathbf{1}_{S_n(O_{F_0})}\bigr) 
	   = -\Int(g)\cdot\log q,
\]
cf. the precise statement of Conjecture \ref{AFLconj} for the AFL conjecture.

The goal of this paper is to give a proof of the AFL conjecture over $F_0=\BQ_p$ when $p\geq n$, for an open dense subset of regular semisimple elements (i.e., the set of ``strongly regular semisimple elements" in the sense of \cite{Y}), cf. Theorem \ref{thm AFL}. This restriction is harmless for  the  relative trace formula approach to the arithmetic Gan--Gross--Prasad conjecture.

In fact, we also obtain a proof of the Jacquet--Rallis fundamental lemma (FL) conjecture over $p$-adic field, a theorem due to Yun \cite{Y} and Gordan \cite{Go} for $p$ large, which is an identity between two orbital integrals
\[
\Orb(\gamma, \mathbf{1}_{S_n(O_{F_0})}\bigr) 
	   = \Orb(g, \mathbf{1}_{K_0}),
\] 
cf. the precise statement of Conjecture \ref{FLconj}. The idea is similar to the proof of the AFL and is easier to explain. For our proof of the FL, the main input is a study of a ``partially linearized" version of the Jacquet--Rallis  RTF, which we call a semi-Lie algebra version. This is closely related to the RTF of Yifeng Liu to the Fourier--Jacobi periods/cycles \cite{Liu14,Liu18}.
The advantage of the linearization is to gain more ``symmetry", i.e.,  there is an ``action" on the RTF (changing test functions) by $\SL_2$ under the Weil representation. The $\SL_2$-modularity plays the role in the global setting of the Fourier transform in the local harmonic analysis, a crucial ingredient in \cite{Z14}  to prove the smooth transfer conjecture of Jacquet--Rallis.

Now we give a little more detail of our approach.  Let $F_0$ be a totally real number field, and $F$ a CM quadratic extension of $F_0$. Let $V$ be an $F/F_0$-hermitian space with $\dim_F V=n$ and $\U(V)$ the associated isometry group (a reductive group over $F_0$). Consider the (diagonal) action of $\U(V)$ on the product $\U(V)\times V$, where the two factors are viewed as affine varieties over $F_0$ endowed with the conjugation action and the standard action respectively. For unexplained notation, we refer the reader to Notation \S\ref{notation} and the main body of the paper. To any Schwartz function $\Phi\in \CS((\U(V)\times V)(\BA_0))$, we can associate a kernel function 
$$
\CK_{\Phi}(g)= \sum_{(\delta,u)\in (\U(V)\times V)(F_0)}\Phi(g^{-1} (\delta,u)),\quad g\in \U(V)(\BA_0),
$$
which is left invariant under $\U(V)(F_0)$. Then, as one usually does in the theory of relative trace formula, one may study the distribution on $(\U(V)\times V)(\BA_0)$ (at least for certain nice test functions $\Phi$),
$$
\BI(\Phi)=\int_{[\U(V)]} \CK_{\Phi}(g)\,dg.
$$
Here $[G]\colon=G(F_0)\bs G(\BA_0)$ for an algebraic group $G$ over $F_0$.
Similarly, one can start with the (diagonal) action of $\GL_{n,F_0}$ on the product $S_n\times  V_n'$ where $V_n'={\rm M}_{1,n}\times {\rm M}_{n,1}$ is the product of the space of column and row vectors, cf. \S\ref{s:gpthsetup}. To any Schwartz function $\Phi'\in \CS((S_n\times V_n')(\BA_0))$, we have a similar kernel function $\CK_{\Phi'}$ and a distribution (for nice test functions $\Phi'$)
$$
\BJ(\Phi')=\int_{[\GL_{n,F_0}]} \CK_{\Phi'}(g)\,\eta_{F/F_0}\circ\det(g)\,dg.
$$
By the smooth transfer between $\Phi$ and $\Phi'$ through their orbital integrals (relative to the group actions here), one can match the distributions $\BI$ and $\BJ$.

Now, due to the presence of the linear factors $V$ and $V_n'$ respectively, the Weil representation $\omega$ of $\SL_2(\BA_0)$ acts on $\CS((\U(V)\times V)(\BA_0))$ and $\CS((S_n\times V_n')(\BA_0))$, hence on the distributions $\BI$ and $\BJ$,
$$
\BI(h,\Phi)\colon=\BI(\omega(h)\Phi),\quad\text{and} \quad\BJ(h,\Phi')\colon=\BJ(\omega(h)\Phi')
$$
where $h\in \SL_2(\BA_0)$. Moreover, the action is ``modular" in the sense that $h\mapsto \BI(h,\Phi)$ and $\BJ(h,\Phi')$ are left invariant under $\SL_2(F_0)$, as an application of the Poisson summation formula. In other words, we may enrich the kernel function to a two-variable one
$$
\CK_{\Phi}(g,h)= \sum_{(\delta,u)\in (\U(V)\times V)(F_0)}\omega(h)\Phi(g^{-1} (\delta,u)),\quad g\in \U(V)(\BA_0),\,  h\in \SL_2(\BA_0).
$$
The natural question now is how the Weil representation fits into the comparison of the two distributions. From \cite{Z14} and \cite{Xue} one can deduce that the Weil representation commutes with  smooth transfer, cf. Theorem \ref{thm Weil ST} in the appendix. 

Both  distributions $\BI$ and $\BJ$ can be expanded as a sum over orbital integrals.
 Then the $\SL_2$-modularity amounts to certain recursive relations between the orbital integrals appearing in $\BI$ and $\BJ$. One may hope that the recursive relations are ample enough to allow us to extract identities such as the aforementioned fundamental lemma, starting from some simple identities that can be verified directly. This resembles the situation in the geometric approach (cf. \cite{Ngo}, \cite{Y} ) where one also needs to verify some simple cases directly as a starting point before applying the ``perverse continuation principle". 
 
 The idea does not work directly to yield a proof of the Jacquet--Rallis FL; however, it does work if we take two additional inputs. The first input is to consider  a ``slice" of the semi-Lie algebra version. For example, we fix a suitable monic polynomial $\alpha$ and denote by $\U(V)(\alpha)$ the subscheme of $\U(V)$ consisting of elements with characteristic polynomial equal to $\alpha$. We  then introduce a kernel function,
$$
\CK_{\Phi, \alpha}(g)= \sum_{(\delta,u)\in (\U(V)(\alpha)\times V)(F_0)}\Phi(g^{-1} (\delta,u)),\quad g\in \U(V)(\BA_0).
$$
Here the sum runs only over a subset of $\U(V)(F_0)$-orbits on  $(\U(V)\times V)(F_0)$. Similarly we define a distribution
$$
\BI_{\alpha}(\Phi)=\int_{[\U(V)]} \CK_{\Phi,\alpha}(g)\,dg.
$$
This still keeps the action  of $\SL_2(\BA_0)$ under the Weil representation $\omega$
\begin{align}\label{eq:I g0}
\BI_{\alpha}(h,\Phi)=\BI_{\alpha}(\omega(h)\Phi),\quad h\in \SL_2(\BA_0).
\end{align}We have the similar construction for $S_n\times V_n'$. Clearly by varying $\alpha$ we have refined the relations between  the orbital integrals appearing in $\BI$ and $\BJ$.
 In the local situation, this sliced version was utilized in \cite{Z14} to prove the existence of smooth transfer by an induction argument. Here we are exploiting the global analog, i.e., the $\SL_2(F_0)$-modularity of \eqref{eq:I g0} and its counterpart for $\BJ$.

 Another input is to impose that $\U(V)$ is compact at archimedean places, and at the same time to plug in a (weaker version of) Gaussian test function, cf. \S\ref{s:arch RTF}. This simplifies the spectra of the $\SL_2$-automorphic forms  $\BI_{\alpha}(\cdot,\Phi)$ and its counterpart on $S_n\times V_n'$, to the extent that the spectra are finite. In fact, in our case,  they lie in a finite dimensional vector space corresponding to classical holomorphic modular forms with known levels and weights.

 The two inputs allow us to deduce the  Jacquet--Rallis fundamental lemma by induction on the dimension of $V$, at least when $p\geq \dim V$. 
 
 Now that we have explained our approach to the FL, 
let us move to the AFL conjecture. We have indicated that the extra symmetry is the $\SL_2$-modularity of the kernel function, which follows from the Poisson summation formula. In the arithmetic setting, the extra symmetry is a version of the modularity of generating series of special divisors in the arithmetic Chow groups of the integral models of unitary Shimura varieties (e.g. in the recent work of Bruinier--Howard--Kudla--Rapoport--Yang \cite{BHKRY}). 

To take advantage of the modularity, we consider the semi-Lie algebraic version of the AFL conjecture, which has appeared in Mihatsch's thesis \cite[\S8]{M-Th} and in Liu's work \cite[Conj.\ 1.11]{Liu18}. In the semi-Lie algebraic version,  we consider the intersection numbers of the Kudla--Rapoport divisors (KR divisors, for short) \cite{KR-U1} and the (derived) fixed point locus of an automorphism of the RZ space. We show in  \S\ref{s:AFL} that there is an inductive structure similar to the smooth transfer and the fundamental lemma. More precisely, it is possible to reduce the special case when the KR divisor is (formally) smooth to the AFL in one-dimension lower. This is still hardly useful if we only work on the local moduli space. Therefore we introduce a global version of the fixed point locus, called ``the derived CM cycle", or ``the fat big CM cycle", being a ``thickened" variant of the ``big CM cycle" in the work of Bruinier--Kudla--Yang \cite{BKY} and  Howard \cite{Ho-kr}. The naively defined CM cycle may have dimension larger than expected. However, we note that it is a union of connected components of the fixed point locus of a Hecke correspondence (over the integral model), cf. \S\ref{ss Hk}. Therefore there is a natural derived structure on  the naive CM cycle, and the derived CM cycle has virtual dimension one, as expected. 

By the modularity of generating series of special divisors mentioned above, we obtain a modular form (with known level and weight) by taking the (arithmetic) intersection numbers (cf. \eqref{eqn AIT S}) of a fixed (derived) CM cycle with special divisors, cf. \S\ref{ss:Int mod}.  The rest is then similar to the proof of the FL conjecture. The resulting modular form is the arithmetic analog of \eqref{eq:I g0}. By induction, together with a special case of the AFL (cf. Prop. \ref{AFL DVR}), one may assume that the $\xi$-th Fourier coefficients are known if $\xi$ is prime to a certain finite set of places. The desired equality for {\em all} Fourier coefficients then follows from the modularity of the generating series and a density principle for the Fourier coefficients of holomorphic modular forms (cf. Lemma \ref{lem density}). Finally, one deduces the AFL conjecture from the global identity, together with a local constancy property of the intersection numbers on RZ spaces, cf. Theorem \ref{prop LC}.

In our approach, it is important to understand the archimedean local intersection (i.e., the values of Green functions, cf. \S\ref{s:arch ht}), and correspondingly the derivatives of the archimedean orbital integrals for Gaussian test functions (cf. \S\ref{s:arch RTF}). After subtracting the archimedean terms, the intersection numbers and derivative of orbital integrals at non-archimedean places all lie in $\BQ$-linear span of $\log p$ for a finite set of primes $p$. One can then separate the contribution from different primes by the linear independence of logarithms of prime numbers.

We have restricted the paper to the case $F_0=\BQ$ since in a few places there are missing ingredients in the literature and some of them are subtle. However, we have tried to present most of the arguments in the general totally real field case, especially  in the analytic side of RTF.

We would like to point out some earlier works related to the AFL conjecture. The author  proved the AFL for low ranks of the unitary group ($n=2$ and $3$) in \cite{Z12}. Rapoport, Terstiege and the author \cite{RTZ} proved it for arbitrary rank $n\leq p$ and \emph{minuscule} group elements. A Lie algebraic version (in the case of artinian intersection) was studied by Mihatsch in \cite{M-AFL,M-Th},  simplifying the proof and generalizing the result in \cite{Z12}. Finally, in the minuscule case, Li and Zhu in \cite{LZ} have given a simplified proof of \cite{RTZ}; recently, He, Li, and Zhu \cite{HLZ} have also removed the restriction on the residue characteristic.

During the preparation of this paper, the author learned that Beuzart-Plessis \cite{BP19} has given a purely local proof of the Jacquet--Rallis fundamental lemma for all $p$-adic fields with $p$ odd, by induction and using a more precise version (i.e., a local relative trace formula) of the compatibility between the {\em local} Weil representation (mainly the Fourier transform) and smooth transfer. It is an interesting question whether there is a purely local proof of the AFL along the line of his proof.

\subsection{Acknowledgements} We thank Chao Li, Andreas Mihatsch, Michael Rapoport, Chen Wan and the referee for their comments. An earlier version of the paper has been circulated in the ARGOS seminar in the spring 2019, and in a seminar in Morningside center Beijing in the summer 2019. The author would like to thank Michael Rapoport and Ye Tian for communicating comments from their seminars, which have helped the author improve the paper.

\subsection{Notation}\label{notation}

\subsubsection*{Notation on algebra}

\begin{itemize}
\item $\BR_+$: the set of positive real numbers. 
\item
Let $F$ be a field of character zero. For  a reductive group $H$ acting on an affine variety $X$, we say that a point $x\in X(F)$ is
\begin{itemize}
\item  $H$-semisimple if $Hx$ is Zariski closed in $X$ (when $F$ is a local field, equivalently,  $H(F)x$ is closed in $X(F)$ for the analytic topology);
\item $H$-regular if the stabilizer $H_x$ of $x$ is trivial.
\end{itemize}
And we say that $x$ is {\em regular semisimple} if it is regular and semisimple.  We denote by $X(F)_\rs$ the set of regular semisimple elements, and $[X(F)]_\rs$ the set of regular semisimple $H(F)$-orbits. We denote the categorical quotient by $X_{/\!\!/H}$  with the natural map $X\to X_{/\!\!/H}$.

\item 
For global fields, unless otherwise stated, $F$ denotes a CM number field and $F_0$ denotes its (maximal) totally real subfield of index $2$.  We denote by $a\mapsto \ov a$ the nontrivial automorphism of $F/F_0$.   Let $F_{0,+}$ (resp., $F_{0,\geq0}$) the set of totally positive (resp., semi-positive) elements in $F_0$.
\item We denote $\bH=\SL_2$ as an algebraic group over $F_0$. Denote by $B$ the Borel subgroup of upper triangular matrices, $N$ its unipotent radical.  
\item 
We use the symbols $v$ and $v_0$ to denote places of $F_0$, and $w$ and $w_0$ to denote places of $F$.  We write $F_{0,v}$ for the $v$-adic completion of $F_0$, and we set $F_v := F \otimes_{F_0} F_{0,v}$; thus $F_v$ is isomorphic to $F_{0,v} \times F_{0,v}$ or to a quadratic field extension of $F_{0,v}$ according as $v$ is split or non-split in $F$. We write $O_{F_0,v} \subset F_{0,v}$ for the ring of integers. We use analogous notation for other fields in place of $F_0$ and other finite places in place of $v$.
\item 
Unless otherwise stated, we write $\BA$, $\BA_{0}$, and $\BA_F$ for the adele rings of $\BQ$, $F_0$, and $F$, respectively.  We systematically use a subscript $f$ for the ring of finite adeles, and a superscript $p$ for the adeles away from the prime number $p$.

\item 
For an abelian scheme $A$ over a locally noetherian scheme $S$ on which the prime number $p$ is invertible, we write $\RT_p(A)$ for the $p$-adic Tate module of $A$ (regarded as a smooth $\BZ_p$-sheaf on $S$) and $\RV_p(A) := \RT_p(A) \otimes \BQ$ for the rational $p$-adic Tate module (regarded as a smooth $\BQ_p$-sheaf on $S$). When $S$ is a $\BZ_{(p)}$-scheme, we similarly write $\wh \RV^p(A)$ for the rational prime-to-$p$ Tate module of $A$.  When $S$ is a scheme in characteristic zero, we write $\wh \RV(A)$ for the full rational Tate module of $A$.
\item 
We use a superscript $\circ$ to denote the operation $-\otimes_\BZ \BQ$ on groups of homomorphisms of abelian schemes, so that for example $\Hom^\circ(A,A') := \Hom(A,A')\otimes_\BZ \BQ$.

\item All Chow groups and $K$-groups have $\BQ$-coefficients.

\item  Given a discretely valued field $L$, we denote the completion of a maximal unramified extension of it by $\breve L$.

\item  We write $1_n$ for the $n \times n$ identity matrix.  Let $
{\rm M}_{n,m}(R)$ denote the $R$-module of $n\times m$-matrices with coefficients in a ring $R$.

\item \label{Herm2Quad}
For a vector space $V$  over a field $F_0$ (of characteristic not equal to $2$), a quadratic form $\fkq: V\to F$ has an associated symmetric bilinear pairing defined by
\begin{align}\label{eq:q2bi}
\pair{x,y}=\fkq(x+y)-\fkq(x)-\fkq(y),\quad x,y\in V.
\end{align}
In particular, 
\begin{align}\label{eq:q2bi2}
\pair{x,x}=2\fkq(x).
\end{align}
For a quadratic field extension $F$ of $F_0$, an $F/F_0$-hermitian space is an $F_0$-vector space $V$ endowed with an $F_0$-linear action of $F$ and an ``$F/F_0$-hermitian form", i.e., a map $\pair{\cdot,\cdot}: V\times V\to F$ that is $F$-linear on the first factor, and conjugate-linear on the second factor. Its dimension will be the dimension as an $F$-vector space. It  induces  a symmetric bi-$F_0$-linear pairing by $(x,y)\mapsto \tr_{F/F_0}\pair{x,y}\in F_0$. In particular, the corresponding quadratic form on $V$ is 
\begin{align}\label{eq:her2q}
\fkq(x)=\pair{x,x}\in F_0.
\end{align} We will treat $V$ as an affine variety over $F_0$, and for $\xi\in F_0$ we denote by $V_\xi$ the subscheme defined by $\fkq(x)=\xi$.

\item For a $F/F_0$-hermitian space $V$ over a non-archimedean local field, and an $O_F$-lattice $\Lambda\subset V$ (of full rank), we denote by $\Lambda^\vee$ its dual lattice under the hermitian form. 

\item \label{ss:mon poly}Let $R$ be a commutative ring. We denote by $\LNSch_{/R}$ the category of locally noetherian schemes over $\Spec R$. We denote by $R[T]_{\deg=m}$ the set of monic polynomials with coefficients in $R$ of degree $m$.
\end{itemize}

\subsubsection*{Notation on automorphic forms}

\begin{itemize}
\item 
Fix the non-trivial additive character $\psi=\psi_{\BQ}\circ\tr_{F_0/\BQ}: F_0\bs\BA_{0}\to \BC^\times$ where $\psi_{\BQ}$ is the standard one and $\tr_{F_0/\BQ}:F_0\bs \BA_{0}\to\BQ\bs \BA$ is the trace map. For $\xi\in F_0$ we denote by $\psi_\xi$ the twist $\psi_\xi(x)=\psi(\xi x)$.

\item For a smooth algebraic variety $X$ over a local field $F$, we denote $\CS(X(F))$ by the space of Schwartz functions on $X(F)$.
When $F$ is non-archimedean, this is the same as the space of locally constant functions with compact support. When $F$ is archimedean, $\CS(X(F))$ consists of smooth functions $\phi$ on $X(F)$ such that, for every algebraic differential operator $D$ on $X$, the function $D\phi$ is bounded. 
Similarly, for a smooth algebraic variety $X$ over a global field $F$, we denote $\CS(X(\BA))$ by the space of Schwartz functions on $X(\BA)$.

\item $\CH=\{\tau=b+ia\in\BC\mid a>0\}$: the complex upper half plane. \item For $\xi\in\BR$ and $k\in\BZ$, the weight $k$ Whittaker function on $\SL_2(\BR)$ is defined by
\begin{align}\label{Whit}
W^{(k)}_{\xi}(h)= |a|^{k/2} e^{2\pi i \xi (b+ai)} \chi_k(\kappa_\theta),
\end{align}
where we write $h\in \SL_2(\BR)$ according to the Iwasawa decomposition 
\begin{align}\label{h infty}
h=\left(\begin{matrix} 1 & b \\
& 1
\end{matrix}\right)
\left(\begin{matrix} a^{1/2} & \\
& a^{-1/2}
\end{matrix}\right)\,\kappa_\theta,\quad a\in\BR_{+},\quad b\in \BR,
\end{align}
and
\begin{align}\label{kappa in SO2}
\kappa(\theta)=\left(\begin{matrix}\cos\theta&\sin\theta\\- \sin\theta&
\cos\theta\end{matrix}\right)\in \SO(2,\BR).
\end{align}
Here the weight $k$-character of $\SO(2,\BR)$, for $k\in\BZ$, is defined by
\begin{align}\label{chi}
\chi_k(\kappa_\theta)= e^{i k  \theta}.
\end{align} 

\item
The principal congruence subgroups of $\SL_2(\BZ)$
\begin{align*}
\Gamma(N)
=\left\{\gamma\in\SL_2(\BZ)\Big|\gamma=\left(\begin{matrix}a &b\\c&
d\end{matrix}\right)\equiv\left(\begin{matrix}1 &0\\ 0&1
\end{matrix}\right) \mod N\right\}
\end{align*}

\item $\CA_{\rm hol}(\Gamma, k)$: the space of holomorphic modular forms of level $\Gamma$, weight $k$, for $\Gamma$ where $\Gamma(N)\subset \Gamma\subset\SL_2(\BZ)$. For any subfield $L\subset\BC$, we denote by $\CA_{\rm hol}(\Gamma, k)_L$ the $L$-vector space consisting of $f\in \CA_{\rm hol}(\Gamma, k)$ whose Fourier coefficients in the $q$-expansion at the cusp $i\infty$ all lie in $L$. 
Fixing an embedding  $\ov\BQ\incl\BC$, the $\BC$-vector space $\CA_{\rm hol}(\Gamma, k)$ has  a $\ov\BQ$-structure via the $q$-expansion at the cusp $i\infty$, i.e.,  $\CA_{\rm hol}(\Gamma, k)=\CA_{\rm hol}(\Gamma, k)_{\ov\BQ}\otimes_{\ov\BQ} \BC$. For any $L$-vector space $W$, we have  an $L$-vector space 
\begin{align}\label{def A hol W}
\CA_{\rm hol}(\Gamma, k)_L\otimes_{L} W.
\end{align}
We will view this vector space as the space of formal power series in $q^{1/N}$ with coefficients in $W$ 
$$
\sum_{\xi\geq 0,\xi\in \frac{1}{N}\BZ} A_\xi q^\xi, \quad A_\xi\in W
$$
where there exist elements $f_i\in \CA_{\rm hol}(\Gamma, k)_L$ indexed by  a finite set $I$ whose $q$-expansion at the cusp $i\infty$ 
are given by $\sum_{\xi\geq 0,\xi\in \frac{1}{N}\BZ} a_\xi(f_i)q^\xi\in L[\![q^{1/N}]\!]$, and elements $w_i\in W, i\in I$, such that
$$
A_\xi=\sum_{i\in I} a_\xi(f_i) w_i,\quad \text{for all $\xi$}.
$$

\item
$\CA_{\rm hol}(\bH(\BA_{0}),K, k)$: the space of automorphic forms (with moderate growth) on $\bH(\BA_0)$,  invariant under $K\subset\bH(\BA_{0,f})$, and parallel weight $k$ under the action of $\prod_{v\in \Hom(F_0,\BR)}\SO(2,\BR)$, holomorphic (i.e., annihilated by the element $\frac{1}{2}\left(\begin{matrix}i &1\\1&
-i\end{matrix}\right)$ in the complexifed Lie algebra of $\bH(F_{0,v})\simeq\SL_2(\BR)$ for every $v\in \Hom(F_0,\BR)$).
This is a finite dimensional vector space over $\BC$, and it has a $\ov\BQ$-structure via the $q$-expansion at the cusp $i\infty$. For any subfield $L\subset\BC$ and any $L$-vector space $W$, we define  $\CA_{\rm hol}(\bH(\BA_{0}),K, k)_L$ similar to  $\CA_{\rm hol}(\Gamma, k)_L$, and 
\begin{align}\label{def A hol W tot}
\CA_{\rm hol}(\bH(\BA_{0}),K, k)_L\otimes_L W,
\end{align}
similar to $\CA_{\rm hol}(\Gamma, k)_L\otimes_L W$ as above.

\item
To a (parallel) weight $k$ function  $\phi :\bH(\BA_0)\to \BC$ and $h_f\in \bH(\BA_{0,f})$, we define $\phi_{h_f}^\flat$ to be the function:
\begin{equation}\label{phi2phi flat}
	\begin{gathered}
   \xymatrix@R=0ex{
	\phi_{h_f}^\flat\colon &  \prod_{v\mid\infty} \CH \ar[r]  &  \BC\\
		&\tau=(\tau_v)_{v\mid\infty} \ar@{|->}[r]  & |a_\infty|^{-k/2}   \phi (h_\infty,h_f)
	}
	\end{gathered},
\end{equation}
where $ h_\infty=(h_v)_{v\mid\infty}, h_v=\left(\begin{matrix} 1 & b_v \\
& 1
\end{matrix}\right)
\left(\begin{matrix} a_v^{1/2} & \\
& a_v^{-1/2}
\end{matrix}\right), \tau_v=b_v+a_v i\in \CH$ and  $|a_\infty|=\prod_{v\mid\infty}|a_v|$.
When $h_f=1$, we simply write it as $\phi^\flat$. If $\phi\in \CA_{\rm hol}(\bH(\BA_0),K, k)$, then $\phi^\flat_{h_f}\in \CA_{\rm hol}(\Gamma, k)$ where $\Gamma=h_f K h_f^{-1}\cap \bH(F_0)$.

\item For a left $N(F_0)$-invariant  continuous function $\phi:\bH(\BA_0)\to\BC$, its $\xi$-th Fourier coefficient  for $\xi\in F_0$ is defined as the function
\begin{align}\label{eq:def F coeff}
h\in\bH(\BA_0)\mapsto W_{\phi}(h)\colon=\int_{F_0\bs \BA_0}  \phi\left[\left(\begin{matrix} 1&b\\
&1
\end{matrix}\right) h\right]\psi_{-\xi}( b) db.
\end{align}
Then there is a Fourier expansion (by an absolute convergent sum): for $h\in  \bH(\BA_0)$,
\begin{align}\label{eq:def F exp}
\phi(h)=\sum_{\xi\in F_0} W_{\phi,\xi}(h).
\end{align}
\item The case $F_0=\BQ$:  the $\BC$-vector space
$\CA_{\rm exp}(\bH(\BA),K, k)$ consists of smooth functions $\phi$ on $\bH(\BA)$ with at worst exponential growth (i.e., for every $h_f\in \bH(\BA_f)$, there exists a constant $C$ such that $
|\phi(h_\infty h_f)|\leq e^{Ca}
$ when $a\to\infty$, where $h_\infty\in \SL_2(\BR)$ denotes the matrix \eqref{h infty}), invariant under $K\subset\bH(\BA_f)$ and weight $k$ under the action of $\SO(2,\BR)$, such that $\frac{1}{2}\left(\begin{matrix}i &1\\1&
-i\end{matrix}\right) \phi\in \CA_{\rm hol}(\bH(\BA),K, k-2)$.
  This is related to the space $ \CA^!_k(\rho_L^\vee)$ in \cite[Def.\ 2.8, pp.2104]{ES}, noting that the differential operator $\frac{1}{2}\left(\begin{matrix}i &1\\1&
-i\end{matrix}\right) $ is the Maass lowering operator.
This is an infinite dimensional vector space over $\BC$.

\end{itemize}

\part{Local theory}
Throughout this part, $F_0$ is a field of characteristic zero, and $F$ a quadratic \'etale $F_0$-algebra.
\section{FL and variants}\label{s:FL var}

\subsection{Group-theoretic setup}\label{s:gpthsetup}

 Let
\[
   e := (0,\dotsc,0,1) \in {\rm M}_{n,1}(F)=F^n,
\]
be a column vector, and $e^\ast \in  {\rm M}_{1,n}(F)\simeq {\rm M}_{n,1}(F)^\ast =(F^n)^\ast$ the transpose of $e$.
Consider the embedding of algebraic groups over $F$,
\begin{equation}\label{GL_n-1 emb}
	\begin{gathered}
   \xymatrix@R=0ex{
	   \GL_{n-1} \ar[r]  &  \GL_n\\
		\gamma_0 \ar@{|->}[r]  &  \diag(\gamma_0,1)
	}
	\end{gathered};
\end{equation}
this identifies $\GL_{n-1}$ with the subgroup of points $\gamma$ in $\GL_n$ such that $\gamma e = e $ and $e^\ast \gamma= e^\ast$. 

We introduce the algebraic group $G'$  over $F_0$ and its subgroups,
\begin{align*}
   G' &:= \Res_{F/F_0}(\GL_{n-1} \times \GL_n),\\
	H_1' &:= \Res_{F/F_0} \GL_{n-1},\\
	H_2' &:= \GL_{n-1}\times\GL_n.
\end{align*}
Here $H_1'$ is embedded diagonally, and $H_2'$ is embedded in the obvious way. We consider the natural right action of $H_1'\times H_2'$ on $G'$,
\[
(h_1,h_2)\cdot    \gamma  = h_1^{-1}\gamma h_2.
\]

Consider the symmetric space
\begin{equation}\label{Sn def}
    S := S_n := \{\,g\in \Res_{F/F_0}\GL_n\mid g\ov g=1_n\,\},
\end{equation}
and its tangent space at $1_n$, called ``the Lie algebra" of $S_n$, 
\begin{equation}\label{fks def}
   \fks := \fks_n := \bigl\{\, y \in \Res_{F/F_0}\M_n \bigm| y + \ov y = 0 \,\bigr\} .
\end{equation}
Set
\[
   H' :=\GL_{n-1} .
\]
Then $H'$ acts on $S_n$ and $\fks_n$ by conjugation
$$h\cdot \gamma =h^{-1} \gamma h.
$$

We also consider a variant (arising from the Fourier--Jacobi period \cite{GGP,Liu14}). Let 
\begin{equation}\label{V' def}
V'_{n-1}=F_0^{n-1}\times (F_0^{n-1})^\ast,
\end{equation}
and consider the (diagonal) action of $H'$ on the product $S_{n-1}\times V'_{n-1}$,
$$
h\cdot(\gamma, (u_1,u_2))=(h^{-1} \gamma h, ( h^{-1}u_1,u_2 h)).
$$
The action of $H'$ on the its Lie algebra $\fks_{n-1}\times
V_{n-1}'$ is defined similarly.

Next let $V^\sharp$ be an $F/F_0$-hermitian space of dimension $n \geq 2$. We fix a non-isotropic vector $u_0 \in V^\sharp$, which we call the \emph{special vector}. We denote by $V$ the orthogonal complement of $u_0$ in $V^\sharp$.  We define the algebraic group $G$ over $F_0$ and its subgroups,
\begin{align}\label{eq:def GV}
	G &:= \U(V^\sharp),\notag\\
   H &:= \U(V),\\
	G_V&:= H \times G.\notag
	\end{align}	
We have the natural action of $H\times H$ on $G_V$, and the conjugation action of $H$ on $G$. We also consider the adjoint action of $H$ on the Lie algebra $\fkg=\fku(V^\sharp)$ of $G$. When $\dim V=1 $, the Lie algebra $\fku(V)$ is denoted by $\fku(1)$, which is canonically isomorphic to $F^{-}$, the $(-1)$-eigenspace of $F$ under the Galois conjugation.

We also need the variant arising from the RTF for the Fourier--Jacobi period \cite{Liu14}): the (diagonal) action of $H=\U(V)$ on the product $\U(V)\times V$ and $\fku(V)\times V$, where the two factors are endowed with the adjoint action (on the group and the Lie algebra) and the standard action respectively.

\subsection{Orbit matching}\label{ss: orb match}

There is  a natural bijection of orbit spaces of \emph{regular semisimple} elements, 
\begin{align}\label{eq:orb mat 0} 
\xymatrix{
 \coprod_{V} \bigl[( \U(V^\sharp) (F_0)\bigr]_\rs  \ar[r]^-\sim& \bigr[S_n(F_0)\bigr]_\rs},
\end{align}
and
\begin{align}\label{eq:orb mat 1}
 \xymatrix{\coprod_{V} \bigl[(\U(V)\times V)(F_0) \bigr]_\rs	   \ar[r]^-\sim&  [(S_{n-1}\times
V'_{n-1})(F_0)]_\rs},
\end{align}
cf.  \cite{Z12} and \cite{Liu14}, where the disjoint union runs over the set of isometry classes of $F/F_0$-hermitian spaces $V$, and the larger space $V^\sharp=V\oplus F\cdot u_0$ is then determined uniquely by demanding the special vector $u_0$ to have norm one (or any fixed number in $F_0^\times$ when varying $V$). Here the left (resp. right) hand sides denote the orbits under the action of the group $\U(V)$ (resp. $\GL_{n-1}$). The bijections define a {\em matching relation} between regular semisimple orbits. In both cases, there are also similar injections for orbits on the Lie algebras:
  \begin{align}\label{eq:orb mat lie0} \xymatrix{
 \coprod_{V} \bigl[( \fku(V^\sharp) (F_0)\bigr]_\rs  \ar[r]^-\sim& \bigr[\fks_n(F_0)\bigr]_\rs},
\end{align}
and
\begin{align}\label{eq:orb mat lie1}
 \xymatrix{\coprod_{V} \bigl[(\fku(V)\times V)(F_0) \bigr]_\rs	   \ar[r]^-\sim&  [(\fks_{n-1}\times
V'_{n-1})(F_0)]_\rs}.
\end{align}

We recall how the map \eqref{eq:orb mat 1} is defined.
Choose an $F$-basis for $V$ and complete it to a basis for $V^\sharp$ by adjoining $u_0$.  This identifies $V$ with $F^{n-1}$ and $V^\sharp$ with $F^n$ in such a way that $u_0$ corresponds to the column vector $  e := (0,\dotsc,0,1)$ in $F^n$, and hence determines embeddings of groups $\U(V^\sharp) \inj \Res_{F/F_0} \GL_n$.    An element $g\in\U(V) (F_0)_\rs$ and an element $\gamma\in S_n(F_0)_\rs$ are said to \emph{match} if these two elements, when considered as elements in $ \Res_{F/F_0} \GL_n(F_0)$, are conjugate under $ \Res_{F/F_0} \GL_{n-1}$. The matching relation is independent of the choice of embeddings and induces a bijection \cite[\S2]{Z12}. Similarly, we view elements in $(S_{n-1}\times
V'_{n-1})(F_0)$ as elements in $\Res_{F/F_0} {\rm M}_{n,n}(F_0)$ by
$$
 (\gamma, (u_1,u_2))\mapsto  \left(\begin{matrix}\gamma &u_1\\ u_2&0
\end{matrix}\right).
$$
And we view elements $(g,u) \in (\U(V)\times V)(F_0)$ as elements in  $\Res_{F/F_0} {\rm M}_{n,n}(F_0)$ 
$$
 (g, u)\mapsto  \left(\begin{matrix}g &u\\ u^\ast&0
\end{matrix}\right).
$$Here we view $u\in V(F_0)$ as the corresponding element in $\Hom(V^\perp,V)$ sending $u_0\in V^\perp=F\cdot u_0$ to $u$, and  $u^\ast$ is the element in $\Hom(V,V^\perp)=\Hom(V,F \cdot u_0) $ defined by $u'\mapsto \pair{u',u} u_0$.
Then, an element $(g,u)\in(\U(V)\times V)(F_0)_\rs$ and an element $(\gamma, (u_1,u_2)) \in \left(S_{n-1}(F_0)\times
F_0^{n-1}\times (F_0^{n-1})^\ast\right)_\rs$ are said to \emph{match} if these two elements, when considered as elements in $ \Res_{F/F_0} {\rm M}_{n,n}(F_0)$, are conjugate under $ \Res_{F/F_0} \GL_{n-1}$.

Equivalently, $(g,u)\in(\U(V)\times V)(F_0)_\rs$ matches  $(\gamma, (u_1,u_2)) \in \left( S_{n-1}(F_0)\times
F_0^{n-1}\times (F_0^{n-1})^\ast\right)_\rs$  if and only if  the following invariants are equal
$$
\det(T\, {\bf 1}_{n-1}+g)=\det(T\,{\bf 1}_{n-1}+\gamma),\quad\text{and}\quad \pair{g^iu,u}= u_2 \gamma^i u_1,\quad 0\leq i\leq n-2. 
$$
Here $\det(T\,{\bf 1}_{n-1}+g)\in F[T]_{\deg=n-1}$ is the characteristic polynomial of $g$ (we remind the reader that $F[T]_{\deg=n-1}$ denotes the set of monic polynomials with coefficients in $F$ of degree $n-1$, cf.\,\S\ref{notation}). In fact, these invariants define natural identifications of the categorical quotients $(\U(V)\times V)_{/\!\!/ \U(V)}$ and $(S_{n-1}\times V_{n-1}')_{/\!\!/ \GL_{n-1}}$ with an $F_0$-subscheme of the affine space $\Res_{F/F_0}(F[T]_{\deg=n-1}\times F^{n-1})$, and we denote this $F_0$-subscheme by $\CB_{n-1}$:
\begin{align}\label{eqn:def Bn}
\xymatrix{\CB_{n-1}\ar@{^(->}[r]& \Res_{F/F_0}(F[T]_{\deg=n-1}\times F^{n-1}).}
\end{align} We refer to \cite{Z14} for the analogous case  $\U(V^\sharp)_{/\!\!/ \U(V)} \simeq S_{n /\!\!/ \GL_{n-1}}$. Similarly, the characteristic polynomial defines a natural identification of 
$\U(V)_{/\!\!/ \U(V)}$ and $S_{n-1/\!\!/ \GL_{n-1}}$ with an $F_0$-subscheme of the affine space $\Res_{F/F_0}(F[T]_{\deg=n-1})$, which will be denote by $\CA_{n-1}$:
\begin{align}\label{eqn:def An}
\xymatrix{\CA_{n-1}\ar@{^(->}[r]& \Res_{F/F_0}(F[T]_{\deg=n-1}).}
\end{align}
More precisely, $\CA_{n-1}$ is the $F_0$-scheme of  
 conjugate self-reciprocal monic polynomials $\alpha\in F[T]_{\deg=n-1}$, i.e.,
$$T^{\deg(\alpha)} \, \alpha(T^{-1})=\alpha(0)\ov \alpha(T),$$ where $\ov \alpha$ is the coefficient-wise Galois conjugate of $\alpha$ (in particular, $\alpha(0)\ov{\alpha(0)}=1$).  Moreover both $\CA$ and $\CB$ have natural integral models over $O_{F_0}$, so we will freely talk about their points over any $O_{F_0}$-algebra. 

For $\alpha\in \CA_{n-1}(F_0)$, we will denote $S_{n-1}(\alpha)$ its preimage under the natural morphism  $S_{n-1}\to \CA_{n-1}$. For $\xi\in F_0$, we will denote by $V_{n-1,\xi}'$ the subscheme of $V_{n-1}'$ defined by $u_2u_1=\xi$. We denote by $[S_{n-1}(\alpha)(F_0)]$ (resp. $[(S_{n-1}(\alpha)\times V_{n-1}')(F_0)]$ and $[(S_{n-1}(\alpha)\times V_{n-1,\xi}')(F_0)]$)  the set of $\GL_{n-1}(F_0)$-orbits in $S_{n-1}(\alpha)(F_0)$ (resp. $(S_{n-1}(\alpha)\times V_{n-1}')(F_0)$ and $(S_{n-1}(\alpha)\times V_{n-1,\xi}')(F_0)$). Similar notation applies to unitary groups.

\subsection{Orbital integral matching: smooth transfer}\label{ss: transfer}

We recall orbital integrals \cite[\S2.2]{RSZ2}. Now let $F/F_0$ be a quadratic extension of local fields of characteristic zero (the split $F=F_0\times F_0$ is similar and simpler). Let 
$$\xymatrix{\eta=\eta_{F/F_0}\colon F_0^\times\ar[r]&\{\pm 1\}}$$ 
be the quadratic character associated to $F/F_0$ by local class field theory.

To simplify the exposition we consider the non-archimedean case, though the archimedean case requires very little change.
Then there are exactly two isometry classes of $F/F_0$-hermitian spaces of dimension $n-1$, denoted by $V_0$ and $V_1$. When $F/F_0$ is unramified, we will assume that $V_0$ has a self-dual lattice. Then the orbit bijections are now
\[ \xymatrix{ \bigl[( \U(V_0^\sharp) (F_0)\bigr]_\rs 
 \coprod\bigl[( \U(V_1^\sharp) (F_0)\bigr]_\rs  \ar[r]^-\sim& \bigr[S_n(F_0)\bigr]_\rs},
\]
and
\[
 \xymatrix{ \bigl[(\U(V_0)\times V_0)(F_0) \bigr]_\rs \coprod \bigl[(\U(V_1)\times V_1)(F_0) \bigr]_\rs	   \ar[r]^-\sim&  [(S_{n-1} \times V_{n-1}')(F_0)]_\rs}.
\]

For $\gamma \in S_n(F_0)_\rs$, $f' \in \CS(S_n(F_0))$, and $s \in \BC$, we define
\begin{equation}\label{Orb(gamma,f',s)}
   \Orb(\gamma,f',s) := \int_{\GL_{n-1}(F_0)} f'(h^{-1}\gamma h) \lv \det h \rv^s \eta(h)\, dh,
\end{equation}
where $\lv\phantom{a}\rv$ denotes the normalized absolute value on $F_0$, where we set
\[
   \eta(h) := \eta(\det h).
\] We define the special values
\begin{equation}\label{Orb(gamma,f',s=0)}
   \Orb(\gamma, f') :=\omega(\gamma) \Orb(\gamma, f', 0)
	\quad\text{and}\quad
	\del(\gamma, f') := \omega(\gamma)\,\frac{d}{ds} \Big|_{s=0} \Orb({\gamma},  f',s),
\end{equation}
where the transfer factor $\omega(\gamma)$ is to be explicated below by \eqref{S transfer factor}. Here, we have included the transfer factor in the special values of the orbital integrals, different from \cite[\S2.2]{RSZ2}.

For $(\gamma, u') \in (S_{n-1} \times V_{n-1}')_\rs(F_0)$, $\Phi' \in \CS((S_{n-1} \times
V_{n-1}' )(F_0))$, and $s \in \BC$, we define
\begin{equation}\label{Orb(gamma,f',s)}
   \Orb((\gamma,u'),\Phi',s) := \int_{\GL_{n-1}(F_0)} \Phi'(h\cdot (\gamma, u')) \lv \det h \rv^s \eta(h)\, dh,
\end{equation}
and define their special values similar to \eqref{Orb(gamma,f',s=0)}, replacing the transfer factor $\omega(\gamma)$ by $\omega(\gamma,u')$ to be explicated below by \eqref{SV transfer factor}. 

On the unitary side, for $g \in \U(V^\sharp)(F_0)_\rs$ and $f \in \CS(\U(V^\sharp)(F_0))$, we define
\[
   \Orb(g,f) := \int_{\U(V)(F_0)} f(h^{-1}gh) \, dh.
\]
For
$(g,u)\in(\U(V)\times V)(F_0)_\rs$ and  $\Phi \in \CS((\U(V)\times V)(F_0))$, we define
\begin{align}\label{eq:orb U lie}
   \Orb((g,u),\Phi) := \int_{\U(V)(F_0)} \Phi(h\cdot (g,u)) \, dh.
\end{align}

Finally, we define an explicit transfer factors, cf. \cite[\S2.4]{RSZ2}. First fix an extension $\wt\eta$ of the quadratic character $\eta$ from $F_0^\times$ to $F^\times$ (not necessarily of order $2$). If $F$ is unramified, then we take the natural extension $\wt\eta(x)=(-1)^{v(x)}$. For $S_n$, we take the transfer factor
\begin{equation}
\label{S transfer factor}
   \omega(\gamma) := \wt\eta\bigl(\det(\gamma)^{-\lfloor n/2\rfloor } \det(\gamma^ie)_{0 \leq i \leq n-1}\bigr), \quad \gamma \in S_n(F_0)_\rs.
\end{equation}
For $(\gamma,u') \in \left(S_{n-1} \times
V_{n-1}'\right)(F_0)_\rs$ where $u'=(u_1,u_2)\in  V_{n-1}(F_0)=F_0^{n-1}\times (F_0^{n-1})^\ast$, we take
\begin{equation}
\label{SV transfer factor}
   \omega(\gamma,u') := \wt\eta\bigl(\det(\gamma)^{-\lfloor (n-1)/2\rfloor } \det(\gamma^iu_1)_{0 \leq i \leq n-2}\bigr) .
\end{equation}
Similarly we define transfer factors on $\fks_n$ and  $\fks_{n-1}\times V_{n-1}'$.

\begin{definition}
A function $f' \in \CS(S_n(F_0))$ and a pair of functions $(f_0,f_1) \in \CS(\U(V^\sharp_0)(F_0)) \times \CS(\U(V^\sharp_1)(F_0))$ are (smooth) \emph{transfers} of each other if for each $i \in \{0,1\}$ and each $g \in \U(V^\sharp_i)(F_0)_\rs$,
\[
   \Orb(g,f_i) =  \Orb(\gamma,f')
\]
whenever $\gamma \in S(F_0)_\rs$ matches $g$.
\end{definition}
\begin{definition}\label{def st loc}
A function $\Phi' \in \CS((S_{n-1} \times
V_{n-1})(F_0))$ and a pair of functions $(\Phi_0,\Phi_1) \in \CS((\U(V_0)\times V_0)(F_0)) \times \CS((\U(V_1)\times V_1)(F_0))$ are   (smooth) \emph{transfers} of each other if for each $i \in \{0,1\}$ and each $(g,u) \in (\U(V_i)\times V_i)(F_0)_\rs$,
\begin{align}\label{eq:def st loc}
   \Orb((g,u),\Phi_i) = \Orb((\gamma, u'),\Phi')
\end{align}
whenever $(\gamma, u')\in \left(S_{n-1} \times
V_{n-1}'\right)(F_0)_\rs$ matches $(g,u)$.
\end{definition}

The definitions made above easily extend verbatim to the setting of the full Lie algebras $\fku(V)\times V$ and $\fks_{n-1}\times V'_{n-1}$.  Finally, we remark that the definitions extend to the archimedean local field extension $F/F_0=\BC/\BR$, where one only needs to replace the pair of functions $(\Phi_0,\Phi_1)$ by a tuple of functions $\{\Phi_V\}_V$ indexed by the set of isometry classes of $F/F_0$-hermitian spaces $V$, as in \eqref{eq:orb mat 1} and \eqref{eq:orb mat lie1}.  We will not repeat the detail here.

\subsection{Review of the FL conjecture}\label{s:FL}
We review the FL conjecture, cf.~\cite{JR, Z12, RSZ2}. Let $F/F_0$ be an unramified quadratic extension of $p$-adic field for an {\em odd} prime $p$.
Assume furthermore that the special vectors $u_i \in V_i$ have norm one (or any fixed unit in $O_{F_0}$).  Then the hermitian space $V_i^\sharp$ is again split for $i=0$ and non-split for $i=1$.  We write $G_i = \U(V_i^\sharp)$, $\fkg_i = \Lie G_i$, and $H_i = \U(V_i)$. Fix a self-dual $O_F$-lattice
\[
   \Lambda_0 \subset V_0,
\]
which exists and is unique up to $H_0(F_0)$-conjugacy.  Let
\[
   \Lambda_0^\sharp := \Lambda_0 \oplus O_Fu_0 \subset V_0^\sharp,
\]
which is again self-dual.  We denote by
\[
   K_0\subset H_0(F_0)
\]
the stabilizer of $\Lambda_0$, and by
\[
   K_0^\sharp \subset G_0(F_0)
	\quad\text{and}\quad
	\fkk_0^\sharp \subset \fkg_0(F_0)
\]
the respective stabilizers of $\Lambda_0$.  Then $K_0$ and $K_0^\sharp$ are both hyperspecial maximal subgroups.

We normalize the Haar measures on the groups
\[
 \GL_{n-1}(F_0),\quad\text{and}\quad \U(V_0)(F_0)
\]
by assigning each of the respective subgroups
\[
   \GL_{n-1}(O_{F_0}),\quad \text{and}\quad K_0
\]
measure one.  

With respect to these normalizations, the Jacquet--Rallis Fundamental lemma conjecture is the following statement, cf. \cite[\S3]{RSZ2}. Note that the semi-Lie algebra version below is essentially the Fourier--Jacobi case arising from the relative trace formula of Yifeng Liu \cite{Liu14}.

\begin{conjecture}[Jacquet--Rallis Fundamental lemma conjecture]\label{FLconj}
\hfill
\begin{altenumerate}
\renewcommand{\theenumi}{\alph{enumi}}
\item\label{FLconj gp}
\textup{(The group version)} The characteristic function $\mathbf{1}_{S_n(O_{F_0})} \in \CS(S_n(F_0))$ transfers to the pair of functions $(\mathbf{1}_{K_0},0) \in \CS(G_0(F_0))\times \CS(G_1(F_0))$.

\item\label{FLconj smilie}
\textup{(The semi-Lie algebra version)} The characteristic function $\mathbf{1}_{(S_{n-1}\times  V'_{n-1})(O_{F_0}))}\in  \CS((S_{n-1} \times V'_{n-1})(F_0))$ transfers to the pair of functions $(\mathbf{1}_{K_0\times \Lambda_0},0) \in \CS((H_0\times V_0)(F_0))\times \CS((H_1\times V_1)(F_0))$.

\end{altenumerate}
\end{conjecture}
\begin{remark}\label{rem:FL Lie}
There is also a Lie algebra version: the characteristic function $\mathbf{1}_{\fks_n(O_{F_0})} \in  \CS(\fks_n(F_0))$ transfers to the pair of functions $(\mathbf{1}_{\fkk_0},0) \in \CS(\fkg_0(F_0))\times \CS(\fkg_1(F_0))$. This Lie algebra version is equivalent to the group version, at least when $p$ is odd, cf. \cite[\S2.6]{Y}.
\end{remark}
\begin{remark}
We note that the equal characteristic analog of the FL conjecture was proved by Z.~Yun for $p> n$, cf.\ \cite{Y}; J.~Gordon deduced the $p$-adic case for $p$ large, but unspecified, cf.\ \cite{Go}.   
\end{remark}

It is straightforward to check a  special case. 
\begin{proposition}
\label{FL DVR}
The semi-Lie algebra version FL holds for $(g,u)\in (G_0\times V_0)(F_0)_\rs$ when $g$ is regular semisimple (i.e. $F[g]$ is a product of fields with total degree equal to $\dim V$) and generates a maximal order $O_F[g]$ (in $F[g]$).
\end{proposition}
\begin{proof}
This is easy to check, e.g., \cite[Lem.\ 2.5.5]{Y} for the Lie algebra version; but the argument is the same for the semi-Lie algebra version.
\end{proof}

\begin{proposition}\label{prop FL L2G} 
Fix $F/F_0$.  Assume that $q\geq n$ where $q$ denotes the cardinality of the residue field of $O_{F_0}$. Then
\begin{altenumerate} 
\item\label{L2G i} In Conjecture \ref{FLconj}, part \eqref{FLconj gp} and  part \eqref{FLconj smilie} are equivalent.

\item\label{L2G ii} In Conjecture \ref{FLconj}, part \eqref{FLconj gp} for $S_{n-1}$ implies part \eqref{FLconj smilie} for  $(g,u)\in (H_0\times V_0)(F_0)_{\rs}$ where the norm of $u$ is a unit.

\end{altenumerate} 
\end{proposition}
\begin{proof}
We will prove a similar statement, namely Proposition \ref{prop AFL L2G}, for the AFL conjecture where the situation is more delicate, and we omit the argument here and only point out that the proof also works here.

\end{proof}

\section{AFL and variants}\label{s:AFL}
For section \S\ref{s:AFL} and \S\ref{s:red AFL} we let  $F$ be an unramified quadratic field extension of a $p$-adic local field $F_0$ for an odd prime $p$.

\subsection{The AFL conjecture and variants}\label{ss:AFL}
For any $n \geq 1$, we recall the construction of the Rapoport--Zink formal moduli scheme $\CN_n = \CN_{n, F/F_0}$ associated to unitary groups, cf. \cite[\S4]{RSZ2}.  For $\Spf O_{\breve F}$-schemes $S$, we consider triples $(X, \iota, \lambda)$, where
 \begin{enumerate}
 \item[$\bullet$]
 $X$ is a $p$-divisible group of absolute height $2nd$ and dimension $n$ over $S$, where  $d:=[{F_0}: \BQ_p]$, 
 \item[$\bullet$]  $\iota$ is an action of $O_{F}$ such that the induced action of $O_{F_0}$ on $\Lie X$ is via the structure morphism $O_{F_0}\to \CO_S$,  and
 \item[$\bullet$] $\lambda$ is a principal ($O_{F_0}$-relative) polarization. 
 \end{enumerate}
  Hence $(X, \iota|_{O_{F_0}})$ is a formal $O_{F_0}$-module of relative height $2n$ and dimension $n$. We require that the Rosati involution $\Ros_\lambda$   on $O_{F}$ is the non-trivial Galois automorphism in $\Gal({F}/{F_0})$, and that the \emph{Kottwitz condition} of signature $(n-1,1)$ is satisfied, i.e.
\begin{equation}\label{kottwitzcond}
   \charac \bigl(\iota(a)\mid \Lie X\bigr)=(T-a)^{n-1}(T-\ov a) \in \CO_S[T]
	\quad\text{for all}\quad
	a\in O_{F} . 
\end{equation} 
An isomorphism $(X, \iota, \lambda) \isoarrow (X', \iota', \lambda')$ between two such triples is an $O_{F}$-linear isomorphism $\varphi\colon X\isoarrow X'$ such that $\varphi^*(\lambda')=\lambda$.

Over the residue field $\ov k$ of $O_{\breve {F}}$, there is a triple $(\BX_n, \iota_{\BX_ n}, \lambda_{\BX_n})$ such that $\BX_n$ is  supersingular, unique up to $O_F$-linear quasi-isogeny compatible with the polarization. We fix such a triple which we call a {\em framing object} (for the functor $\CN_n$). Then $\CN_n$ (pro-)represents the functor over $\Spf O_{\breve F}$ that associates to each $S$ the set of isomorphism classes of quadruples $(X, \iota, \lambda, \rho)$ over $S$, where the final entry is an $O_F$-linear quasi-isogeny of height zero defined over the special fiber,
\[
   \rho \colon X\times_S\ov S \to \BX_n \times_{\Spec \ov k} \ov S,
\]
such that $\rho^*((\lambda_{\BX_n})_{\ov S}) = \lambda_{\ov S}$. Here $\rho$ is called a {\em framing}. The formal scheme $\CN_n$ is  smooth over $\Spf O_{\breve {F}}$ of relative dimension $n-1$.

For $n \geq 2$, define the product $\CN_{n-1, n}:=\CN_{n-1}\times_{\Spf O_{\breve {F}}}\CN_n $. It is a (locally Noetherian) formal scheme of (formal) dimension $2(n-1)$, formally smooth over $\Spf O_{\breve {F}}$.

When $n=1$, we have 
the (unique up to isomorphism) formal $O_F$-module $\BE$ (with signature $(1,0)$) over $\ov k$ and its canonical lift $\CE$ over $O_{\breve F}$, as well as the ``conjugate'' objects $\ov\BE$ and $\ov\CE$ (with signature $(0,1)$).  For $n \geq 2$, there is a natural closed embedding of formal schemes
\begin{equation}\label{eqn:delta CN}
   \delta=\delta_{\CN_{n-1}}\colon 
	\xymatrix@R=0ex{
	   \CN_{n-1} \ar[r]  &  \CN_n\\
		(X, \iota, \lambda, \rho) \ar@{|->}[r]  &  \bigl(X\times {\CE}, \iota\times \iota_{\CE}, \lambda\times \lambda_{\CE}, \rho\times\rho_{\CE}\bigr),
		}
\end{equation}
where we set $\BX_1 = \ov\BE$ and inductively take
\begin{equation}\label{BX_n unram}
   \BX_n = \BX_{n-1} \times \BE
\end{equation}
as the framing object for $\CN_n$.  Let 
\begin{equation}\label{eqn:Delta CN}
   \Delta_{\CN_{n-1}} \colon \CN_{n-1} \xra{(\id_{\CN_{n-1}},\delta)} \CN_{n-1}\times_{\Spf O_{\breve F}}\CN_n = \CN_{n-1,n}
\end{equation}
be the graph morphism of $\delta$.  Then
\begin{equation}\label{eqn:Delta}
   \Delta := \Delta_{\CN_{n-1}}(\CN_{n-1})
\end{equation}
is a closed formal subscheme of half the formal dimension of $\CN_{n-1, n}$. Note that
\begin{equation}\label{Aut cong U}
   \Aut^\circ (\BX_n,\iota_{\BX_n},\lambda_{\BX_n}) \cong \U\bigl(\BV_n\bigr)(F_0), 
\end{equation}
where the left-hand side is the group of quasi-isogenies of the framing object $(\BX_n,\iota_{\BX_n},\lambda_{\BX_n})$, and $$
\BV_n\colon=\Hom^\circ_{O_F}(\BE, \BX_n)
$$
is the hermitian space on which the hermitian form is induced by the principle polarizations on $\BX_n$ and $\BE$. Note that $\BV_n$ does not contain a self-dual lattice.

More concretely
$$
\U\bigl(\BV_n\bigr)(F_0)=\left\{g\in \End_F(\BV_n)\mid gg^\ast=\id\right\}.
$$
Here we denote by $g^\ast=\Ros_{\lambda_{\BX_n}}(g)$ the Rosati involution. 
Then the group $\U\bigl(\BV_n\bigr)(F_0)$ acts naturally on $\CN_n$ by acting on the framing: 
$$g\cdot (X,\iota,\lambda,\rho) = (X,\iota,\lambda, g \circ \rho).
$$  Furthermore $\BV_n$ contains a natural special vector $u_0$ given by the inclusion of $\BE$ in $\BX_n = \BX_{n-1} \times \BE$ via the second factor. The norm of $u_0$ is $1$.  Then $\BV_n$ is a non-split hermitian space of dimension $n$. Therefore, in the setting of \S\ref{ss: transfer}, we can choose an identifications  $V_1^\sharp=\BV_n$ and $V_1=\BV_{n-1}$ compatible with the natural inclusions on both sides. Hence we obtain an action of $H_1(F_0)$ on $\CN_{n-1}$, of $G_1(F_0)$ on $\CN_n$, and of $G_{V_1}(F_0)=(\U(\BV_{n-1})\times \U(\BV_{n}))(F_0)$ on $\CN_{n-1, n}$, cf. \eqref{eq:def GV}; and furthermore the maps $\delta_{\CN_{n-1}}$ and $\Delta_{\CN_{n-1}}$ are equivariant with respect to the respective embeddings $H_1(F_0) \inj G_1(F_0)$ and $H_1({F_0})\hookrightarrow G_{V_1}({F_0})$.

 For  $g \in G_{V_1}(F_0)_\rs$, we denote by $\Int(g)$ the intersection product on $\CN_{n-1, n}$ of $\Delta$ with its translate $g\Delta$, defined through the derived tensor product of the structure sheaves, cf. \eqref{eqn:chi FG},
\begin{equation}\label{defintprod}
   \Int(g) := \la \Delta, g\Delta\ra_{\CN_{n-1, n}} := \chi({\CN_{n-1, n}},  \CO_\Delta\Ltimes\CO_{g\Delta}) . 
\end{equation}
We similarly define $\Int(g)$ for $g \in G_1(F_0)_\rs$,
\begin{equation}\label{defint G}
   \Int(g) := \bigl\la \Delta, (1 \times g)\Delta \bigr\ra_{\CN_{n-1,n}}.
\end{equation}
In both cases, when $g$ is regular semisimple, the right-hand side of this definition is finite since the (formal) schematic intersection $\Delta\cap g\Delta$ is a proper {\em scheme} over $\Spf O_{\breve{F}}$. We refer to the appendix \ref{s:appB} for the terminology regarding various $K$-groups of formal schems, following the work of Gillet--Soul\'e for schemes in \cite{GS87}.

Now we introduce a new variant of the above intersection number $\Int(g)$ via the Kudla--Rapoport special divisors \cite{KR-U1}.  This variant is closely related to in the AFL conjecture in the context of Fourier--Jacobi cycles in the work of Yifeng Liu \cite[Conj.\ 1.11]{Liu18}. A special case has also appeared in Mihatsch's thesis \cite[\S8]{M-Th}.

Recall from \cite{KR-U1}, for every non-zero $u\in \BV_n$, Kudla and Rapoport have defined a special divisor $\CZ(u)$ in $\CN_n$. This is the locus where the quasi-homomorphism $u\colon \BE\to \BX_n$ lifts to a homomorphism from $\CE$ to the universal object over $\CN_n$. By \cite[Prop.\ 3.5]{KR-U1}, $\CZ(u)$ is a relative divisor (or empty). Then the morphism $\delta$ in \eqref{eqn:delta CN} induces  an obvious closed embedding 
\begin{align}\label{eqn:Zu0}
\xymatrix{\CN_{n-1}\ar[r]^-\sim& \CZ(u_0)}
\end{align} for the unit norm special vector $u_0$, which is an isomorphism  by \cite[Lem.\ 5.2]{KR-U1}. It follows from the definition that if $g\in \U(\BV_n)(F_0)$, then
\begin{align}\label{act KR}
g \CZ(u)=\CZ(gu).
\end{align}

For simplicity we will write $\CN_n\times\CN_n$ for the fiber product $\CN_n\times_{\Spf O_{\breve F}}\CN_n$.
For $g\in \U(\BV_n)(F_0)$, let $\Gamma_g\subset \CN_n\times\CN_n$ be the graph of the automorphism of $ \CN_n$ induced by $g$. We define {\em the (naive) fixed point locus}, denoted  by $\CN^g_n$, as the (formal) schematic intersection (i.e., fiber product of formal schemes)
\begin{align}\label{Ng}
\CN_n^g\colon=\Gamma_g\,\cap \Delta_{\CN_n},
\end{align}
viewed as a closed formal subscheme of $\CN_n$.
We also form a {\em  derived fixed point locus}, denoted by $\LN^g_n$, i.e., the derived tensor product
\begin{align}\label{der Ng}
\LN^g_n\colon=\Gamma_g\,\jiao \Delta_{\CN_n}:=\CO_{\Gamma_g}\Ltimes_{\CO_{\CN_n\times\CN_n}}\CO_{\Delta_{\CN_n}}
\end{align}
viewed as an element in $K_0^{\CN_n^g}(\CN_n)$, cf. \S\ref{ss:G gp}.

For a pair $(g,u)\in  (\U(\BV_n)\times  \BV_n)(F_0)_\rs$, we define, cf. \eqref{eqn:chi FG},
\begin{equation}\label{def int g u}
   \Int(g,u) := \la\,\LN_n^ g, \CZ(u)\ra _{\CN_n} := \chi\left({\CN_{ n}}, \, \LN_n^ g \Ltimes_{\CO_{\CN_n}} \,\CO_{\CZ(u)}\right) . 
\end{equation}
Similar to \eqref{defintprod} and \eqref{defint G}, when $(g,u)$ is regular semisimple, $\CN^g_n\cap\CZ(u)$ is a proper {\em scheme} over $\Spf O_{\breve{F}}$ and hence the right-hand side of this definition is finite. The number $   \Int(g,u)$ depends only on its $\U(\BV_n)(F_0)$-orbit.
\begin{remark}
By the projection formula for the closed immersion $\Delta: \CN_n\to \CN_n\times\CN_n$, we obtain an equality in $K_0^{\Gamma_g\cap \Delta(\CZ(u))}(\CN_n\times\CN_n)$,
  $$
\RR\Delta_\ast(\LN_n^g \Ltimes_{\CO_{\CN_n}}\CO_{\CZ(u)})=\CO_{\Gamma_g}\Ltimes_{\CO_{\CN_n\times\CN_n}} \CO_{\Delta(\CZ(u))},
  $$
  where we have used $\RR\Delta_\ast(\CO_{\CZ(u)})=\CO_{\Delta(\CZ(u))}$ for a closed immersion. Therefore, an equivalent definition of the intersection number \eqref{def int g u} is
  \begin{equation*}
   \Int(g,u)= \chi\left({\CN_{ n}\times\CN_n}, \, \CO_{\Gamma_g}\Ltimes_{\CO_{\CN_n\times\CN_n}} \CO_{\Delta_{\CZ(u)}}\right) . 
\end{equation*}
This also appears in the AFL in the context of Fourier--Jacobi cycles in \cite{Liu18}.
\end{remark}

\begin{conjecture}[Arithmetic fundamental lemma conjecture]\label{AFLconj rs}\hfill
\begin{altenumerate}
\renewcommand{\theenumi}{\alph{enumi}}
\item
\textup{(The group version)}\label{AFL gp} 
Suppose that $\gamma\in S_n({F_0})_\rs$ matches an element $g\in  \U(\BV_n)(F_0)_\rs$. Then 
\[
\del\bigl(\gamma, \mathbf{1}_{S_n(O_{F_0})}\bigr) 
	   = -\Int(g)\cdot\log q. 
\]
\item
\textup{(The semi-Lie algebra version)}\label{AFL lie}
Suppose that $(\gamma,u')\in (S_{n}\times V_n')({F_0})_\rs$ matches an element $(g,u)\in ( \U(\BV_{n})\times  \BV_{n})(F_0)_\rs$. Then 
\[
\del\bigl((\gamma,u'), \mathbf{1}_{(S_{n}\times V_n')(O_{F_0})}\bigr) 
	   = -\Int(g,u)\cdot\log q. 
\]
\end{altenumerate}
\end{conjecture}

\begin{remark}
We refer to \cite[\S4, Conj.\ 4.1 (a)]{RSZ2} for the homogeneous group version of AFL involving the intersection numbers \eqref{defintprod} rather than \eqref{defint G}; it is equivalent to part \eqref{AFL gp} of Conjecture \eqref{AFLconj rs}.
\end{remark}
\begin{remark}
Mihatsch \cite{M-AFL} has pointed out that a naive formulation of Lie algebra version of AFL is not well behaved (unless the formal schematic intersection is artinian), unlike the case of FL (cf. Remark \ref{rem:FL Lie}). Therefore the semi-Lie algebraic version seems to be the best possible linearization of the AFL conjecture.
\end{remark}

\begin{definition}
\begin{altenumerate}
\renewcommand{\theenumi}{\alph{enumi}}
\item
A regular semisimple element $(g,u) \in (\U(V)\times  V)(F_0)$ is called strongly regular semisimple (``{\rm srs}" for short) if 
$g\in \U(V)(F_0)$ is semisimple with respect to the conjugation action of $\U(V)$. 
\item A  regular semisimple element $g \in \U(V^\sharp)(F_0)$ with respect to the conjugation action of $\U(V)$ 
 is called strongly regular semisimple (``{\rm srs}" for short) if  it is also semisimple with respect to the conjugation action of $ \U(V^\sharp)$. 
\end{altenumerate}
\end{definition}
\begin{definition}\begin{altenumerate}
\renewcommand{\theenumi}{\alph{enumi}}
\item
A regular semisimple element $(\gamma,u')\in (S_{n-1}\times V'_{n-1})(F_0)$ is called strongly regular semisimple (``{\rm srs}" for short) if 
$\gamma\in S_{n-1}(F_0)$ is semisimple with respect to the conjugation action of $\GL_{n-1,F_0}$. 
\item A regular semisimple element  $\gamma'\in S_{n}(F_0)$ with respect to the conjugation action of $\GL_{n-1,F_0}$  is called strongly regular semisimple  (``{\rm srs}" for short) if it is also semisimple with respect to the conjugation action of $\GL_{n,F_0}$.
\end{altenumerate}
\end{definition}
\begin{remark}
On the Lie algebras the notion of  ``strongly regular semisimple" has appeared in \cite{Y}.
\end{remark}

\begin{conjecture}[Arithmetic fundamental lemma conjecture for strongly regular semisimple elements]\label{AFLconj}\hfill
\begin{altenumerate}
\renewcommand{\theenumi}{\alph{enumi}}
\item
\textup{(The group version)}\label{AFL gp} 
Suppose that $\gamma\in S_n({F_0})_\srs$ matches an element $g\in  \U(\BV_n)(F_0)_\srs$. Then 
\[
\del\bigl(\gamma, \mathbf{1}_{S_n(O_{F_0})}\bigr) 
	   = -\Int(g)\cdot\log q. 
\]
\item
\textup{(The semi-Lie algebra version)}\label{AFL lie}
Suppose that $(\gamma,u')\in (S_{n-1}\times V'_{n-1})({F_0})_\srs$ matches an element $(g,u)\in ( \U(\BV_{n-1})\times  \BV_{n-1})(F_0)_\srs$. Then 
\[
\del\bigl((\gamma,u'), \mathbf{1}_{(S_{n-1}\times V'_{n-1})(O_{F_0})}\bigr) 
	   = -\Int(g,u)\cdot\log q. 
\]
\end{altenumerate}
\end{conjecture}

\subsection{A  special case of AFL}

\begin{proposition}\label{AFL DVR}
Let $p>n$.
Conjecture \ref{AFLconj} part \eqref{AFL lie} (i.e., the semi-Lie algebra version AFL) holds for $(g,u)\in ( \U(\BV_{n-1})\times  \BV_{n-1})(F_0)_\srs$ when $O_F[g]$ is a maximal order (in $F[g]$).
\end{proposition}
\begin{proof}This follows from   \cite[Cor.\ 9.9]{M-Th} (as well as the fact that the assertion holds when $n=2$).
 When $F_0=\BQ_p$ , this can also be deduced from \cite{Ho-kr}. \end{proof}

\section{Relation between the two versions of AFL}
\label{s:red AFL}
In this section, we continue to let $F$ be an unramified quadratic field extension of a $p$-adic local field $F_0$ for an odd prime $p$.

\subsection{Orbits in $\U(\BV_n)$}

We recall that the Cayley map is the rational morphism
\begin{align}\label{Cayley}
\xymatrix@R=0ex{ \fkc=\fkc_n\colon& \fku(\BV_n)  \ar[r]& \U(\BV_n)  \\
&x\ar@{|->}[r]  &-\frac{1-x}{1+x} .}
\end{align}Here $\frac{1+x}{1-x}$ denotes $(1-x)^{-1}(1+x)=(1+x)(1-x)^{-1}$ since the two factors commute.
 Its inverse is 
$$
\fkc^{-1}(g')=\frac{1+g'}{1-g'}.
$$ 

By definition $\BV_{n}=\Hom^\circ_{O_F}(\BE,\BX_n)$ and $\BX_{n}=\BX_{n-1}\times \BE$, we decompose
$$\BV_{n}=\BV_{n-1}\oplus \End^\circ_{O_F}(\BE)=\BV_{n-1}\oplus F\, u_0 .
$$
Accordingly, write $g'\in \U(\BV_{n})$ in the matrix form
\begin{align}\label{diag g'}
   \xymatrix{
	   \BX_{n-1}\times \BE \ar[rr]^-{g'=\left(\begin{matrix} h & u\\
w^\ast& d
\end{matrix}\right)} &		   &  \BX_{n-1}\times \BE} 
\end{align}
where $\ast$ denotes the map $\Hom^\circ_{O_F}(\BE, \BX_{n-1})\to \Hom^\circ_{O_F}( \BX_{n-1},\BE) $ induced by polarizations on $\BX_{n-1}$ and $\BE$, and
$$
h\in \End^\circ_{O_F}(\BX_{n-1}), \quad u, w\in \BV_{n-1},\quad d\in \End^\circ_{O_F}(\BE).
$$

\begin{lemma}\label{cayley U}Let $g'\in \U(\BV_n)$ be as in \eqref{diag g'}.
Write
\begin{align}\label{eq: x' g'}
x'=\fkc_n^{-1}(g')=\left(\begin{matrix} x & \wt u\\
-\wt u^\ast& e
\end{matrix}\right)\in  \fku(\BV_n),
\end{align}
and define
\begin{align}\label{eq: def g}
 g\colon=\fkc_{n-1}(x)\in  \U(\BV_{n-1}).
\end{align}
Then
\begin{align}\label{g u w}
\begin{cases}g=h+(1-d)^{-1}u\,w^\ast,\\
\wt u=2(1-d)^{-1}(1-g)^{-1} u,\\
\det(1-g')=(1-d)\det(1-g),\\
gw=\epsilon_d\, u,
\end{cases}
\end{align}
where we define
\begin{align}\label{eq:ep d}
\epsilon_d\colon=\frac{1-\ov d}{1-d}.
\end{align}
\end{lemma}

\begin{proof}

By definition of $\fkc_{n}^{-1}$, we expand the equality $1+g'=(1-g')x'$
$$
\left(\begin{matrix}1+ h & u\\
w^\ast&1+ d
\end{matrix}\right)=\left(\begin{matrix} 1-h & -u\\
-w^\ast& 1-d
\end{matrix}\right) \left(\begin{matrix} x & \wt u\\
-\wt u^\ast& e
\end{matrix}\right)
$$
to obtain
\begin{align*}
\begin{cases}1+h=(1-h)x+u \wt u^\ast,\\
w^\ast=-w^\ast x-(1-d)\wt u^\ast .\\
\end{cases}
\end{align*}
The second equality yields
$$
\wt u^\ast=-(1-d)^{-1}w^\ast(1+x),$$
which implies that
\begin{align}\label{eqn:u}
\wt u=-(1-\ov d)^{-1}(1-x) w.
\end{align}
Plug into the the first equality:
$$
1+h=(1-h)x-(1-d)^{-1} u w^\ast(1+x),
$$and this implies that
$$
1+h+(1-d)^{-1} u w^\ast=(1-h -(1-d)^{-1} u w^\ast)x.
$$
It follows that 
$$
g=\fkc_{n-1}(x)=h+(1-d)^{-1} u w^\ast
$$
and this proves the first equality in \eqref{g u w}.

Now note that the condition for $g'g'^\ast=1$ amounts to
\begin{align}\label{eqn gg'=1}
hh^\ast+u u^\ast=1, \quad h w+\ov d u=0, \quad w^\ast w+d\ov d=1.
\end{align}
The last equality in \eqref{g u w} now follows:
\begin{align*}
gw&=hw+(1-d)^{-1}u\,w^\ast  w
\\&=(-\ov d + (1-d\ov d)(1-d)^{-1})u
\\&=\frac{1-\ov d}{1-d} u.
\end{align*}
Now we return to \eqref{eqn:u}, noting that $1-x=-2(1-g)^{-1}g$,
$$
\wt u=2(1-\ov d)^{-1}(1-g)^{-1}g w=2(1-d)^{-1}(1-g)^{-1}u.
$$
This proves the second equality in \eqref{g u w}.

Finally, by $1-g'=\left(\begin{matrix} 1-h & -u\\
-w^\ast& 1-d
\end{matrix}\right)$ and  the first equality in \eqref{g u w}, we have 
\begin{align*}
\det(1-g')&=(1-d)\det((1-h)-(1-d)^{-1} uw^\ast)\\
&=(1-d)\det(1-g).
\end{align*}
This proves the third equality in \eqref{g u w}, and completes the proof.
\end{proof}

 We now define a rational map   by the formulas in Lemma \ref{cayley U}
\begin{align}\label{def g'2gu}
\xymatrix@R=0ex{\fkr: \U(\BV_{n}) \ar[r]&  \U(\BV_{n-1})\times  \BV_{n-1} \\
g'\ar@{|->}[r] & \left(g,\frac{u}{(1-d)\sqrt{\epsilon}}\right),}
\end{align}where $\epsilon\in O_{F_0}^\times$ is chosen such that $F=F_0[\sqrt{\epsilon}].$ 
We also define a variant
\begin{align}\label{def g'2gu1}
\xymatrix@R=0ex{\fkr^\nat: \U(\BV_{n}) \ar[r]&  \U(\BV_{n-1})\times  \BV_{n-1} \\
g'\ar@{|->}[r] & \left(g,\frac{\wt u}{\sqrt{\epsilon}}\right).}
\end{align}

Following the notation in Lemma \ref{cayley U}, let $\U(\BV_{n})^\circ$ be the open sub-variety of  $\U(\BV_{n})$ defined by 
$$
1-d\neq 0,\quad \text{and}\quad \det(1-g')\neq 0.
$$
Let $\left(\U(\BV_{n-1})\times  \BV_{n-1}\times\fku(1)\right)^\circ$ be the open sub-variety of  $\U(\BV_{n-1})\times  \BV_{n-1}\times\fku(1)$ defined by 
$$
 \det(1-g)\neq 0, \quad \text{and}\quad \det\left(1+x'\right)\neq 0.
$$
Here $(g,\wt u,e)\in \U(\BV_{n-1})\times  \BV_{n-1}\times\fku(1)$ and $x'$ is as in \eqref{eq: x' g'} where $x=\fkc_{n-1}^{-1}(g)$.
\begin{lemma}\label{lem:inv cay}
The map $\fkr$ together with $e\in \fku(1)$ (cf. \eqref{eq: x' g'}) induce an isomorphism, equivariant under the action of $\U(\BV_{n-1})$,
\begin{align*}
\xymatrix@R=0ex{\wt \fkr=(\fkr,e)\colon \U(\BV_{n})^\circ\ar[r]^-{\sim}&  \left(\U(\BV_{n-1})\times  \BV_{n-1}\times \fku(1)\right)^\circ
\\g'\ar@{|->}[r] & (\fkr(g'), e).}
\end{align*}
The same holds if we replace $\fkr$ by $\fkr^\nat$.
\end{lemma}

\begin{proof}
By \eqref{g u w} we have $$
\det(1-g')=(1-d)\det(1-g),
$$
and by $1-d\neq 0$, it follows that $\det(1-g)\neq 0$. Then the map $x\mapsto \fkc(x)$ is well defined since $1-g=\frac{1}{1+x}$. Therefore the rational map $\wt \fkr=(\fkr,e)$ is defined on $\U(\BV_{n})^\circ$

To reverse the map $\wt \fkr$, let $(g,u,e)\in \left(\U(\BV_{n-1})\times  \BV_{n-1}\times \fku(1)\right)^\circ$. First we send $g$ to $\fkc^{-1}(g)=x$ (this is defined since $\det(1-g)\neq 0$). Then we define $\wt u$  by $\wt u=2\sqrt{\epsilon} (1-g)^{-1}u$, cf. \eqref{g u w} and \eqref{def g'2gu}.   Finally, we apply Cayley map $\fkc$ \eqref{Cayley} to $\left(\begin{matrix} x & \wt u\\
-\wt u^\ast& e
\end{matrix}\right)$ to obtain $g'$ (the Cayley map is well-defined by the second condition $\det\left(1+x'\right)\neq 0$ when defining $ \left(\U(\BV_{n-1})\times  \BV_{n-1}\times \fku(1)\right)^\circ$).  It is easy to see that the composition of above maps is defined on $ \left(\U(\BV_{n-1})\times  \BV_{n-1}\times \fku(1)\right)^\circ$ and defines an inverse to the rational map $\wt \fkr$. The desired assertion for $\wt \fkr$ follows. It is easy to see the assertion for $\fkr^\nat$.

\end{proof}

  We may apply the same construction to $\xi g'$ for $\xi\in F^1=\ker(\Nm:F^\times\to F_0^\times)$:
\begin{align}\label{def g'2gu2}
\xymatrix@R=0ex{\fkr_\xi: \U(\BV_{n}) \ar[r]&  \U(\BV_{n-1})\times  \BV_{n-1} \\
g'\ar@{|->}[r] &\fkr(\xi g')}.
\end{align}
 We define the variant $\fkr^\nat_\xi$ similar to \eqref{def g'2gu1}.

 \begin{lemma}\label{lem ss U}
 \hfill
 \begin{altenumerate} 
\item  An element  $g'\in \U(\BV_{n})^\circ(F_0)$ is regular semisimple (with respect to the conjugation action of $\U(\BV_{n-1})$ for $\BV_n=\BV_{n-1}\oplus F\, u_0$) if and only if $\fkr(g')=(g,u)$ is  regular semisimple as an element in $(\U(\BV_{n-1})\times \BV_{n-1})(F_0)$.
\item Let  $g'\in \U(\BV_{n})^\circ(F_0)_{\srs}$. Then, for all but finitely many $\xi\in F^1$, the element $\xi g'\in \U(\BV_{n})^\circ(F_0)$ and $\fkr_\xi(g')\in\left(\U(\BV_{n-1})\times \BV_{n-1}\right) (F_0)_{\srs}$. 
\item Let  $(g,u)\in (\U(\BV_{n-1})\times \BV_{n-1})(F_0)_{\srs}$. Then, for all but finitely many $e\in \fku(1)$, the element $(g,u,e)\in \left(\U(\BV_{n-1})\times  \BV_{n-1}\times \fku(1)\right)^\circ(F_0)$ and $ \wt\fkr^{-1}(g,u,e)\in  \U(\BV_{n})^\circ(F_0)_\srs$.
\end{altenumerate}

 \end{lemma}
 \begin{proof}
 The regular semi-simplicity for $g'\in \U(\BV_{n})(F_0)$ is equivalent to the vectors $$\{g'^i \,u_0\mid  0\leq i\leq n-1\}$$ being a basis of $ \BV_{n}$ (as an $F$-vector space). By the decomposition \eqref{diag g'}, this is  equivalent to $\{h^i u,  0\leq i\leq n-2\}$ being a basis of $ \BV_{n-1}$. By the first equality in \eqref{g u w}, we can show inductively that, for all $ 1\leq i\leq n-2$,  $g^i u-h^iu$ lies in the span of $u,hu,\cdots, h^{i-1}u$. This proves part (i).

Let $P(\lambda)=\det(\lambda+h)$ be the characteristic polynomial of $h$, and let 
$$
Q(\lambda)=\det(\lambda+h)\cdot w^\ast (\lambda+h)^{-1} u,
$$
which is a polynomial in $\lambda$ of degree $n-2$. Then the characteristic polynomial of $g'$ can be written as
 \begin{align}\label{det g'}
 \det(\lambda+g')=(\lambda+d)P(\lambda)-Q(\lambda)  .
 \end{align}
Since $g'\in \U(\BV_{n})^\circ(F_0)_{\srs}$ (particularly, regular semisimple relative to the $\U(\BV_{n})$-conjugation action), this polynomial in $\lambda$ has only simple roots.

Let $\fkr_\xi(g')=(g_\xi,u_\xi)$ and
now we study how the characteristic polynomial of $g_\xi$ (or equivalently, of $\xi^{-1}g_\xi$) depends on $\xi$ . By the first equality in \eqref{g u w}, 
\begin{align*}
\det(\lambda+\xi^{-1}g_\xi)&=\det\left(\lambda+h+\frac{\xi}{1-d\xi} u w^\ast\right).
\end{align*}
Set $$
t=\frac{\xi}{1-d\xi} .
$$
Then
\begin{align*}
\det(\lambda+\xi^{-1}g_\xi)&=\det(\lambda+h)\det\left (1+t \, u w^\ast (\lambda+h)^{-1}\right)\\
&=\det(\lambda+h)\left(1+ t\,   w^\ast (\lambda+h)^{-1} u\right)\\
&=\det(\lambda+h)+t \, \det(\lambda+h)\,w^\ast (\lambda+h)^{-1} u
\\&=P(\lambda)+t \,Q(\lambda).
\end{align*}
Here in the second equality we have used the fact that $u w^\ast\in \End(\BV_{n-1})$ is of rank at most one.

Let $R(\xi)$ be the GCD of  $P(\lambda)$ and $Q(\lambda)$.  By the semi-simplicity  of $g'$, the polynomial $R(\lambda)$  is multiplicity free. 
Fix an algebraic closed field $\Omega\supset F$. Since there are only finitely many  $t\in\Omega$ such that $P/R+t \,Q/R$ and $R$ have common roots, the question is reduced to the case $R=1$ (and possibly smaller $n$). Now assume $R=1$. Then $P+t \,Q\in F[t,\lambda]$ is an irreducible (over $\Omega$) polynomial in $t,\lambda$, hence defines an irreducible curve $C$ in $\bA^2_F$ (the affine plane in $t,\lambda$), and $t$ defines a non-constant rational morphism to the projective line $ C\to \bP^1_F$. The polynomial $P+t \,Q$ has a repeated root precisely when the rational morphism is ramified at $t$. Hence there are only finitely many such $t\in \Omega$.  This proves part (ii).

Part (iii) is proved similarly to part (ii).

 \end{proof}
 
 \subsection{ Reduction of the intersection numbers}

 We  recall from \eqref{eqn:delta CN} that $\delta:\CN_{n-1}\to \CN_n$ is the embedding whose image is the special divisor $\CZ(u_0)$ for a unit $u_0\in \End^\circ_{O_F}(\BE)$, cf. \eqref{eqn:Zu0}. Consider $$
\xymatrix{ \CN_{n-1}\times \CN_{n-1} \ar[rr]^-{(\delta,\,\delta)}&  &  \CN_n\times \CN_n}
$$ and let $\xymatrix{ \pi_2: \CN_{n-1}\times\CN_{n-1}\ar[r]& \CN_{n-1}}$ be the projection to the second factor. We have the following pull-back formula for the graph of an automorphism.
\begin{lemma}\label{lem pullback}
 Let $g'\in \U(\BV_{n})^\circ (F_0)$ be such that $1-d\in O^\times_F$, and let $(g,u)=\fkr(g')\in (\U(\BV_{n-1})\times \BV_{n-1})(F_0)$. Then
 \begin{align}
 \label{eqn:int step1}
(\delta,\,\delta)^\ast\Gamma_{g'}\simeq \Gamma_{g} \cap\,\pi_2^\ast\CZ(u),
\end{align}
where $(\delta,\,\delta)^\ast$ is the naive pull-back, i.e., the fiber product
\[
   \xymatrix{  (\delta,\,\delta)^\ast\Gamma_{g'} \ar[r] \ar[d] \ar@{}[rd]|*{\square}  &  \Gamma_{g' }\ar[d]\\
 \CN_{n-1}\times \CN_{n-1} \ar[r]^-{(\delta,\,\delta)}  &  \CN_n\times \CN_n
   }.
\]

 Moreover, if $u$ is non-zero, then 
 \begin{align}
 \label{eqn: Ltimes=times}
 \CO_{ \CN_{n-1}\times \CN_{n-1}}\Ltimes_{\CO_{\CN_n\times \CN_n}}\CO_{\Gamma_{g'}}=\CO_{  \Gamma_{g} \cap\,\pi_2^\ast\CZ(u)}= \CO_{\Gamma_{g }}\Ltimes \CO_{\pi_2^\ast \CZ(u)},
\end{align}
as elements in $K_0'(\Gamma_{g} \cap\,\pi_2^\ast\CZ(u))$.

\end{lemma}
\begin{remark}
By \eqref{g u w}, we have $gw=\epsilon_d u$. Since $d\neq 1$,  $\epsilon_d=\frac{1-\ov d}{1-d} $ is a unit in $O_F$, and hence we may replace $\pi_2^\ast \CZ(u)$ by $\pi_1^\ast \CZ(w)$ in the above statements.

\end{remark}
\begin{proof}
We prove the natural map on $S$-points are the identity map.
Let $(X_1,X_2)$ be an $S$-point of  $\CN_{n-1}\times_{\Spf O_{\breve F}}\CN_{n-1}$, and let $X'_i=X_i\times \CE$ (in the notation we have omitted $S$ and the obvious additional structure $\iota,\lambda$ etc.).

We start from $(X_1,X_2)$ on the graph $\Gamma_{g'}$, i.e., there exists (uniquely)  $\varphi':X'_1\to X'_2$ lifting $g'$. Write $\varphi'$ in the matrix form
\[
   \xymatrix{
	   X_1\times \CE \ar[rr]^-{\varphi'=\left(\begin{matrix} \varphi & \psi\\
\psi'^\ast& d
\end{matrix}\right)} &		   & X_2\times \CE} 
\]
which lifts the diagram \eqref{diag g'}.
We then need to construct elements in $\Gamma_{g }\cap\,\pi_2^\ast \CZ(u)$. The subtle point is that $X_1$ and $X_2$ are different, whereas  the  $\BX_{n}$ in the target and the source in the map $g'$  of \eqref{diag g'} are (unfortunately) identified.

 First we  have $X_2\in\CZ(u)$ (note that the $u$ in $\fkr(g')=(g,u)$ differs from the $u$ in  \eqref{diag g'} only by a unit $(1-d)\sqrt{\epsilon}$, hence we ignore the difference in this proof). Consider the homomorphism $$
\xymatrix{\wt\varphi\colon=\varphi+\frac{ \psi\psi'^\ast}{1-d}\colon X_1\ar[r]&X_2}.
$$  
This is a lifting of $g\in \U(\BV_n)$   by Lemma \ref{cayley U}, hence we have constructed $(X_1,X_2)$ on $ \Gamma_{g} \cap\,\pi_2^\ast\CZ(u)$. Again by Lemma \ref{cayley U}, $\psi'$ lifting $\epsilon_d g^{-1}u$ (and $\epsilon_d=\frac{1-\ov d}{1-d}$ is a unit), hence 
 $$\psi'=\epsilon_d \,\wt\varphi^{-1} \psi= \epsilon_d \,\wt\varphi^{\ast} \psi
 $$ 
 can be recovered from $\wt\varphi$ and $\psi$.  The desired isomorphism follows.

Now we prove the second part of the lemma. We assume that $u$ is non-zero. If $\CZ(u)$ is empty, then clearly both sides vanish. Now we assume that $\CZ(u)$ is a (non-empty) relative divisor. Now note that the dimension of the intersection is as expected. Since both $\Gamma_{g'}$ and  $\CN_{n-1}\times \CN_{n-1}$ are local complete intersection in the ambient  $\CN_{n}\times \CN_{n}$, Lemma \ref{proper int} shows that higher ${\rm Tor}$ all vanish. This proves the first equality in \eqref{eqn: Ltimes=times}; the second one is proved similarly.  

\end{proof}

Recall  from \eqref{eqn:Delta} that $\Delta$ is the image of the closed embedding $\Delta_{\CN_{n-1}}\colon\CN_{n-1}\to \CN_{n-1,n}=\CN_{n-1}\times\CN_n$, cf.  \eqref{eqn:Delta CN}.
\begin{proposition}\label{prop: int g'=gu}
 Let $g'\in \U(\BV_{n})^\circ(F_0) $ be such that $1-d\in O^\times_F$, and let $(g,u)=\fkr(g')\in (\U(\BV_{n-1})\times \BV_{n-1})(F_0)$.  Assume further that the vector $u\neq 0$ in $\BV_{n-1}$. Then
 \[
   \Delta \jiao (\id\times g')  \Delta =\LN_{n-1}^g    \jiao \CZ(u)
   \]
   as elements in $K_0'(\CN_{n-1}^g    \cap \CZ(u) )$. In particular, if $g'$ is regular semisimple (hence so is $(g,u)$ by Lemma \ref{lem ss U} (i)), then
$$
\Int(g')=\Int(g,u).
$$
\end{proposition}
\begin{proof}
Consider the following two  cartesian squares, where we have applied  Lemma \ref{lem pullback} \eqref{eqn:int step1} to the middle term in the top row, 
\[
   \xymatrix{\CN_{n-1}^g    \cap  \CZ(u) \ar[r] \ar[d] \ar@{}[rd]|*{\square}  &  \Gamma_{g }\cap\,\pi_2^\ast \CZ(u) \ar[r] \ar[d] \ar@{}[rd]|*{\square}  &  \Gamma_{g' }\ar[d]\\
      \CN_{n-1} \ar[r]^-{\Delta}  &  \CN_{n-1}\times \CN_{n-1} \ar[r]^-{(\delta,\,\delta)}  &  \CN_n\times \CN_n
   }.
\]
We obtain equalities as elements in $K_0'(\CN_{n-1}^g    \cap \CZ(u) )$,
\begin{eqnarray*}
 &&\CO_{  \CN_{n-1}}\Ltimes_{\CO_{\CN_n\times \CN_n}}\CO_{  \Gamma_{g'}} \\
 & =& \CO_{  \CN_{n-1}}\Ltimes_{  \CO_{\CN_{n-1}\times \CN_{n-1}} }  (\CO_{\CN_{n-1}\times \CN_{n-1}}\Ltimes_{ \CO_{\CN_{n}\times \CN_{n}}} \CO_{  \Gamma_{g'}}  )  \\
  &=& \CO_{  \CN_{n-1}}\Ltimes_{  \CO_{\CN_{n-1}\times \CN_{n-1}} }  \CO_{  \Gamma_{g }\cap\,\pi_2^\ast \CZ(u)}  \quad \quad \quad \quad \quad \quad \mbox{(Lemma \ref{lem pullback} \eqref{eqn: Ltimes=times})}\\
   &=&( \CO_{  \CN_{n-1}}\Ltimes_{  \CO_{\CN_{n-1}\times \CN_{n-1}} }  \CO_{\Gamma_{g }})\Ltimes_{  \CO_{\CN_{n-1}\times \CN_{n-1}} }   \CO_{\pi_2^\ast \CZ(u) }\quad  \mbox{(Lemma \ref{lem pullback} \eqref{eqn: Ltimes=times})}\\
    &=&\LN_{n-1}^g    \Ltimes_{  \CO_{\CN_{n-1}} }   \CO_{ \CZ(u) } \quad \quad \quad \quad \quad \quad  \quad \quad \quad\mbox{(By \eqref{der Ng})}.
\end{eqnarray*}

Similarly, we have  two  cartesian squares
\[
   \xymatrix{ \Delta\cap  (1\times g')\Delta \ar[r] \ar[d] \ar@{}[rd]|*{\square}  & (1\times g')\Delta \ar[r] \ar[d] \ar@{}[rd]|*{\square}  &  \Gamma_{g' }\ar[d]\\
      \CN_{n-1} \ar[r]^-{(\id_{\CN_{n-1},\,\delta})}  &  \CN_{n-1}\times \CN_{n} \ar[r]^-{(\delta,\,\id_{\CN_n})}  &  \CN_n\times \CN_n
   },
\]
with similar equalities as elements in $K_0'( \Delta\cap  (1\times g')\Delta)=K_0'(\CN_{n-1}^g \cap \CZ(u) )$, which lead to
\[
   \Delta \jiao_{\CN_{n-1,n}} (1\times g')\Delta =\CN_{n-1}\,  \jiao_{\CN_n\times\CN_n} \Gamma_{g'}.
\]
This completes the proof.

\end{proof}

\subsection{Reduction of orbital integrals}
We use the Cayley map for $S_n$ (a rational morphism)
\begin{align}\label{CayleyS}
\xymatrix@R=0ex{ \fkc=\fkc_n :\fks_{n} \ar[r]& S_n \\
y\ar@{|->}[r]  &-\frac{1-y}{1+y} }.
\end{align}
 Its inverse is 
$$
\fkc^{-1}(\gamma)=\frac{1+\gamma}{1-\gamma}.
$$

Similar to $\U(\BV_n)$, we now write $\gamma'\in S_n$ according to the decomposition $F^n=F^{n-1}\oplus F u_0$: $$
\gamma'=\left(\begin{matrix} a& b\\
c& d
\end{matrix}\right).
$$
\begin{lemma}\label{cayley Sn}
Let
\begin{align}\label{eq: gamma2y}
y'=\fkc_n^{-1}(\gamma')=\left(\begin{matrix} y & \wt b\\
\wt c& e
\end{matrix}\right)\in  \fks_n,\quad \text{and}\quad \gamma=\fkc_{n-1}(y)\in  S_{n-1}.
\end{align}
Then
\begin{align}\label{gamma u w}
\begin{cases}\gamma=a+(1-d)^{-1} bc,\\
\wt b=2(1-d)^{-1}(1-\gamma)^{-1} b,\\
\wt c=-2c (1-d)^{-1}(1-\gamma)^{-1}, \\
\gamma \ov b=\epsilon_d \,b,
\end{cases}
\end{align}
where we recall that
$\epsilon_d=\frac{1-\ov d}{1-d}$, cf. \eqref{eq:ep d}.
\end{lemma}

\begin{proof}
Similar to the proof of \ref{cayley U}, we obtain
\begin{align*}
\begin{cases}1+a=(1-a)y+b \wt c,\\
c=-c y-(1-d)\wt c .\\
\end{cases}
\end{align*}
We then obtain
$$
\wt c=-(1-d)^{-1}c(1+y),
$$and
$$
1+a+(1-d)^{-1} bc=(1-a-(1-d)^{-1} bc)y.
$$
It follows that 
$$
\gamma=\fkc_{n-1}(y)=a+(1-d)^{-1} bc.
$$
The remaining assertions follow similarly.

\end{proof}

 We now define a rational map  by the formulas in Lemma \ref{cayley Sn}
 \begin{align}\label{def S2S 0}
\xymatrix@R=0ex{\fkr : S_{n} \ar[r]&  S_{n-1}\times  V'_{n-1} \\
\gamma'\ar@{|->}[r] & \left(\gamma,\left(\frac{\wt b}{\sqrt{\epsilon}},\frac{ \wt c}{\sqrt{\epsilon}} \cdot   (1-y^2)^{-1}\right)\right).}
\end{align}
From \eqref{gamma u w}, and  the fact that $y\in \fks_{n-1}\imp y^2\in {\rm M}_{n,n}$, it follows that the last component of $\fkr(\gamma')$ indeed lies in $V'_{n-1}=F_0^{n-1}\times (F_0^{n-1})^\ast$.
We also define a variant:
\begin{align}\label{def S2S 1}
\xymatrix@R=0ex{\fkr^\nat: S_{n} \ar[r]&  S_{n-1}\times  V'_{n-1} \\
\gamma'\ar@{|->}[r] & \left(\gamma,\left(\frac{\wt b}{\sqrt{\epsilon}},\frac{\wt c}{\sqrt{\epsilon}}\right)\right).}
\end{align}

Following the notation in Lemma \ref{cayley Sn}, let $S_n^\circ$ be the open sub-variety of  $S_{n}$ defined by 
$$
1-d\neq 0,\quad \text{and}\quad \det(1-\gamma')\neq 0.
$$
Let $\left(S_{n-1}\times  V'_{n-1}  \times \fks_1 \right)^\circ$ be the open sub-variety of  $S_{n-1}\times  V'_{n-1}  \times \fks_1$ defined by 
$$
 \det(1-\gamma)\neq 0, \quad \text{and}\quad\det\left(1+y'\right)\neq 0. 
$$
\begin{lemma}
The map $\fkr$ together with $e\in \fks_1$ (cf. \eqref{eq: gamma2y}) induce an isomorphism (between two open sub-varieties), equivariant under the action of $\GL_{n-1}$,
\begin{align*}
\xymatrix@R=0ex{\wt \fkr=(\fkr,e)\colon S_n^\circ\ar[r]^-{\sim}&  \left(S_{n-1}\times  V'_{n-1}\times \fks_1 \right)^\circ
\\ \gamma'\ar@{|->}[r] & ( \fkr(\gamma'), e).}
\end{align*}
The same holds if we replace $\fkr$ by $\fkr^\nat$.

\end{lemma}
\begin{proof}
The proof of Lemma \ref{lem:inv cay} still works, and we omit the detail.
\end{proof}

  We may apply the same construction to $\xi \gamma'$ for $\xi\in F^1=\ker(\Nm:F^\times\to F_0^\times)$:
\begin{align}\label{def S2S 2}
\xymatrix@R=0ex{\fkr_\xi: S_{n}  \ar[r]&  S_{n-1}\times  V'_{n-1}  \\
\gamma'\ar@{|->}[r] &\fkr(\xi \gamma').}
\end{align}
 We define $\fkr^\nat_\xi$ similar to \eqref{def S2S 1}.

 \begin{lemma}\label{lem ss Sn}
  \hfill
 \begin{altenumerate} 
\item An element   $\gamma'\in S_{n}^\circ(F_0)$ is regular semisimple if and only if 
$\fkr_\xi(\gamma')\in (S_{n-1}\times V'_{n-1})(F_0)$ is regular semisimple.
\item Let $\gamma'\in S_{n}^\circ(F_0)_\srs$. Then, for all but finitely many $\xi\in F^1$, the element  $\xi\gamma'\in S_{n}^\circ(F_0)$ and $\fkr_\xi(\gamma')\in (S_{n-1}\times V'_{n-1})(F_0)_{\srs}$.
\item  Let  $(\gamma,u')\in (S_{n-1}\times V'_{n-1})(F_0)_{\srs}$. Then, for all but finitely many $e\in \fks_1$, the element $(\gamma,u',e)$ lies in $ \left(S_{n-1}\times V'_{n-1}\times \fks_1\right)^\circ(F_0)$ and $ \wt\fkr^{-1}(\gamma,u',e)\in  S_n^\circ(F_0)$ is strongly regular semisimple.
\end{altenumerate} 
 \end{lemma}
 
 \begin{proof}
 The same argument as the proof of Lemma \ref{lem ss U} works here. Hence we omit the detail.
 \end{proof}

  \begin{lemma}\label{lem match}
If $\gamma'\in S_{n}(F_0)_\srs$ and $g'\in \U(\BV_{n})(F_0)_\srs$  match, then the following pairs also match (whenever they are well-defined for $\xi\in F^1$ under the rational maps):
\begin{itemize}
\item
 $\fkr_\xi^\nat(\gamma') \in (S_{n-1}\times V'_{n-1})(F_0)_\srs$ and $\fkr^\nat_\xi(g')\in (\U(\BV_{n-1})\times\BV_{n-1})(F_0)_\srs$; 
 \item  $\fkr_\xi(\gamma') \in (S_{n-1}\times V'_{n-1})(F_0)_\srs$ and $\fkr_\xi(g')\in (\U(\BV_{n-1})\times\BV_{n-1})(F_0)_\srs$.
 \end{itemize}
 \end{lemma}
 
  \begin{proof}
We retain the notation in Lemma \ref{lem ss U} and Lemma \ref{cayley Sn}. We may assume $\xi=1$. By choosing a basis of $\BV_{n-1}$ and of $\BV_n=\BV_{n-1}\oplus F \,u_0$,  we write $g'\in M_{n,n}(F)$ in matrix form, cf. the discussion on matching orbits in \S\ref{ss: orb match}.  Since the inverse Cayley map (cf. \eqref{Cayley}, \eqref{CayleyS}) preserve the matching conditions, $\fkc^{-1}(\gamma')$ and $\fkc^{-1}(g')$ also match. It follows that the two elements denoted by $e$ in their lower right corner are equal. Moreover, there exists $k\in \GL_{n-1}(F)$ such that
$$
\left(\begin{matrix} x & \wt u\\
-\wt u^\ast& e
\end{matrix}\right)=\left(\begin{matrix} k^{-1} & \\
&1
\end{matrix}\right)\left(\begin{matrix} y & \wt b\\
\wt c& e
\end{matrix}\right)\left(\begin{matrix} k & \\
&1
\end{matrix}\right),
$$
or equivalently,
 $$
x=k^{-1}\, y\, k,\quad \wt u= k^{-1}\, \wt b,\quad -\wt u^\ast=\wt c\, k.
$$
It follows that
$$g=\fkc(x)=k^{-1}\, \fkc(y)\, k= k^{-1}\gamma \,k,
$$
and hence
$$
\left(\begin{matrix} g & \frac{\wt u}{\sqrt{\epsilon}} \\
\left(\frac{\wt u}{\sqrt{\epsilon}}\right)^\ast& e
\end{matrix}\right)=\left(\begin{matrix} k^{-1} & \\
&1
\end{matrix}\right)\left(\begin{matrix} \gamma &  \frac{\wt b}{\sqrt{\epsilon}}\\
 \frac{\wt c}{\sqrt{\epsilon}}& e
\end{matrix}\right)\left(\begin{matrix} k & \\
&1
\end{matrix}\right).
$$
This proves the first part.

By Lemma \ref{cayley U} \eqref{g u w}, $\wt u=2(1-d)^{-1}(1-g)^{-1} u$, we obtain 
$$ u=2^{-1}(1-d)(1-g)\wt u=(1-d)(1+x)^{-1}\wt u.
$$We compute the invariants of $(g, u)$. For $0\leq i\leq n-1$,
\begin{align*}
u^\ast g^i u&= (1-d)(1-\ov d)\wt u^\ast (1+x^\ast)^{-1}g^i(1+x)^{-1}\wt u\\
&=(1-d)(1-\ov d)\wt u^\ast (1-x^2)^{-1}g^i\wt u,
\end{align*}
where we have used that $g$ and $x$ commute, and $x^\ast=-x$. In terms of the invariants of $(\gamma, \wt b,\wt c)$ This last quantity is equal to
\begin{align*}
u^\ast g^i u&=(1-d)(1-\ov d)\wt u^\ast (1-x^2)^{-1}g^i\wt u\\
&=-(1-d)(1-\ov d)\wt c\, k\,  (1-x^2)^{-1}g^i  \,k^{-1} \wt b\\
&=-(1-d)(1-\ov d)\wt c\,   (1-y^2)^{-1}\, \gamma^i  \, \wt b.
\end{align*}
Obviously $g$ and $\gamma$ have the same characteristic polynomial. It follows that $\fkr(g')=(g,\frac{u}{\sqrt{\epsilon}(1-d)})$ has the same set of invariants as $$ \left(\gamma,\left( \sqrt{\epsilon}^{-1}\wt b,\sqrt{\epsilon}^{-1} \wt c\cdot   (1-y^2)^{-1}\right)\right)=\fkr(\gamma').
$$ This completes the proof of the second part.

 \end{proof}

  \begin{lemma}\label{lem Orb red}
Let $\gamma'\in S_{n}(F_0)_\rs$ and $g'\in \U(\BV_{n})(F_0)_\rs$ be a matching pair, and $\xi\in F^1$. Assume that 
 \begin{align}\label{eq:assmp int}
 1- \xi d\in O_F^\times,\quad \text{and}\quad \det(1-\xi\gamma')\in O_F^\times.
 \end{align} 
Then
 \begin{align*}
  \Orb(\gamma',{\bf 1}_{S_n(O_{F_0})},s)&=\Orb\left(\fkr_\xi^\nat(\gamma'), {\bf 1}_{(S_{n-1}\times  V'_{n-1})(O_{F_0})}, s\right)
 \\&=\Orb\left(\fkr_\xi(\gamma'), {\bf 1}_{(S_{n-1}\times  V'_{n-1})(O_{F_0})}, s\right).
 \end{align*} 

 \end{lemma}
 
   \begin{proof}It suffices to prove the assertions for $\xi=1$.
 We also consider the orbital integral on the Lie algebra $\fks_n$. Since $\det(1-\gamma')\in O_F^\times$ by assumption \eqref{eq:assmp int}, and the Cayley map is equivariant under the $\GL_{n-1}(F_0)$, 
   $$
   h\cdot \fkc^{-1}(\gamma') \in \fks_{n}(O_{F_0}) \quad \text{ if and only if } \quad   h\cdot \gamma'  \in S_{n}(O_{F_0}).
   $$It follows that
   $$
   \Orb(\fkc^{-1}(\gamma'), {\bf 1}_{\fks_{n}(O_{F_0})},s)= \Orb(\gamma',{\bf 1}_{S_n(O_{F_0})},s).
   $$
  Similarly, by $\det(1+y)=(1-d)^{-1}\det(1-\gamma')$ and \eqref{eq:assmp int}, we know that $\det(1+y)\in O_F^\times$. Therefore,
  $$
   h^{-1} y h \in \fks_{n-1}(O_{F_0}) \quad \text{ if and only if } \quad   h ^{-1}\gamma h  \in S_{n-1}(O_{F_0}).
   $$
It follows that (note that $d$ and $e$ are now in $O_F$ and $\fks_1(O_{F_0})$ respectively)
   $$
   \Orb(\fkc^{-1}(\gamma'), {\bf 1}_{\fks_{n}(O_{F_0})},s)=\Orb\left(\fkr^\nat(\gamma'), {\bf 1}_{(S_{n-1}\times  V'_{n-1})(O_{F_0})}, s\right).
   $$
This proves the first equality.

We now simply denote
$$
\wt c':= \wt c \cdot  (1-y^2)^{-1}
$$
so that 
$$
\fkr (\gamma')=\left(\gamma, \left( \wt b/ \sqrt{\epsilon}, \wt c'/\sqrt{\epsilon} \right )\right).
$$
Note now that $\det(1-y^2)=\Nm\det(1+y)\in O_{F_0}^\times$ under our assumption. Therefore,
when $h^{-1}\gamma h\in S_{n-1}(O_{F_0})$,  we have
 $$
\frac{ \wt c' }{\sqrt{\epsilon}}h \in O_{F_0}^n \quad \text{ if and only if } \quad  \frac{ \wt c}{\sqrt{\epsilon}}h  \in O_{F_0}^n.
   $$
This immediately implies the second equality.   
   
 \end{proof}

 \subsection{ Relation between the two versions of AFL}

\begin{proposition}
\label{prop AFL L2G}
 Assume that $q\geq n$ where $q$ denotes the cardinality of the residue field of $O_{F_0}$. Then
\begin{altenumerate} 
\item\label{AFL i} in Conjecture \ref{AFLconj}, part \eqref{AFL gp} for $\BV_{n}$ is equivalent to part \eqref{AFL lie} for $\BV_{n-1}$.
\item\label{AFL ii} in Conjecture \ref{AFLconj}, part \eqref{AFL gp} for $S_{n-1}$ implies part \eqref{AFL lie} for $\BV_{n-1}$ and  $(g,u)\in (\U(\BV_{n-1})\times \BV_{n-1})(F_0)_\srs$ where the norm of $u$ is a unit. 

\end{altenumerate} 

\end{proposition}

\begin{remark}Similar results hold for Conjecture  \ref{AFLconj rs} for regular semisimple elements.
\end{remark}

\begin{proof}

For part (i), let $g'\in \U(\BV_{n})(F_0)_\srs$. We may assume that  $d\in O_F$  and the characteristic polynomial of both $g'$ and $g$ have integral coefficients (otherwise both sides of part \eqref{AFL gp} vanish). Since $q+1>n$, there exists $\xi\in F^1$ such that $\det(1-\xi g')\in O_F^\times$ is a unit (looking at the reduction of the characteristic polynomial modulo the uniformizer $\varpi_F$ of $O_F$).  Since both side of part \eqref{AFL gp} for $\BV_{n}$ are invariant under the substitution $g'\mapsto \xi g'$, we may just assume that $g'$ has the property that $d\in O_F$ and $\det(1- g')\in O_F^\times$. From the third equality in \eqref{g u w}  and the integrality of $\det(1-g)$ it follows that $1-d\in O_F^\times$. Now $g'  \in\U(\BV_{n})^\circ(F_0)_\srs$, so that we may apply the map $\fkr$. By  Lemma \ref{lem ss U}, we may adjust $\xi\in F^1$ within the same residue class $\!\!\!\mod \varpi_F$ such that the image $\fkr(g')=(g,u)$ lies in $(\U(\BV_{n-1})\times \BV_{n-1})(F_0)_\srs$.
Now, by Proposition \ref{prop: int g'=gu}, 
$$
\Int(g')=\Int(g,u).
$$ 

Now we consider the orbital integral.  By Lemma \ref{lem Orb red}
$$
\del(\gamma',{\bf 1}_{S_n(O_{F_0})})=\del\left(\fkr(\gamma'), {\bf 1}_{(S_{n-1}\times  V'_{n-1})(O_{F_0})}\right).
$$
Here we refer to \cite[Lem.\ 11.9]{RSZ1} for the comparison of the transfer factors. By Lemma \ref{lem match}, $\fkr(\gamma') \in (S_{n-1}\times V'_{n-1})(F_0)_\srs$ and $\fkr(g')\in (\U(\BV_{n-1})\times\BV_{n-1})(F_0)_\srs$ match. This shows that part \eqref{AFL lie} for $\BV_{n-1}$ implies part \eqref{AFL gp} for $\BV_{n}$.

For the inverse direction, we start from $(g,u)\in (\U(\BV_{n-1})\times \BV_{n-1})(F_0)_\srs$. Again it suffices to prove part \eqref{AFL lie}  when  the invariants of $(g,u)$ are all integers. By multiplying a suitable $\xi\in F^1$, we may assume $\det(1-g)\in O_F^\times$. Then $\det(1+x)\in  O_F^\times$. By Lemma \ref{lem ss U} part (iii), there exists $e\in \fku(1)(O_{F_0})$ such that $\det(1+x')\in O_F^\times$,  $(g,u,e)$ lies in $ \left(\U(\BV_{n-1})\times  \BV_{n-1}\times \fku(1)\right)^\circ$ and $g'= \wt\fkr^{-1}(g,u,e)\in  \U(\BV_{n})^\circ_\srs$. Then $\det(1-g')\in O_F^\times$ and hence $(1-d)\in O_F^\times$ (by the third equality in \eqref{g u w}), and we may therefore apply Proposition \ref{prop: int g'=gu}. Similar procedure proves the desired identity between orbital integrals.  This shows that part \eqref{AFL gp} for $\BV_{n}$ implies part \eqref{AFL lie} for $\BV_{n-1}$.

For part (ii) we note that for $g\in \U(\BV_{n-1})(F_0)_\srs$, the pair $(g,u_0)\in (\U(\BV_{n-1})\times \BV_{n-1})(F_0)_\srs$, and it is easy to see
$$
\Int(g)=\Int(g,u_0).
$$
One can show that the orbital integrals are equal easily, and we leave the detail to the reader. 

\end{proof}

\section{Local constancy of intersection numbers}\label{s:loc const}

This section is not used until \S\ref{s:pf AFL}.

\subsection{Local constancy of the function $\Int(g,\cdot)$}
\label{ss:loc con}
We recall the Bruhat--Tits stratification of the underlying reduced scheme $\CN_{n,\red}$ of $\CN_n$, following the work of Vollaard--Wedhorn \cite{VW} (see also \cite[\S2.2]{KR-U1}). The scheme $\CN_{n,\red}$ admits a stratification by Deligne--Lusztig varieties of dimensions $0,1,\dots,\lfloor \frac{n-1}{2}\rfloor$, attached to unitary groups in an odd number of variables and to Coxeter elements, with strata parametrized by the vertices of the Bruhat--Tits complex of the special unitary group for the non-split $n$-dimensional $F/F_0$-hermitian space $\BV_n$. The vertices of the  Bruhat--Tits complex are bijective to the vertex lattices in $\BV_n$ where an $O_F$-lattice (of full rank) $\Lambda
\subset\BV_n$ is called a vertex lattice if $\Lambda\subset\Lambda^\vee\subset \varpi^{-1}\Lambda$.  The parametrization of the strata by vertex lattices in $\BV_n$ is compatible with the action of the group $\U(\BV_n)$ on $\CN_{n,\red}$ (cf. \eqref{Aut cong U}) and on $\BV_n$. The type of  a vertex lattice $\Lambda$ is by definition the integer $t(\Lambda):=\dim_k \Lambda^\vee/ \Lambda$. Denote by $\CV(\Lambda)$ the corresponding (generalized) Deligne--Lusztig variety; it is smooth projective of dimension $\frac{t(\Lambda)-1}{2}$, cf. {\it loc. cit.}. Note that the type $t(\Lambda)$ is necessarily odd because the $F/F_0$-hermitian space $\BV_n$ is non-split.

\begin{lemma}\label{lem LC}
Let $n\geq 3$, and $\Lambda\subset\BV_n$ a vertex lattice of maximal type (i.e., type $2[(n-1)/2]+1$).
  Let $C\in\Ch_{1,\CV(\Lambda)}(\CN_{n,\rm red})$. Then the function 
\begin{equation*}
   \Int_C\colon 
	\xymatrix@R=0ex{
	   \BV_n \ar[r]  &  \BQ\\
		u \ar@{|->}[r]  & \chi(\CN_n, C\jiao \CZ(u))
		}
\end{equation*}
is locally constant and compactly supported. Here, even though the function is only defined for $u\neq 0$, the local constancy around $u=0$ is to be interpreted as that the function takes a constant value for all $u\neq 0$ in a neighborhood of $0\in \BV$.
\end{lemma}

\begin{proof}
The proof is essentially the same as \cite[Corollary\ 6.2.2]{LZ1}, noting that $  \Int_C$ is a linear combination of $\Int_{\CV(\Lambda')}$ for  vertex lattices $\Lambda'\supset\Lambda$ of type $3$.

\end{proof}

\begin{proposition} \label{prop C infty}
Fix a regular semisimple element $g\in \U(\BV_{n})$. Let $$
 \BV_{n,g} \colon=\{u\in \BV_n\mid (g,u) \text{ is not regular semisimple}\}.
$$
Then the function 
\begin{equation*}
   \Int(g,\cdot)\colon 
	\xymatrix@R=0ex{
	   \BV_n\setminus \BV_{n,g} \ar[r]  &  \BQ\\
		u \ar@{|->}[r]  & \Int(g,u)=\chi(\CN_n,\,\LN_n^g\jiao \CZ(u))
		}
\end{equation*}
is locally constant.
\end{proposition}
\begin{proof}
We first observe that, when $u \in\BV_n\setminus \BV_{n,g} $, the formal scheme $ \CN_n^g\cap \CZ(u)=\CN_n^g\cap \CZ(u)\cap  \CZ(gu)\cap \cdots\cap \CZ(g^{n-1}u)$ is a noetherian scheme. It follows that, since $g$ is fixed, the scheme $\CN_n^g\cap \CZ(u)$ depends only on the lattice spanned by $g^iu, i=0,1,\cdots, n-1$.  For any $u\in \BV_n\setminus \BV_{n,g}$, there is  a small open neighborhood $\CU$ of $u$ in $\BV_n\setminus \BV_{n,g}$ such that this lattice does not vary when varying $u\in \CU$. Let us fix such an open neighborhood $\CU$. Therefore, without changing the notation, we may and will simply work with the restriction of the relevant coherent sheaves to a suitable fixed noetherian open formal subscheme $\CN_n^\circ$ of $\CN_n$.

In the codimension filtration \eqref{Fil K}, the classes of $\CO_{\Gamma_g}$ and $\CO_{\Delta_{\CN_n}}$  belong to $F^{n-1} K_0^{\Gamma_g}(\CN_n\times\CN_n)$ and $F^{n-1} K_0^{\Delta_{\CN_n}}(\CN_n\times\CN_n)$  respectively. Therefore, by \eqref{cup Fil} the class $\LN^g_n$ (cf. \eqref{der Ng}) lies in the filtration $F^{2n-2}K_0^{\CN_n^g}(\CN_n\times\CN_n)=F^{n-1}K_0^{\CN_n^g}(\CN_n)$. Since $\CZ(u)$ is a Cartier divisor, the Euler--Poincar\'e characteristic $\chi(\CN_n, Z\jiao \CZ(u))$ vanishes if $Z$ is a (noetherian) zero dimensional subscheme of $\CN_n$. Therefore, it suffices to consider  $\LN_n^{g}\in {\rm Gr}^{n-1}K_0^{ \CN^g_n}(\CN_n)$ (see \eqref{Gr K} for the definition of the graded groups ${\rm Gr}^{\ast}K_0$). We may represent (the restriction to $\CN_n^\circ$ of) $\LN_n^{g}$ by a finite sum $\sum_C {\rm mult_C}\cdot [\CO_{C}]$ where  ${\rm mult_C}\in \BQ$ and all $C$ are one dimensional, formally reduced (i.e., the sheaf $\CO_C$ has trivial nilradical),  irreducible, and closed formal subschemes of $\CN_n^g$. 
Fix such a $C$. It suffices to show that the function 
\begin{equation*}
   \Int_C\colon 
	\xymatrix@R=0ex{
	  \CU\ar[r]  &  \BQ\\
		u \ar@{|->}[r]  & \chi(\CN_n,\,C\jiao \CZ(u))
		}
\end{equation*}
is locally constant.

There are the following (mutually exclusive) two cases for $C$ 
\begin{itemize}\item  $C$ is a closed formal subscheme of  $\CN_{n,\red}$.
\item   $C$ is not a closed formal subscheme of  $\CN_{n,\red}$.
\end{itemize}

For the first case, we can assume that $C\subset\CV(\Lambda)$ for some vertex lattice $\Lambda$ of maximal type. Then we have proved an even stronger statement in Lemma \ref{lem LC}.  

Now we consider the second case. We let $\wt C$ be the normalization of $C$, and $\pi:\wt C\to C$ the normalization morphism (this is a finite morphism by the excellence of $C$). It suffices to show the local constancy of $u\in \CU\mapsto \chi(\CN_n,  \pi_\ast\CO_{ \wt C}\Ltimes \CO_{\CZ(u)} )$. Note that $\CZ(u)$ is a Cartier divisor on $\CN_n$. Then $\wt C\times_{\CN_n} \CZ(u)$ has expected dimension (otherwise $C\subset \CZ(u)$; then $C\subset C\cap \CZ(u)\subset\CN_n^g\cap \CZ(u)$ and hence $C$ is a closed subscheme of $\CN_{n,\red}$). Therefore we have equalities in $K_0'(\CN_{n}^g    \cap \CZ(u) )$:
\begin{eqnarray*}
 \pi_\ast\CO_{ \wt C}\Ltimes_{\CO_{\CN_n}} \CO_{\CZ(u)}&=&  \pi_\ast\CO_{ \wt C}\otimes_{\CO_{\CN_n}} \CO_{\CZ(u)}\\
 &=& \pi_\ast \CO_{\wt C}\otimes_{\CO_{\CN^g_n}} ( \CO_{ \CN_n^g}\otimes_{\CO_{\CN_n}} \CO_{\CZ(u)}) \\
 &=&  \pi_\ast\CO_{ \wt C}\otimes_{\CO_{\CN^g_n}} \CO_{ \CN_n^g\cap \CZ(u)}.
\end{eqnarray*}
It follows that
\begin{equation}\label{eq:chi C u}
\chi(\CN_n, \pi_\ast \CO_{ \wt C}\Ltimes\CO_ {\CZ(u)})=\chi(\CN_n, \pi_\ast \CO_{ \wt C}\otimes_{\CO_{\CN^g_n}} \CO_{ \CN_n^g\cap \CZ(u)}).
\end{equation}We have seen in the beginning of the proof that the sheaf $\CO_{ \CN_n^g\cap \CZ(u)}$ does not change when varying $u$ in $\CU$. Therefore the Euler--Poincar\'e characteristic \eqref{eq:chi C u} is a constant for $u\in    \CU$.
\end{proof}

\begin{remark}\label{rem:no B3}
We could avoid \eqref{cup Fil} in the proof of Proposition \ref{prop C infty} as follows. Now we can not conclude that $\LN_{n, \sV}^{g}\in F^{n-1}K_0^{ \CN^g_{n,\sV}}(\CN_n)$. Therefore, in the two cases of the proof, the closed formal subscheme $C$ may have dimension higher than one. For the first case, Lemma \ref{lem LC} holds for any $C\in \Ch_{i,\CV(\Lambda)}(\CN_{n,\rm red})$ of arbitrary dimension $i$ (see \cite[\S6.4]{LZ1}). For the second case, we still have  \eqref{eq:chi C u} and therefore the proof above still works. 
\end{remark}

\subsection{Local constancy of the function $\Int(\cdot,\cdot)$}

\begin{lemma} \label{submer}
Fix $(g_0,u_0,e_0)\in \left(\U(\BV_{n})\times  \BV_{n}\times \fku(1)\right)^\circ$ such that $ g'=\wt\fkr^{-1}(g_0,u_0,e_0)\in  \U(\BV_{n+1})_\srs$ (cf.  Lemma \ref{lem ss U} for the notation).  Then the map (defined on some open subsets of $F_0$-varieties)
\begin{equation*}
{\rm char} (g_0,\cdot,\cdot)\colon
	\xymatrix@R=0ex{
	  \BV_{n}\times \fku(1)  \ar[r]  &\U(\BV_{n+1})_{/\!\!/\U(\BV_{n+1})} \\
		(u,e) \ar@{|->}[r]  & {\rm char\,poly}(\wt\fkr^{-1}(g_0,u,e))
		}
\end{equation*}
is submersive (i.e., the induced map on tangent spaces is surjective) at $(u_0,e_0)$. 

Here $\U(\BV_{n+1})_{/\!\!/\U(\BV_{n+1})}$ denotes the categorical quotient (with respect to the the conjugation action) and ${\rm char\,poly}$ denotes the characteristic polynomial.
\end{lemma}
\begin{proof}
The question is local on the source. Tracing the definition back to \eqref{def g'2gu} and Lemma \ref{lem:inv cay}, we may reduce the question  to the Lie algebra version: for a fixed $\left(\begin{matrix} x_0 & u_0\\
- u_0^\ast& e_0
\end{matrix}\right)\in  \fku(\BV_{n+1})_{\srs} $, the map 
\begin{equation*}
	\xymatrix@R=0ex{
	  \BV_{n}\times \fku(1)  \ar[r]  &\fku(\BV_{n+1})_{/\!\!/\U(\BV_{n+1})}		}
\end{equation*}
sending $(u,e)$ to the characteristic polynomial of $\left(\begin{matrix} x_0 & u\\
- u^\ast& e
\end{matrix}\right)\in  \fku(\BV_{n+1})$  is  submersive at $(u_0,e_0)$. 

Note that a complete set of generators of invariants relative to the $\U(\BV_{n})$-action on $\fku(\BV_{n+1})$ is given by: 
$$
{\rm char\,poly}(x), \quad e, \quad u^\ast x^j u,\quad 0\leq j\leq n-1,
$$
where $x'= \left(\begin{matrix} x & u\\
- u^\ast& e
\end{matrix}\right)\in \fku(\BV_{n+1})$
cf. \cite{Z12}. It is easy to see that an equivalent set is 
$$
{\rm char\,poly}(x), \quad {\rm char\,poly}(x').
$$

Therefore, it suffices to show the analogous map \begin{equation*}
	\xymatrix@R=0ex{
	  \BV_{n}  \ar[r]  & \prod_{i=0}^{n-1} F^{(-1)^j}	}
\end{equation*}
sending $u\in\BV_n$ to the  invariants $$
 u^\ast x^j u,\quad 0\leq j\leq n-1,
$$
  is  submersive at $u_0$. Here $ F^{(-1)^j}$ is the $(-1)^j$-eigenspace of $F$ under the Galois conjugation. Now the assertion follows from the regular semi-simplicity of $x_0$, which reduces the question to the case $n=1$, but for  the product of field extensions of $F$. This is routine and we omit the detail.

\end{proof}

\begin{theorem} \label{prop LC}
The function 
\begin{equation*}
   \Int(\cdot,\cdot)\colon 
	\xymatrix@R=0ex{
	  (\U(\BV_n)\times \BV_n)(F_0)_\srs  \ar[r]  &  \BQ\\
		(g,u) \ar@{|->}[r]  & \Int(g,u)
		}
\end{equation*}
is locally constant. Its support is compact modulo the action of $\U(\BV_n)(F_0)$.
\end{theorem}
\begin{remark}
See the forthcoming work of Mihatsch \cite{M-LC} for a different proof, which also yields the local constancy on the regular semisimple locus.
\end{remark}

\begin{proof}
We may assume that the invariants of $(g,u)$ are all integers. We now fix such a pair  $(g,u)$ and we want to show the local constancy near $(g,u)$.

First, by the argument in the proof of part (i) of Proposition \ref{prop AFL L2G}, there exists $g'= \wt\fkr^{-1}(g,u,e)\in  \U(\BV_{n+1})^\circ(F_0)_\srs$ such that 
\begin{align}\label{eqn g'=gu}
\Int(g')=\Int(g,u).
\end{align} 
In fact, by the same argument the equality holds if we replace $(g,u,e)$ by any element $(g^\sharp,u^\sharp,e^\sharp)$ near it, and $g'$ by the respective image $g'^{\sharp}$ under the map $ \wt\fkr^{-1}$.

On the other hand, we may write 
$$\Int(g')=\Int(g',u'_0),$$
 where $u'_0\in \BV_{n+1}$ is the fixed unit normed vector that induces the embedding $\CN_{n}\incl \CN_{n+1}$. 
We now apply Proposition \ref{prop C infty} to $(g', u'_0)$:
$$\Int(g',u'_0)=\Int(g',u'),$$
where $u'\in \BV_{n+1}$ is close to $u'_0$. In particular the equality holds for $u'= h u_0'$ for $h\in \U(\BV_{n+1})$  in a small neighborhood of $1$.  By the invariance under $\U(\BV_{n+1})$, we have for $u'= h u_0'$
 $$\Int(g',u')= \Int(h^{-1}g'h,u'_0).
 $$ 
 It follows that $\Int(g',u'_0)= \Int(h^{-1}g'h,u'_0)$ and hence
 \begin{align}\label{eqn u'0}
 \Int(g')=\Int(h^{-1}g'h)
\end{align} for $h\in \U(\BV_{n+1})(F_0)$  in a small neighborhood of $1$.  This shows that, as a function on the quotient $[\U(\BV_{n+1})/\!\!/\U(\BV_{n})](F_0)$, $\Int(g')$ is constant on those elements near  $g'$ and having the same characteristic polynomial (as $g'$).
  
 Now we claim that the desired local constancy near $(g,u)\in (\U(\BV_n)\times \BV_n)(F_0)_\srs$ follows from the following two properties
 \begin{enumerate}
 \item the local constancy in the $u$-variable (for a fixed $g$), by Proposition \ref{prop C infty}; 
\item the invariance  \eqref{eqn u'0} under conjugation by elements $h$ near $1\in \U(\BV_{n+1})$.
 \end{enumerate}
 
To show the claim, let $g'^{\sharp}$ be an element  in a small neighborhood of $g'$.
By Lemma \ref{submer}, there exists a neighborhood $\Omega\subset  \BV_{n}\times \fku(1)$ of $(u,e)$ such that $g'^{\sharp}$ is conjugate (by an element $h\in\U(\BV_{n+1})(F_0)$ near $1$) to $\wt\fkr^{-1}(g,u^\sharp,e^\sharp)$ for some $(u^\sharp,e^\sharp)\in \Omega$.  By the invariance  \eqref{eqn u'0}, we have 
$$
 \Int(g'^{\sharp})=\Int( \wt\fkr^{-1}(g,u^\sharp,e^\sharp)).
$$
 
By \eqref{eqn g'=gu} (and the remark following it), 
$$
\Int( \wt\fkr^{-1}(g,u^\sharp,e^\sharp))=\Int(g,u^\sharp).
$$
By Proposition \ref{prop C infty} for the local constancy  in $u$, 
$$
\Int( g,u^\sharp)=\Int(g,u).
$$
Again by \eqref{eqn g'=gu}  $
\Int( g,u)=\Int(g')
$, we obtain $ \Int(g'^{\sharp})=\Int(g')$.   The desired local constancy of 
$\Int(g,u)$ follows from \eqref{eqn g'=gu} (and the remark following it).

To show the compactness of the support modulo $\U(\BV_n)(F_0)$, it suffices to show the {\em claim}:  the support is contained in the union of  compact subsets $$K_\Lambda\times \Lambda \subset (\U(\BV_n)\times \BV_n)(F_0),$$ where $\Lambda$ runs over  all vertex lattices, and  $K_\Lambda$ is the stabilizer of $\Lambda$. Then the desired compactness follows 
from the fact that  the group $\U(\BV_n)(F_0)$ acts transitively on the set of vertex lattices of any given type $t$ (and there are only finitely many possible types $t=1,3,\cdots, 2[(n-1)/2]+1$). 
Now we show the claim. If $\Int(g,u)\neq 0$, then there exists a point $x\in \CN_{n}\left(\ov k\right)$ lying on $\CZ(u)$ and $\CN_n^g$. Let $\CV(\Lambda)$ for some vertex lattice $\Lambda$ be the smallest stratum containing the point $x$. Then $g\Lambda=\Lambda$ (otherwise the intersection $g \CV(\Lambda)\cap \CV(\Lambda)$ is non-empty and is a strictly smaller stratum), and $u\in \Lambda$ by \cite[Prop.\ 4.1]{KR-U1}. Therefore $(g,u)\in K_\Lambda\times \Lambda $ as desired.  
\end{proof}

\part{Global theory}

\section{Shimura varieties and their integral models}\label{s:int model}
In this section we recall the construction of the integral models of unitary Shimura varieties, following \cite{BHKRY,RSZ3,RSZ3.5}. In fact, rather than the full strength of {\it loc. cit.}, we only need a regular integral  model away from a suitable finite set of primes: the key is to keep those primes where the relevant hermitian space  locally contains a self-dual lattice. 

\subsection{Shimura varieties}
\label{ss:S data}

Let $F$ be a CM number field with maximal totally real subfield $F_0$ and nontrivial $F/F_0$-automorphism $a \mapsto \ov a$. Let $n$ be a positive integer. A \emph{generalized CM type of rank $n$} is a function $r\colon \Hom_\BQ(F, \ov\BQ)\to \BZ_{\geq 0}$, denoted $\varphi \mapsto r_\varphi$, such that
\begin{equation}
   r_\varphi+r_{\ov\varphi}=n \quad \text{for all}\quad \varphi.
\end{equation}
  Here $\ov\varphi$ denotes the pre-composition of $\varphi$ by the nontrivial $F/F_0$-automorphism.
When $n=1$, a generalized CM type is the same as a usual CM type (i.e., a half-system $\Phi$ of complex embeddings of $F$), via $\Phi=\{\varphi\in \Hom_\BQ(F, \ov\BQ)\mid r_\varphi=1\}$. 

Let $(V, \sform)$ be an $F/F_0$-hermitian vector space of dimension $n$. Fix a CM type $\Phi$ of $F$. Then the signatures of $V$ at the archimedean places determine a generalized CM type $r$ of rank $n$ (and vice versa), by the following recipe
\begin{equation}\label{CM type}
   \sig V_\varphi=(r_\varphi, r_{\ov\varphi}), \quad \varphi \in \Phi, \quad V_\varphi := V \otimes_{F,\varphi}\BC.
\end{equation}
Let $G^\BQ$ be the group of unitary similitudes of $(V, \sform)$,
\begin{align*}
G^\BQ &:= \bigl\{\, g \in \Res_{F_0/\BQ} \GU(V) \bigm| c(g)\in \BG_m \,\bigr\},
	\end{align*}
 considered as an algebraic group over $\BQ$ (with similitude factor in $\BG_m$), where $c$ denotes the similitude map.

Given  $\Phi$, $r$ and $V$, we define a Shimura datum $(G^\BQ, \{ h_{G^\BQ} \})$ as follows (cf. \cite[\S2.2]{RSZ3.5}).  For each $\varphi\in\Phi$, choose a $\BC$-basis of $V_\varphi$ with respect to which the matrix of $\sform$ is given by
\begin{equation}\label{diagr}
   \diag(1_{r_\varphi}, -1_{r_{\ov\varphi}}).
\end{equation}
Then $\{ h_{G^\BQ} \}$ is the $G^\BQ(\BR)$-conjugacy class of the homomorphism 
$$
h_{G^\BQ}:\Res_{\BC/\BR} \BG_m \to   G^\BQ_\BR,
$$
 defined with respect to the inclusion
\begin{equation}\label{G^Q decomp}
   G^\BQ(\BR) \subset \GL_{F \otimes \BR}(V \otimes \BR) \xra[\undertilde]{\Phi} \prod_{\varphi \in \Phi} \GL_\BC(V_\varphi),
\end{equation}
by $h_{G^\BQ}=(h_{G^\BQ,\varphi})_{\varphi\in\Phi}$ with the component  $h_{G^\BQ,\varphi}$  (
on the $\BR$-points) \[
   h_{G^\BQ,\varphi}\colon z \in \BC^\times\mapsto \diag(z \cdot 1_{r_\varphi}, \ov z \cdot 1_{r_{\ov\varphi}}).
\]
Then the reflex field $E(G^\BQ, \{ h_{G^\BQ} \})$ is the reflex field $E_r$ of $r$, which is the subfield of $\ov\BQ$ defined by
\begin{equation}
   \Gal(\ov\BQ/E_r)= \bigl\{\, \sigma\in \Gal(\ov\BQ/\BQ) \bigm| \sigma^*(r)=r \,\bigr\} .
\end{equation}

Now, in addition to the CM type $\Phi$, we also fix a distinguished element $\varphi_0 \in \Phi$.
From now on we assume that the generalized CM type $r$ is of {\em strict
fake Drinfeld type} relative to $\Phi$ and $\varphi_0$, in the sense of \cite{RSZ3.5}, i.e.,
\[
   r_\varphi=
   \begin{cases}
      n-1,  &  \text{$\varphi = \varphi_0$};\\
     \text{$n$},  &  \varphi \in \Phi \ssm \{\varphi_0\}.
   \end{cases}
\]

The first special case is when $n = 1$ and $V$ is totally positive definite, i.e., $V$ has signature $(1,0)$ at each archimedean place\footnote{Here we follow the convention of \cite{RSZ3.5}, which differs from \cite{RSZ3} where the space  $V$ is totally \emph{negative} definite.}.  In this case, we write $Z^\BQ := G^\BQ$ (a torus over \BQ) and $h_{Z^\BQ} := h_{G^\BQ}$. The reflex field of $(Z^\BQ,\{h_{Z^\BQ}\})$ is $E_\Phi$, the reflex field of $\Phi.$

For general $n$, we set \begin{equation*}
   \wt G := Z^\BQ \times_{\BG_m} G^\BQ,
\end{equation*}
where the two maps are respectively given by $\Nm_{F/F_0}$ and the similitude character.  We form a Shimura datum for $\wt G$ by 
\[
 h_{\wt G}\colon \BC^\times \xra{(h_{Z^\BQ}, h_{G^\BQ})} \wt G(\BR) .	
\]Then the reflex field $E \subset \ov\BQ$ of $(\wt G, \{h_{\wt G}\})$ is the composite $E_\Phi E_r$ (cf. \cite[\S3.2]{RSZ3.5}). In particular, the field $F$ is a subfield of $E$ via $\varphi_0$.

In  \cite[Rem. 3.2 (iii)]{RSZ3} (also \cite[\S2.3]{RSZ3.5}) the authors also defined a Shimura datum $(\Res_{F_0/\BQ} G,\{h_G\})$ where $G$ is the unitary group $G=\U(V)$ (an algebraic group over $F_0$); this gives the Shimura variety in the Gan--Gross--Prasad conjecture, cf. \cite[\S27]{GGP}. Note that, there is a natural isomorphism 
$$
\begin{gathered}
	\xymatrix@R=0ex{
      \wt G \ar[r]^-\sim &  Z^\BQ \times \Res_{F_0/\BQ}  G\\
	   (z, g) \ar@{|->}[r]  &  (z, z^{-1}g)
	}
	\end{gathered}
$$and, when $K_{\wt G}=K_{Z^\BQ}\times K_G$ is a decomposable compact open subgroup of  $\wt G(\BA_f)$, we have a product decomposition of the Shimura varieties  over $E$
\begin{equation}\label{prod shim}
\Sh_{K_{\wt G}}\bigl(\wt G, \{h_{\wt G}\}\bigr)\simeq
   \Sh_{K_{Z^\BQ}}\bigl(Z^\BQ,\{h_{Z^\BQ}\}\bigr) \times \Sh_{K_G}\bigl(\Res_{F_0/\BQ} G,\{h_G\}\bigr).
\end{equation} 

\subsection{Integral models}
\subsubsection{The auxiliary moduli problem for $Z^\BQ$}
\label{sss:M0}
We recall the moduli problem $\CM_0$ over $O_{E_\Phi}$ of \cite[\S3.2]{RSZ3}. For a scheme $S$ in $\LNSch_{/O_{E_\Phi}}$, we define $\CM_0(S)$ to be the groupoid of triples $(A_0, \iota_0, \lambda_0)$, where 
\begin{altitemize}
\item $A_0$ is an abelian scheme over $S$;
\item $\iota_0\colon O_F\to \End(A_0)$ is an $O_F$-action satisfying the Kottwitz condition: 
\begin{equation}\label{kottcondA_0}
   \charac\bigl(\iota(a)\mid\Lie A_0\bigr) = \prod_{\varphi\in\Phi}\bigl(T-\varphi(a)\bigr)
   \quad\text{for all}\quad
   a\in O_F;
\end{equation}
and
\item $\lambda_0$ is a principal polarization on $A_0$ such that the induced Rosati involution via $\iota_0$ coincides with the Galois involution on $O_F$.
\end{altitemize}
A morphism between two objects $(A_0,\iota_0,\lambda_0)$  and $(A'_0,\iota'_0,\lambda'_0)$ in this groupoid is an $O_F$-linear isomorphism $\mu_0\colon A_0\to A'_0$ under which $\lambda_0'$ pulls back to $\lambda_0$. Then the functor $\CM_0$ is represented by a Deligne--Mumford stack, finite and \'etale over $\Spec O_{E_{\Phi}}$, cf.\ \cite[Prop.\ 3.1.2]{Ho-kr}. 

The generic fiber $M_0$ of $\CM_0$ is a disjoint union of copies of the Shimura variety  $\Sh_{K^\circ_{Z^\BQ}}\bigl(Z^\BQ,\{h_{Z^\BQ}\}\bigr)$ where $K^\circ_{Z^\BQ}$ is the unique maximal compact subgroup of $Z^\BQ(\BA_f)$ (cf. \cite[Lem.\ 3.4]{RSZ3} specializing to the ideal $\fka=O_{F_0}$).  To avoid the possible emptiness of $\CM_0$, we assume that $F/F_0$ is ramified throughout this paper (cf. \cite[Rem.\ 3.5 (ii)]{RSZ3}). For our purpose, it suffices to work with a fixed copy of the Shimura variety  $\Sh_{K^\circ_{Z^\BQ}}\bigl(Z^\BQ,\{h_{Z^\BQ}\}\bigr)$ in the disjoint union in $M_0$, and by abuse of notation we will still denote it by $M_0$ and by $\CM_0$ the corresponding smooth integral model for the rest of the paper.

We also introduce a level structure for $\CM_0$.  We let $K_{Z^\BQ}=\prod_{p}K_{Z^\BQ,p}\subset Z^\BQ(\BA_f)$ be an open subgroup such that the prime-to-$\fkd$ components remain maximal. Analogous to $\CM_0$, there is a moduli functor $\CM_{0,K_{Z^\BQ}} $ with $K_{Z^\BQ}$-level structure, whose generic fiber is $\Sh_{K_{Z^\BQ}}\bigl(Z^\BQ,\{h_{Z^\BQ}\}\bigr)$. The construction is not important to this paper  and we omit the detail (cf. \cite[\S C.3]{Liu18}); it suffices to mention that an object in the groupoid $\CM_{0,K_{Z^\BQ}}(S)$ will be denoted by $(A_0,\iota_0,\lambda_0,\ov\eta_0)$ where $\ov\eta_0$ denotes a $K_{Z^\BQ}$-level structure.

\subsubsection{The RSZ  integral model for $\left(\wt G, \{h_{\wt G}\}\right)$}
\label{sss:RSZ}

We now follow \cite[\S5.1]{RSZ3} and \cite[\S6.1]{RSZ3.5} to define the moduli interpretation of our Shimura varieties associated to the Shimura datum $\left(\wt G, \{h_{\wt G}\}\right)$ for a certain special level structure. When $F_0=\BQ$, this is closely related to \cite{KR-U2,BHKRY}. In fact for this paper we only need an integral model over a suitable Zariski open subscheme of $\Spec O_E$.

Let $\sD_0$ denote the finite set consisting of all non-archimedean places $v$ of $F_0$ such that
\begin{itemize}
\item  the residue characteristic of $v$ is $2$, or
 \item $v$ is ramified in $F$,  or
\item $v$ is  inert in $F$ where $V_v$ is non-split.
\end{itemize}
Let $\sD$ be a finite set of  non-archimedean places containing $\sD_0$, such that $\sD$ is pull-back from a set of places $\sD_\BQ$ of $\BQ$.  Define 
$$
   \fkd=\prod_{p\mid \sD_\BQ}p.
$$
We will consider the Shimura variety for  $\left(\wt G, \{h_{\wt G}\}\right)$ with level-structure at the finite set of places dividing $\fkd$.

For every non-archimedean $v\notin \sD$, we fix a self-dual $O_{F_v}$-lattice $\Lambda_v^\circ\subset V_v$, i.e.,
\begin{equation}\label{eqn:sd lat v}
   \Lambda^\circ_v= (\Lambda_v^\circ)^\vee,
   \end{equation}
where we recall that $(\Lambda_v^\circ)^\vee$ denotes the dual lattice with respect to the hermitian forms on $V_v$.
Let \begin{equation}\label{eqn:hyper K}
K^\circ_{G,v}\subset \U(V_v)(F_{0,v})
   \end{equation}
    be the stabilizer of the lattice $\Lambda_v^\circ$, a hyperspecial compact open subgroup of $\U(V_v)(F_{0,v})$. Let $K_G=\prod_v K_{ G,v}\subset G(\BA_{0,f})$ be a compact open subgroup such that $K_{ G,v}= K^\circ_{G,v}$ for all $v\nmid\fkd$.  Let $K_{Z^\BQ}=\prod_{p}K_{Z^\BQ,p}\subset Z^\BQ(\BA_f)$ be an open subgroup such that the prime-to-$\fkd$ components remain maximal. Accordingly we define  $$K_{\wt G}=K_{Z^\BQ}\times K_{ G}\subset \wt G(\BA_f).$$

Recall that $E=E_{\Phi}E_r$ is the reflex field of $(\wt G, \{h_{\wt G}\})$.
\begin{definition}\label{def RSZ glob}
The functor $\CM_{K_{\wt G}}(\wt G)$ associates to each scheme $S$ in $\LNSch_{/O_{E}[1/\fkd]}$ the groupoid of tuples $(A_0,\iota_0,\lambda_0,A,\iota,\lambda,\ov\eta)$, where
\begin{altitemize}
\item $(A_0,\iota_0,\lambda_0,\ov\eta_0)$ is an object of $\CM_{0,K_{Z^\BQ}}(S)$; 
\item $A$ is an abelian scheme over $S$;
\item $\iota\colon O_F[1/\fkd] \to \End(A)\otimes_\BZ\BZ[1/\fkd] $ is an action satisfying the Kottwitz condition of signature $$(( n-1,1)_{\varphi_0}, (n,0)_{\varphi\in\Phi\ssm\{\varphi_0\}})$$ on $O_F[1/\fkd]$; and
\item $\lambda:A \to A^\vee$ is a prime-to-$\fkd$ principle polarization whose Rosati involution inducing the Galois involution on $O_F[1/\fkd]$ with respect to $\iota$;
\item $\ov\eta$ is a $\prod_{v\mid\fkd}K_{G,v}$-orbit  of isometries of hermitian modules  
(as smooth $F_{\fkd}=\prod_{v\mid \fkd}F_{ v} $-sheaves on $S$ endowed with its natural hermitian form induced by the polarization)
\begin{align}\label{level}
    \xymatrix{ \eta\colon  \RV_\fkd(A_0, A) \ar[r]^-\sim& V(F_{0,\fkd})},
\end{align}
where
$$
 \RV_\fkd(A_0, A)\colon=\prod_{p\mid\fkd}  \RV_p(A_0, A), \quad \text{and} \quad   \RV_p(A_0, A)=\Hom_{F\otimes_\BQ\BQ_p}(V_p(A_0), V_p(A)),
 $$
 and
\begin{align}\label{eq:def V d}
 V(F_{0,\fkd})\colon= \prod_{p\mid \fkd}V\otimes_{\BQ}\BQ_p= \prod_{v\mid \fkd}V\otimes_{F_0}F_{0,v}.
\end{align}More precisely, this is understood in the sense of, e.g., \cite[Rem.\ 4.2]{KR-U2}.  Fixing any geometric point $\ov s$ of a connected scheme $S$, the rational Tate module $ \RV_\fkd(A_0, A)$ is a smooth $F_{\fkd}=\prod_{v\mid \fkd}F_{ v} $-sheaf on $S$ determined by  the rational Tate module $ \RV_\fkd(A_{0,\ov s}, A_{\ov s}) $ together with the action of the fundamental group $\pi_1(S,\ov s)$.  Moreover, the polarizations on $A_0$ and $A$ defines an $F_{\fkd}$-valued hermitian forms $\pair{\cdot,\cdot}$ on $ \RV_\fkd(A_{0,\ov s}, A_{\ov s}) $:
$$
\pair{x,y}= \lambda_0^{-1}\circ y^\vee\circ \lambda \circ x\in \End _{F_\fkd}\left(V_\fkd(A_{0,\ov s})\right)=F_\fkd.
$$ 
Then the level structure $\ov\eta$ is a $\prod_{v\mid\fkd}K_{G,v}$-orbit of  isometries of hermitian modules  
\begin{align}\label{level s}
  \xymatrix{ \eta\colon   \RV_\fkd(A_{0,\ov s}, A_{\ov s})\ar[r]^-\sim & V(F_{0,\fkd})},
\end{align}
that is required to be stable under the action of $\pi_1(S,\ov s)$. The notion of $\prod_{v\mid\fkd}K_{G,v}$-level structure is independent of the choice of the geometric point $\ov s$ on $S$.

\item Finally, we impose the Eisenstein condition (cf. \cite[\S5.2]{RSZ3.5}) for every place $w\nmid\fkd$ of $F$ \footnote{The Eisenstein condition at a place $w$ of $F$ \cite[\S5.2]{RSZ3.5} follows from the Kottwitz condition if we assume that $w$ is unramified over $p$. Therefore we only have this condition at finitely many places $w$ ramified over $p$.}. 
\end{altitemize}

A morphism between two objects $(A_0,\iota_0,\lambda_0,\ov\eta_0,A,\iota,\lambda,\ov\eta)$  and $(A'_0,\iota'_0,\lambda'_0,\ov\eta_0',A',\iota',\lambda',\ov\eta')$ is an  isomorphism  $(A_0,\iota_0,\lambda_0,\ov\eta_0) \isoarrow (A_0',\iota_0',\lambda_0',\ov\eta_0')$ in $\CM_{0,K_{Z^\BQ}}(S)$ and an $O_F$-linear  $\fkd$-isogeny  $A \to A'$, pulling  $\lambda'$ back to $\lambda$ and  $\ov\eta'$ back to $\ov\eta$.

\end{definition}

\begin{theorem}The functor $\CM_{K_{\wt G}}(\wt G)$ is represented by a Deligne--Mumford stack. The morphism $\CM_{K_{\wt G}}(\wt G)\to\Spec O_{E}[1/\fkd]$ is separated of finite type, and smooth of relative dimension $n-1$. 
\end{theorem}
\begin{proof}This follows from \cite[Th.\ 5.2]{RSZ3} (when all $p\nmid \fkd$ are unramified in $F$) and \cite[Th.\ 6.2]{RSZ3.5}, except we note that here we have omitted the sign conditions in {\it loc. cit.}.  However,  the sign conditions hold automatically for all places away from $\fkd$ (cf. \cite[Rem.\ 6.5 (i)]{RSZ3.5}) and therefore these two theorems still apply to the current situation. 
\end{proof}

Note that when both  $\prod_{p\mid\fkd}K_{Z^\BQ,p}$ and  $\prod_{p\mid\fkd}K_{G,p}$ are small enough, the functor $\CM_{K_{\wt G}}(\wt G)$ is represented by a quasi-projective scheme, smooth over $\Spec O_E[1/\fkd] $. We will make this smallness assumption for the rest of the paper.

By \cite[Prop.\ 3.5]{RSZ3}, the generic fiber $M_{K_{\wt G}} (\wt G)$ of $\CM_{K_{\wt G}}(\wt G)$ is isomorphic to the canonical model of $\Sh_{K_{\wt G}}(\wt G, \{h_{\wt G}\})$.  We also recall the moduli
functor $M_{K_{\wt G}} (\wt G)$ over $\Spec E$, for {\em any} compact open subgroup $K_{\wt G}\subset \wt G(\BA_f)$ of the form $K_{\wt G}=K_{Z^\BQ}\times K_{ G}$.  Similar to Definition \ref{def RSZ glob}, the functor $M_{K_{\wt G}}(\wt G)$ associates to each scheme $S$ in $\LNSch_{/E}$ the groupoid of tuples $(A_0,\iota_0,\lambda_0,\ov\eta_0,A,\iota,\lambda,\ov\eta)$, where everything is the same as  Definition \ref{def RSZ glob} with the following minor change. Now 
$\iota\colon F \to \End^\circ(A) $ is an $F$-action,
 $\lambda:A \to A^\vee$ is a polarization, and $\ov\eta$ is a $K_{ G}$-orbit  of isometries of $\BA_{F,f}/\BA_{F_0,f}$-hermitian modules  
\begin{align}\label{level gen}
    \xymatrix{ \eta\colon 
    \wh \RV (A_0, A) \ar[r]^-\sim& V(\BA_{F_0,f}) },
\end{align}
where 
$$
\wh\RV(A_0, A)\colon=\prod_{p}  \RV_p(A_0, A),
$$ 
and $V(\BA_{F_0,f}) =V\otimes _{F_0}\BA_{F_0,f}$. The rest is the same as Definition \ref{def RSZ glob}, with the appropriate modification of the definition of morphisms in the groupoid, cf. \cite[\S3.2]{RSZ3}.

\section{Kudla--Rapoport divisors and the derived CM cycles}

In this section we consider two type of special cycles on the integral models of Shimura varieties introduced in the previous section:
\begin{itemize}
\item the Kudla--Rapoport  special divisors \cite{KR-U2}, and 
 \item the derived CM cycle, which is a variant of the (1-dimensional) ``big CM cycle" of Bruinier--Kudla--Yang and  Howard \cite{BKY, Ho-kr}.   
\end{itemize}
The derived CM cycle is the main novel  geometric construction of this paper.

We make the following notational assumption: in Part 2 of the paper, all Schwartz functions on totally disconnected topological spaces are $\BQ$-valued.  The reason for this assumption is to define elements in various Chow group or $K$-groups with $\BQ$-coefficients.

\subsection{The global Kudla--Rapoport divisors on $M_{K_{\wt G}}(\wt G)$ over $\Spec E$}
We first define the global Kudla--Rapoport divisors on the canonical model $M_{K_{\wt G}}(\wt G)$ of $\Sh_{K_{\wt G}}(\wt G, \{h_{\wt G}\})$ over $\Spec E$, introduced at the end of \S\ref{sss:RSZ}, for an arbitrary compact open subgroup of the form $K_{\wt G}=K_{Z^\BQ}\times K_{ G}$. We follow  \cite{KR-U2} when $F_0=\BQ$, and \cite[Def.\ 4.21]{Liu18} and \cite[\S3.5]{RSZ3.5} for a general totally real field $F_0$.

Let $\xi\in F_{0,+}$ and let $\mu\in  V(\BA_{0,f})  /K_{G}$ be  a $K_{G}$-orbit.

\begin{definition}\label{def:KR gen}
For each scheme $S$ in $\LNSch_{/ E}$, the $S$-points of the KR cycle  $Z(\xi, \mu)$ is the groupoid of tuples
 $(A_0,\iota_0,\lambda_0,A,\iota,\lambda, \ov\eta, u )$, where
 \begin{altitemize}
\item  $(A_0,\iota_0,\lambda_0,\ov\eta_0,A,\iota,\lambda,\ov\eta)\in M_{K_{\wt G}}(\wt G)(S)$, and
 \item $u\in \Hom^\circ_F(A_0,A)$ such that $\pair{u,u}=\xi$, and $\ov\eta (u)$ is a homomorphism   in the $K_{G}$-orbit $\mu$. Here $\pair{\cdot,\cdot}$ denotes the hermitian form on $\Hom^\circ_F(A_0,A)$ induced by the polarization $\lambda_0$ and $\lambda$:
$$
\pair{x,y}= \lambda_0^{-1}\circ y^\vee\circ \lambda \circ x\in \End ^\circ_F(A_0)\simeq F
$$ for $x,y \in\Hom^\circ_F(A_0,A)$.
\end{altitemize}

A morphism between two objects $(A_0,\iota_0,\lambda_0,\ov\eta_0,A,\iota,\lambda,\ov\eta,u)$  and $(A'_0,\iota'_0,\lambda'_0,\ov\eta_0',A',\iota',\lambda',\ov\eta', u')$ is an  isomorphism  $(A_0,\iota_0,\lambda_0,\ov\eta_0) \isoarrow (A_0',\iota_0',\lambda_0',\ov\eta_0')$ in $\CM_0(S)$ and an $F$-linear  isogeny $\varphi:A \to A'$, compatible with $\lambda$ and $\lambda'$, and with $\ov\eta$ and $\ov\eta'$, and such that
$u'=u\circ \varphi.$

\end{definition}
Forgetting $u$ defines a natural morphism $i: Z(\xi, \mu)\to M_{K_{\wt G}}(\wt G)$, and we defer to Proposition
\ref{prop Z mu} for its properties. In particular, the push-forward defines a class in the Chow group $\Ch^1(M_{K_{\wt G}}(\wt G))$. For $\phi\in\CS(V(\BA_{0,f}))^{K_{ G}}$, we define
\begin{align}\label{global KR E}
Z(\xi, \phi)\colon=\sum_{\mu\in V_\xi(\BA_{0,f})/K_{G}} \phi(\mu)\,Z(\xi, \mu),
\end{align}viewed as an element in the Chow group $\Ch^1(M_{K_{\wt G}}(\wt G))$. Here $V_\xi$ is defined in \S\ref{Herm2Quad} after \eqref{eq:her2q}. Note that \eqref{global KR E} is a finite sum due to the compactness of the support of $\phi$ (and $G(\BA_{0,f})$ acts transitively on $V_\xi(\BA_{0,f})$ when $\xi\neq 0$).

\subsection{The global Kudla--Rapoport divisors on the integral model $\CM_{K_{\wt G}}(\wt G)$}

We now consider the moduli function $\CM=\CM_{K_{\wt G}}(\wt G)$  with level structure at primes dividing $\fkd$, cf. Definition \eqref{def RSZ glob}. Here  $K_{\wt G}$ is of the form $K_{\wt G}=K_{Z^\BQ}\times K_{ G}$ with $K_G=\prod_v K_{ G,v}\subset G(\BA_{0,f})$ such that $K_{ G,v}= K^\circ_{G,v}$ for $v\nmid\fkd$ where $K^\circ_{G,v}$ is the stabilizer of the self-dual lattice $\Lambda_v^\circ$, cf. \eqref{eqn:sd lat v} and \eqref{eqn:hyper K}.

Let $\xi\in F_{0,+}$ and $\mu\in V(F_{0,\fkd})/K_{G,\fkd}$. Here  $V(F_{0,\fkd})$ is as in \eqref{eq:def V d}.

\begin{definition}For each scheme $S$ in $\LNSch_{/ O_E[1/\fkd]}$, the $S$-points of the KR cycle  $\CZ(\xi, \mu)$ is the groupoid of tuples
 $(A_0,\iota_0,\lambda_0,\ov\eta_0, A,\iota,\lambda, \ov\eta, u )$ where
 \begin{altitemize}
\item  $(A_0,\iota_0,\lambda_0,\ov\eta_0,A,\iota,\lambda,\ov\eta)\in\CM_{K_{\wt G}}(\wt G)(S)$, and
 \item $u\in \Hom_{O_F}(A_0,A)\otimes_\BZ \BZ[1/\fkd]$ such that $\pair{u,u}=\xi$, and $\ov\eta (u)$ is a homomorphism   in the $K_{G,\fkd}$-orbit $\mu$. Here $\pair{\cdot,\cdot}$ denotes the hermitian form induced by the polarization $\lambda_0$ and $\lambda$:
$$
\pair{x,y}= \lambda_0^{-1}\circ y^\vee\circ \lambda \circ x\in \End _{O_F}(A_0)\otimes_\BZ \BZ[1/\fkd]\simeq O_F[1/\fkd].
$$ 
\end{altitemize}

A morphism between two objects $(A_0,\iota_0,\lambda_0,\ov\eta_0,A,\iota,\lambda,\ov\eta,u)$  and $(A'_0,\iota'_0,\lambda'_0,\ov\eta_0',A',\iota',\lambda',\ov\eta', u')$ is an  isomorphism  $(A_0,\iota_0,\lambda_0,\ov\eta_0) \isoarrow (A_0',\iota_0',\lambda_0',\ov\eta_0')$ in $\CM_{0,K_{Z^\BQ}}(S)$ and an $O_F$-linear prime-to-$\fkd$ isogeny $\varphi:A \to A'$, compatible with $\lambda$ and $\lambda'$, and with $\ov\eta$ and $\ov\eta'$, and such that
$u'=u\circ \varphi.$

\end{definition}
Forgetting $u$ defines a natural morphism $i: \CZ(\xi, \mu)\to \CM_{K_{\wt G}}(\wt G)$.
\begin{proposition}
\label{prop Z mu}
\begin{altenumerate}
\renewcommand{\theenumi}{\alph{enumi}}\item
The morphism $i: \CZ(\xi, \mu)\to \CM_{K_{\wt G}}(\wt G)$  is representable, finite and unramified. 
\item The morphism $i$ defines \'etale locally a Cartier divisor. Moreover, the morphism $\CZ(\xi, \mu)\to \Spec O_E[1/\fkd]$ is flat.
\end{altenumerate}
\end{proposition}
\begin{proof}When $F_0=\BQ$, part (a) follows from \cite[Prop.\ 2.9]{KR-U2}, part (b) from  \cite[\S2.5]{BHKRY}. For a general totally real $F_0$, both follow from \cite[Prop.\ 4.22]{Liu18}.
\end{proof}

To a function $\phi_\fkd\in\CS(V_\fkd)^{K_{G,\fkd}}$, we associate $\phi={\bf 1}_{\Lambda^\fkd}\otimes\phi_\fkd\in\CS(V(\BA_{0,f}))$, where $\Lambda^\fkd=\prod_{v\nmid\fkd}\Lambda_v^\circ$ for the self-dual lattice $\Lambda_v^\circ$ in \eqref{eqn:sd lat v}. Then we define 
\begin{align}\label{global KR}
\CZ(\xi, \phi)\colon=\sum_{\mu\in V_\xi(F_{0,\fkd})/K_{G,\fkd}} \phi_\fkd(\mu)\,\CZ(\xi, \mu),
\end{align}viewed as an element in the Chow group $\Ch^1(\CM_{K_{\wt G}}(\wt G))$.  For such functions $\phi_\fkd$ and the associated $\phi$, the generic fiber of $\CZ(\xi, \phi)$ is $Z(\xi, \phi)$ defined by \eqref{global KR E} (specializing to the current level $K_{\wt G}$).  In particular, the generic fiber of $\CZ(\xi, \mu)$ is the union of the  KR cycles $Z(\xi, \mu')$ in Definition \ref{def:KR gen}, for suitable $\mu'\in  V(\BA_{0,f})  /K_{G}$.

\subsection{Special divisors in the formal neighborhood of the basic locus}\label{ss: KR in ss loc}
We consider the restriction of the KR divisors to the formal completion of $\CM=\CM_{K_{\wt G}}(\wt G)$ along the basic locus.

Let $\nu\nmid\fkd$ be a non-archimedean place of $E$. Its restriction to $F$ ($F_0$, resp.) is a place denoted by $w_0$ ($v_0$, resp.). Assume that $v_0$ is {\em inert}. We recall from  \cite[\S8, in the proof of Th.\ 8.15]{RSZ3} the non-archimedean uniformization along the basic locus:
\begin{equation}\label{eq unif1}
    \CM_{  O_{\breve E_\nu}}\sphat\colon= \bigl(\CM_{(\nu)} \otimes_{O_{E,(\nu)}} O_{\breve E_\nu}\bigr)\sphat \,
	   = \wt G' (\BQ)\Big\bs \Bigl[ \CN'\times \wt G (\BA_f^p) /K_{\wt G}^p \Bigr].
\end{equation}
Here the hat on the left-hand side denotes the completion along the basic locus in the geometric special fiber of $\CM_{(\nu)}$. The group $\wt G'$ is an inner twist of $\wt G$. More precisely, the group $\wt G'$ is associated to the ``nearby" hermitian space $V'$, that is positive definite at all archimedean places, and  isomorphic to $V$, locally at all non-archimedean places except at $v_0$. Let $\CN=\CN_{n, F_{w_0}/F_{v_0}}\to \Spf O_{\breve F_{w_0}}$ be the RZ space introduced in \S\ref{ss:AFL}, and take its base change $\CN_{O_{\breve E_\nu}}=\CN \mathbin{\wh\otimes}_{O_{\breve F_{w_0}}}O_{\breve E_\nu}$. Then as in {\it loc. cit.}\footnote{The formal scheme in the Rapoport--Zink uniformization theorem  is the ``absolute" RZ space of PEL-type in \cite{M-Th} rather than the ``relative" RZ space $\CN_{n, F_{w_0}/F_{v_0}}$ in \S\ref{ss:AFL}. These two RZ spaces coincide by \cite[Th.\ 3.1]{M-Th}, noting that the Eisenstein condition imposed in \cite{M-Th} for the signature $(r,s)=(n-1,1)$ reduces to the Eisenstein condition in \cite[\S5.2, case (2)]{RSZ3.5} for the definition of the moduli space $\CM=\CM_{K_{\wt G}}(\wt G)$ in this paper.}, we may rewrite \eqref{eq unif1} as 
\begin{equation}\label{eq unif2}
  \CM_{  O_{\breve E_\nu}}\sphat \,
	   = \wt G'(\BQ) \Big\bs \Bigl[  \CN_{O_{\breve E_\nu}} \times \wt G (\BA_f^{v_0}) /K_{\wt G}^{v_0} \Bigr].
\end{equation}
Here, by abuse of notation, we denote 
$$
\wt G (\BA_f^{v_0}) /K_{\wt G}^{v_0} = 
    \wt G (\BA_f^{p})/K_{\wt G}^{p}\times \bigl(Z^\BQ(\BQ_p)/K_{Z^\BQ, p}\bigr)\times \prod_{v\in S_p\ssm\{v_0\}}G(F_{0,v})/ K_{G,v},
$$
where $S_p$ denotes the set of places of $F_0$ above $p$. For the action of the group $\wt G'(\BQ)$ in \eqref{eq unif2}, we fix an isomorphism $\wt G'(\BA_{f}^{v_0})\simeq \wt G(\BA_{f}^{v_0})$.

Note that the uniformization \eqref{eq unif2}  induces a projection to a discrete set (in fact an abelian group)
\begin{equation}\label{unif proj}
\xymatrix{  \CM_{  O_{\breve E_\nu}}\sphat \ar[r]&  Z^\BQ(\BQ)\bs\bigl(Z^\BQ(\BA_f)/K_{Z^\BQ}\bigr).}
\end{equation}
This gives a partition of the formal scheme $\CM_{  O_{\breve E_\nu}}\sphat$, each fiber is naturally isomorphic to 
\begin{equation}\label{unif G}
\CM_{  O_{\breve E_\nu}, 0}\sphat:= G'(F_0) \Big\bs \Bigl[  \CN_{O_{\breve E_\nu}} \times  G (\BA_{0,f}^{v_0}) /K_{G}^{v_0} \Bigr].
\end{equation}

Recall that we have the local KR divisors $\CZ(u)$ on $\CN=\CN_{n,F_{w_0}/F_{0,v_0}}$ for  each $u\in V'\otimes F_{0,v_0}\simeq \Hom^\circ(\BE,\BX_n)$, the hermitian space of local special homomorphisms (for some fixed framing objects $\BE$ and $\BX_n$ in the uniformization \eqref{eq unif2} above). For a pair $(u,g) \in V'(F_0)\times G (\BA_{0,f}^{v_0})/K_G^{v_0}$ with $u\neq 0$, we define the product divisor on $\CN_{O_{\breve E_\nu}} \times  G (\BA_{0,f}^{v_0}) /K_{G}^{v_0}$
\begin{align}\label{eq:KR v0}
\CZ(u,g)_{K_{G}^{v_0}}=\CZ(u)\times {\bf 1}_{  g\,K_G^{v_0}}.
\end{align}
We then consider
\begin{align*}
\sum\CZ(u',g')_{K_{G}^{v_0}},
\end{align*}where the sum is over $(u',g')$ in the $G'(F_0)$-orbit of the pair $(u,g)$ (for the diagonal action of $G'(F_0)$ on $V'(F_0)\times G (\BA_{0,f}^{v_0})/K_G^{v_0}$. Since the sum is $G'(F_0)$-invariant, it descends to a divisor on the quotient $\CM_{  O_{\breve E_\nu}, 0}$ in \eqref{unif G}, which we denote by $[\CZ(u,g)]_{K_{G}^{v_0}}$.

\begin{proposition}\label{KR v0}
Let $\xi\in F_{0,+}$.
 The restriction of the special divisor $ \CZ(\xi,\phi)$ to each fiber of the above projection \eqref{unif proj}   is  the sum
 \begin{equation}\label{KR v0 ug}
 \sum_{(u,g)\in G'(F_0)\bs(V'_\xi(F_0)\times G (\BA_{0,f}^{v_0})/K_G^{v_0}) } \phi^{v_0}(g^{-1}u)\cdot [\CZ(u,g)]_{K_{G}^{v_0}},
 \end{equation}
 viewed as a divisor on \eqref{unif G}. This is a finite sum. 
\end{proposition}

\begin{remark}
This is similar to the description of the special divisors over the complex number, cf. \eqref{KR C xi}.
\end{remark}

\begin{proof}
This follows from the proof of \cite[Prop.\ 6.3]{KR-U2}, also cf. \cite[\S4.2]{Liu18}.
\end{proof}

\subsection{Fat big CM cycles}
\label{ss:FB CM}
We introduce a fat variant of the ``big CM cycle" in \cite{BKY, Ho-kr} on our moduli space $\CM=\CM_{K_{\wt G}}(\wt G)$ with level structure at primes dividing $\fkd$ (cf. Definition \ref{def RSZ glob}).

Fix an $\alpha\in \CA_n(F_0)\subset F[T]_{\deg=n}$ (cf. the end of \S\ref{ss: orb match}). If $\alpha$ has no repeated roots, then 
\begin{align}\label{eqn:def F'}
F'=F[T]/(\alpha)
\end{align}
is a semi-simple $F$-algebra. There is a unique $F_0$-linear involution on $F'$ sending $T\to T^{-1}$ and extending the Galois involution for $F/F_0$. Then the fixed subalgebra $F_0'$ is a product of totally real field extensions of $F_0$, and $F'\simeq F\otimes _{F_0}F_0'$ with the involution of $F'/F_0'$ induced from that of $F/F_0$.

Now let $\alpha\in \CA_n(O_{F_0}[1/\fkd])$ be irreducible over $F$. Then the algebra $F'$ in \eqref{eqn:def F'} is a field. Throughout the rest of the paper we will assume  $F'$ is a CM extension of $F_0'$. This is a necessary condition for the functor in Definition \ref{def CM0} below to be non-empty. We denote
\begin{align}\label{eqn:def Ra}
R_\alpha=O_F[1/\fkd][T]/(\alpha),
\end{align}
viewed as a sub-ring of the CM field $F'$.

\begin{definition} \label{def CM0}
The functor $\CCM(\alpha)=\CCM_{K_{\wt G}}(\alpha)$ associates to each scheme $S$ in $\LNSch_{/O_{E}[1/\fkd]}$ the groupoid of tuples $(A_0,\iota_0,\lambda_0,\ov\eta_0,A,\iota,\lambda, \ov\eta,\varphi)$ where $(A_0,\iota_0,\lambda_0,\ov\eta_0,A,\iota,\lambda, \ov\eta)\in \CM_{K_{\wt G}}(\wt G)$ and  $\varphi\in\End_{O_F}(A)\otimes \BZ[1/\fkd]$ such that 
\begin{itemize}
\item the polynomial $\alpha$ annihilates the endomorphism $\varphi$;
\item  $\varphi$ is compatible with $\lambda$, i.e.,  $\varphi^\ast\lambda=\lambda$, or equivalently, the Rosati involution sends $\varphi$ to $\varphi^{-1}$; and 
\item $\varphi$ preserves the  level structure $\ov\eta$, i.e., we have a commutative diagram 
\begin{align*}
  \xymatrix{  \RV_\fkd(A_0, A) \ar[d]^{\varphi}\ar[r]^-{\eta_1}& V(F_{0,\fkd})\ar[d]^{\id}\\
 \RV_\fkd(A_0, A) \ar[r]^-{\eta_2}& V(F_{0,\fkd}),}
\end{align*}
for some $\eta_1,\eta_2\in\ov\eta$. 
\end{itemize}
Morphisms in the groupoid are defined in the obvious way. 
\end{definition}

We have a natural forgetful map
\[
 \begin{gathered}
	\xymatrix@R=0ex{
\CCM_{K_{\wt G}}(\alpha)   \ar[r] & \CM_ {K_{\wt G}}(\wt G).
	}
	\end{gathered}
\]
We call $\CCM_{K_{\wt G}}(\alpha) $ the (naive) fat big CM cycle, or simply CM cycle.

\begin{remark} The abelian scheme $A$ in the moduli functor $\CCM(\alpha)$ carries an $R_\alpha$-action where the $T$ in \eqref{eqn:def Ra} acts by $\varphi$. Our moduli functor $\CCM(\alpha)$ is analogous to the big CM cycle defined in \cite{Ho-kr} where $R_\alpha$ is replaced by the ring of integers $O_{F'}$ in $F'=F[T]/(\alpha)$. A minor difference is that we do not impose any Kottwitz signature condition in our Definition \ref{def CM0}, while \cite[Def. 3.11]{Ho-kr} does. The main new feature of our moduli functor $\CCM(\alpha)$ is that we allow $R_\alpha$ to be a non-maximal order in $O_{F'}[1/\fkd]$. As a result, it could have very complicated structure in positive characteristic (e.g., with large dimensional components). A complete understanding of the geometric structure of $\CCM(\alpha)$ seems a hard question (e.g., to determine all of its irreducible components in its special fibers), and the AFL type identity in this paper gives us a partial answer.  
\end{remark}

We also define a twisted variant of $\CCM(\alpha) $. 
\begin{definition}\label{def CM}
Let $g\in G(F_{0,\fkd})$.
The functor $\CCM(\alpha,g)=\CCM_{K_{\wt G}}(\alpha,g)$ associates to each $O_{E}[1/\fkd]$-scheme $S$ the groupoid of tuples $(A_0,\iota_0,\lambda_0,\ov\eta_0,A,\iota,\lambda, \ov\eta,\varphi)$ where $(A_0,\iota_0,\lambda_0,\ov\eta_0,A,\iota,\lambda, \ov\eta)\in \CM_{K_{\wt G}}(\wt G)$ and  $\varphi\in\End_{O_F}(A)\otimes \BZ[1/\fkd]$ such that 
\begin{itemize}
\item the polynomial $\alpha$ annihilates the endomorphism $\varphi$;
\item  $\varphi$ is compatible with $\lambda$, i.e.,  $\varphi^\ast\lambda=\lambda$, or equivalently,  the Rosati involution sends $\varphi$ to $\varphi^{-1}$; and 
\item We have a commutative diagram 
\begin{align*}
  \xymatrix{  \RV_\fkd(A_0, A) \ar[d]^{\varphi}\ar[r]^-{\eta_1}& V(F_{0,\fkd})\ar[d]^{ g}\\
 \RV_\fkd(A_0, A) \ar[r]^-{\eta_2}& V(F_{0,\fkd}),}
\end{align*}
for some $\eta_1,\eta_2\in\ov\eta$. 
\end{itemize}
Morphisms are defined in the obvious way. 
\end{definition}

Denote by $\Ram(\alpha)$ the set of non-archimedean places $v\nmid \fkd$ of $F_0$ where $R_\alpha=O_F[1/\fkd][T]/(\alpha)$ is non-maximal (i.e., $R_{\alpha,v}:=R_\alpha\otimes _{O_{F_0}}O_{F_0,v}$ is not a product of DVRs). 

\begin{proposition}\label{pro bad fiber}
Let $\alpha\in \CA_n(O_{F_0}[1/\fkd])$ be irreducible over $F$. Let $g\in\prod_{v\mid\fkd}G(F_{v})$. 
\begin{altenumerate}
\renewcommand{\theenumi}{\alph{enumi}}\item
The morphism 
$\CCM(\alpha,g)\to \CM$ is representable, finite and unramified. 
\item
The morphism 
$\CCM(\alpha,g)  \to \Spec O_{E}[1/\fkd]$ is proper. Its restriction to the open sub-scheme $\Spec O_{E}[1/\fkd]\setminus \Ram(\alpha)$ is  finite \'etale.
\end{altenumerate}
\end{proposition}

\begin{proof}The first part follows similarly to Proposition \ref{prop Z mu} (by the theory of Hilbert scheme, the morphism is representable by a disjoint union of schemes of finite type; it is of finite type by the first condition $\alpha(\varphi)=0$  \footnote{Alternatively, one can argue using Lemma  \ref{lem:char poly cons} and Lemma \ref{CCM=Hk}.}; it is then quasi-finite because there are only finitely many ways to endow an action of the order $R_\alpha$ to a given $(A,\iota,\lambda)$ over an arbitrary field;  the unramifiedness follows from the rigidity of quasi-isogeny; by the valuative criterion by the N\'eron property of abelian scheme, the morphism is proper, and hence finite).

The properness of $\CCM(\alpha,g)  \to \Spec O_{E}[1/\fkd]$ follows by the valuative criterion (the toric part of a semi-abelian scheme will have too small dimension to carry an action of $R$). Finally, the argument of \cite[Prop.\ 3.1.2(3)]{Ho-kr} still holds to show the finiteness and \'etaleness over $\Spec O_{E}[1/\fkd]\setminus \Ram(\alpha)$: at every place above $v\notin \Ram(\alpha)$, the local order $R_{\alpha,v}$ is maximal and hence the $p$-divisible group has formal multiplication by a local maximal order.  \end{proof}

\subsection{Hecke correspondences and their fixed point loci}
\label{ss Hk}
We first introduce the characteristic polynomial of an endomorphism of an abelian variety.  Then we apply it to study the fixed point loci of Hecke correspondences.

Let $k$ be an arbitrary field, and $A$ an abelian variety over $k$. Then we define the characteristic polynomial of $\varphi\in \End^\circ(A)$, denoted by $\charac_\BQ(\varphi)$, as follows
$$
\charac_\BQ(\varphi)=\det(T-\varphi| \RV_\ell(A))\in \BQ_\ell[T]_{\deg=2\dim A},
$$
where $\ell$ is any prime different from the  characteristic of $k$, and $\RV_\ell(A)$ denotes the rational $\ell$-adic Tate module of $A$.   Similarly, if $\iota:F\to \End^\circ(A)$ is an $F$-action, and $\varphi\in \End_F^\circ(A)$, then $\RV_\ell(A)$ is a free $F\otimes_\BQ\BQ_\ell$-module of rank $n:=\frac{2\dim A}{[F:\BQ]}$. We then define 
$$
\charac_F(\varphi)={\rm det}_{ F\otimes_\BQ\BQ_\ell}(T-\varphi| \RV_\ell(A))\in F\otimes_\BQ\BQ_\ell[T]_{\deg=n},
$$
viewing $\RV_\ell(A)$ as a free $F\otimes_\BQ\BQ_\ell$-module. 

\begin{lemma}\label{lem:char poly}
\begin{altenumerate}
\renewcommand{\theenumi}{\alph{enumi}}\item The characteristic polynomial 
$\charac_\BQ(\varphi)\in \BQ[T]_{\deg=2\dim A}$ and is independent of the choice of $\ell$. 
\item If $\iota:F\to \End^\circ(A)$ is an $F$-action, and $\varphi\in \End_F^\circ(A)$, then the characteristic polynomial  $\charac_F(\varphi)\in F[T]_{\deg=n}$ and is independent of the choice of $\ell$.
\end{altenumerate}
\end{lemma}

\begin{proof}
The  characteristic polynomial $\charac_\BQ(\varphi)$ is determined by its value at $T=m\in \BQ\subset \End^\circ(A)$, in which case we have
$$
\charac_\BQ(\varphi)(m)=\deg(m-\varphi)\in \BQ_{\geq 0}.
$$
This proves the first part.  

Let $\tr_\BQ(\varphi)$ be the negation of the coefficient of $T^{2\dim A-1}$ in the polynomial $\charac_\BQ(\varphi)$. 
Then we obtain a $\BQ$-linear map $\tr_\BQ:  \End^\circ(A)\to \BQ $.
Then knowing $\tr_\BQ(\varphi^i) $ for all $i\geq 0$ is equivalent to knowing $\charac_\BQ(\varphi)$. If $\varphi$ commutes with the $F$-action $\iota: F\to \End^\circ(A)$, we define $\tr_F(\varphi)\in F$, characterized  by
$$
\tr_{F/\BQ}(a\tr_F(\varphi))=\tr_\BQ(\iota(a) \varphi),\quad \text{for all } a\in F.
$$From $\tr_F(\varphi^i)\in F$  for all $i\geq 0$, there exists a unique polynomial in $ F[T]_{\deg=n}$ recovering the characteristic polynomial $\charac_F(\varphi)$. This proves the second part.

\end{proof}

\begin{lemma}\label{lem:char poly cons}
Let $S$ be a connected locally noetherian scheme, $A\to S$ an abelian scheme. 
\begin{altenumerate}
\renewcommand{\theenumi}{\alph{enumi}}\item If $\varphi\in \End^\circ(A)$, then the function $s\in S\mapsto \charac_\BQ(\varphi)\in \BQ[T]_{\deg=2\dim A}$ is constant.
\item Let $\iota:F\to \End^\circ(A)$  and $\varphi\in \End^\circ_F(A)$. Then the function $s\in S\mapsto \charac_F(\varphi)\in F[T]_{\deg=n}$ is constant.

\end{altenumerate}
\end{lemma}

\begin{proof}It suffices to show the assertion when some rational prime $\ell$ is invertible on $S$ (otherwise, choose two primes $\ell_1\neq\ell_2$, and cover $\Spec\BZ$ by open sub-schemes $\Spec\BZ[1/\ell_1]$ and $\Spec\BZ[1/\ell_2]$, then pull back to cover $S$). 
Then the local constancy follows from the fact that the rational $\ell$-adic Tate module $\RV_\ell(A)$
is a lisse \'etale sheaf on $S$.
\end{proof}

We now define Hecke correspondences.
\begin{definition}\label{def Hk}
Let $g\in G(F_{0,\fkd})\subset G(\BA_{0,f})$.
The functor $\Hk_{[K_G\, g\, K_{G}]}$ associates to each $O_{E}[1/\fkd]$-scheme $S$ the groupoid of tuples 
$$
(A_0,\iota_0,\lambda_0,\ov\eta_0,A,\iota,\lambda, \ov\eta, A',\iota',\lambda', \ov\eta', \varphi)
$$ 
where $(A_0,\iota_0,\lambda_0,\ov\eta_0,A,\iota,\lambda, \ov\eta), (A_0,\iota_0,\lambda_0,\ov\eta_0,A',\iota',\lambda', \ov\eta')\in \CM_{K_{\wt G}}(\wt G)(S)$,  and a quasi-isogeny $\varphi\in\Hom_{O_F}(A,A')\otimes \BZ[1/\fkd]$ such that 
\begin{itemize}
 \item $\varphi$  is compatible with $\lambda$ and $\lambda'$, i.e., $\varphi^\ast\lambda'=\lambda$; 
 \item There exist $\eta\in \ov\eta$ and $\eta'\in\ov \eta'$ such that the diagram 
\begin{align*}
  \xymatrix{  \RV_\fkd(A_0, A) \ar[d]^{\varphi}\ar[r]^-{\eta}& V(F_{0,\fkd})\ar[d]^{ g}\\
 \RV_\fkd(A_0, A') \ar[r]^-{\eta'}& V(F_{0,\fkd})}
\end{align*}
commutes. Here the left vertical map on rational Tate modules is induced by $\varphi$. Note that this is to be understood similarly to the definition of level structure (cf. Definition \ref{def RSZ glob}).
\end{itemize}
Morphisms are defined in the obvious way. 
\end{definition}
We have a natural morphism
\begin{align*}
  \xymatrix{ {\rm Hk}_{[K_G\, g\, K_{G}]}  \ar[r]&\CM\times_{O_E[1/\fkd]} \CM.}
\end{align*}This morphism is finite,
and the projection to any one factor is a finite \'etale morphism.

Now consider the fiber product, called the ``fixed point locus of the Hecke correspondence ${\rm Hk}_{[K_G\, g\, K_{G}]}$"
\begin{align*}
  \xymatrix{\CM_{[K_G\, g\, K_{G}]}\colon= {\rm Hk}_{[K_G\, g\, K_{G}]}\times _{\CM\times\CM}\Delta_\CM \ar[d] \ar[r] &{\rm Hk}_{[K_G\, g\, K_{G}]} \ar[d] 
  \\ \CM \ar[r]^-{\Delta}& \CM\times_{O_E[1/\fkd]} \CM.}
\end{align*}
Since $\CM$ is a scheme over $O_E[1/\fkd]$ (under our smallness assumption on the compact open $K_{\wt G}$), an object in $\CM_{[K_G\, g\, K_{G}]}(S)$ can be represented by $(A_0, \iota_0, \lambda_0, \ov\eta_0,A, \iota, \lambda, \ov\eta, \varphi)$.

By Lemma \ref{lem:char poly} and Lemma \ref{lem:char poly cons}, we obtain a locally constant map (for the Zariski topology on the source)
\begin{align}\label{char pol}
  \xymatrix{\charac_F\colon \CM_{[K_G\, g\, K_{G}]}  \ar[r] &  F[T]_{\deg= n},}
\end{align}
which sends a point  $(A_0, \iota_0, \lambda_0, \ov\eta_0,A, \iota, \lambda, \ov\eta, \varphi)$ in   $\CM_{[K_G\, g\, K_{G}]}$ to $\charac_F(\varphi)$.
The image is a finite set by Lemma \ref{lem:char poly cons} because the source is of finite type and hence has only finitely many connected components.
It follows that the fixed point locus $\CM_{[K_G\, g\, K_{G}]}$ is a disjoint union of open and closed subschemes, indexed by the image under the map \eqref{char pol}:
\begin{align}\label{decom Fix}
\CM_{[K_G\, g\, K_{G}]}=\coprod_{\alpha\in\Im(\charac_F)}\charac_F^{-1}(\alpha).
\end{align}

\begin{lemma}\label{lem:self-rec}
If $\alpha\in  F[T]_{\deg= n}$ lies in the image of the map \eqref{char pol}, then it is conjugate self-reciprocal, and all of its coefficients lie in $ O_{F}[1/\fkd]$ (i.e.,  $\alpha\in \CA_n(O_{F_0}[1/\fkd])$ in the notation \eqref{eqn:def An}).
\end{lemma}

\begin{proof}
Suppose that $\alpha$ is the image of a point $(A_0, \iota_0, \lambda_0, \ov\eta_0,A, \iota, \lambda, \ov\eta, \varphi)$ over a algebraically closed field $k$. Since the endomorphism $\varphi$ preserves the polarization $\lambda$, it preserves the  hermitian form on $\RV_\ell(A_0,A)$ for any $\ell\nmid\fkd$. Then the first assertion follows from the easy fact that the characteristic polynomial of an element preserving a hermitian form is conjugate self-reciprocal. To show that $\alpha\in O_F[1/\fkd][T]$, it suffices to show for every prime $\ell\nmid\fkd$, we have $\alpha\in O_{F}\otimes_\BZ \BZ_\ell[T]$ when viewing $\alpha\in F\otimes_\BQ\BQ_\ell[T]$. If $\ell$ is different from the characteristic of the field $k$, the Tate module $\RT_\ell(A)$ is a free $ O_{F}\otimes_\BZ \BZ_\ell$-module. The endomorphism $\varphi$ preserves $\RT_\ell(A)$ and hence its characteristic polynomial has coefficients in $O_{F}\otimes_\BZ \BZ_\ell$. If $\ell$ is equal to the characteristic of the field $k$, the desired integrality follows using the Dieudonn\'e $\RM(A)$, a free $O_F\otimes_{\BZ} W(k)$-module of rank $n$, where $W(k)$ is the ring of Witt vectors of $k$. 
\end{proof}

\begin{remark}
Similarly, if $\varphi\in\End^\circ(A)$ preserves a polarization $\lambda:A\to A^\vee$ (i.e., $\varphi^\ast\lambda=\lambda$) then $\charac_\BQ(\varphi)$ is  self-reciprocal  (i.e., $T^{2\dim A}\charac_\BQ(\varphi)(T^{-1})=\charac_\BQ(\varphi)(T)$).
\end{remark}

Finally, we relate the fixed point locus to the  twisted CM cycle $\CCM(\alpha,g)$ in Definition \ref{def CM}.  
\begin{lemma}\label{CCM=Hk}
Let $\alpha\in \CA_n(O_{F_0}[1/\fkd])$ be irreducible over $F$. Let $g\in G(F_{0,\fkd})$. Then the fiber of the map $\charac_F$ \eqref{char pol} above the polynomial $\alpha$ is canonically isomorphic to the twisted CM cycle $\CCM(\alpha,g)$ in Definition \ref{def CM}. 
\end{lemma}
\begin{proof}By Definition \eqref{def Hk},  the fiber of the map $\charac_F$ \eqref{char pol} above $\alpha$ is the functor whose $S$-points are the groupoid of tuples $(A_0,\iota_0,\lambda_0,\ov\eta_0,A,\iota,\lambda, \ov\eta,\varphi)$ satisfying the same conditions as in Definition \ref{def CM}, except the first one, i.e., $\alpha(\varphi)=0$. This condition is equivalent to the  condition on the characteristic polynomial of $\varphi$  by Cayley--Hamilton theorem and the assumption that $\alpha$ is irreducible.

\end{proof}

\subsection{Derived CM cycle $\LCM(\alpha,g)$ }\label{der CM}

In \S\ref{ss Hk},  the twisted CM cycle $\CCM(\alpha,g)$ is recognized as a union of connected components of the fixed point locus $\CM_{[K_G\, g\, K_{G}]}$, cf. \eqref{char pol}:
\begin{align}\label{eq:int Hk}
  \xymatrix{ \CCM(\alpha,g)\ar[r]&\CM_{[K_G\, g\, K_{G}]} \ar[d] \ar[r] &{\rm Hk}_{[K_G\, g\, K_{G}]} \ar[d] 
  \\
& \CM \ar[r]^-{\Delta}& \CM\times_{O_E[1/\fkd]} \CM.}
\end{align}
This allows us to defines a derived CM cycle, by taking the restriction of the derived tensor product 
\begin{align}\label{def LCM}
 \LCM(\alpha,g):=[ \CO_{\Hk_{[K_G\, g\, K_{G}]}} \Ltimes  \CO_{\CM} ] \mid_{ \CCM(\alpha,g)}\in K_0' (\CCM(\alpha,g)).
\end{align}
Moreover, since $\Delta$ is a regular immersion of codimension $n-1$, this element belongs to the filtration $F^{n-1}K_0^{\CCM(\alpha,g)}({\rm Hk}_{[K_G g K_{G}]})$, and hence by \eqref{cup Fil}\footnote{Strictly speaking the assertion  \eqref{cup Fil} only applies to two closed subschemes $Y,Z$ of $X$. Here, the right-most morphism in \eqref{eq:int Hk} is finite and hence preserves the dimension of any closed subscheme. Therefore we may apply \eqref{cup Fil} to the image of this morphism.}
\begin{align}\label{def LCM fil}
 \LCM(\alpha,g)\in F_1 \,K_0' (\CCM(\alpha,g)).
\end{align}

We extend the derived CM cycle to a weighted version. 
 Let $\CS\left(G(F_{0,\fkd}),K_{G,\fkd}\right)$ be the space of bi-$K_{G,\fkd}$-invariant Schwartz functions on $G(F_{0,\fkd})$. For $\phi_\fkd\in \CS\left(G(F_{0,\fkd}),K_{G,\fkd}\right)$, we denote $\phi_0={\bf 1}_{K_G^\fkd}\otimes\phi_\fkd\in\CS(G(\BA_{0,f}), K_G)$ (here $K_G^\fkd=\prod_{v\nmid \fkd}K^\circ_{G,v}$). We then define  $\LCM(\alpha,\phi_0)$ as the sum of above twisted CM cycles
\begin{align}\label{CM phi0}
\LCM(\alpha,\phi_0)=\sum_{g\in K_{ G}\bs G(\BA_{0,f})/ K_{ G} } \phi_0(g)\,\LCM(\alpha,g),
\end{align}
where we regard  each summand $\LCM(\alpha,g)$ as an element in $$\bigoplus_{g\in K_{ G,\fkd}\bs G(F_{0,\fkd})/ K_{ G,\fkd}}K_0'(\CCM(\alpha,g)).
$$ Moreover,  these elements lie in the filtration, cf. \eqref{def LCM fil}, 
\begin{align}\label{def LCM fil phi}
\LCM(\alpha,\phi_0) \in \bigoplus_{g\in  K_{ G,\fkd}\bs G(F_{0,\fkd})/ K_{ G,\fkd}} F_1\,K_0'(\CCM(\alpha,g)).
\end{align}

\subsection{Hecke correspondences in the formal neighborhood of the basic locus}
\label{ss Hk loc}

We now consider the restriction of the Hecke correspondence $\Hk_{K_G\, g\, K_{G}}$ to the formal neighborhood of the basic locus at a non-archimedean place $v_0\nmid \fkd$ of $F_0$, {\em inert} in $F$, via the RZ uniformization \eqref{eq unif2}. We resume the notation there. We consider the fiber product (in the category of locally noetherian formal schemes)
\begin{align*}
  \xymatrix{ {\rm Hk}_{[K_G\, g\, K_{G}]}\sphat  \ar[d] \ar[r] &{\rm Hk}_{[K_G\, g\, K_{G}]} \ar[d] 
  \\  \CM_{  O_{\breve E_\nu}}\sphat \times_{\Spf O_{\breve E_\nu}}  \CM_{  O_{\breve E_\nu}}\sphat \ar[r]& \CM\times_{O_E[1/\fkd]} \CM.}
\end{align*}
The commutative diagram in fact lives over the base $Z^\BQ(\BQ)\bs\bigl(Z^\BQ(\BA_f)/K_{Z^\BQ}\bigr)$, cf. \eqref{unif proj}. Therefore it suffices to consider the fiber (cf. \eqref{unif G}) over any fixed element of $Z^\BQ(\BQ)\bs\bigl(Z^\BQ(\BA_f)/K_{Z^\BQ}\bigr)$. It follows immediately that
\begin{proposition}\label{Hk ss}
Let $$
\Hk^{(v_0)}_{[K_G\, g\, K_{G}] }:=\{ (g_1,g_2)\in G(\BA_{0,f}^{v_0}) /K_{G}^{v_0}\times G(\BA_{0,f}^{v_0}) /K_{G}^{v_0}\mid g_1^{-1}g_2 \in   K_G g K_{G} \}
$$
with the two obvious projection maps, and the diagonal action by $G'(F_0)$ from the left multiplication.
Then the fiber of the Hecke correspondence ${\rm Hk}_{[K_G\, g\, K_{G}]}\sphat  $ over any fixed element of $Z^\BQ(\BQ)\bs\bigl(Z^\BQ(\BA_f)/K_{Z^\BQ}\bigr)$ (cf. \eqref{unif G}) can be identified with 
\begin{align*}
  \xymatrix{ {\rm Hk}_{[K_G\, g\, K_{G}],0}\sphat  \ar[d] \ar[r] ^-\sim&  G'(F_0) \Big\bs \Bigl[  \CN_{O_{\breve E_\nu}} \times  \Hk^{(v_0)}_{[K_G\, g\, K_{G}] }   \Bigr] \ar[d] 
  \\  \CM_{  O_{\breve E_\nu},0}\sphat \times \CM_{  O_{\breve E_\nu},0}\sphat \ar[r]^-\sim&   G'(F_0) \Big\bs \Bigl[  \CN_{O_{\breve E_\nu}} \times  G (\BA_{0,f}^{v_0}) /K_{G}^{v_0} \Bigr]\times  G'(F_0) \Big\bs \Bigl[  \CN_{O_{\breve E_\nu}} \times  G (\BA_{0,f}^{v_0}) /K_{G}^{v_0} \Bigr]}
\end{align*}
where the right vertical map is induced by the diagonal $\CN_{O_{\breve E_\nu}}\to \CN_{O_{\breve E_\nu}}\times \CN_{O_{\breve E_\nu}}$, and the two projection maps from $\Hk^{(v_0)}_{[K_G\, g\, K_{G}]}$.
\end{proposition}

\subsection{CM cycles in the formal neighborhood of the basic locus}

\label{ss: CM basic}

We now consider the restriction of the fat big CM cycle and its derived version to the formal neighborhood of the basic locus at a non-archimedean place $v_0\nmid \fkd$ of $F_0$, {\em inert} in $F$, via the RZ uniformization \eqref{eq unif2}. We resume the notation there. 

Let $\alpha\in \CA_n(O_{F_0}[1/\fkd])$ be irreducible over $F$. We denote by $\CCM\sphat(\alpha)$ (resp., $\CCM\sphat(\alpha,g)$) the formal completion along the basic locus of the CM cycle $\CCM(\alpha)$ (resp.,  $\CCM(\alpha,g)$ for $g\in G(F_{0,\fkd})$). We denote the derived CM cycle $$ \LCM\sphat(\alpha,g)\in  \,K_0' (\CCM\sphat(\alpha,g)),$$ 
and for $\phi_0={\bf 1}_{K_G^\fkd}\otimes\phi_\fkd\in\CS(G(\BA_{0,f}),K_G)$,
 $$\LCM\sphat(\alpha,\phi_0)\in \bigoplus_{g\in K_{ G,\fkd}\bs G(F_{0,\fkd})/ K_{ G,\fkd}}K_0'(\CCM\sphat(\alpha,g)).
$$

For $\delta\in G'(F_{0,v_0})$, let $\CN^{\delta}$ be the fixed point locus of $\delta$ on the RZ space $\CN$ for $F_{w_0}/F_{0,v_0}$(cf. \S\ref{ss:AFL}), and let $\CN_{O_{\breve E_\nu}} ^{\delta}$ be its base change to $O_{\breve E_\nu}$. 
For $(\delta,h)\in G'(F_0)\times G (\BA_{0,f}^{v_0})/K_G^{v_0}$, we define a closed formal subscheme of
$\CN_{O_{\breve E_\nu}} \times  G (\BA_{0,f}^{v_0}) /K_{G}^{v_0}$:
\begin{align}\label{CM v0}
\CCM(\delta, h)_{K_G^{v_0}}= \CN_{O_{\breve E_\nu}} ^{\delta }  \times {\bf 1}_{h K_{G}^{v_0}}.
\end{align}
We consider the sum
\begin{align}\label{CM v0 qt}
\sum  \,\,\CCM(\delta', h')
\end{align}
 over all $(\delta',h')\in G'(F_0)\times G (\BA_{0,f}^{v_0})/K_G^{v_0}$ in the $G'(F_0)$-orbit of $(\delta,h)$. Here $G'(F_0)$ acts diagonally on $G'(F_0)\times G (\BA_{0,f}^{v_0})/K_G^{v_0}$ by $g\cdot (\delta,h)=(g\delta g^{-1} ,gh)$. The sum is $G'(F_0)$-invariant and hence descends to the quotient formal scheme \eqref{unif G}, which we denote by $\bigl[\CCM(\delta, h)\bigr]_{K_G^{v_0}}$.

Furthermore, we have a derived version of \eqref{CM v0} and \eqref{CM v0 qt} by replacing the naive fixed point locus $\CN_{O_{\breve E_\nu}} ^{\delta}$ in \eqref{CM v0} by the derived fixed point locus $\LN_{O_{\breve E_\nu}} ^{\delta}$ defined by \eqref{der Ng}.

We then have an analog of Proposition \ref{KR v0}.
\begin{proposition}\label{CCM v0}
Let $\alpha\in \CA_n(O_{F_0}[1/\fkd])$ be  irreducible over $F$. 
\begin{altenumerate}
\item The restriction  of $ \CCM\sphat(\alpha)$ to each fiber of the projection \eqref{unif proj}   is  the disjoint union
$$
\coprod_{(\delta ,h) }\quad \,\bigl[\CCM(\delta, h)\bigr]_{K_G^{v_0}},
$$
where the index runs over the set
$$\left\{(\delta ,h)\in G'(F_0)\bs\bigl (G'(\alpha)(F_0)\times G (\BA_{0,f}^{v_0})/K_G^{v_0}\bigr)\mid h^{-1}\delta h\in K_G^{v_0}\right\}.
$$

\item Let $\phi_0={\bf 1}_{K_G^\fkd}\otimes\phi_\fkd\in\CS(G(\BA_{0,f}),K_G)$ where $\phi_\fkd\in \CS\left(G(F_{0,\fkd}),K_{G,\fkd}\right)$. The restriction  of $\LCM\sphat(\alpha,\phi_0)$ in \eqref{CM phi0} to each fiber of the projection \eqref{unif proj}   is the sum
$$
\sum_{(\delta ,h)\in G'(F_0)\bs\bigl(G'(\alpha)(F_0)\times G (\BA_{0,f}^{v_0})/K_G^{v_0}\bigr) } \phi^{v_0}_0(h^{-1}\delta h)\cdot \left[\LCM(\delta, h)\right]_{K_G^{v_0}},
$$
as an element in the group 
$\bigoplus_{g\in K_{ G,\fkd}\bs G(F_{0,\fkd})/ K_{ G,\fkd}}K_0'(\CCM\sphat(\alpha,g))$.

\end{altenumerate}
\end{proposition}
\begin{remark}
One can define an analog of the cycle $\LCM(\alpha,\phi_0)$ on a semi-global integral model (i.e., over the localization $O_{E,(\nu)}$ of $O_E$ at a place $\nu$ above $v_0$, cf. \cite[\S4]{RSZ3}) where one allows more general level structure $K_{G}^{v_0}$ away $v_0$, and therefore allows $\phi_0={\bf 1}_{K_{G,v_0}}\otimes\phi^{v_0}\in\CS(G(\BA_{0,f}))$ where $\phi^{v_0}\in \CS\bigl( G (\BA_{0,f}^{v_0}), K_G^{v_0}\bigr)$.
\end{remark}

\begin{proof}
We only prove part (i); part (ii) concerning the derived version follows along the same line.

Over the formal scheme \eqref{unif G}, 
$ \CCM(\alpha)$ consists of $G'(F_0) $-cosets of $(X, hK_{G}^{v_0})\in   \CN_{O_{\breve E_\nu}} \times  G (\BA_{0,f}^{v_0}) /K_{G}^{v_0} $ together with an isomorphism $\varphi_{v_0}: X\to X$ and $g\in G (\BA_{0,f}^{v_0})$, satisfying the following conditions: there exists $\delta\in G'(F_0) $ such that the  endomorphism of the framing object $\BX_n$ induced by $\varphi_{v_0}$ is $\delta$, and both $g$ and $\delta$ fix $  h K_{G}^{v_0}$ and they induce the same automorphism  of $  hK_{G}^{v_0}$; the polynomial $\alpha$ annihilates $g$ and $\varphi_{v_0}$ (or equivalently $\delta$ by the rigidity of quasi-isogeny).  In particular, $\delta\in G'(\alpha)(F_0)$ by the irreducibility of $\alpha$. 

Here we view $G (\BA_{0,f}^{v_0}) /K_{G}^{v_0} $ as a groupoid in which the automorphism group of $h K_{G}^{v_0}$ is isomorphic to $h K_{G}^{v_0} h^{-1}$. If both $\delta$ and $g$ fix $  hK_{G}^{v_0}$ and induce the same automorphism  of $  hK_{G}^{v_0}$, then $g=\delta$ (``rigidity away from $v_0$"). It follows that the condition $g$ fixing $  h K_{G}^{v_0}$ is equivalent to $\delta  h K_{G}^{v_0}= h K_{G}^{v_0}$, i.e.,  $h^{-1}\delta  h\in K_{G}^{v_0}$.

The condition on the existence of  a quasi-isogeny $\varphi_{v_0}$ lifting $\delta$ amounts to $X\in \CN_{O_{\breve E_\nu}} ^{\delta }  $.

Therefore, for a fixed $\delta\in G'(\alpha)(F_0)$, the desired pairs $(X, h K_{G}^{v_0})$ are exactly those lying on  $\CN_{O_{\breve E_\nu}} ^{\delta }  \times {\bf 1}_{h K_{G}^{v_0}}$ subject to the condition $h^{-1}\delta  h\in K_{G}^{v_0}$. Then it remains to sum over all $\delta\in G'(\alpha)(F_0)$ to complete the proof of part (i).
\end{proof}

\section{Modular generating functions of special divisors}
\label{s:gen div}

In this section we collect a few modularity results due to various authors for the generating functions of special divisors with valued in Chow groups, and in a reduced version of arithmetic Chow groups. 

\subsection{Generating functions of special divisors on $M_{K_{\wt G}}(\wt G)$} 
We first define the generating functions of special divisors on the canonical model $M_{K_{\wt G}}(\wt G)$ over $\Spec E$. The moduli functor is introduced at the end of \S\ref{s:int model}, for an arbitrary compact open subgroup $K_{\wt G}$ of the form $K_{\wt G}=K_{Z^\BQ}\times K_{ G}$.  For $\phi\in\CS(V(\BA_{0,f}))^{K_{ G}}$,  and $\xi\in F_{0,+}$, we have defined the divisor $Z(\xi, \phi)\in \Ch^1(M_{K_{\wt G}}(\wt G))$ by \eqref{global KR E}. When $\xi=0$, we define
\begin{equation}\label{KR C 0}
Z(0,\phi)=-\phi(0) \,c_1({\bf\omega})\in \Ch^1(M_{K_{\wt G}}(\wt G)),
\end{equation}
where $\bf\omega$ is the automorphic line bundle \cite{K-duke}, and $c_1$ denotes the first Chern class.

In \S\ref{ss:weil} we will recall the Weil representation $\omega$ of $\bH(\BA_{0,f})$ on $\CS(V(\BA_{0,f}))^{K_G}$. We define the generating function on $\bH(\BA_0)$ by 
\begin{equation}\label{gen  E}
Z(h,\phi)=Z(0,\omega(h_f)\phi)W^{(n)}_{0}(h_\infty)+\sum_{\xi\in F_{0,+}} Z(\xi,\omega(h_f)\phi)\, W^{(n)}_{\xi}(h_\infty),
\end{equation}
where $h=(h_\infty,h_f)\in\bH(\BA_{0})$, $h_\infty=(h_v)_{v\mid \infty} \in  \prod_{v\mid\infty}\SL_{2}(F_{0,v})$, and 
$$
W^{(n)}_{\xi}(h_\infty)=\prod_{v\mid \infty} W^{(n)}_{\xi}(h_v),
$$ cf. \eqref{Whit} for the weight $n$ Whittaker function $W^{(n)}_{\xi}$  on $\SL_2(\BR)$. 

 \begin{theorem}\label{thm:mod E}
 The generating function $Z(h,\phi)$ lies in $ \CA_{\rm hol}(\bH(\BA_0),K, n)_{\ov\BQ}\bigotimes_{\ov\BQ} \Ch^1(M_{K_{\wt G}}(\wt G))_{\ov\BQ}$, where $K\subset \bH(\BA_{0,f})$ is a compact open subgroup which fixes $\phi\in \CS(V(\BA_{0,f}))$ under the Weil representation. 
  \end{theorem}
 We refer to \eqref{def A hol W tot} (and \eqref{def A hol W}) for the definition of the vector space $ \CA_{\rm hol}(\bH(\BA_0),K, n)_{\ov\BQ}$. One can replace the field $\ov\BQ$ by a number field, but it will not be more useful in this paper.
  
The result has an analog for orthogonal Shimura varieties, which is due to Borcherds when $F_0=\BQ$ (generalizing Gross--Kohnen--Zagier theorem), and \cite{YZZ1} for totally real fields $F_0$; Bruinier also gave a proof in \cite{Br12} where he also constructed the automorphic Green function we will use later. By the embedding trick \cite[\S3.2, Lem.\ 3.6]{Liu1}, this result implies the analogous modularity for Shimura varieties $\Sh_{K_G}\bigl(\Res_{F_0/\BQ} G,\{h_G\}\bigr)$.\footnote{In the unitary case, one expect to obtain a $\RU(1,1)$-automorphic form. However, the $\SL_2$-automorphic form suffices for our purpose, and in fact the extra information in $\RU(1,1)$-modularity is not useful for us at all because the analytic side only has $\SL_2$-modularity.} Then the assertion in the theorem above follows from the fact that, after base change to $\BC$, $M_{K_{\wt G}}(\wt G)$ is a disjoint union of copies of $\Sh_{K_G}\bigl(\Res_{F_0/\BQ} G,\{h_G\}\bigr)$, cf. \eqref{prod shim}.

\subsection{Complex uniformization of special divisors} \label{ss:KR C}

We now study special divisors over the complex numbers. The situation is analogous to the description of the special divisors in the formal neighborhood of the basic locus, cf. \S\ref{ss: KR in ss loc}.

We start with the complex uniformization of our Shimura varieties. This is very much similar to the \eqref{eq unif2}. Let $\nu: E\incl \BC$ be a complex place of the reflex field $E$. Its restriction to $F$ ($F_0$, resp.) is a place denoted by $w_0$ ($v_0$, resp.).  Let $M_{\nu,\BC}=M_{K_{\wt G}}(\wt G)\otimes_{E,\nu}\BC$ be the complex orbifold  via $\nu$. Let $V'$ be the ``nearby" hermitian space, i.e., the unique one  that is positive definite at all archimedean places except $v_0$ where the signature is $(n-1,1)$, and  isomorphic to $V$ locally at all non-archimedean places. 
Then let $G'$ be the unitary group (viewed as a $\BQ$-algebraic group) associated to $V'$. Let $\CD_{v_0}$ be the Grassmannian of negative definite $\BC$-lines in $V'\otimes_{F,w_0}\BC$. Then we have a complex uniformization 
\begin{equation}\label{eq unif C}
  M_{\nu,\BC}\,	   = \wt G'(\BQ) \Big\bs \Bigl[  \CD_{v_0} \times \wt G (\BA_{f}) /K_{\wt G} \Bigr].
\end{equation}
When the embedding $\nu: E\incl \BC$ is the natural one for the reflex field $E$ (recall that it is a subfield of $\BC$), the uniformization is \cite[Rem.\ 3.2, Prop.\ 3.5]{RSZ3}, and in general it follows from the proof of {\it loc. cit.}. 

Analogous to \eqref{unif proj}, we have a partition by the projection 
\begin{equation}\label{unif C proj}
\xymatrix{    M_{\nu,\BC} \ar[r]&  Z^\BQ(\BQ)\bs\bigl(Z^\BQ(\BA_f)/K_{Z^\BQ}\bigr),}
\end{equation}
where each fiber is naturally isomorphic to 
\begin{equation}\label{unif G C}
  M_{\nu,\BC,0}\colon= G'(F_0) \Big\bs \Bigl[  \CD_{v_0} \times  G (\BA_{0,f}) /K_{G} \Bigr].
\end{equation}
 Here we  fix an isomorphism $G'(\BA_{0,f})\simeq G(\BA_{0,f})$.

Now we return to describe the complex uniformization of the special divisors. For each $u\in V'(F_0)$ with totally positive norm, let $\CD_{v_0,u}\subset \CD_{v_0}$ be the  space of negative definite $\BC$-lines perpendicular to $u$.\footnote{The codimension one analytic space $\CD_{v_0,u}$ on $\CD_{v_0}$ is the archimedean analog of the local KR divisor $\CZ(u)$ on $\CN$ in \S\ref{ss: KR in ss loc}.} For a pair $(u, g)\in V'(F_0)\times G(\BA_{0,f})/K_G$, we define 
\begin{align}\label{KR infty}
Z(u,g)_{K_G}= \CD_{v_0,u}\times {\bf 1}_{g \,K_G}.
\end{align}
We consider the sum
\begin{align}\label{KR infty qt}
\sum Z(u',g')_{K_G},
\end{align}
over $(u',g')$ in the $G'(F_0)$-orbit of the pair $(u,g)$ (for the diagonal action of $G'(F_0)$ on $V'(F_0)\times G (\BA_{0,f})/K_G$. The sum is $G'(F_0)$-invariant and hence descends to a divisor on the quotient \eqref{unif G C}, denoted by $[Z(u,g)]_{K_G}$.

Then, we have an archimedean analog of Proposition \ref{KR v0} for the special divisor $Z(\xi,\phi)$ defined  by \eqref{global KR E}. In the case of $F_0=\BQ$ and a special level structure, this is proved in \cite[\S3.3]{KR-U2}; in general, the proof in {\it loc. cit.} works verbatim and hence we omit the detail.
\begin{proposition}\label{prop KR infty}
Let $\xi\in F_{0,+}$. Then the restriction of the special divisor $Z(\xi,\phi)\otimes_{E,\nu}\BC$ to each fiber of the projection \eqref{unif proj}   is 
\begin{equation}\label{KR C xi}
 \sum_{(u,g)\in G'(F_0)\bs(V'_\xi(F_0)\times G (\BA_{0,f})/K_G) }\phi(g^{-1}u)\cdot [Z(u,g)]_{K_G}.
\end{equation}
\end{proposition}

\begin{remark}We may rewrite the above result into a form that has appeared in the formula of special divisors in
\cite[\S1]{YZZ1}. Let $G'_u\subset G'$ the stabilizer of $u$ under the action of $G'$  on $V'$, viewed as an algebraic group over $F_0$. Instead of \eqref{KR infty}, we define
\begin{align*}
\wt Z(u,g)_{K_G}:= \CD_{v_0,u}\times {\bf 1}_{G'_u(\BA_{0,f})\,g \,K_G}.
\end{align*}
Similarly we denote its image in the quotient \eqref{unif G C} by  $[\wt Z(u,g)]_{K_G}$. 
Then we may rewrite the sum as \eqref{KR C xi}
\begin{equation*}
 \sum_{u \in G'(F_0)\bs V'_\xi(F_0)}\, \sum_{g\in G'_u (\BA_{0,f})\bs G (\BA_{0,f})/K_G }\phi(g^{-1}u)\cdot [\wt Z(u,g)]_{K_G}.
\end{equation*}
This is exactly the formula in {\it loc. cit.}.
\end{remark}

\subsection{Green functions}\label{ss:K Green}
We recall the Green functions of Kudla \cite{K}, and the automorphic Green functions (cf. \cite{OT,Br12}. The former is more convenient when comparing with the analytic side, while the latter is more suitable for proving (holomorphic) modularity of generating series. The difference between them is studied by Ehlen--Sankaran in \cite{ES} when $F_0=\BQ$.

We first recall Kudla's Green functions,  defined for the orthogonal case in \cite{K} which can be carried over easily to the unitary case (cf. \cite[\S4B]{Liu1}). 
Let $u\in V'(F_0)$ be as in the previous subsection. Let $ z\in \CD_{v_0}$. 
 Let $u_z$ be the orthogonal projection to the negative definite $\BC$-line $z$ of $V'\otimes_{F,w_0}\BC$.
  Define
\begin{align}\label{Ruz}
R(u,z)=\pair{u_z,u_z}=\frac{\pair{u,\wt z}^2}{\pair{\wt z,\wt z}},
\end{align}
where $\wt z$ is any  $\BC$-basis of the line $z$.

We will need the exponential integral defined by
\begin{align}\label{Ei}
\Ei(-r)=-\int^{\infty}_r \frac{e^{-t}}{t}dt,\quad r>0.
\end{align}
This function has  a logarithmic singularity around $0$, more precisely, when $r\to 0^+$,
$$
\Ei(-r)=\gamma+\log r+ \sum_{n=1}^\infty\frac{(-r)^n}{n\cdot n!}.
$$
Here $\gamma$ is the Euler constant.

Let $h_\infty=(h_v)_{v\mid\infty}\in \prod_{v\mid\infty}\SL_{2}(F_{0,v})$ and $h_v=\left(\begin{matrix} 1& b_v\\
& 1\end{matrix}\right)\left(\begin{matrix} \sqrt{a_v}& \\
& 1/\sqrt{a_v}\end{matrix}\right)\kappa_{v}$ in the Iwasawa decomposition, cf. \eqref{h infty}. For each {\em non-zero} vector $u\in V(F_0)$, Kudla \cite{K} defined a Green function on $\CD_{v_0}$, parameterized by $h_\infty$
\begin{align}\label{Gr Ku1}
{\bf\CG}^{\bf K}(u, h_\infty)(z)=-{\rm Ei}(2\pi a_{v_0}\, R(u,z)),\quad z\in \CD_{v_0}\setminus \CD_{v_0,u}.
\end{align}
It has logarithmic singularity along the divisor $\CD_{v_0,u}$. Note that this is defined for {\em every} non-zero vector $u\in V'(F_0)$ (in particular,  $u$ may have null-norm). If $\CD_{v_0,u}$ is empty, the function is then smooth on $\CD_{v_0}$. When $u=0$, we set
\begin{align}\label{Gr Ku m=0}
{\bf\CG}^{\bf K}(0, h_\infty)=-\log|a_{v_0}|.
\end{align}

Now we descend the Green function on $\CD_{v_0}$ to the quotient \eqref{unif G C}: for all $\xi \in F_0$, define
\begin{align}\label{Gr Ku2}
{\bf\CG}^{\bK}(\xi,h_\infty, \phi)=\sum \phi(g^{-1}u)\cdot \left({\bf\CG}^{\bf K}(u, h_\infty)\times  {\bf 1}_{g \,K_G} \right) 
\end{align}
where the sum is over the double coset $(u,g)\in G'(F_0)\bs\left(V'_\xi(F_0)\times G(\BA_{0,f})/K_G\right)$. This defines a Green function for the divisor $Z(\xi,\phi)$, cf. \cite[Prop.\ 4.9]{Liu1}.

We now recall the automorphic Green function \cite{OT,Br12,BHKRY}. Since the role of them are indirect to this paper, we just say that there is a Green function $\CG^{\bf B}(\xi,\phi)$ for each $\xi\in F_{0,+}$, and $\phi\in\CS(V(\BA_{0,f}))$, cf. \cite[\S7.3]{BHKRY}.

We define the generating function of the difference of the two Green functions 
\begin{equation}\label{geo error}
\CZ_{v_0,\corr}(h,\phi)\colon=\sum_{\xi\in F_0}\left( \CG^{\bK}(\xi,h_\infty,\omega(h_f)\phi)-{\bf\CG}^{\bB}(\xi,\omega(h_f)\phi)\right) W^{(n)}_\xi(h_\infty),
\end{equation}
where the notation is the same as \eqref{gen  E}. We note that this definition depends on the archimedean place $v_0$ of $F_0$, though it is omitted in the right hand side of the equality.

The following theorem is due to Ehlen--Sankaran \cite{ES}.
 \begin{theorem}\label{thm ES}
 Assume that $F_0=\BQ$.
 The generating function $\CZ_{\infty,\corr}(h,\phi)$
  lies in the space $\CA_{\rm exp}(\bH(\BA_0), K,n)$, in the sense that, for every point $[z,g]\in   M_{\nu,\BC}$, the value of the generating functions at $[z,g]$ lies in $\CA_{\rm exp}(\bH(\BA_0), K, n)$. Here $K\subset \bH(\BA_{0,f})$ is a compact open subgroup which fixes $\phi\in \CS(V(\BA_{0,f}))$ under the Weil representation.
 \end{theorem}
 
 \begin{proof}
 In \cite[Th.\ 3.6]{ES}, the authors proved the assertion for orthogonal groups, from which the case of unitary groups follows (e.g, by \cite[the proof of Th.\ 4.13, p. 2131]{ES}).
 \end{proof}

\subsection{Modularity in the arithmetic Chow group $\wh\Ch^1_{\circ}(\CM)$}
\label{ss:arCh}

We will use the Gillet--Soul\'e  arithmetic intersection theory cf.\  \cite{GS,Gi} (in the non-proper case, cf. \cite{BKK}). We first recall the arithmetic Chow group $\wh\Ch^1(\CM)$ (with $\BQ$-coefficient) for a regular flat scheme (possibly non-proper) $\CM\to\Spec O_E$. Elements are represented by arithmetic divisors, i.e., $\BQ$-linear combinations of tuples $\left(Z, (g_{Z,w})_{w\in\Hom_\BQ(E,\ov\BQ)}\right)$, where $Z$ is a divisor on $\CM$ and $g_{Z,w}$ is a Green function of $Z_{w}(\BC)$ on the complex manifold $\CM_{w}(\BC)$ via the embedding $w:E\incl\ov \BQ\subset\BC$ (cf.\ \cite[\S3.3]{GS}). 
Principal arithmetic divisors are tuples associated to rational functions $f\in E(\CM)^\times$:
$$
\left(\div(f),(- \log |f|^2_{w})_{ w\in\Hom_\BQ(E,\ov\BQ)}\right).
$$
(e.g., when $E=\BQ$, we have $\bV_\fkp=(0,2\log|p|)$ in $\wh\Ch^1(\CM)$, where $\bV_p$ is the fiber of $\CM$ over a prime $p$.)

Now it is clear we can extend the same definition to a regular flat scheme $\CM\to\Spec O_E \setminus S$ for a finite set $S$ of {\em non-archimedean} places. We still denote it by $\wh\Ch^1(\CM)$. 

\begin{remark}If we start with a regular flat scheme $\CM\to\Spec O_E$, and a finite set $S$, then two groups $\wh\Ch^1(\CM)$ and $\wh\Ch^1(\CM^S)$ for $\CM^S=\CM\times_{\Spec O_E}\Spec O_E  \setminus S$ are related as follows. We denote by $\Ch_{|S|}^1(\CM)$ the subgroup of $\wh\Ch^1(\CM)$ consisting of elements supported at the fibers above $\nu\in S$. This is a finite dimensional vector space. Then there is a natural isomorphism
$$\xymatrix{
 \wh\Ch^1(\CM)/\Ch_{|S|}^1(\CM)\ar[r]^-\sim&\wh\Ch^1(\CM^S).
}$$
\end{remark}
 
 Now we specialize to our interest, the moduli space $\CM=\CM_{K_{\wt G}}(\wt G)$ introduced in Definition \ref{def RSZ glob}. Let $S$ be the set of places  $\nu\mid\fkd$. Recall that the morphism $\CM_{K_{\wt G}}(\wt G)\to\Spec O_{E}[1/\fkd]=\Spec O_E\setminus S$ is smooth.
 
 Let $\phi\in \CS(V(\BA_{0,f}))^{K_G}$ be of the form $\phi={\bf 1}_{\Lambda^\fkd}\otimes\phi_\fkd$ (cf. \eqref{global KR}). For $\xi\in F_{0,+}$, we  endow the special divisor $\CZ(\xi,\phi)$ (cf. \eqref{global KR}) with the automorphic Green function $\CG^{\bf B}(\xi,\phi)$. Denote by $\wh\CZ^\bB(\xi,\phi)$ the resulting element in $\wh\Ch^1 (\CM)$.
 When $\xi=0$, we define
\begin{equation}\label{KR C ar}
\CZ^\bB(0,\phi)=-\phi(0) \,c_1({\bf\wh\omega})\in \wh\Ch^1(\CM),
\end{equation}
where $\wh{\bf\omega}=({\bf\omega},|\!|\cdot|\!|_{\rm Pet})$ is the extension of the automorphic line bundle $\omega$ to the integral model $\CM$, endowed with its Petersson metric \cite[\S7.2]{BHKRY}. 

We define the generating series with coefficients in the arithmetic Chow group $\wh\Ch^1(\CM)$
\begin{align}\label{eq: mod gen OE}
\wh\CZ^{\bB}(\tau,\phi)=\sum_{\xi\in F_{0},\,\, \xi\geq 0}\wh\CZ^\bB(\xi, \phi)\, q^\xi,
\end{align}
 where
\begin{align}\label{eq: q xi}
 \tau=(\tau_v )_{v\mid \infty}\in \prod_{v\mid\infty}\CH,\quad q^\xi\colon=e^{2\pi i \tr_{F_0/\BQ} (\tau \xi)}.
\end{align}
 The following theorem can be deduced from  \cite{BHKRY}.
  \begin{theorem}[Bruinier--Kudla--Howard--Rapoport--Yang]
\label{thm BHKRY}
 Let $F_0=\BQ$.
 The generating series $\wh\CZ^{\bB}(\cdot,\phi)$  lies in $\CA_{\rm hol}(\Gamma(N), n)_{\ov\BQ}\bigotimes_{\ov\BQ}\wh\Ch^1(\CM)_{\ov\BQ}$, where $N$ depends only on $\phi$ and all prime factors of $N$ are contained in $S$.
 \end{theorem}
 
 \begin{proof}In  \cite{BHKRY}
 the authors proved a stronger version (i.e., Theorem B in {\it loc. cit.})  in a maximal level case (with principle polarization) over the full ring of integers of $E$. Since the arithmetic Chow group of $\CM$ considered here omits a finite set of bad places $S$ (including primes ramified in $F$), the computation of divisors of the regularized theta lifts and Borcherds product on the integral models over $\Spec O_E[1/\fkd]$ of {\it loc. cit.} still applies to our (even simpler) situation.  

 \end{proof}

\section{Local intersection: non-archimedean places}
\label{s:int}

\subsection{Arithmetic intersection theory}\label{ss:AIT}We first recall an arithmetic intersection pairing on
 a pure dimensional flat (not necessarily proper) morphism $\CM\to\CB=\Spec O_E$ of regular schemes with smooth generic fiber. Let $\wt\CZ_{1,c}(\CM)$ be the group of proper ({\em over the base $\CB$}) 1-cycles on $\CM$ (with $\BQ$-coefficient).
Then there is an arithmetic intersection pairing between two $\BQ$-vector spaces (cf.  \cite[\S2.3]{BGS}  when the ambient scheme is proper)
\begin{align}\label{eqn AIT S}
 \begin{gathered}
	\xymatrix@R=0ex{
(\cdot,\cdot)\colon\quad \wh\Ch^1(\CM)\times\wt \CZ_{1,c}(\CM)  \ar[r] & \BR.
	}
	\end{gathered}
\end{align}

Now let $S$ be a finite set of non-archimedean places of $E$, and $S_p$ the subset of places above $p$.  Let $\CM\to \Spec O_{E}\setminus S$, and $\wh\Ch^1(\CM)$ its arithmetic Chow group defined in \S\ref{ss:arCh}. Consider the quotient of $\BR$ by a finite dimensional $\BQ$-vector space:
\begin{align}
\label{RS}
\BR_S:=\BR/{\rm span}_\BQ\{\log p:\# S_p\neq 0\}
\end{align}
which is an (infinite dimensional) $\BQ$-vector space. Then the definition of \cite{BGS} works directly if we replace the base $\Spec O_E$ by $\Spec O_{E} \setminus S$ (i.e., without an integral model over the full ring of integers $O_E$), and yields a pairing with valued in $\BR_S$
 \begin{align}\label{eqn AIT S0}
 \begin{gathered}
	\xymatrix@R=0ex{
(\cdot,\cdot)\colon \quad\wh\Ch^1(\CM)\times\wt \CZ_{1,c}(\CM)   \ar[r] & \BR_S.
	}
	\end{gathered}
\end{align}
 Note that the cycles in  $\wt \CZ_{1,c}(\CM) $  are assumed to be {\em  proper} over $\Spec O_{E} \setminus S$. 
 
 We note that, by  \cite[Prop.\ 2.3.1 (ii)]{BGS}, for cycles in $\wt \CZ_{1,c}(\CM) $ supported on special fibers, the pairing only depends on their rational equivalence classes. This motivates us to define a quotient group 
$ \CZ_{1,c}(\CM) $  of $\wt \CZ_{1,c}(\CM)$ by the subgroup generated by 1-cycles that are supported on {\em proper} subschemes $Y$ of the special fibers and are rationally equivalent to zero on $Y$. We have the resulting pairing
 \begin{align}\label{eqn AIT S}
 \begin{gathered}
	\xymatrix@R=0ex{
(\cdot,\cdot)\colon \quad\wh\Ch^1(\CM)\times \CZ_{1,c}(\CM)   \ar[r] & \BR_S.
	}
	\end{gathered}
\end{align}

We now let  $\CM=\CM_{K_{\wt G}}(\wt G)$ be the moduli stack introduced in Definition \ref{def RSZ glob}. 
We apply the above pairing to $\CM=\CM_{K_{\wt G}}(\wt G)\to\CB=\Spec O_{E}\setminus S$, for any finite set $S$ containing all places $\nu\mid\fkd$. We define an element in $\CZ_{1,c}(\CM)$ starting from the derived CM cycle $\LCM(\alpha,g)$ \eqref{def LCM}, which is an element in $F_1\,K_0'(\CCM(\alpha,g))$, \eqref{def LCM fil}. The finite morphism $\CCM(\alpha,g)\to \CM$ induces a homomorphism 
$$
K_0'(\CCM(\alpha,g))\to K_{0,\CCM(\alpha,g)}'(\CM)
$$
preserving the respective filtrations, where $K_{0,\CCM(\alpha,g)}'(\CM)$ denotes the $K$-group of coherent sheaves with support on the image of $\CCM(\alpha,g)$. Since  $\CCM(\alpha,g)\to \CB$ is proper and the generic fiber of $\CCM(\alpha,g)$ is zero dimensional (cf., Prop. \ref{pro bad fiber} (b)),  there is a natural homomorphism $\Ch_{1,\CCM(\alpha,g)}(\CM)\to \CZ_{1,c}(\CM)$.  We now consider the composition
$$\xymatrix{ F_1K_0'(\CCM(\alpha,g))\ar[r]&
{\rm Gr }_1 K_{0,\CCM(\alpha,g)}'(\CM)\ar[r]^\sim& \Ch_{1,\CCM(\alpha,g)}(\CM)\ar[r]&\CZ_{1,c}(\CM) ,}
$$
where the isomorphism in the middle is \cite[Th.\ 8.2]{GS87}, and ${\rm Gr }_1$ denotes the grading $F_1/F_0$. By abuse of notation, we still denote by $\LCM(\alpha,g)$  the image in $\CZ_{1,c}(\CM)$ of the element $\LCM(\alpha,g)\in F_1K_0'(\CCM(\alpha,g))$ (cf. \eqref{def LCM fil})  under the above composition. 
\subsection{Intersection of special divisors and CM cycles}\label{ss:Int mod}
For the rest of the article, we let $\Phi=\otimes_{v_0}\Phi_{v_0}\in \CS((G\times V)(\BA_{0,f}))$ be of the form $\phi_0\otimes\phi$, where 
\begin{itemize}
\item
 $\phi_0={\bf 1}_{K_G^\fkd}\otimes\phi_{0,\fkd}$ and $\phi_{0,\fkd}\in \CS\left(G(F_{0,\fkd}),K_{G,\fkd}\right)$ (cf. \eqref{CM phi0}), and
\item
  $\phi={\bf 1}_{\Lambda^\fkd}\otimes\phi_\fkd$ and $\phi_\fkd\in \CS(V(F_{0,\fkd}))^{K_{G,\fkd}}$ (cf. \eqref{global KR}).
\end{itemize}
Recall from \S\ref{ss:FB CM} that we have also fixed a conjugate self-reciprocal polynomial $\alpha\in O_{F}[1/\fkd][T]_{\deg=n}$, irreducible over $F$. We define a generating series using the intersection pairing \eqref{eqn AIT S}
\begin{align}\label{int g0}
\Int(\tau,\Phi)\colon
=\frac{1}{\tau(Z^\BQ)\cdot  [E:F]} \left( \wh \CZ^{\bB}(\tau, \phi),\quad  \LCM(\alpha,\phi_0)\right),
\end{align}
where $\CZ^{\bB}(\tau, \phi)$ is \eqref{eq: mod gen OE}, and 
\begin{align}\label{def tau Z}
\tau(Z^\BQ):= \#Z^\BQ(\BQ)\bs\bigl(Z^\BQ(\BA_f)/K_{Z^\BQ}\bigr).
\end{align}
\begin{remark}
By Theorem \ref{thm BHKRY}, when $F_0=\BQ$, this is a holomorphic modular form (of weight $n$, and level depending only on $\phi$) with coefficients in the $\ov\BQ$-vector space $\BR_{S,\ov\BQ}:=\BR_S\otimes_{\BQ}\ov\BQ$, i.e.,
\begin{align}\label{eq: Int mod}
\Int(\cdot,\Phi)\in \CA_{\rm hol}(\Gamma(N), n)_{\ov\BQ}\otimes_{\ov\BQ}\BR_{S,\ov\BQ}.
\end{align}
Our results in this and the next section are still valid for general totally real fields $F_0$ since they do not use the modularity.
\end{remark}

Similarly we define for each $\xi\in F_{0,+}$
\begin{align}\label{int g0 xi}
\Int(\xi,\Phi)\colon
=\frac{1}{\tau(Z^\BQ)\cdot  [E:F]} \left(\wh \CZ^\bB(\xi, \phi) ,\quad\LCM(\alpha,\phi_0) \right).
\end{align} 
When $\xi=0$, this is  by definition
 \begin{align}\label{int g0 xi0}
\Int(0,\Phi)
=-\frac{1}{\tau(Z^\BQ)\cdot  [E:F]} \left({\bf \wh{\omega}} ,\quad  \LCM(\alpha,\phi_0)\right)\phi(0).
\end{align}
Then by \eqref{eq: mod gen OE}, 
\begin{align}\label{eq: Int=sum xi}
\Int(\tau,\Phi)
=\sum_{\xi\in F_0,\,\xi\geq 0} \Int(\xi,\Phi) \,q^\xi.
\end{align}

Now let $\xi\neq 0$. We will  express the arithmetic intersection number \eqref{int g0 xi} in terms of the local intersection numbers from the AFL over good places and the archimedean local intersection.

\subsection{The support of the intersection}

We first study the intersection of the special divisor $\CZ(\xi,\phi)$ and the CM cycle $\LCM(\alpha,\phi_0)$. First we have the following analog to \cite[Th.\ 8.5]{RSZ3}.

\begin{theorem}\label{int generic}
Let $\xi\neq 0$ and $\Phi=\otimes_{v_0}\Phi_{v_0}\in \CS((G\times V)(\BA_{0,f}))^{K_G}$. Let $S$ be a finite set of places containing all places $\nu\mid\fkd$ and such that  at $v_0\notin S$, $\Phi_{v_0}={\bf 1}_{K^\circ_{G, v_0}}\otimes {\bf 1}_{\Lambda_{v_0}^\circ}$.

 Then the following statements on the  support of the intersection of the special divisor $\CZ(\xi,\phi)$ and the CM cycle $\CCM(\alpha,\phi_0)$ on $\CM $ hold. 
\begin{altenumerate}
\item\label{int generic i} The support does not meet the generic fiber. 
\item\label{int generic ii} Let $\nu\notin S$ be a place of $E$ lying over a place of $F_0$ which splits in $F$. Then the  support does not meet the special fiber 
   $\CM\otimes_{O_E} \kappa_{\nu}$. 
\item\label{int generic iii} Let $\nu\notin S$ be a place of $E$ lying over a place of $F_0$ which does not split in $F$. Then the support meets the special fiber 
   $\CM \otimes_{O_E} \kappa_{\nu}$ only in its basic locus.
\end{altenumerate}
\end{theorem}

\begin{proof}
The proof of \cite[Th.\ 8.5]{RSZ3} goes through verbatim (since $\alpha$ is irreducible over $F$, the pair $(g,u)$ is regular semisimple for any non-zero vector $u$  in $V(F_0)$).
\end{proof}

Since their generic fibers do not intersect by Theorem \ref{int generic}, the intersection pairing $\Int(\xi,\Phi)$ localizes to a sum over all places of $E$. We define
\begin{equation}\label{int for globwith}
 \begin{aligned}
  \Int_\nu^\natural (\xi,\Phi)&:=\bigl\la  \wh \CZ^\bB(\xi, \phi) ,\quad  \LCM(\alpha,\phi_0)\ra_{{\nu}}\, \log q_\nu,
   \end{aligned}
\end{equation}
where $q_\nu$ is the cardinality of the residue field of $O_{E, (\nu)}$ for non-archimedean $\nu$ (see below for the archimedean case). Here we recall that the local intersection number $\bigl\la\cdot,\cdot\ra_{{\nu}}$  is defined for a non-archimedean place $\nu$ through the Euler--Poincar\'e characteristic of a derived tensor product on $\CM\otimes_{O_E} O_{E, (\nu)}$, cf.~\cite[4.3.8(iv)]{GS}.  For an archimedean place $\nu$, the local intersection number is the value of the Green function at the complex point of the CM cycle:
\begin{equation}\label{int for globwith}
  \Int_\nu^\natural (\xi,\Phi):=\bigl\la  \CG^\bB_\nu(\xi, \phi) ,\quad\LCM(\alpha,\phi_0)_{\nu,\BC} \ra \,\log q_\nu,
\end{equation}
where by definition $\log q_\nu=2$ for complex places $\nu$ (and $1$ if $\nu$ were a real place).

 For a place $v_0$ of $F_0$, we set
  \begin{equation}\label{intloc}
   \Int_{v_0}(\xi,\Phi) :=\frac{1}{\tau(Z^\BQ)\cdot  [E:F]} \sum_{\nu\mid v_0} \Int^\natural_\nu(\xi,\Phi).
\end{equation}
Then we have a decomposition into a sum over places $v_0$ of $F_0$
\begin{equation}\label{locInt}
   \Int(\xi,\Phi) = \sum_{v_0} \Int_{v_0}(\xi,\Phi).
\end{equation}
Combining \eqref{eq: Int=sum xi}, we obtain a decomposition of the generating function of arithmetic intersection numbers
\begin{align}\label{eq: Int=sum v}
\Int(\tau,\Phi)=\Int(0,\Phi)+\sum_{v_0} \Int_{v_0}(\tau,\Phi) ,
\end{align}
where 
\begin{align}\label{eq: Int v=sum xi}
\Int_{v_0}(\tau,\Phi)\colon=\sum_{\xi\in F_{0,+}} \Int_{v_0}(\xi,\Phi)\,q^\xi ,
\end{align}

\begin{corollary} \label{coro split}(to Theorem \ref{int generic}) If $v_0$ is split in $F/F_0$, then
\begin{align}
 \Int_{v_0}(\xi,\Phi)=0.
 \end{align}
 \end{corollary}
 \subsection{Local intersection: inert non-archimedean places}
 Now let $v_0$ be a  place of $F_0$ inert in $F$, and $w_0$ the unique place of $F$ above $v_0$. The notation here follows \S \ref{ss: KR in ss loc}.

\begin{theorem}\label{thm inert}
Assume that $v_0\nmid \fkd$ and $\Phi=\Phi_{v_0}\otimes\Phi^{v_0}$ where
 $$\Phi_{v_0}={\bf 1}_{K^\circ_{G, v_0}}\otimes {\bf 1}_{\Lambda_{v_0}^\circ}.$$  
 Then 
\begin{equation}\label{sum inert}
	\Int_{v_0}(\xi,\Phi) =  
2\log q_{v_0} \sum_{(\delta,u)\in [(G'(\alpha)\times V'_\xi)(F_0)]} \Int_{v_0}(  \delta,u ) \cdot \Orb\left((\delta,u), \Phi^{v_0}\right).
\end{equation}
Here $\Int_{v_0}(  \delta,u )$ is the quantity defined in the AFL conjecture (semi-Lie algebra version) for the unramified quadratic extension $F_{w_0}/F_{0,v_0}$, cf. \eqref{def int g u}, and the orbital integral is the product of the local orbital integral defined by \eqref{eq:orb U lie} with Haar measures on $G(F_{0,v})$ such that $\vol(K_{G,v})=1$. 
 
\end{theorem}

\begin{proof}
The proof goes along a similar line to \cite[Th.\ 3.11]{Z12} and \cite[Th.\ 8.15]{RSZ3}.

First, by Theorem \ref{int generic} (iii), the intersection only takes place in the basic locus. Hence it suffices to consider the question in the formal completion along the basic locus. We now fix a place $\nu$ of $E$ above $v_0$. Now by Proposition 
\ref{KR v0}, and Proposition \ref{CCM v0}, it suffices to consider the intersection number for each fiber of the projection \eqref{unif proj}, and multiply the result by the factor  $\tau(Z^\BQ)$ (hence canceling the factor $\tau(Z^\BQ)$ in the denominator of \eqref{intloc}). Therefore we consider only the intersection on the fiber $\CM_{  O_{\breve E_\nu}, 0}\sphat$, cf. \eqref{unif G}. 

Recall that by Proposition \ref{KR v0}, the restriction to $\CM_{  O_{\breve E_\nu}, 0}\sphat$ of the special divisor $ \CZ(\xi,\phi)$ is
 \begin{equation*}
 \sum_{(u,g')\in G'(F_0)\bs(V'_\xi(F_0)\times G (\BA_{0,f}^{v_0})/K_G^{v_0}) }  \phi^{v_0}(g'^{-1}u)
 \cdot \bigl[\CZ(u,g')\bigr]_{K_{G}^{v_0}},
 \end{equation*}
 and by Proposition \ref{CCM v0} the restriction of the derived CM cycle $\LCM(\alpha,\phi_0)$ is the sum
\begin{align*}
\sum_{(\delta ,h)\in G'(F_0)\bs(G'(\alpha)(F_0)\times G (\BA_{0,f}^{v_0})/K_G^{v_0}) } \phi^{v_0}_0(h^{-1}\delta h)\cdot \left[\LCM(\delta, h)\right]_{K_G^{v_0}}.
\end{align*}
We may compute the intersection number by pulling-back to the covering formal scheme $\CN_{O_{\breve E_\nu}} \times  G (\BA_{0,f}^{v_0}) /K_{G}^{v_0}$
 in the uniformization \eqref{unif G}. The intersection number  $ \LCM(\alpha,\phi_0)\jiao \CZ(\xi,\phi)\,\log q_{\nu}$  (restricted to $\CM_{  O_{\breve E_\nu}, 0}\sphat$) is equal to a sum  of 
$$ 
 \phi^{v_0}_0(h^{-1}\delta h) \phi^{v_0}(g'^{-1}u)  \cdot 
 \LCM( \delta, h)_{K_{G}^{v_0}}\,\jiao  \CZ(u,g')_{K_{G}^{v_0}} \cdot \,\log q_\nu,
 $$
over $G'(F_0)$-orbits (via diagonal action) of tuples $(\delta,h, u,g')$:
$$(\delta,h)\in G'(\alpha)(F_0)\times G (\BA_{0,f}^{v_0})/K_G^{v_0},\quad \text{ and }\quad (u,g')\in V'_\xi(F_0)\times G (\BA_{0,f}^{v_0})/K_G^{v_0}.$$ 
Here, we are abusing the notation $\jiao$ to denote the Euler--Poincare characteristics of the corresponding derived tensor product.

By \eqref{eq:KR v0} and \eqref{CM v0}, we obtain
$$
 \LCM( \delta, h)_{K_{G}^{v_0}}\,\jiao  \CZ(u,g')_{K_{G}^{v_0}}\cdot \log q_\nu
=\LN_{O_{\breve E_\nu}}^{\,\delta}\,\jiao\, \CZ(u)_{O_{\breve E_\nu}}  \log q_\nu
\cdot {\bf 1}_{ K_G^{v_0}}(g'^{-1}h).
$$
The first term  is equal to  
\begin{align*}
\LN_{O_{\breve E_\nu}}^{\,\delta}\,\jiao\, \CZ(u)_{O_{\breve E_\nu}}\,\log q_\nu
&=[E_{\nu}:F_{w_0}]\cdot\left(\LN^{\,\delta}\,\jiao\, \CZ(u)\right)\,\log q_{w_0}\\
&=2[E_{\nu}:F_{w_0}]\cdot \Int_{v_0}( \delta,u )\log q_{v_0}.
 \end{align*}
 Here the factor $2$ is due to $q_{w_0}=q_{v_0}^2$. 
In particular, it is invariant under the (diagonal) action of $G'(F_0)$ on the product $(G'(\alpha)\times V'_\xi)(F_0)$.

The second term $(g',h)\in (G(\BA_{0,f}^{v_0})/K_G^{v_0})^2\mapsto {\bf 1}_{  K_G^{v_0}}(g'^{-1}h)$ is also invariant under the (diagonal) $G'(F_0)$-action. 
For a fixed pair $(\delta,u)$,  we obtain
\begin{eqnarray*}
 \sum_{(g',h) \in     (G(\BA_{0,f}^{v_0})/K_G^{v_0})^2} && \phi^{v_0}_0(h^{-1}\delta h) \phi^{v_0}(g'^{-1}u)\cdot {\bf 1}_{  K_G^{v_0}}(g'^{-1}h) 
\\ &=&\sum_{h\in G(\BA_{0,f}^{v_0})/K_G^{v_0}}\phi^{v_0}_0(h^{-1}\delta h) \phi^{v_0}(h^{-1}\cdot u)\\
&=&\int_{G (\BA_{0,f}^{v_0}) } \phi_0^{v_0}(h^{-1}\delta\, h) \phi^{v_0}(h^{-1}\cdot u) \, \,dh
\\&=&\Orb\left((\delta,u), \Phi^{v_0}\right),
\end{eqnarray*}
 where we note that the Haar measure on $G'(\BA_f^{v_0})$ is normalized such that $\vol(K_G^{v_0})=1$.

To summarize, the intersection number  $ \LCM(\alpha,\phi_0)\jiao \CZ(\xi,\phi)\,\log q_{\nu}$  (restricted to $\CM_{  O_{\breve E_\nu}, 0}\sphat$) is equal to
$$2[E_{\nu}:F_{w_0}]
\sum_{(\delta,u)}  \Orb\left((\delta,u), \Phi^{v_0}\right)\cdot \Int_{v_0}( \delta,u )\log q_{v_0},
$$
where the sum is over $G'(F_0)$-orbits of pairs $ (\delta,u)\in (G'(\alpha)\times V'_\xi)(F_0).$

Finally the sum over all places $\nu\mid v_0$ will cancel the factor $[E:F]$ in \eqref{intloc}, by $$\sum_{\nu\mid w_0}e_{\nu/w_0}f_{\nu/w_0}=\sum_{\nu\mid w_0}d_{\nu/w_0}=[E:F],$$
where $e_{\nu/w_0}$ (resp., $f_{\nu/w_0}, d_{\nu/w_0}$) denotes the ramification degree (resp., inert degree, degree) of the extension $E_\nu/F_{w_0}$. This completes the proof.
\end{proof}

\section{Local intersection: archimedean places}\label{s:arch ht}
The goal of this section is to compute the local intersection at $\nu$ of $E$ above an archimedean place  $v_0$ of $F_0$. In fact we  will  replace the automorphic Green function by Kudla's Green function, i.e., we consider the analog of \eqref{int for globwith}:
\begin{equation}\label{int for infty K}
  \Int_\nu^{\natural,\bK} (\xi,\Phi):=\bigl\la  \CG^\bK_\nu(\xi, \phi) ,\quad\LCM(\alpha,\phi_0)_{\nu,\BC} \ra \,\log q_\nu.
\end{equation}
When $F_0=\BQ$ the difference is addressed by Theorem   \ref{thm ES}. Similar to \eqref{intloc}, we set for $\xi\in F_0$,
  \begin{equation}\label{intloc K}
   \Int_{v_0}^\bK(\xi,\Phi) :=\frac{1}{\tau(Z^\BQ)\cdot  [E:F]} \sum_{\nu\mid v_0} \Int^{\natural,\bK}_\nu(\xi,\Phi).
\end{equation}
We note that by \eqref{Gr Ku1} and \eqref{Gr Ku m=0}, there is a parameter $h_\infty\in \bH(F_0\otimes_\BQ \BR)$ implicitly in the above expression.

The strategy is analogous to Theorem \ref{thm inert}.  We follow the notation in \S\ref{ss:KR C} and \S\ref{ss:K Green}. 

\begin{theorem}\label{thm infty}
Let $\Phi\in\CS((G\times V)(\BA_{0,f}))$. Let $\xi\neq 0$. Then we have
\begin{equation}\label{sum infty}
	\Int_{v_0}^{\bK}(\xi,\Phi) =  
\sum_{(\delta,u)\in [(G'(\alpha)\times V'_\xi)(F_0)]} \Int_{v_0}(\delta,u ) \cdot \Orb\left((\delta,u), \Phi\right).
\end{equation}
 Here $\Int_{v_0}(  \delta,u )$ is  defined  as the special value of the function 
\begin{align}\label{ht infty}
\Int_{v_0}(  \delta,u )= \,{\bf\CG}^{\bf K}(u, h_\infty)(z_\delta),
\end{align}
where $z_\delta$ is the unique fixed point of $\delta$ on $\CD_{v_0}$. Moreover, the point $z_\delta$ does not lie on any $\CD_{v_0, u}$ for non-zero vector $u\in V'(F_0)$.

\end{theorem}

\begin{proof}

The proof goes along the same line as that of Theorem \ref{thm inert}, so we will not repeat the detail, except to prove the claim on the point $z_\delta$. Consider the $n$-dimensional $\BC$-vector space $V'\otimes_{F,w_0}\BC$ with the induced hermitian form.
 If a negative definite $\BC$-line is fixed by $\delta$,  it must be an eigen-line for $\delta$, which must be unique by the signature $(n-1,1)$ condition on the hermitian form on $V'\otimes_{F,w_0}\BC$. If $z_\delta$  lies on a divisor $\CD_{v_0, u}$ for non-zero vector $u\in V'(F_0)$, it also lies on $\CD_{v_0, \delta^i\cdot u}$, the translation of $\CD_{v_0, u}$ under $\delta^i$, for all $i\in\BZ$. Equivalently,  the line $z_\delta$ is perpendicular to all $\delta^i\cdot u\in V'\otimes_{F,w_0}\BC$. Since $u\in V_0(F)$ is non-zero vector, and its characteristic polynomial $\alpha$ of $\delta$ is irreducible over $F$, the vectors $\delta^i \cdot u$ span $V'$ over $F$, hence they also span $V'\otimes_{F,w_0}\BC$ over $F_{w_0}$. Contradiction!

\end{proof}

It remains to compute \eqref{ht infty}, or equivalently $R(u,z_{\delta})$ defined by \eqref{Ruz}. The element $\delta\in G'(\alpha)(F_0)$ induces an action of the CM field $F'$ (cf. \eqref{eqn:def F'}) on $V'$ and makes $V'$ into a one-dimensional $F'/F_0'$-hermitian space $(W, \pair{\cdot,\cdot}_{F_0'})$ satisfying 
\begin{align}\label{Herm F'/F}
 \begin{gathered}
	\xymatrix@R=0ex{
(\RR_{F'/F}W,\, \tr_{F'/F} \pair{\cdot,\cdot}_{F_0'})     \ar[r]^-\sim & (V', \, \pair{\cdot,\cdot})
	}
	\end{gathered}.
\end{align}
Here $\RR_{F'/F}W$ denotes the ``restriction of scalar" of $W$, i.e., to view it as an $F$-vector space.

\begin{remark}
The above construction $\delta\in G'(\alpha)(F_0)\mapsto (W, (\cdot,\cdot)_{F_0'}) $  defines a bijection between the set of $G'(F_0)$-conjugacy classes in $G'(\alpha)(F_0)$ and  the set of one-dimenional $F'/F_0'$-hermitian spaces (up to isometry) satisfying \eqref{Herm F'/F}. In fact, fixing a $\delta_0\in G'(\alpha)(F_0)$, we denote by $W_0$  the associated $F'/F_0'$-hermitian space, and $T$ the centralizer of $\delta_0$ in $G'$. Then $T$ is an anisotropic $F_0$-torus isomorphic to $\Res_{F_0'/F_0}\RU(W_0)$. Now the  pointed set of $G'(F_0)$-conjugacy classes in $G'(\alpha)(F_0)$ (with the conjugacy class of $\delta_0$ as  the distinguished element) is bijective to the pointed set $\ker(H^1(F_0,T)\to H^1(F_0,G'))$. Moreover, the pointed set $\ker(H^1(F_0,T)\to H^1(F_0,G'))$ is naturally isomorphic to the pointed set of one-dimenional $F'/F_0'$-hermitian spaces (up to isometry) satisfying \eqref{Herm F'/F} (with the $F'/F_0'$-hermitian space $W_0$ as the distinguished element). Similar remark applies to local fields rather than $F/F_0$ (except that the torus $T$ may not be anisotropic). 
\end{remark}

It follows from \eqref{Herm F'/F} that the $F'/F_0'$-hermitian space $W$ has signatures $(1,0)$ for all but one archimedean place $v'_0$ of $F_0'$ over $v_0$.
We define a refined invariant  
\begin{align}\label{xi'}
\xi'=\fkq'(u)\in F_0',
\end{align}
where $\fkq'$ is the quadratic form attached to the $F'/F_0'$-hermitian form on $W$, cf. \eqref{eq:her2q}. In particular, $\tr_{F_0'/F_0}(\xi')=\xi$. 

According to the action of $F_0'$, we have an orthogonal direct sum decomposition
$$
V'\otimes_{F,w_0}\BC=\bigoplus_{v'\in\Hom(F_0',\BR), v'|_{F_0}=v_0} \BC_{v'},
$$
where $F_0'$ acts on the line $\BC_{v'}$ through $v':F_0'\incl \BR$. Then there is a unique negative-definite summand, say $\BC_{v'_0}$ for a place $v_0'$ above $v_0$.  It follows that 
\begin{align}\label{R g0}
R(u,z_{\delta})=v_0'( \fkq'(u))=-|\xi'|_{v_0'},
\end{align}
where the last equality is due to the fact $v_0'( \fkq'(u))<0$.

\begin{corollary}\label{coro ht infty}
Under the same assumptions as Theorem \ref{thm infty}, we have
\begin{equation}\label{sum Ei}
	\Int^\bK_{v_0}(\xi,\Phi) =-  
\sum_{} \Ei(-2\pi |\xi'|_{v_0'})\cdot \Orb\left((\delta,u), \Phi\right),
\end{equation} where the sum runs over the $G'(F_0)$-orbits $(\delta,u)$ in the product 
$(G'(\alpha)\times V'_\xi)(F_0)$, $\xi'=\fkq'(u)$ is the refined invariant defined by\eqref{xi'}, and  $v'_0 \mid v_0$ is the unique archimedean place of $F_0'$ where $\xi'$ is negative.

\end{corollary}

Finally, we address the difference  between the two Green functions. Define, for any place $v\mid\infty$ of $F_0$ and $h\in \bH({\BA_0})$,
\begin{align}\label{B-Kv}
\Int_v^{\bK-\bB}(h,\Phi)=\frac{1}{\tau(Z^\BQ)\cdot  [E:F]} \left(\CZ_{v,\corr}(h,\phi),\quad \LCM(\alpha,\phi_0) \right),
\end{align}
cf. \eqref{geo error}, and define
\begin{align}\label{B-K}
\Int^{\bK-\bB}(h,\Phi)=\sum_{v\mid \infty}\Int^{\bK-\bB}_v(h,\Phi).
\end{align}
We note that the definition works without any reference to the integral models $\CM$, hence makes sense for all $\phi_0\in  \CS\left(G(\BA_{0,f}),K_{G}\right)$ and $ \phi\in\CS(V(\BA_{0,f}))^{K_G}$.
\begin{corollary}[to Theorem \ref{thm ES}]\label{cor:mod int infty}
 Let $F_0=\BQ$.   Then the function $h\in \bH(\BA_0)\mapsto \Int^{\bK-\bB}(h,\Phi)$ belongs to  $\CA_{\rm exp}(\bH(\BA_0),K,n)$, where $K\subset \bH(\BA_{0,f})$ is a compact open subgroup which fixes $\phi\in \CS(V(\BA_{0,f}))$ under the Weil representation.
\end{corollary}

\section{Weil representation and RTF}\label{s:weil} 

Starting from this section, we study a  partially linearized  version of the Jacquet--Rallis relative trace formula, and the ``action" on the RTF by $\SL_2(\BA_0)$ under the Weil representation (by changing testing functions on the linear factor of the RTF).
\subsection{Weil representation and theta functions}\label{ss:weil}
For now we let $F$ be a global field. Let $(V,\fkq)$ be a (non-degenerate) quadratic space over $F$ of {\em even} dimension $d$, where $\fkq:V\to F$ is the quadratic form with  the associated symmetric bilinear pairing $\pair{\cdot,\cdot}:V\times V\to F$ by \eqref{eq:q2bi}. Let $O(V)=O(V,\fkq)$ be the isometry group, viewed as an algebraic group over $F$.

Let  $\CS(V(\BA_F))$ be the space of Schwarz functions.  The product group $O(V)(\BA_F)\times \SL_2(\BA_F)$ acts on $\CS(V(\BA_F))$ via the Weil representation denoted by $\omega$: for $\phi\in \CS(V(\BA_F))$, the function $\omega(g,h)\phi$ is defined by
$$
(\omega(g,h)\phi)(x)=(\omega(h))\phi(g^{-1}x), \quad (g,h)\in O(V)(\BA_F)\times \SL_2(\BA_F),
$$
where the action of $\SL_2(\BA_F)$ is defined as follows. Let $\chi_V=\prod_{v}\chi_{V_v}$ be the quadratic character of $F^\times\bs\BA_F^\times$ defined by
$$
\chi_{V}(a)=(a,(-1)^{d/2}\det(V))_F,
$$
where $(\cdot,\cdot)$ is the Hilbert symbol over $F$, and $\det(V)\in F^\times/ (F^{\times})^2$ is the determinant of the moment matrix $\frac{1}{2}(\pair{x_i,x_j})_{1\leq i,j\leq d}$ of any $F$-basis $x_1,\cdots, x_d$ of $V$. For a place $v$ of $F$, and $\phi_v\in\CS(V(F_v))$, the action of $\SL_2(F_v)$ is determined by
\begin{align}\label{eqn weil}
\omega_v\left(\begin{matrix} a& \\
& a^{-1}
\end{matrix}\right)\phi_v(x)&=\chi_{V_v}(a)|a|_v^{d/2}\phi_v(ax),\notag\\
\omega_v\left(\begin{matrix} 1&b \\
&1
\end{matrix}\right)\phi_v(x)&=\psi_v(b \fkq(x))\phi_v(x),\\
\omega_v\left(\begin{matrix} &1\\
-1&
\end{matrix}\right)\phi_v(x)&=\gamma_{V_v}\,\wh\phi_v(x) ,\notag
\end{align}
where $\gamma_{V_v}$ is  the Weil constant (a fourth root of unity under our assumption that $\dim V_v$ is even), and the Fourier transform is defined by
$$
\wh\phi_v(x) =\int_{V(F_v)} \phi_v(y)\psi_v\left(\pair{x,y}\right)\, dy.
$$
Here $dy$ is a self-dual Haar measure on $V(F_v)$.

For $\phi\in \CS(V(\BA_F))$, we 
define the theta function by the absolute convergent sum
$$
\theta_{\phi}(g,h)=\sum_{\xi\in V}\omega(g,h)\phi(\xi),\quad (g,h)\in O(V)(\BA_F)\times \SL_2(\BA_F).
$$
This is left invariant under $O(V)(F)\times \SL_2(F)$.

\subsection{Automorphic kernel functions}\label{ss:ker}
In this subsection we work with a fairly general setting. It serves to explain the idea behind the more explicit setting in  later sections. 

Let $G$ be a connected reductive algebraic group over $F$, acting on $V$ and preserving the quadratic form $\fkq$ (i.e., the homomorphism $G\to \GL(V)$ factors through $O(V,\fkq)$).
Let $X$ be an affine variety over $F$ with an action of $G$, and $X_{/\!\!/G}$ the categorical quotient. Consider the diagonal action $r$ of $G$ on $X\times V$. Then $G(\BA_F)$ acts on $\CS((X\times V)(\BA_F))$.
 
The group $\SL_2(\BA_F)$ acts on  $\CS((X\times V)(\BA_F))$ through the second factor $V$ via the Weil representation. Note that now the formula \eqref{eqn weil} for the action of $\SL_2(\BA_F)$ is only applied to the second coordinate, e.g., locally at $v$, the element $\left(\begin{matrix} &1\\-1&
\end{matrix}\right)$ acts on $\CS((X\times V)(F_v))$ by (up to the Weil constant $\gamma_{V_v}$) the partial Fourier transform w.r.t. the $V$-component.

Let $\alpha \in X_{/\!\!/G}(F)$ be a fixed semi-simple element and $X(\alpha)$ the preimage of $\alpha$ (under the quotient map $X\to X_{/\!\!/G}$). Let $\phi_0\in \CS(X(\BA_F))$ and $\phi\in \CS( V(\BA_F))$. 
We define the automorphic kernel function associated to $\Phi=\phi_0\otimes\phi \in \CS((X\times V)(\BA_F))$,
\begin{align}\label{eqn:kernel}
\CK_{\Phi,\alpha}(g,h)&:=
\sum_{(x,u)\in( X(\alpha)\times V)(F)}\phi_0( g^{-1}  \cdot x)\omega(h)\phi(g^{-1}\cdot u)\\
&=
\sum_{(x,u)\in( X(\alpha)\times V)(F)}\omega(h)\Phi(g^{-1}\cdot(x, u)),\notag
\end{align}
where $g\in G(\BA_F), h\in \SL_2(\BA_F)$. This is again left invariant under $G(F)\times \SL_2(F)$. It follows that 
\begin{align}
\label{eqn Z h}
h\in\SL_2(\BA_F)\mapsto \BJ(h,\Phi)\colon =\int_{[G]}\CK_{\Phi,\alpha}(g,h)dg,
\end{align}
when absolutely convergent, is left invariant under $\SL_2(F)$. The same applies if we replace the pure tensor $\phi_0\otimes\phi$ by more general function $\Phi$ in $\CS((X\times V)(\BA_F))$ (this does not make any essential difference at non-archimedean places, but does at archimedean places).

Now we return to our earlier convention. Let $F_0$ be a totally real field, and $F/F_0$ a CM field extension. Let $$\xymatrix{\eta=\eta_{F/F_0}\colon F_0^\times\bs \BA_{0}^\times\ar[r]&\{\pm 1\}}$$ be the quadratic character by class field theory. 
Note that now $F_0$ plays the role of the base field $F$ in above discussion.

\subsection{The case of unitary groups}\label{ss:u}
Now we consider the Jacquet--Rallis RTF for unitary groups. Let $V$ be a $F/F_0$-hermitian space of dimension $n$.
Let $G=\RU(V)$ be the unitary group and $X=G$ with the conjugation action by $G$. Let $\alpha\in\CA_n(F_0)$ be irreducible over $F$ (cf. the end of \S\ref{ss: orb match}). We rewrite the kernel function \eqref{eqn:kernel} according to  $(\delta,u)\in (G(\alpha)\times  V)(F_0)$ regular semisimple (equivalently  $u\neq 0$ by the irreducibility of $\alpha$) or not, and then \eqref{eqn Z h} becomes
\begin{align}\label{Z phi u}
\BJ(h,\Phi)=&\int_{[G]} \sum_{(\delta,u)\in( G(\alpha)\times V)(F_0)}r(g)\omega(h)\Phi(\delta, u)\,dg\\
=& \int_{[ G] }\sum_{\delta\in G(\alpha)(F_0)}\omega(h)\Phi(g^{-1} \cdot \delta,  0) \,dg+ \int_{[G]}\sum_{ (\delta,u)\in( G(\alpha)\times V)(F_0)\atop u\neq 0}\omega(h)\Phi(g^{-1}\cdot (\delta,  u))\,dg. \notag
\end{align}

The summands in \eqref{Z phi u} are related to the global Jacquet--Rallis (relative) orbital integral (for the $G$-action on $G\times V$) of $\omega(h)\Phi$. For $\Phi\in  \CS((G\times V)(\BA_0))$ and a regular semisimple $(\delta,u)\in(G\times V)(F_0) $, we define
\begin{align}\label{orb u}
\Orb((\delta,u),\Phi)\colon=\int_{G(\BA_0)}\Phi(g^{-1}\cdot (\delta,  u))\,dg .
\end{align}
For $\delta\in G(\alpha)(F_0)$, we define
\begin{align}\label{orb u0}
\Orb((\delta,0),\Phi):=
\vol([G_\delta])\int_{G_\delta(\BA_{0})\bs G(\BA_{0})}\Phi(g^{-1} \cdot \delta,  0)\,dg,
\end{align}
where $G_\delta$ is the centralizer of $\delta$ in $G$, an anisotropic $F_0$-torus.

The first summand in \eqref{Z phi u} is a sum over the set of $G(F_0)$-conjugacy classes in $G(\alpha)(F_0)$ (cf. \S\ref{ss: orb match}),
\begin{align}\label{eqn:cst term}
\sum_{ \delta\in [G(\alpha)(F_0)]} \Orb ((\delta,0),\omega(h)\Phi).
\end{align}
There are only finitely many non-zero terms (uniformly in $h\in\bH(\BA_0)$), and hence the sum is absolutely convergent. The second summand in \eqref{Z phi u} is a sum over the set, denoted by $[(G(\alpha)\times V)(F_0)]_\rs$, of regular semisimple (equivalently, $u\neq 0$) $G(F_0)$-orbits in $(G(\alpha)\times V)(F_0)$.
$$
\sum_{(\delta,u)\in  [(G(\alpha)\times V)(F_0)]_{\rs}} \Orb ((\delta,u),\omega(h)\Phi).
$$
We first justify the convergence. We recall that $\bH=\SL_{2,F_0}$.
\begin{lemma}
\label{lem SL2 inv U}
\begin{altenumerate}
\renewcommand{\theenumi}{\alph{enumi}}
\item 
For any $\Phi\in  \CS((G\times V)(\BA_0))$, we have
$$
\sum_{(\delta,u)\in  [(G\times V)(F_0)]_{\rs}} \int_{G(\BA_0)}|\Phi(g^{-1}\cdot (\delta,  u))|\,dg<\infty.
$$
In particular, the same holds if we only sum over $(\delta,u)\in  [(G(\alpha)\times V)(F_0)]_{\rs}$.
\item Assume that $\Phi$ is $K_\infty$-finite for the maximal compact $K_\infty=\prod_{v}\SO(2,\BR)$ of $\bH(F_{0,\infty})$. Then the sum in $\BJ(h,\Phi)$ converges absolutely and uniformly for $h$ in any compact subset of $\bH(\BA_0)$. In particular, the function  $h\in\bH(\BA_0)=\SL_2(\BA_0)\mapsto \BJ(h,\Phi) $ is smooth and left invariant under $\bH(F_0)$.
\end{altenumerate}
\end{lemma}

\begin{proof}
For part (a) we prove a stronger result
$$
\sum_{(\delta,u)\in  [(G\times V)(F_0)]_{\rs}} \int_{G(\BA_0)}|\Phi(g^{-1}\cdot (\delta,  u))|\,dg<\infty.
$$
We first note that this is easy when $\Phi$ has compact support, since the sum would then have only finitely many non-zero terms. Next we may and do assume that $\Phi=\prod_{v}\Phi_v$ is a pure tensor (otherwise we can dominate $\Phi$ by a finite sum of pure tensors).

 We refer to \cite[\S A.1]{BP} for the terminology in our proof. Consider $X=G\times V$, and the categorical quotient $\CB=X_{/\!\!/ G}$ with the  natural map $\pi: X\to \CB$. The regular semisimple locus in $\CB$ (resp., $X$), denoted by $\CB_\rs$ (resp., $X_\rs$), is defined by $\Delta\neq 0$ (resp., $\Delta\circ\pi\neq 0$) for a regular function $\Delta$ on $\CB$ (resp., its pull-back to $X$).  Then the restriction $\pi_\rs: X_\rs\to \CB_\rs$ is a $G$-torsor. 
 
We fix a norm $|\!|\cdot |\!|_X$ on $X(\BA_0)$, which is the product of local norms $|\!|\cdot |\!|_{X_v}$ on $X(F_{0,v})$. Similarly we fix norms on $X_\rs(\BA_0)$ and $X_\rs(F_{0,v})$. For the norm on the affine line, we denote by $|\!| t |\!|$ (resp., $|\!| t |\!|_v$) for $t \in \BA_{0}$ (resp., $t\in F_{0,v}$), defined by $|\!| t |\!|_v=\max\{1,|t|_v\}$,  cf. \cite[\S A.1, before (1)]{BP}.

For $\Phi\in \CS(X(\BA_0))$, we have
\begin{align}\label{eqn:sch2norm}
|\Phi(x)|\ll |\!|x|\!|_X^{-d_1},\quad \text{for all $x\in X(\BA_0)$},
\end{align}
for any constant $d_1>0$ \footnote{Here, for two functions $f_1, f_2$,  the notation $f_1\ll f_2$ means that there is a constant $c$ such that $f_1\leq c f_2$.}.
By \cite[Prop. A.1.1\,(iii)]{BP}, for all $x\in X_\rs(\BA_0)$,
\begin{align*}
|\!|x|\!|_{X_\rs}\sim |\!|x|\!|_{X}\, |\!| \Delta\circ\pi(x)^{-1}|\!|.
\end{align*}
In particular, there exists a constant $d_2>0$ such that, for all $x\in X_\rs(\BA_0)$,
\begin{align*}
|\!|x|\!|_{X_\rs}^{d_2}\ll |\!|x|\!|_{X}\, |\!| \Delta\circ\pi(x)^{-1}|\!|.
\end{align*}
By  \eqref{eqn:sch2norm} and \cite[Prop. A.1.1\,(vii)]{BP}, for any $d_3>0$, there exists $d_1>0$ large enough such that
\begin{align}
\int_{G(\BA_0)} |\Phi(g\cdot x)|\,dg &\ll |\!| \Delta\circ\pi(x)^{-1}|\!|^{d_1}  \int_{G(\BA_0)}  |\!|g\cdot x|\!|_{X_\rs}^{-d_1d_2}\notag\\
&\ll  |\!| \Delta\circ\pi(x)^{-1}|\!|^{d_1} |\!|\pi_{\rs} (x)|\!|_{\CB_\rs}^{-d_3},\label{eqn:est orb}
\end{align}
for all $x\in X_{\rs}(\BA_0)$. Note that this implies similar estimates: for any $d_3>0$, there exists $d_1>0$ such that
\begin{align}\label{eqn:est orb v}
\int_{G(F_{0,v})} |\Phi_v(g\cdot x_v)|\,dg \ll  |\!| \Delta\circ\pi(x_v)^{-1}|\!|_v^{d_1} |\!|\pi_{\rs} (x_v)|\!|_{\CB_{\rs,v}}^{-d_3}
\end{align}holds  for every place $v$, and
\begin{align}\label{eqn:est orb S}
\prod_{v\notin S}\int_{G(F_{0,v})} |\Phi_v(g\cdot x_v)|\,dg \ll\prod_{v\notin S} |\!| \Delta\circ\pi(x_v)^{-1}|\!|_v^{d_1} |\!|\pi_{\rs} (x_v)|\!|_{\CB_{\rs,v}}^{-d_3}
\end{align}holds for any finite set $S$ of places, where $(x_v)_{v\notin S}\in X_{\rs}(\prod_{v\notin S}F_{0,v})$.  Here we emphasize that $d_1, d_3$ can be made independent of $v$.

Now we claim that for any $d_3>0$ there exists $d_1>0$ such that
\begin{align}\label{eqn:est orb v inf}
\int_{G(F_{0,v})} |\Phi_v(g\cdot x_v)|\,dg \ll  | \Delta\circ\pi(x_v)|_v^{-d_1} |\!|\pi_{\rs} (x_v)|\!|_{\CB_{\rs,v}}^{-d_3}
\end{align}
holds for every place $v$ (here we emphasize that $d_1,d_3$ can be made independent of $v$). Indeed, if  $ | \Delta\circ\pi(x_v)|_v\leq 1$, then $ |\!| \Delta\circ\pi(x_v)^{-1}|\!|_v=| \Delta\circ\pi(x_v)^{-1}|_v$, and hence \eqref{eqn:est orb v inf} follows from  \eqref{eqn:est orb v} in this case. Now suppose that $ | \Delta\circ\pi(x_v)|_v\geq 1$, then $ |\!| \Delta\circ\pi(x_v)^{-1}|\!|_v=1$,  it follows from \eqref{eqn:est orb v} that
\begin{align}\label{eqn:est orb v1}
\int_{G(F_{0,v})} |\Phi_v(g\cdot x_v)|\,dg &\ll  |\!|\pi_{\rs} (x_v)|\!|_{\CB_{\rs,v}}^{-d'_3}
\end{align}
for any constant $d'_3>0$. By \cite[Prop. A.1.1\,(ii)]{BP} applied to the morphism $\Delta:\CB_{\rs}\to \BA^1$  (here $\BA^1$ denotes the affine line), there exists a constant $d_4>0$ such that
 $$|\!|\Delta(b)|\!|_v\ll |\!|b|\!|_{\CB_{\rs,v}} ^{d_4}
 $$ for all $b\in \CB_\rs(F_{0,v})$.

Note that $| \Delta(b)|_v\leq |\!| \Delta(b)|\!|_v$. Choosing $d_3'=d_3+d_1d_4$ in \eqref{eqn:est orb v1}, we arrive at the estimate \eqref{eqn:est orb v inf} (for any constants $d_1,d_3>0$).

 Since the support of the non-archimedean component $\Phi^\infty=\prod_{v\nmid\infty}\Phi_v$ is compact, its image under $\Delta\circ\pi$ is also compact in $\BA_{0,f}$. It follows that for $x\in X_{\rs}(\BA_{0})\cap \supp(\Phi)$ 
\begin{align}\label{eqn:prod form0}
\prod_{v\nmid\infty} |\!| \Delta\circ\pi(x)^{-1}|\!|_v\ll \prod_{v\nmid\infty} | \Delta\circ\pi(x)^{-1}|_v=\prod_{v\nmid\infty} | \Delta\circ\pi(x)|_v^{-1}.
\end{align}
It follows that, when $x=(x_v)_v\in X_\rs(\BA_0)$, 
for any $d_3>0$ there exists $d_1>0$ such that
\begin{eqnarray*}
&& \int_{G(\BA_0)} |\Phi(g\cdot x)|\,dg\\ &=&  \prod_{v\mid\infty} \int_{G(F_{0,v})} |\Phi_v(g\cdot x_v)|\,dg\, \prod_{v\nmid\infty} \int_{G(F_{0,v})} |\Phi_v(g\cdot x_v)|\,dg \\
  &\ll &  \prod_{v\mid\infty} | \Delta\circ\pi(x_v)|_v^{-d_1} |\!|\pi_{\rs} (x_v)|\!|_{\CB_{\rs,v}}^{-d_3} \prod_{v\mid\infty}  |\!| \Delta\circ\pi(x_v)^{-1}|\!|_v^{d_1} |\!|\pi_{\rs} (x_v)|\!|_{\CB_{\rs,v}}^{-d_3}
  \mbox{(By \eqref{eqn:est orb v inf} and \eqref{eqn:est orb S})}\\
   &\ll&  \prod_{v} | \Delta\circ\pi(x_v)|_v^{-d_1} |\!|\pi_{\rs} (x_v)|\!|_{\CB_{\rs,v}}^{-d_3} \quad \quad \quad \quad \quad \quad \quad \quad \quad \quad \quad \quad \quad \quad \quad  \mbox{(By\eqref{eqn:prod form0})}.
\end{eqnarray*}

Finally, let $x\in X_{\rs}(F_0)$. By the product formula $\prod_{v} | \Delta\circ\pi(x)|_v=1$, we obtain  
$$
 \int_{G(\BA_0)} |\Phi(g\cdot x)|\,dg\ll  |\!|\pi_{\rs} (x)|\!|_{\CB_{\rs}}^{-d_3} 
$$
for any constant $d_3>0$. The desired convergence then follows from \cite[Prop. A.1.1\,(v)]{BP} (applied to $\CB_{\rs}$):
$$
\sum_{b\in \CB_{\rs}(F_0)}  |\!|b |\!|_{\CB_{\rs}}^{-d_3} <\infty
$$
for $d_3$ large enough.

To show part (b), it suffices to show that the constant implicit in $\ll$ of \eqref{eqn:sch2norm} can be made uniform for $\omega(h)\Phi$ for $h\in \bH(\BA_0)$ in a neighborhood of $1$. The function $\Phi^\infty$ is invariant under a compact open of $\bH(\BA_{0,f})$. So it suffices to consider $h\in  \bH(F_{0,\infty})$. By the $K_\infty$-finiteness assumption, it suffices to consider upper triangular elements in $  \bH(F_{0,\infty})$. Then it is easy to see the constant can be made uniform by the formula \eqref{eqn weil}.

\end{proof}

To summarize, we obtain
\begin{align}\label{Z phi u1}
\BJ(h,\Phi)=\sum_{(\delta,u)\in [(G(\alpha)\times V)(F_0)]}\Orb((\delta,u),\omega(h)\Phi).
\end{align}
When $\Phi_\infty$ is $K_\infty$-finite, it follows easily that, for $\xi\in F_0^\times$, the $\xi$-th Fourier coefficient of $\BJ(\cdot,\Phi)$ is equal to
\begin{align}\label{eqn: xi coeff U}
\sum_{(\delta,u)\in [(G(\alpha)\times V_\xi)(F_0)]}\Orb((\delta,u),\omega(h)\Phi).
\end{align}
Here we refer to \eqref{eq:def F coeff} for the definition of Fourier coefficients.

\subsection{The case of general linear groups}\label{ss gln}

Now we consider the Jacquet--Rallis RTF for general linear groups. Let $V_0=F_0^n$ be the $n$-dimensional vector space of column vectors over $F_0$. We identify the dual vector space $V_0^\ast=\Hom_{F_0}(V_0,F_0)$  with the space of row vectors. Consider the natural quadratic form
on $V'=V_0\times V_0^\ast$:
\begin{align}\label{eq: q on V'}
\fkq\colon 
\xymatrix@R=0ex{ V_0\times V_0^\ast \ar[r] &F_0\\
u'=(u_1,u_2)\ar@{|->}[r]& u_2(u_1)}.
\end{align}
Let $$
\xymatrix@R=0ex{ \pair{\cdot,\cdot}\colon V'\times V' \ar[r] &F_0}.
  $$
  be the the associated symmetric bilinear pairing (so that $\pair{u',u'}=2\fkq(u')$).  
  Let $G'=\GL(V_0)$  act on $V'$ by $({\rm std, std^\vee})$. Then $G'\simeq \GL_{n,F_0}$ via the given identification $V_0=F_0^n$. Consider the diagonal action of $G'$ on $S_n\times V'$, cf. \eqref{Sn def}.

Now let $\alpha\in \CA_n(F_0)$ be irreducible over $F$. We rewrite the kernel function \eqref{eqn:kernel} according to $(\gamma,u')\in (S_n(\alpha)\times  V')(F_0)$ regular semisimple or not. By the irreducibility of $\alpha$,  precisely three orbits are not regular semisimple, i.e., $(\gamma,u')$ where $u'$ are 
\begin{align}\label{nil orb}
\begin{cases}
\{(0,0)\}, \\ 0_+:= \{(u_1,0): u_1\in V_0(F_0)\setminus \{0\} \}, \\
 0_-:=\{(0,u_2): u_2\in V_0^\ast(F_0)\setminus \{0\} \}.
\end{cases}
\end{align}They will be called ``(relative) nilpotent";
and the last two are regular (i.e., with trivial stabilizers). 

We define the (global) Jacquet--Rallis (relative) orbital integral (for the $G'$-action on $S_n\times V'$). For $\Phi'\in  \CS((S_n\times V')(\BA_0))$, and a regular semisimple $(\gamma,u')\in(S_n\times V')(F_0) $, we define
\begin{align}\label{orb gl}
\Orb((\gamma,u'),\Phi',s):=\int_{G'(\BA_0)}\Phi'(g^{-1}\cdot (\gamma,  u'))\,|\det(g)|_{F_0}^s\eta(g)\,dg.
\end{align}
Here and thereafter we will simply denote by $\eta$ the character   $\eta\circ \det$ of $G'(\BA_0)$. 
The global orbital integral is a product of local orbital integrals
\begin{align}\label{orb gl loc}
\Orb((\gamma,u'),\Phi'_{v},s):=\int_{G'(F_{0,v})}\Phi'_{v}(g^{-1}\cdot (\gamma,  u'))\,|\det(g)|_{v}^s\eta(g)\,dg.
\end{align}

We consider a one-parameter family of \eqref{eqn Z h}: for $\Phi'\in \CS((S_n\times V')(\BA_0))$ and $s\in\BC$, we will define
\begin{align*}
\BJ(h,\Phi',s)=&\int_{[G']} \left(\sum_{(\gamma,u')\in (S_n(\alpha)\times V')(F_0)}r(g)\omega(h)\Phi'(\gamma, u')\right)|\det(g)|_{F_0}^s\eta(g)\,dg.
\end{align*}
Similar to the unitary case, we write it as a sum over orbits:
\begin{align}\label{def Z exp}
\BJ(h,\Phi',s)=& \BJ(h,\Phi',s)_{0}+ \BJ(h,\Phi',s)_{\rs},
\end{align}
where $$
\BJ(h,\Phi',s)_{\rs} =\sum_{(\gamma,u')\in [(S_n(\alpha)\times V')(F_0)]_\rs} \Orb((\gamma,u'),\omega(h)\Phi',s),
$$
and the term $ \BJ(h,\Phi',s)_{0}$ is the sum over the two regular nilpotent orbits in \eqref{nil orb}, which will be defined in \S\ref{ss: n>1} by an analytic continuation. We have discarded the orbit $\{(0,0)\}$ since $\eta$ is then a non-trivial character on the stabilizer.

We first justify the convergence for the regular semisimple part $\BJ(h,\Phi',s)_{\rs}$. We will defer the $\bH(F_0)$-invariance to  the next section, cf. Theorem \ref{thm inv gl}. 

\begin{lemma}\label{lem conv gl}
\begin{altenumerate}
\renewcommand{\theenumi}{\alph{enumi}}
\item 
For any $\Phi'\in  \CS((S_n\times V')(\BA_0))$, the sum
$$
\sum_{ (\gamma,u')\in [(S_n\times V')(F_0)]_\rs} \int_{G'(\BA_0)}|\Phi'(g^{-1}\cdot (\gamma,  u'))|\det(g)|_{F_0}^s\,dg <\infty,
$$converges absolutely and uniformly for $s$ in any compact subset in $\BC$. In particular, the same holds if we only sum
over $(\gamma,u')\in [(S_n(\alpha)\times V')(F_0)]_\rs$.
\item Assume that $\Phi'$ is $K_\infty$-finite for the maximal compact $K_\infty=\prod_{v}\SO(2,\BR)$ of $\bH(F_{0,\infty})$. Then the sum in $\BJ(h,\Phi',s)_\rs$ converges absolutely and uniformly for $(h,s)$ in any compact subset of $\bH(\BA_0)\times\BC$.
\end{altenumerate}

\end{lemma}

\begin{proof}
This follows the same argument as the proof of Lemma \ref{lem SL2 inv U}. 
\end{proof}

\section{RTF with Gaussian test functions}\label{s:arch RTF}
We now simplify \eqref{Z phi u1} (resp., \eqref{def Z exp}) for a fixed $\alpha\in \CA_n(F_0)$ by plugging in a Gaussian test function at every archimedean place.

\subsection{Gaussian test functions: the compact unitary group case}
\label{ss:cpt u}
Now let $F/F_0$ be the the archimedean  local field extension $\BC/\BR$. Let $V$  be an $n$-dimensional {\em positive definite}  hermitian space with the unitary group $G=\U(V)$. We define a special test function, called the Gaussian test function (cf. \cite[\S7]{RSZ3}) in the semi-Lie algebra setting,
\begin{align}\label{gauss gp}
\Phi(g,u)={\bf 1}_{G(\BR)}(g)\cdot e^{-\pi \pair{u,u}}\in \CS((G\times V)(\BR)).
\end{align}
 Since it is invariant under $G(\BR)$,
its orbital integrals \eqref{eq:orb U lie} take a very simple form. 
\begin{align}\label{orb gauss}
\Orb((g,u),\Phi)=e^{-\pi \pair{u,u}}. 
\end{align}
Here we normalize the Haar measure on  $G(\BR)$ such that $\vol(G(\BR))=1$.

We explicate the action of $\SL_2(\BR)$ by the Weil representation (for the fixed additive character $\psi:x\in\BR\mapsto e^{2\pi i x}$). Write $h\in \SL_2(\BR)$ according to the Iwasawa decomposition 
\begin{align}\label{h infty1}
h=\left(\begin{matrix} 1 & b \\
& 1
\end{matrix}\right)
\left(\begin{matrix} a^{1/2} & \\
& a^{-1/2}
\end{matrix}\right)\,\kappa_\theta,\quad a\in\BR_{+},\quad b\in \BR,
\end{align}
where $\kappa(\theta)$ is as in  \eqref{kappa in SO2}.
 First of all, the Gaussian test functions above are eigen-vectors of weight $k=n$ under the action of the maximal compact $\SO(2,\BR)$ of $\SL_2(\BR)$, i.e.,
\begin{align}\label{SO2}
\omega(\kappa_\theta)\Phi=\chi_n(\kappa_\theta)\Phi,
\end{align}
where $\chi_n$ is the character \eqref{chi}. In general, for $h$  of the form \eqref{h infty1},
\[
\omega(h)\Phi(g,u)=\chi_n(\kappa_\theta)\,{\bf 1}_{G(\BR)}(g)\otimes |a|^{1/2} e^{\pi i( b+ia) \pair{u,u}}.
\]

\subsection{Gaussian test functions: the general linear group case}

On the general linear group side, we define Gaussian test functions to be any smooth transfer of the Gaussian test functions on the unitary side (cf. \cite[\S7]{RSZ3}). We recall the bijection of regular semisimple orbits \eqref{eq:orb mat 1} and \eqref{eq:orb mat lie1}. Note that in the disjoint union, one component is from the positive definite hermitian space $V$. We then defined the notion of transfer at the end of \S\ref{ss: transfer}.

\begin{definition}\label{gauss}
 We call  $\Phi'\in \CS((S_n\times V_n')(\BR))$ a  Gaussian test function if it is a transfer of the tuple $\{\Phi_V\}_V$ where $\Phi_V$ is the Gaussian test functions \eqref{gauss gp} for the positive definite hermitian space $V$, and $\Phi_V=0$ for all the other (isometry classes of)  hermitian spaces $V$.
\end{definition}

 It is expected that Gaussian test functions exist.  However, it seems very difficult to explicate the Gaussian test functions on  $(S_n\times V')(\BR)$ (with one exception: the case $n=1$). Fortunately a weaker version suffices for our purpose.
We only need a partial matching, i.e., only Schwartz functions that have matching orbital integrals for elements with a fixed component on $S_n$; we will name them ``partial Gaussian test functions".

We call the subset $T_n$ of diagonal elements in $S_n(\BR)$ the compact Cartan subspace of $S_n(\BR)$.  We have
$$\xymatrix{
T_n\ar[r]^-\sim &\RU(1)(\BR)^n.}
$$
Let  $T_n^\rs$ denote the open subset of the regular semisimple elements in the Cartan subspace $T_n$ (i.e., those with distinct diagonal entries).
\begin{definition}\label{weak gauss}Let $\Omega$ be a compact subset of $T_n^\rs$. We call  $\Phi'\in \CS((S_n\times V'_n)(\BR))$ a partial Gaussian test function (relative to $\Omega$) if, for all regular semisimple $(\gamma,u')\in\Omega \times V'$ matching $(\delta,u)\in (\U(V)\times V)(\BR)$ for the positive definite hermitian space $V$, we have
\begin{align}\label{eq:def gauss}
\Orb((\gamma,u'),\Phi')=\Orb((\delta,u),\Phi),
\end{align}
where in the right hand side $\Phi$ is the Gaussian test functions \eqref{gauss gp}, and $\Orb((\gamma,u'),\Phi')=0$ whenever a regular semisimple $(\gamma,u')$ matches an orbit from non-positive-definite hermitian spaces in \eqref{eq:orb mat 1}.

\end{definition}

Now we construct  ``partial Gaussian test functions" explicitly, for any compact subset $\Omega$ of $T_n^\rs$. We first consider the case $n=1$, and then reduce the general case to $n=1$.

\subsection{Gaussian test functions when $n=1$}
Assume $n=\dim V=1$. Then $G'(\BR)\simeq \BR^\times$, and the symmetric space $S_1(\BR)$ is compact.  The orbital integrals have been defined in \S\ref{ss: transfer}, cf. \eqref{Orb(gamma,f',s)}. Since the $G'$-action on $S_1$ is trivial, we simply work with the vector space component and suppress the $\gamma\in S_1(\BR)$ in the orbital integrals.  
 
Let $V=\BC$ be 1-dimensional hermitian space (with the standard norm), and let
$$
\phi(z)=e^{-\pi \,z\ov z}\in\CS(V).
$$
Then we have $\wh\phi=\phi$.

Let $V'\simeq\BR\times \BR$, with $\BR^\times$-action $$t\cdot (x,y)=(t^{-1}x,ty)
$$ Recall from \eqref{eq: q on V'} that the quadratic form on $V'$ is $\fkq(x,y)=xy$.
We consider the following Schwartz function in the Fock model,
\begin{align}\label{infty phi'1}
\phi'(x,y)=2^{-3/2}(x+y)e^{-\frac{1}{2}\pi(x^2+y^2)}\in \CS(\BR\times \BR).
\end{align}
It has the symmetry 
$$
\phi'(x,y)=\phi'(y,x),\quad \phi'(-x,-y)=-\phi'(x,y).
$$

Recall that the $K$-Bessel function is defined as
\begin{align*}
K_s(c)=\frac{1}{2}\int_{\BR_{+}}e^{-\frac{1}{2}c(u+1/u) }u^s\frac{du}{u}, \quad c>0, s\in\BC.
\end{align*}

\begin{lemma}\label{lem gaussian n=1}
Let $\xi\in\BR^\times$. Then
\begin{align*}
\Orb((1,\xi),\phi', s)=2^{-1/2}|\xi|^{(-s+1)/2}\left(K_{(s+1)/2}(\pi|\xi|)+\eta(\xi)K_{(s-1)/2}(\pi|\xi|)\right).
\end{align*}
In particular, 
\begin{align*}
\Orb((1,\xi),\phi')=\begin{cases}e^{ -\pi \xi} ,& \xi>0,\\
0, & \xi<0,
\end{cases}
\end{align*}
and when $\xi<0$,
\begin{align*}
\del((1,\xi),\phi')= \frac{1}{2}e^{-\pi \xi}\,\Ei(-2\pi|\xi|).
\end{align*}Here $\Ei$  is the exponential integral \eqref{Ei}.

\end{lemma}
\begin{remark}
Here the special value at $s=0$ has taken into account of the transfer factors, cf. \S\ref{ss: transfer}.
\end{remark}
\begin{proof}By definition of orbital integrals \eqref{Orb(gamma,f',s)} (except we have suppressed the $\fks_1$ and $S_1$ component), we have 
\begin{align*}
\Orb((1,\xi),\phi', s)&=2^{-1/2}\int_{\BR_{+}}(t+\eta(\xi)|\xi|/t)e^{-\frac{1}{2}\pi(t^2+\xi^2/t^2)} t^{-s}\frac{dt}{t}\\
&=2^{-1/2}|\xi|^{(-s+1)/2}\int_{\BR_{+}} (t+\eta(\xi)/t)e^{-\frac{1}{2}\pi |\xi|(t^2+1/t^2)} t^{-s}\frac{dt}{t}\\
&=2^{-1/2}|\xi|^{(-s+1)/2}\int_{\BR_{+}} e^{-\frac{1}{2}\pi |\xi|(t^2+1/t^2)} (t^{-s+1}+\eta(\xi) t^{-s-1})\frac{dt}{t}\\
&=2^{-3/2} |\xi|^{(-s+1)/2}\int_{\BR_{+}} e^{-\frac{1}{2}\pi |\xi|(u+1/u)} (u^{(-s+1)/2}+\eta(\xi) u^{(-s-1)/2})\frac{du}{u}\\
&=2^{-1/2}|\xi|^{(-s+1)/2}\left(K_{(-s+1)/2}(\pi|\xi|)+\eta(\xi)K_{(-s-1)/2}(\pi|\xi|)\right).
\end{align*}

To evaluate at $s=0$, we note
$$
K_{1/2}(\xi)=\sqrt{\frac{\pi}{2}}\frac{e^{-\xi}}{\xi^{1/2}}.
$$
Also we note that the transfer factor \eqref{SV transfer factor} takes value one at elements of the form $(1,\xi)$, applied to $F/F_0=\BC/\BR$.

The assertion for the first derivative follows from the following identity \cite{Ob},
$$
\frac{d}{ds}\Big|_{s=1/2}K_{s}(y)=-\sqrt{\frac{\pi}{2}}  \frac{e^{y} }{y^{1/2}}\,\Ei(-2y),\quad y>0.
$$

\end{proof}

We now explicates the action of $\SL_2(\BR)$ by the Weil representation $\omega$. Similar to the unitary case, the Gaussian test functions above are eigen-vectors of weight $k=n=1$ under the action of the maximal compact $\SO(2,\BR)$, cf. \eqref{SO2}, \eqref{chi}. Write $h\in \SL_2(\BR)$ according to the Iwasawa decomposition 
\begin{align*}
h=\left(\begin{matrix} 1 & b \\
& 1
\end{matrix}\right)
\left(\begin{matrix} a^{1/2} & \\
& a^{-1/2}
\end{matrix}\right)\,\kappa_\theta,\quad a\in\BR_{+},\quad b\in \BR,
\end{align*}
where $ \kappa_\theta\in\SO(2,\BR)$ is as in \eqref{kappa in SO2}.

\begin{lemma}\label{lem gaussian weil}
Let $\xi\in\BR^\times$. Then
\begin{align*}
\Orb((1,\xi),\omega(h)\phi', s)=2^{-1/2}\chi_1(\kappa_\theta)a|\xi|^{(-s+1)/2}\left(K_{(-s+1)/2}(\pi a|\xi|)+\eta(\xi)K_{(-s-1)/2}(\pi a|\xi|)\right).
\end{align*}
In particular, 
\begin{align*}
\Orb((1,\xi),\omega(h)\phi')=\begin{cases}a^{1/2}e^{ \pi i \xi(b+ia)} ,& \xi>0,\\
0, & \xi<0,
\end{cases}
\end{align*}
and when $\xi<0$,
\begin{align*}
\del((1,\xi),\omega(h)\phi')=\frac{1}{2}\chi_1(\kappa_\theta)a^{1/2} \,e^{\pi i  |\xi|(b-ia )}\,\Ei(2\pi a|\xi|).
\end{align*}Here $\Ei$  is the exponential integral \eqref{Ei}.

\end{lemma}
\begin{proof}
This follows by straightforward computation using Lemma \ref{lem gaussian n=1}, and the formulas \eqref{eqn weil} defining the Weil representation in \S\ref{ss:weil}. 
\end{proof}

\subsection{Partial Gaussian test functions: general $n$}\label{ss:weak gau}

We will use  the Iwasawa decomposition of the group $G'(\BR)=\GL_n(\BR)$,
\begin{align}\label{eq:Iwa GLn}
G'(\BR)=ANK,
\end{align}
 where $K=\SO(n,\BR)$, $N$ the group of unipotent upper triangular matrices,  and $A\simeq(\BR^\times)^n$ the diagonal torus. We have 
a homeomorphism 
\begin{align}\label{eq:Iwa GLn prod}
G'(\BR)\simeq A N\times_{\mu_2^{n-1}} K 
\end{align}
 as real manifolds, where the fiber product is over the intersection $AN\cap K$, which is equal to
$$K\cap A=\ker(\mu_2^n\to \mu_2)\simeq \mu_2^{n-1}.$$
We will take the natural Haar measure on each factor (e.g., the measure $\frac{dt}{|t|}$ on $\BR^\times$ and the product measure on $(\BR^\times)^n\simeq A$) and take the induced measure on $G'(\BR)$ by the above product \eqref{eq:Iwa GLn prod}.

Note that the torus $A$ is the stabilizer of a regular semisimple element in the Cartan subspace $T_n$.
Then $NK\cdot T_n^\rs$ (the conjugation action) defines an open subset $S_n^{c,\rs}$ (``c" is for ``compact") in $S_n$:
$$
\xymatrix@R=0ex{ NK\times T_n^\rs\ar[r]^{\sim}&S_n^{c,\rs}\subset S_n\\
(h, t)\ar@{|->}[r]& h^{-1} th .}
$$
The map is a $K\cap A$-torsor, and  induces a $K\cap A$-torsor:
\begin{align}\label{KAN}
\xymatrix@R=0ex{NK\times T_n^\rs\times (V_0\times V_0^\ast)\ar[r]&S_n^{c,\rs}\times  (V_0\times V_0^\ast)\\
(h, t, u')\ar@{|->}[r]& (h^{-1} th, h\cdot u') .}
\end{align}
 
 Now let $\Omega\subset T_n^\rs$ be any compact subset. We consider functions on $NK\times T_n^\rs \times (V_0\times V_0^\ast)$ of the form $\Psi=\phi_0\otimes \phi'$, with $\phi'\in\CS(V_0\times V_0^\ast)$ and 
\begin{align}\label{weak gau}
\phi_0=\varphi_N\otimes \varphi_K\otimes \varphi_{T_n},
\end{align}
where
\begin{enumerate}
\item the function $\varphi_{T_n}\in C_c^\infty(T_n^\rs)$ satisfies  $\varphi_{T_n}|_{\Omega}= {\bf 1}_{ \Omega}$,
\item the function $\varphi_N \in  C_c^\infty(N)$ satisfies $ \int_{N}\varphi_N(n)dn=1$,
\item  the function $\varphi_K$ is a constant multiple of ${\bf 1}_{K}$ such that $\int_K\varphi_K(k)dk=1$,
\item  the function  $\phi'$ is invariant under the finite group $K\cap A$.\end{enumerate}
By the $K\cap A$-invariance of $\phi_0$ and $\phi'$, the function $\Psi=\phi_0\otimes \phi'$  descends along the map \eqref{KAN}  to a Schwartz function $\Phi'^{c}$ on
$S_n^{c,\rs}\times  (V_0\times V_0^\ast)$. Then the extension-by-zero of $\Phi'^{c}$, denoted by $\Phi'$, is a Schwartz function on $S_n\times  (V_0\times V_0^\ast)$.

Finally we specify $\phi'$ on $V_0\times V_0^\ast$. Identify $V_0\times V_0^\ast$ with $\BR^n\times \BR^n\simeq (\BR\times\BR)^n$ and we define
\begin{align}\label{infty phi'}
\phi'=2^{-3n/2}\prod_{1\leq i\leq n}(x_i+y_i)e^{-\frac{1}{2}\pi(x_i^2+y_i^2)},
\end{align}
cf. \eqref{infty phi'1} for the case $n=1$. It is obviously invariant under $K\cap A$. Therefore by our recipe this function $\phi'$ (with any $\phi_0$ above) gives us a Schwartz function $\Phi'$ on $S_n\times  (V_0\times V_0^\ast)$.

Now we define the orbital integral $\Orb(u',\phi',s)$ for $u'\in V_0\times V_0^\ast$, relative to the $A$-action on $V_0\times V_0^\ast$, in the obvious way generalizing the case $n=1$, cf.  \eqref{Orb(gamma,f',s)}.

\begin{lemma}\label{lem Phi2phi}
 Let $\gamma\in \Omega\subset T_{n}^\rs$. Then for any regular semisimple $(\gamma,u')$, we have  $$
\Orb((\gamma,u'),\Phi',s)=\Orb(u',\phi',s),
$$
where the left hand side is the local orbital integral \eqref{orb gl loc}.

In particular, by Lemma \ref{lem gaussian n=1} and \eqref{orb gauss},
the function $\Phi'$ is a partial Gaussian test function (relative to the compact subset $\Omega$).

\end{lemma}
\begin{proof}
By the Iwasawa decomposition \eqref{eq:Iwa GLn}, the local orbital integral \eqref{orb gl loc} is equal to
$$
\int_{A}\int_{NK}\Phi'((nk)^{-1}\cdot (\gamma, a^{-1} \cdot u'))\,|\det(a)|^s\eta(a)\,dn\,dk\, da.
$$
By our choice of $\Phi'$, we obtain
\begin{align*}
&\int_{NK}\Phi'((nk)^{-1}\cdot (\gamma,   u'))\,dn\,dk\\
=&\left(\int_{N}\varphi_N(n)dn\int_K\varphi_K(k)dk\right) \varphi_{T_n} (\gamma)  \phi'(u')
\\=& \phi'(u').
\end{align*}
Therefore
$$
\Orb((\gamma,u'),\Phi',s)=\int_{A}  \phi'(a^{-1}\cdot u') |\det(a)|^s\eta(a)da=\Orb(u',\phi',s).
$$ 
This completes the proof.

\end{proof}
\begin{remark}
This result also holds if $u'$ is a regular nilpotent orbit and the orbital integral is regularized by \eqref{orb nat1} and \eqref{Orb nil n} below.
\end{remark}

\subsection{Modular analytic generating functions when $n=1$}
Now we return to the global situation \S\ref{ss gln}. Assume that $n=1$. 
Then we may identify $V'=F_0\times F_0$ and the special orthogonal group $\SO(V',\fkq)$ can be identified with the $F_0$-group $G':=\GL_{1,F_0}$, via the action on the $V'$ by $g\cdot (u_1,u_2)= (g^{-1}u_1,gu_2)$. The map $u'= (u_1,u_2)\mapsto \xi= \fkq(u')=u_1u_2$ identifies the categorical quotient $V'_{/\!\!/ G'}$ with the affine line. Note that regular semisimple orbits (for the $G'$-action) are exactly the fibers over $\xi\neq 0$, and each fiber has exactly one $G'$-orbit.

Let $\phi'\in \CS(V'(\BA_{0}))$.
Consider the integral, 
\begin{align}\label{Zan}
\BJ(\phi', s)=\int_{[G']}\left(\sum_{u'\in V'(F_0)} \phi'(g^{-1}\cdot u')\right)|g|^s\eta(g)\,dg.
\end{align}
The integral is not necessarily convergent, and we define it by a regularization procedure as follows.

Recall from \eqref{def Z exp} that we can write the integrand as a sum over the $G'(F_0)$-orbits in $V'(F_0)$. Then the regular semisimple part is
\begin{align}\label{Zan1}
\sum_{\xi=\fkq(u')\in F_0^\times}\Orb(u', \phi',s),
\end{align}
where
\begin{align}\label{Zan1.5}
\Orb(u', \phi',s):=\int_{G'(\BA_{0})} \phi'(g^{-1}\cdot u') |g|^s\eta(g)\,dg.
\end{align}
By Lemma \ref{lem conv gl}, the sum in \eqref{Zan1}
 converges absolutely and uniformly for $s$ a compact set in $\BC$.

The fiber over $\xi=0$ breaks into three orbits, 
$$\begin{cases}
\{(0,0)\},\\ 0_+= \{(u_1,0): u_1\in F_0^\times \},\\ 0_-=\{(0,u_2): u_2\in F_0^\times \}.
\end{cases}
$$
The stabilizer of the first one is $G'$, and the other two have trivial stabilizer.  Note that $\eta$ is non-trivial on $G'(\BA_{0})$, and hence we {\em define} the integral for the first orbit to be zero.  For the other two orbits, we define
\begin{align}\label{Zan2}
\Orb(0_+, \phi',s):=\int_{\BA_{0}^\times} \phi'(g,0) |g|^s\eta(g)\,dg,
\end{align}
and
\begin{align}\label{Zan3}
\Orb(0_-, \phi',s)=\int_{\BA_{0}^\times} \phi'(0,g^{-1}) |g|^s\eta(g)\,dg=\int_{\BA_{0}^\times} \phi'(0,g) |g|^{-s}\eta(g)\,dg.
\end{align}
Both will be understood as Tate's global zeta integrals. More precisely, 
\begin{align}\label{Zan2 nat}
\Orb(0_+, \phi',s)=L(s,\eta)\prod_{v} \Orb(0_+,\phi'_v,s),
\end{align}
where the local orbital integral for the regular nilpotent $0_+$ is  defined as (the analytic continuation of)
\begin{align}\label{orb nat1}
 \Orb(0_+,\phi'_v,s)\colon=\frac{\int_{F_{0,v}^\times} \phi'_v(g,0) |g|_v^s\eta_v(g)\,dg}{L(s,\eta_v)}.
\end{align}
Note that the local Tate integral \eqref{orb nat1} is absolutely convergent when $\Re(s)>0$,  extends to an entire function of $s$ (a polynomial in $q^{\pm s}_v$ when $v$ is non-archimedean) and equal to one for unramified data. Here $L(s,\eta)$ is the {\em complete} L-function of the Hecke character $\eta$. Similarly for $0_-$, we have
\begin{align}\label{Zan3 nat}
\Orb(0_-, \phi',s)=L(-s,\eta)\prod_{v} \Orb(0_-,\phi'_v,-s),
\end{align}
where 
$$
 \Orb(0_-,\phi'_v,-s)=\frac{\int_{F_{0,v}^\times} \phi'_v(0,g) |g|_v^{-s}\eta_v(g)\,dg}{L(-s,\eta_v)}.
$$

To summarize, we define \eqref{Zan} as the sum of  \eqref{Zan1}, \eqref{Zan2}, and \eqref{Zan3} (or rather, their analytic continuation to $s\in\BC$)
\begin{align}\label{Zan0}
\BJ(\phi', s)=\Orb(0_+, \phi',s)+\Orb(0_-, \phi',s)+\sum_{\xi=\fkq(u')\in F_0^\times}\Orb(u', \phi',s).
\end{align}
Define the analytic generating function on $\bH(\BA_0)$,
$$\BJ(h,\phi', s)=\BJ(\omega(h)\phi', s),\quad h\in \bH(\BA_{0}).
$$
\begin{remark}The function $\BJ(\cdot,\phi', s)$ may be viewed as the generating function of the above relative orbital integrals \eqref{Zan1.5}, \eqref{Zan2}, and \eqref{Zan3}, parameterized by $\xi\in F_0$.  This is the analytic analog of the modular generating function of special divisors in \S\ref{s:gen div}.
\end{remark}
\begin{theorem}\label{thm n=1 gl}
The function $\BJ(h,\phi', s)$ is smooth in $ h\in \bH(\BA_{0})$ and entire in $s\in\BC$. As a smooth function in $ h\in \bH(\BA_{0})$, it is left invariant under $\bH(F_0)$.
\end{theorem}
\begin{proof}
By Lemma  \ref{lem conv gl}, the smoothness and the entireness follow from the same property for each of \eqref{Zan1}, \eqref{Zan2}, and \eqref{Zan3}. To show the $\bH(F_0)$-invariance, we first note that the invariance under the upper triangular elements follow from the definition of the Weil representation and that of the function $\BJ(h,\phi', s)$. It remains to show the invariance under $w=\left(\begin{matrix} & 1\\
-1 & 
\end{matrix}\right)$, i.e.,
the functional equation
\begin{align}\label{eq:FE J phi'}
\BJ(\phi', s)=\BJ(\wh{\phi'}, s)
\end{align}
holds for all $\phi'$.

By Poisson summation formula (note that the action of $G'(\BA_0)$  commutes with the Weil representation)
$$
\sum_{u'\in V'}\phi'(g^{-1}\cdot u')=\sum_{u'\in V'}\wh\phi'(g^{-1}\cdot u'),\quad g\in G'(\BA_0),
$$
or equivalently,
\begin{align}\label{PSF}
\sum_{u'\in V', \,\xi\neq0}\phi'(g^{-1}\cdot u')-&\sum_{u'\in V',\,\xi\neq 0}\wh\phi'(g^{-1}\cdot u')
\\&=-\sum_{u'\in V'_{\xi=0}}\phi'(g^{-1}\cdot u')+\sum_{u'\in V'_{\xi=0}}\wh\phi'(g^{-1}\cdot u'),\quad g\in G'(\BA_0).\notag
\end{align}
We introduce a partial Fourier transform $\phi'\mapsto \CF_1(\phi')$ for one of the two variables,
$$
\CF_1(\phi')(u_1,u_2)=\int_{\BA_0}\phi'(u_2,w_2)\psi(-u_1 w_2)\,dw_2.
$$
Apply Poisson summation formula to the line $u_1=0$:
$$
\sum_{u_2\in F_0} \phi'(0,g^{-1}u_2)=|g|\sum_{u_1\in F_0}\CF_1(\phi')(gu_1,0).
$$
We obtain an alternative expression of the right hand side of \eqref{PSF} as the sum of
$$
-\sum_{u_1\in F_0}(\phi'(gu_1,0)+|g|\CF_1(\phi')(gu_1,0))+\phi'(0,0)
$$ and
$$
\sum_{u'_1\in F_0}(\wh\phi'(gu_1,0)+|g|\CF_1(\wh\phi')(gu'_1,0))-\wh\phi'(0,0).
$$

Denote
 $[G']^1=G'(F_0)\bs G'(\BA_{0})^1$, where
 $$
 G'(\BA_{0})^1:=\ker(|\det|\colon G'(\BA_{0})\to \BR_+).
 $$
 We embed  $\BR_+$ into $G'(\BA_{0})=\BA_0^\times$ by sending $t\in \BR_+$ to $(t_v)$ where
$$
t_v=\begin{cases}t^{1/[F_0:\BQ]},& v\mid\infty,\\
1,& v\nmid\infty.
\end{cases}
$$
Then we have a direct product
\begin{align}\label{G1}
 \begin{gathered}
	\xymatrix@R=0ex{
     G'(\BA_{0})   \ar[r]^-\sim &  G'(\BA_{0})^1\times \BR_+\\
	  g \ar@{|->}[r]  & (g_1, t)
	}
	\end{gathered}
	\end{align}
Since the quotient $[G']^1$ is compact, we may integrate \eqref{PSF} over $[G']^1$ first, and this kills the zero orbits (due to the non-triviality of $\eta|_{[G']^1}$). Then we integrate over $\BR_+$ (now we use the alternative expression of the right hand side). Since the Tate integrals converges absolutely when $\Re(s)>1$, we obtain 
\begin{align*}
&\Orb(0_+,\phi',s)+\Orb(0_+,\CF_1(\phi'),1+s)+\sum_{\xi\in F^\times_0}\Orb(u',\phi',s)
\\ 
=&\Orb(0_+,\wh\phi',s)+\Orb(0_+,\CF_1(\wh\phi'),1+s)+\sum_{\xi\in F^\times_0}\Orb(u',\wh\phi',s)
\end{align*}
when $\Re(s)>1$.
Finally,  we note that by \eqref{Zan3 nat} and the functional equation of Tate integrals,
$$
\Orb(0_-,\phi',s)=\Orb(0_+,\CF_1(\phi'),1+s).
$$
By analytic continuation, this completes the proof of \eqref{eq:FE J phi'} for all $s\in\BC$.

\end{proof}

\begin{remark}\label{rem SE}
The integral \eqref{Zan} can be viewed as the theta lifting for the pair $$(\SO(V',\fkq),\quad\SL_2),$$ from the automorphic representation $\eta|\cdot|^s$ of $\SO(V')\simeq \GL_{1}$ to $\SL_2$.  Therefore, the representation space spanned by $h\mapsto \BJ(h,\phi', s)$ is the space of degenerate Eisenstein series for the induced representation $\Ind_{B(\BA_0)}^{\bH(\BA_0)}(\eta\, |\cdot|^s)$ ($B$ the Borel subgroup of upper triangular matrices). In this way, the two nilpotent orbital integrals become the constant terms of the associated Eisenstein series.
\end{remark}

\begin{lemma}\label{nil infty}
Let $v\mid\infty$, and $\phi_{v}'$ the Gaussian test function \eqref{infty phi'1}. Then the local nilpotent orbital integral \eqref{orb nat1} is equal to
\[
\Orb(0_+, \phi'_v,s)=2^{\frac{s}{2}-1}.
\]
The action of the group $\SL_2(\BR)$ is given as follows, for $h\in\SL_2(\BR)$ in the form \eqref{h infty},
\begin{align*}
\Orb(0_+,\omega(h)\phi', s)=\chi_1(\kappa_\theta)a^{(-s+1)/2} 2^{\frac{s}{2}-1}.
\end{align*}

\end{lemma}
\begin{proof}By  \eqref{infty phi'1}, we obtain
$
\phi'_v(x,0)=2^{-3/2}xe^{-\frac{1}{2}\pi x^2}$. Then $\Orb(0_+, \phi'_v,s)$ is the Tate's local zeta integral at an archimedean place:
\begin{align*}
2\int_{\BR_+}e^{-\frac{1}{2}\pi x^2} |x|^{s+1}\frac{dx}{x}&
=\int_{\BR_+}e^{-\frac{1}{2}\pi x} |x|^{(s+1)/2}\frac{dx}{x}\\
&=(\pi/2)^{-(s+1)/2}\Gamma((s+1)/2).
\end{align*}
Note the local L-factor in \eqref{orb nat1} is by definition 
$$L(s,\eta)=\pi^{-(s+1)/2}\Gamma\left(\frac{s+1}{2}\right).
$$We obtain 
\[
\Orb(0_+, \phi'_v,s)=2^{\frac{s}{2}-1}.
\]

The action of $\SL_2(\BR)$ is determined in the way similar to Lemma \ref{lem gaussian weil}.

\end{proof}

\subsection{Modular analytic generating functions for general $n$}
\label{ss: n>1}
We now return to the setting of \S\ref{ss gln} for general $n$. 

Recall from \S\ref{ss:FB CM} that we have fixed $\alpha\in \CA_n(F_0)\subset F[T]_{\deg=n}$ irreducible over $F$, the field $F'=F[T]/(\alpha)$ and its subfield $F_0'$.  Then $S_n(\alpha)(F_0)$ consists of exactly one $G'(F_0)$-orbit and let us fix a representative $\gamma\in S_n(\alpha)(F_0)$. Denote by $T'$ the stabilizer of $\gamma$, which is isomorphic to $\Res_{F_0'/F_0}\BG_m$. It follows that the character $\eta\circ \det$ (of $G'(\BA_0)$) is nontrivial on $T'(\BA_0)$, which can be identified (via $T'\simeq \Res_{F_0'/F_0}\BG_m$) with the quadratic character associated to $F'/F_0'$ by class field theory
 $$
 \eta'=\eta_{F'/F'_0}: \BA_{F_0'}^\times\to\{\pm 1\}.
 $$

Similar to the $F'/F_0'$-hermitian form \eqref{Herm F'/F}, via the action of $F_0'$, the vector space $V_0$ (hence $V_0^\ast$) carries a structure of a one-dimensional $F_0'$-vector space. Furthermore, we can identify $$\Hom_{F'_0}(V_0,F_0')\simeq V_0^\ast$$ as one-dimensional $F'_0$-vector spaces.  There is a unique bi-$F_0'$-linear symmetric pairing 
\begin{align}\label{eq:def q' V'}
\xymatrix@R=0ex{ \pair{\cdot,\cdot}_{F_0'}\colon V' \times V' \ar[r] &F_0'}
\end{align}
such that $$
\pair{u_1, u_2}=\tr_{F_0'/F_0}\, \pair{u_1, u_2}_{F'_0}. $$
Let \begin{align}\label{eq: q' on V'}
\fkq'\colon 
\xymatrix@R=0ex{ V_0\times V_0^\ast \ar[r] &F_0'}
\end{align}
be the associated quadratic form over $F_0'$.

\begin{definition}
An element $\gamma\in S_n(F_{0})$ is  compact, if 
 locally at all places $v\mid\infty$, it lies in the $G'(F_{0,v})$-orbit of the compact Cartan subspace $T_n(F_{0,v})$.
\end{definition}
Then $\gamma\in S_n(\alpha)(F_0)$ is compact if and only if the field $F'$ is a CM extension of a totally real field $F_0'$, which we have assumed since \S\ref{ss:FB CM}.  

Now, for every $v\mid \infty$, we fix the archimedean $\Phi'_v\in\CS((S_n\times V')(F_{0,v}))$ to be the partial Gaussian test function constructed in \S\ref{ss:weak gau} (relative to a fixed compact neighborhood $\Omega_v\subset S_n(F_{0,v})$ of $\gamma$). Recall that $\Phi_{v}'$ is associated to the function  $\phi'_{v}$ defined by \eqref{infty phi'}.

There are  two regular nilpotent orbits for the $T'$-action on $V'(F_0)$, denoted by $0_{\pm}$ in \eqref{nil orb}. We now define the constant term $\BJ(h,\Phi',s)_{0} $ in \eqref{def Z exp} as the sum of the two regular $\gamma$-nilpotent orbital integrals $\Orb((\gamma,0_{\pm}),\Phi',s)$ in a similar way to \eqref{Zan2}.  More precisely, we define
\begin{align}\label{Zan nat n}
\Orb((\gamma,0_+), \Phi',s):=L(s,\eta')\prod_{v} \Orb((\gamma,0_+),\Phi'_v,s),
\end{align}
where the local orbital integral is defined as
\begin{align}\label{Orb nil n}
 \Orb((\gamma,0_+),\Phi'_v,s)=\frac{\int_{G'(F_{0,v})}\Phi'_{v}(g^{-1}\cdot (\gamma,  0_+))\,|\det(g)|_{v}^s\eta(g)\,dg}{L(s,\eta'_v)}.
\end{align}
Here the denominator is defined as
$$L(s,\eta'_v)=\prod_{v'|v}L(s,\eta'_{v'})$$
where $v'$ runs over all places of $F_0'$ above $v$. Note that $L(s,\eta')=L(s,\Ind_{F'_0}^{F_0}\eta')$.
 We define $ \Orb((\gamma,0_-), \Phi',s)$ similarly. Here we normalize the measure on $G'(F_{0,v})$ such that $\vol(G'(O_{F_{0,v}}))=1$ for all but finitely non-archimedean places $v$.

\begin{lemma}\label{lem nil+}
The integral \eqref{Orb nil n} is
absolutely convergent when $\Re(s)>0$,  extends to an entire function of $s$ (a polynomial of $q^{\pm s}_v$ for non-archimedean $v$). 

Moreover, for a fixed $\gamma$ and a pure tensor $\Phi=\otimes_{v}\Phi_v$ where $\Phi_v'={\bf 1}_{(S_n\times V')(O_{F_0,v})}$  for all but finitely many $v$, the integral \eqref{Orb nil n} is equal to one for all but finitely many places $v$ (depending on $\gamma$ and $\Phi$). 
\end{lemma}
\begin{proof}When $v\mid\infty$, by Lemma \ref{lem Phi2phi}, the desired claim follows from (the product of $n$ copies of) the same claim for $n=1$.

Now let $v$ be non-archimedean, and $\Phi'_v$ as in \S\ref{ss:weak gau}. We fix a large compact subset $\Omega_{v}$ of $G'(F_{0,v})$ such that $\Phi'_{v}(g^{-1}\cdot \gamma,u')=0$ unless $g\in \Omega _v\cdot T'(F_{0,v})$. We introduce a Schwartz function (with a parameter $s\in\BC$) on $V'(F_{0,v})$
\begin{align}\label{Phi'2phi'}
\phi'_{v,s}(u'):=\int_{\Omega_v} \Phi'_{v}(g^{-1}\cdot (\gamma, u ')) \,|\det(g)|_{v}^s\eta_v(g) \,dg.
\end{align}
It is easy to see that it is of the form
\begin{align}\label{phi'2phi}
\phi'_{v,s}=\sum_{1\leq i\leq m}a_i\, \lambda_i^s\, \phi_{v,i},
\end{align}
where
$$ a_i\in\BQ,\quad \lambda_i\in \BQ^{\times}_+, \quad\phi_i\in \CS(V'(F_{0,v})).
$$
Then, for a suitable choice of measure $\,dg$ on $\Omega_v$ in the integral \eqref{Phi'2phi'}
\begin{align}\label{orb' gl loc}
\Orb((\gamma,0_+),\Phi'_{v},s)=\Orb(0_+,\phi'_{v,s},s).
\end{align}
Here we view $V_0$ as a one-dimensional $F'_0$-vector space, and $V^\ast_0$ as its $F_0'$-dual vector space, and the right hand side is \eqref{Zan3 nat} relative to the quadratic extension $F'_v/F'_{0,v}$ at $v$ (i.e., $F'_v$ is the product of $F'_{v'}=F'\otimes_{F'_0} F'_{0,v'}$ over all places $v'$ of $F_0'$ over $v$). This shows that the local orbital integral for $0_+$ is a polynomial of $q_{v'}^{\pm s}, v'|v$, particularly,  an entire function in $s$.

Finally, let us assume that $v$ is unramified in $F'$ and $\Phi_v'={\bf 1}_{(S_n\times V')(O_{F_0,v})}$ (here we implicitly identified $V_0=F_0^n$ and endow it with the natural integral structure). For all but finitely many places $v$, the element $\gamma$ belongs to $S_n(O_{F_0,v})$ and generates the maximal order $O_{F_v'}$ in $F'_v$. Then it is easy to see that $\phi'_{v,s}= {\bf 1}_{V'(O_{F_0,v})}$ in \eqref{orb' gl loc}, and hence the integral is equal to one by the standard computation of Tate's local zeta integral for unramified data.

\end{proof}

\begin{theorem}\label{thm inv gl}
The function $(h,s)\in \bH(\BA_0)\times\BC\mapsto \BJ(h,\Phi', s)$ is entire in $s\in\BC$, and  left invariant under $\bH(F_0)$. 
\end{theorem}
\begin{proof}

 By the proof of Lemma \ref{lem nil+}, \eqref{phi'2phi} and \eqref{orb' gl loc}, there exist a finite collection of
 $$ a_i\in\BQ,\quad \lambda_i\in \BQ^{\times}_+, \quad\phi_i\in \CS(V'(\BA_{0})).
$$
 such that
\begin{align*}
\BJ(h,\Phi',s)=& \sum_{1\leq i\leq m}a_i\, \lambda_i^s\,\BJ(h,\phi'_i,s),
\end{align*}
Here we note that for almost all places the $\phi'_{v,s}= {\bf 1}_{V'(O_{F_0,v})}$ in \eqref{orb' gl loc}, and the $\phi'_i$ are of the form $\otimes_{v}\phi'_{v,i_v}$ for $\phi'_{v,i_v}$ from \eqref{phi'2phi}. The desired claims follow now from Theorem \ref{thm n=1 gl} for $n=1$, applied to the new quadratic extension $F'/F'_0$.
\end{proof}

For simplicity, we combine the two nilpotent orbital integrals into one
\begin{align}\label{eq:}
\Orb((\gamma,0_\pm),\Phi',s)\colon=\Orb((\gamma,0_+),\Phi',s)+ \Orb((\gamma,0_-,\Phi',s).
\end{align}
Then we obtain an expansion as a sum of orbital integrals
\begin{align}\label{Z g2l}
\BJ(h,\Phi',s)=&\sum_{(\gamma,u')\in [(S_n(\alpha)\times V')(F_0)] \atop u'\neq 0} \Orb((\gamma,u'),\omega(h)\Phi',s)\\
&= \Orb((\gamma,0_\pm),\omega(h)\Phi',s)+\sum_{(\gamma,u')\in [(S_n(\alpha)\times V')(F_0)]_\rs} \Orb((\gamma,u'),\omega(h)\Phi',s).\notag
\end{align}
Moreover,  for $\xi\in F_0^\times$, the $\xi$-th Fourier coefficient of $\BJ(\cdot,\Phi',s)$ is equal to
\begin{align}\label{eqn: xi coeff gl}
\sum_{(\gamma,u')\in [(S_n(\alpha)\times V'_\xi)(F_0)]} \Orb((\gamma,u'),\omega(h)\Phi',s).
\end{align}
This is the analog of \eqref{eqn: xi coeff U} on the unitary side.

\subsection{The decomposition of the special value at $s=0$}

We set 
\begin{equation*}
\BJ(h,\Phi'):=\BJ(h,\Phi',0) .
\end{equation*}
\[
\Orb((\gamma,u'),\omega(h)\Phi'):=\Orb((\gamma,u'),\omega(h)\Phi',0).
\]

Then the decomposition \eqref{Z g2l} specializes to
\begin{equation}\label{J s=0}
\BJ(h,\Phi')=\sum_{(\gamma,u')\in [(S_n(\alpha)\times V')(F_0)]\atop u'\neq 0} \Orb((\gamma,u'),\omega(h)\Phi').
\end{equation}

We set 
\begin{equation}\label{delJ}
\begin{aligned}
   \delJ(h,\Phi')&:=\frac{d}{ds}\Big|_{s=0}   \BJ(h,\Phi',s),\\
 \del((\gamma,u'),\Phi'_v) &:= \frac{d}{ds}\Big|_{s=0}  \Orb((\gamma,u'),\Phi'_v,s).
  	\end{aligned}
\end{equation}
(The second equation also applies to the nilpotent orbit, in which case the local orbital integrals are defined by \eqref{Orb nil n}.)

Now we introduce
\begin{equation}\label{delJ}
\begin{aligned}
 \delJ_v(h,\Phi') &:=       \delJ_v(\omega(h)\Phi'),\quad \text{where}\\
       \delJ_v(\Phi') &:=\sum_{(\gamma,u')\in [(S_n(\alpha)\times V')(F_0)]\atop u'\neq0} \del((\gamma,u'),\Phi'_v)\cdot  \Orb((\gamma,u'),\Phi'^{v}).
  	\end{aligned}
\end{equation}
In the nilpotent case, $\Orb((\gamma,0_\pm),\Phi'^{v})$ is interpreted as $L(0,\eta')\prod_{w\neq v}\Orb((\gamma,0_\pm),\Phi'^{v})$, cf. \eqref{Zan nat n}.

We define 
 $$
\del(0_+,\Phi')=\frac{d}{ds}\Big|_{s=0} L(s,\eta)\prod_{v} \Orb((\gamma,0_+),\Phi'_v,0).
$$
and similarly for $\del(0_-,\omega(h)\Phi')$. Then we define
\begin{align}\label{eq: nilp term}
\del(0_\pm,\Phi')=\del(0_+,\Phi')+\del(0_-,\Phi').
   \end{align}

Then by Leibniz's rule, we obtain a decomposition, 
\begin{align}\label{eqn J' dec}
   \delJ(h,\Phi')= \del(0_\pm,\omega(h)\Phi')+\sum_{v}    \delJ_v(\omega(h)\Phi').
   \end{align}
We call $\del(0_\pm,\omega(h)\Phi')$ the nilpotent term; it is part of the constant term (i.e., the $0$-th Fourier coefficient).

\part{Proof of the main theorems}
\section{The proof of FL}
\subsection{Smooth transfer: the global situation}\label{ss:global tran}
In \S\ref{ss: transfer}, we have defined the local transfer factor, cf. \eqref{SV transfer factor}. The definition depends on a choice of an extension $\wt\eta$ of the quadratic character $\eta$ attached to the local quadratic extension. In the global case, we fix an extension of the quadratic character $\eta_{F/F_0}$ of $F_0^\times\bs\BA_0^\times$ to a character $\wt\eta$ of $F^\times\bs\BA^\times$ (not necessarily of order $2$). The transfer factor for a global element then satisfies a product formula, and transforms according to the desired rule, cf. \cite[\S7.3]{RSZ3}.

We are now in the setting of \S\ref{ss: n>1}; in particular we have fixed an irreducible $\alpha\in \CA_n(F_0)\subset F[T]_{\deg=n}$. Let  $\Phi'=\otimes_{v}\Phi'_v \in \CS((S_{n} \times
V'_n)(\BA_0))$ be a pure tensor such that for every $v\mid\infty$,  $\Phi'_v$ is the partial Gaussian test function.
 We define a weaker notion of smooth transfer. Fix an $F_v/F_{0,v}$-hermitian space $V_v$.
\begin{definition}\label{def p tran}
For a fixed $\alpha\in \CA_n(F_{0,v})$, we say that $\Phi_v'$ partially (relative to $\alpha$) transfers to $\Phi_v\in \CS(\U(V_v)\times V_v)(F_{0,v})$, if we only require the equality \eqref{eq:def st loc} in Definition \ref{def st loc} to hold for matching orbits $(\gamma, u')\in (S_{n} (\alpha)\times
V'_n)(F_{0,v})_\rs$ and $(\delta,u)\in (\U(V_v)(\alpha)\times V_v)(F_{0,v})_\rs$; and $ \Orb((\gamma, u'),\Phi'_v)=0$ for any other $(\gamma, u')\in (S_{n} (\alpha)\times
V'_n)(F_{0,v})_\rs$.
\end{definition}

For $\Phi'^{\infty}=\otimes_{v\nmid\infty}\Phi'_v\in \CS((S_{n} \times
V'_n)(\BA_{0,f}))$, we say that it partially transfers to (or matches) $\Phi^{\infty}=\otimes_{v\nmid\infty}\Phi_v \in \CS((\U(V)\times V)(\BA_{0,f}))$ if  $\Phi'_v$ partially transfers to $\Phi_v$ for every $v\nmid\infty$.

\begin{remark}\label{rem split}
At  those places of $F_0$ split in $F$, we will further demand $\Phi_v$ and $\Phi_v'$ to match in an elementary way  analogous to \cite{Z14}.
\end{remark}

\subsection{Comparison}

In this subsection, we compare $\BJ(h,\Phi')$ with  $\BJ(h,\Phi)$ in the ``coherent" case, i.e., $\Phi=\otimes_{v}\Phi_v \in \CS((\U(V)\times V)(\BA_{0}))$ for an $n$-dimensional $F/F_0$-hermitian space $V$. 
 We further assume that $V$ is totally positively definite and $\Phi_v$ is the Gaussian test function for every $v\mid\infty$, cf. \eqref{gauss gp} in \S\ref{ss:cpt u}.

 \begin{proposition}
\label{prop coh}
 The function
 $$
        h\in \bH({\BA_0})\mapsto     \BJ (h,\Phi'), \text{ resp. } \BJ(h,\Phi),
 $$
 lies in $\CA_{\rm hol}(\bH({\BA_0}),K,n)$, where $K$ is a compact open subgroup of $\SL_{2}(\BA_{0,f})$ that acts trivially on both $\Phi$ and $\Phi'$. 
 \end{proposition}
 
 \begin{proof}
 The $K$-invariance follows immediately from the definition of $\BJ (h,\Phi')$ and $\BJ(h,\Phi)$. 

By Theorem \ref{thm inv gl} (resp., Lemma \ref{lem SL2 inv U})  the function $\BJ (\cdot,\Phi')$ (resp., $\BJ (\cdot,\Phi)$) is invariant under $\bH(F_0)$. The weight $n$ condition follows from the action under $\SO(2,\BR)$, by \eqref{SO2} for $\Phi$, and by Lemma \ref{lem gaussian weil} for $\Phi'$.

Finally we need to show the holomorphy on  the complex upper half plane $\CH^{[F_0:\BQ]}$ and at all cusps. Equivalently,  for any $h_f\in \bH({\BA_{0,f}})$, the function  $\BJ ^\flat_{h_f}(\tau,\Phi')$ (resp. $\BJ^\flat_{h_f} (\tau,\Phi)$) defined by \eqref{phi2phi flat} is holomorphic in $\tau\in\prod_{v\mid\infty} \CH$, and holomorphic at the cusp $i\infty$.  

By \eqref{J s=0}, Lemma \ref{lem gaussian n=1}, and Lemma \ref{lem Phi2phi}, the $\xi$-th Fourier coefficient of $\BJ ^\flat_{h_f}(h_\infty,\Phi')$ vanishes unless $\xi\in F_{0}$ and $\xi\geq 0$ (i.e., totally semi-positive). Hence the Fourier expansion takes the form 
$$
\BJ ^\flat_{h_f}(\tau,\Phi')=\sum_{\xi\in F_{0},\,\xi\geq 0} A_\xi \,q^\xi,\quad A_\xi\in \BC,
$$ 
where $A_\xi =0$ unless $\xi$ lies in a (fractional) ideal of $F_0$ depending on $\Phi'_f$ and $h_f$. This shows that 
$\BJ (\cdot,\Phi')\in\CA_{\rm hol}(\bH({\BA_0}),K,n)$. The assertion for $\Phi$ is proved similarly. 
 \end{proof}
 
Now let us fix a regular elliptic compact element $\gamma\in S_n(\alpha)(F_0)$ (cf. \S\ref{ss: n>1}). Let $S$ be a finite set of non-archimedean places of $F_0$ such that
\begin{itemize}
\item $S$ contains all places with residue characteristic $2$,
\item for all $v\in S$, $\Phi_v'$ partially (relative to $\alpha$) transfers to  $\Phi_v\in\CS(\U(V_v)\times V_v)(F_{0,v})$, \footnote{ Transfers exist by the result of \cite{Z14}. Here we only need the weaker result of the existence of partial transfers for fixed $\alpha$, which can be deduced easily from the $n=1$ case. }
 and
\item for every non-archimedean $v\notin S$, the hermitian space $V_v$ is split, $\Phi_v={\bf 1}_{(\U(V)\times V)(O_{F_0,v})} $ (w.r.t. a self-dual lattice in $V_v$), and $\Phi'_v={\bf 1}_{(S_{n} \times V')(O_{F_0,v})} $.
\end{itemize}

We consider the difference 
 $$
    \sE(h)   =\BJ (h,\Phi')- \BJ(h,\Phi),\quad h\in \bH(\BA_0).
 $$

 \begin{theorem}\label{thm FL 0}
 Assume that Conjecture \ref{FLconj} part \eqref{FLconj gp} holds for all $F_{0,v}$ with $v\notin S$ and for $S_{n}$. Then
 $\sE=0$. (Note that we are in the case $\dim V=n$.)
 \end{theorem}
\begin{proof}

Let $B$ (for ``bad") be  the (finite) set of non-archimedean places  $v\notin S$ of $F_0$ where $R_{\alpha,v}=R_\alpha\otimes_{O_{F_0}} O_{F_{0,v}}$ is not a product of DVRs.

By Proposition \ref{prop coh} and our choice of $\Phi$ and $\Phi'$, the function $\sE(h)\in  \CA_{\rm hol}(\bH({\BA_0}),K,n)$ where the compact open subgroup $K=\prod_{v\nmid\infty}K_v\subset \bH({\BA_{0,f}})$ can be chosen such that $K_v$ is of the form \eqref{eqn:Kv} for every $v\in B$. This is easy to see if the additive character $\psi_v$ is of level zero. In general,  it is known how the Weil representation depends on $\psi_v$ and the desired $K_v$-invariance holds for any $\psi_v$ at $v\in B$.

By the vanishing criterion Lemma \ref{lem density} below, it suffices to show, for $\xi\in F_{0}^\times$,
$$
W_{\sE,\xi}(h_\infty)=0,
$$
whenever $(\xi, B)=1$ (i.e., $v(\xi)=0$ for all $v\in B$).

Now let  $(\xi, B)=1$. By \eqref{eqn: xi coeff gl}, the $\xi$-th Fourier coefficient of $\BJ(h_\infty,\Phi')$ is 
$$
\sum_{(\gamma,u')\in [(S_n(\alpha)\times V'_\xi)(F_0)]} \Orb((\gamma,u'),\omega(h_\infty)\Phi').
$$
Similarly,  by \eqref{eqn: xi coeff U}, the $\xi$-th Fourier coefficient of $\BJ^\flat (\tau,\Phi)$ is
$$
\sum_{(\delta,u)\in [(G(\alpha)\times V_\xi)(F_0)]}\Orb((\delta,u),\omega(h_\infty)\Phi).
$$
By our choices of partial Gaussian test functions, for every $v\mid\infty$,
\begin{align}\label{eq: Orb eqty infty}
\Orb((\gamma,u'),\omega(h_v)\Phi'_v)=\Orb((\delta,u),\omega(h_v)\Phi_v)
\end{align}
holds for every $(\gamma,u')$ matching $(\delta,u)$.

We now {\em claim} that the equality
\begin{align}\label{eq: Orb eqty}
\Orb((\gamma,u'),\Phi'_v)=\Orb((\delta,u),\Phi_v)
\end{align}
holds for every non-archimedean place $v$ and every matching pair $(\gamma,u')\in [(S_n(\alpha)\times V'_\xi)(F_0)]$ and $(\delta,u)\in [(G(\alpha)\times V_\xi)(F_0)]$ (when $(\xi, B)=1$).   From the claim it follows that $A_\xi=0$ whenever $(\xi, B)=1$.

To show the claim, first let  $v\in B$. Since $(\xi, B)=1$, $\xi$ is a unit at $v$. Therefore by Proposition \ref{prop FL L2G} (ii), Conjecture \ref{FLconj} part \eqref{FLconj gp}, which we have assumed to hold, implies \eqref{eq: Orb eqty}.

If $v\notin S\cup B$, then $R_{\alpha,v}$ is a maximal order, and 
 \eqref{eq: Orb eqty} follows from Proposition  \ref{FL DVR} when $v$ is inert. When $v$ is split, the identity is trivial.

If $v\in S$, by our assumption on $\Phi_v'$ and $\Phi_v$, they partially match (relative to the fixed $\alpha$), and hence  \eqref{eq: Orb eqty} holds. This proves the claim.

Now by \eqref{eq: Orb eqty infty} and \eqref{eq: Orb eqty} we conclude that $W_{\sE,\xi}(h_\infty)=0$ whenever $(\xi,B)=1$.  This completes the proof.

\end{proof}

 \begin{corollary}
 \label{coro FL0}
  Under the assumption of Theorem \ref{thm FL 0}, we have for all $\xi\in F_{0}^\times$,
$$
\sum_{(\gamma,u')\in [(S_n(\alpha)\times V'_\xi)(F_0)]}\Orb((\gamma,u'),\Phi')=\sum_{(\delta,u)\in [(G(\alpha)\times V_\xi)(F_0)]}\Orb((\delta,u),\Phi).
$$

  \end{corollary}
\begin{proof}
This follows from Theorem \ref{thm FL 0}, comparing the $\xi$-th coefficients of $\BJ(h,\Phi')$ and $\BJ (h,\Phi)$.
\end{proof}

\subsection{A lemma on Fourier coefficients of modular forms}

Let $\phi$ be a continuous function on $\bH(\BA_0)$, left invariant under $\bH(F_0)$. Recall  its Fourier expansion from \eqref{eq:def F coeff} and \eqref{eq:def F exp}. Let $c_v$ be the level of $\psi_v$, i.e., the maximal integer such that $\psi_v$ is trivial on $\varpi_v^{-c_v}O_{F_0,v}$, where $\varpi_v$ is a uniformizer of $F_{0,v}$.

\begin{lemma}
\label{lem density}
Let $B$ be a finite set of non-archimedean places of $F_0$. Assume that $\phi$ is right invariant under a compact open $K=\prod_{v\nmid\infty} K_v\subset \bH(\BA_{0,f})$ where 
\begin{align}
\label{eqn:Kv}
K_v=m( \varpi_v^{c_v})^{-1} \bH( O_{F_0,v}) m( \varpi_v^{c_v}),\quad m( \varpi_v^{c_v})= \left(\begin{matrix}\varpi_v^{c_v} &\\ &1\end{matrix}\right),
\end{align}
for all $v\in B$. Suppose that $W_{\phi,\xi}(h_\infty)$ vanishes identically (as a function in $h_\infty\in \bH(F_{0,\infty})$)   for all $\xi\in F_0^\times$ such that $(\xi,B)=1$ (i.e., $v(\xi)=0$ for all $v\in B$). Then $\phi$ is a constant function. In particular, if $\phi$ is of (parallel) weight $n$ with $n\neq 0$, then $\phi=0$. 
\end{lemma}

\begin{proof}
We prove the assertion by induction on $\# B$. If $B$ is empty (i.e., $W_{\phi,\xi}(h_\infty)=0$ for all $\xi\in F_0^\times$), then $\phi(h_\infty)=W_{\phi,\xi=0}(h_\infty)$ is left invariant under $N(F_{0,\infty})$, and left invariant under $\bH(F_0)\cap K$.  Now note that, for every $v$, $\bH(F_{0,v})$ is generated by $N(F_{0,v})$ and any single element in $\bH(F_{0,v})\setminus B(F_{0,v})$. It follows that $h_\infty\in\bH(F_{0,\infty})\mapsto\phi(h_\infty)$ is constant, and hence $h\in \bH(\BA_{0})\mapsto \phi(h)$ is constant since $\bH(F_0) \,\bH(F_{0,\infty}) $ is dense in $ \bH(\BA_{0})$ by the strong approximation theorem for $\bH=\SL_{2,F_0}$.

Now assume that $B$ contains at least one element, say $v_0\in B$.
 Consider $b_{v_0}\in\varpi_{v_0}^{-c_{v_0}-1}O_{F_0,v_0}$, and the unipotent matrix 
 $$
 n(b_{v_0}):=\left(\begin{matrix}1 & b_{v_0}\\ &1\end{matrix}\right)\in N(F_{0,v_0}) .$$
Consider the function
$$
\wt \phi(h)\colon=\phi(h\,n(b_{v_0}))-\phi(h),\quad h\in \bH(\BA_0).
$$
Then we {\em claim} that $W_{\wt\phi,\xi}(h_\infty)=0$ for all  $\xi\in F_0^\times$ such that $(\xi, B\setminus\{v_0\})=1$.

To show the claim, the case $v_0(b_{v_0})\geq -c_{v_0}$ is obvious and therefore it suffices to consider the case $v_0(b_{v_0})=-c_{v_0}-1$. From the $K_{v_0}$-invariance it follows that $W_{\phi,\xi}(h_\infty)=0$ and $W_{\phi,\xi}(h_\infty n(b_{v_0}))=0$ unless $v_0(\xi)\geq 0$. If $v_0(\xi)\geq 1$, then $v_0(\xi \,b_{v_0})\geq  -c_{v_0}$ and hence $\psi_{v_0}(\xi\, b_{v_0})=1$. It follows that, unless $v_0(\xi)=0$,  we have 
\begin{align*}
W_{\wt\phi,\xi}(h_\infty)&=W_{\phi,\xi}(h_\infty\,n(b_{v_0}))-W_{\phi,\xi}(h_\infty)
\\ &= \psi_{v_0}(\xi \,b_{v_0}) W_{\phi,\xi}(h_\infty)-W_{\phi,\xi}(h_\infty)= 0.
\end{align*}
The claim now follows.
 
By induction, we conclude that $\wt\phi$ is a constant function, i.e., $W_{\wt\phi,\xi}(h)=0$ for all  $\xi\in F_0^\times$. It follows that $W_{\phi,\xi}(h)$ is right invariant under $N(\varpi_{v_0}^{-c_{v_0}-1}O_{F_0,v_0})$. It is well-known that the groups $N(\varpi_{v_0}^{-c_{v_0}-1}O_{F_0,v_0})$ and $ N_{-}( \varpi_{v_0}^{c_{v_0}} O_{F_0,v_0}) \subset K_{v_0}$ generate $\bH(F_{0,v_0})$ (this is equivalent to that $N( \varpi_{v_0}^{-1}O_{F_0,v_0})$ and $N_{-}(O_{F_0,v_0})$ generate $\bH(F_{0,v_0})$; for a proof, see \cite[Proposition 8.1.2]{LZ}). Here $N_-$ denotes the transpose of $N$. It follows that, for all $\xi\in F_0^\times$, $W_{\phi,\xi}$ is right invariant under $\bH(F_{0,v_0})$ and therefore it must vanish.   Finally the assertion follows from the case when $B$ is empty. 

\end{proof}

\subsection{A refinement of Corollary \ref{coro FL0}}

We recall from \eqref{eqn:def Bn} that $\CB=\CB_n$ is identified with the categorical quotients
$(\U(V)\times V)_{/\!\!/ \U(V)}$ and $(S_n\times V'_n)_{/\!\!/ \GL_n}.$

\begin{lemma}\label{lem !}

Fix  $b_0\in \CB(F_0)$. 
Fix a non-archimedean place $v_1$ of $F_0$, split in $F$.  For every place $v\neq v_1$ of $F_0$, we fix a compact  subset   $\Omega_v\subset \CB(F_{0,v})$ containing $b_0$, such that, for all but finitely many non-archimedean places $v$, $\Omega_v$ is equal to $\CB(O_{F_{0,v}})$. Then there exists a  neighborhood   $\Omega_{v_1}\subset \CB(F_{0,v_1})$ of $b_0$ such that
$$
\CB(F_0)\cap \,\prod_{v} \Omega_v=\{ b_0\},
$$
where the intersection is taken inside $\CB(\BA_0)$.
\end{lemma}

\begin{proof}
We may embed $\CB$ as a closed sub-variety of some affine space $Y=\bA^m$ over $F_0$ (e.g., by \eqref{eqn:def Bn}) such that for almost all $v$, $\CB(O_{F_{0,v}})=Y(O_{F_{0,v}})\cap \CB(F_{0,v})$. For every $v\neq v_1$, by the compactness of $\Omega_v$ we may choose a compact subset $\wt\Omega_v\subset Y(F_{0,v})$ such that $\Omega_v=\wt\Omega_v\cap \CB(F_{0,v})$, and such that $\wt\Omega_v=Y(O_{F_{0,v}})$ for almost all $v$. By a standard argument using the product formula (i.e., $\prod_{v}|x|_v=1$ for $x\in F_0^\times$), there must be a small neighborhood   $\wt\Omega_{v_1}\subset Y(F_{0,v_1})$ such that 
$$
Y(F_0)\cap \,\prod_{v} \wt\Omega_v=\{ b_0\}.
$$
Set $\Omega_{v_1}=\wt\Omega_{v_1}\cap \CB(F_{0,v_1})$ to complete the proof.

\end{proof}

We are now ready to refine the result of Corollary \ref{coro FL0}.

\begin{proposition}\label{prop refine}
Under the assumption of Theorem \ref{thm FL 0}, we have, for every $(\delta,u) \in (\U(V)(\alpha)\times V)(F_0)_\rs$ matching $(\gamma, u')\in (S_{n}(\alpha)\times V'_n)(F_0)_\rs$ such that $\xi=\fkq(u)\neq 0$,
$$
\Orb((\gamma,u'),\Phi')=\Orb((\delta,u),\Phi).
$$
\end{proposition}

\begin{proof}
Let $b_0\in \CB(F_0)$ be the (common) image of $(\delta,u)$ and $(\gamma, u')$, and let $\xi=\fkq(u)=\fkq(u')$.
Fix a non-archimedean place $v_1$ of $F_0$, split in $F$. Decompose
$$
\Orb((\delta,u),\Phi)=\Orb((\delta,u),\Phi^{v_1})\Orb((\delta,u),\Phi_{v_1}),
$$
and $$
\Orb((\gamma,u'),\Phi')=\Orb((\gamma,u'),\Phi'^{v_1})\Orb((\gamma,u'),\Phi'_{v_1}),
$$
where the local orbital integral $\Orb((\delta,u),\Phi_{v_1})=\Orb((\gamma,u'),\Phi'_{v_1})$. 
 We may assume that the local orbital integrals at $v_1$ are nonzero (otherwise both sides vanish). It remains to show
\begin{align}\label{??}
\Orb((\delta,u),\Phi^{v_1})=\Orb((\gamma,u'),\Phi'^{v_1}).
\end{align}

For every non-archimedean $v\neq v_1$, we define a compact set $\Omega_v\subset\CB(F_{0,v})$  to be the image of the support of $\Phi_v$ for $v\in S$, and $\Omega_v= \CB(O_{F_{0,v}})$ for all $v\notin S$.

For $v\mid\infty$, we define $\Omega_v$ to be the image of $(\U(V)(\alpha)\times V_\xi)(F_{0,v})$. Since the hermitian space $V\otimes_{F_0}F_{0,v}$ is positive definite, the set $(\U(V)(\alpha)\times V_\xi)(F_{0,v})$ is compact, and hence so is $\Omega_v$.

Now apply Lemma \ref{lem !} to choose a small neighborhood   $\Omega_{v_1}\subset \CB(F_{0,v_1})$ of $b_0$ such that 
$\CB(F_0)\cap \,\Omega=\{ b_0\}$ where $\Omega=\prod_{v}\Omega_v$. Then we choose a {\em point-wise non-negative} function $\wt\Phi_{v_1}$ with non-empty support whose image in $ \CB(F_{0,v_1})$ contains $b_0$ and is contained in $\Omega_{v_1}$. Choose $\wt\Phi'_{v_1}$ to match $\wt\Phi_{v_1}$ in the elementary way (cf. Remark \ref{rem split}). Now apply Corollary \ref{coro FL0} to this new pair of functions $\wt\Phi=\Phi^{v_1}\otimes\wt\Phi_{v_1}$ and $\wt\Phi'=\Phi'^{v_1}\otimes\wt\Phi'_{v_1}$
$$
\sum_{(\gamma,u')\in [(S_n(\alpha)\times V'_\xi)(F_0)]}\Orb((\gamma,u'),\wt\Phi')=\sum_{(\delta,u)\in [(G(\alpha)\times V_\xi)(F_0)]}\Orb((\delta,u),\wt\Phi).
$$
Now the non-zero terms in each side only involve regular semsimple orbits with invariants in $\Omega$: this is clear for the unitary side by our choice of $\Omega$, and it is true for the left hand side because $\Phi'_v$ partially (relative to $\alpha$) transfers to $\Phi_v$ for all $v\in S$ and $v\mid\infty$. It follows that each side has one term left, namely the one with invariant $b_0\in\CB(F_0)$:
$$
\Orb((\gamma,u'),\wt\Phi')=\Orb((\delta,u),\wt\Phi).
$$
 By the point-wise positivity of $\Phi_{v_1}$, the local orbital integral at the place $v_1$ does not vanish. We hence deduce the desired equality \eqref{??}.
This completes the proof.
\end{proof}

\subsection{The proof of FL conjecture}

Now we return to the set up of  Conjecture \ref{FLconj} in \S\ref{s:FL}. 

\begin{theorem} \label{thm FL}
Conjecture \ref{FLconj} holds for $F_0$ with $q\geq n$ (recall that $q$ denotes the cardinality of the residue field of $O_{F_0}$).
\end{theorem}
\begin{proof}
By Proposition \ref{prop FL L2G}, it suffices to  prove  Conjecture \ref{FLconj} part \eqref{FLconj smilie}. We will do so by induction on $\dim V_0$.

The case $\dim V_0=1$ is trivial. Assume now that  Conjecture \ref{FLconj} part \eqref{FLconj smilie}  holds when $\dim V_0=n-1$. Then by Proposition \ref{prop FL L2G} part \eqref{L2G i},  Conjecture \ref{FLconj} part \eqref{FLconj gp}  holds for $S_{n}$ over $F_0$ with $q\geq n$.

  We now want to apply Proposition  \ref{prop refine}. We now change the notation: we denote by $F_0$ a totally real field with a place $v_0$ such that $F_{0,v_0}$ is the local field in Conjecture \ref{FLconj}. We then choose the following local data:
\begin{itemize}
\item  an unramified (local) quadratic extension $F_{w_0}/F_{0,v_0}$,
\item  the split $F_{w_0}/F_{0,v_0}$-hermitian space $V_{v_0}$ of dimension $n$, 
\item an element $(g_{v_0}, u_{v_0})\in (\U(V_{v_0})\times V_{v_0})(F_{0,v_0})_{\srs}$, we further assume that the  characteristic polynomial of $g_{v_0}$ has integral coefficients (in $O_{F_{w_0}}$), $\det(1-g_{v_0})$ is a unit, and $\pair{u_{v_0},u_{v_0}}\neq 0$,
\item  an element $(\gamma_{v_0}, u'_{v_0})\in (S_{n}\times V'_n)(F_{0,v_0})_\srs$ matching  $(g_{v_0}, u_{v_0})$.
\end{itemize}

 To globalize the data, we first use Cayley transform, cf. \eqref{Cayley}.  Let $x_{v_0}=\fkc^{-1}(g_{v_0})=\frac{1+g_{v_0}}{1-g_{v_0}}$, an element in the Lie algebra $ \fku(V_{v_0})\subset \End_{F_{w_0}}(V_{v_0})$. 

We now choose a totally negative element  $\epsilon\in F_0^\times$ such that $v_0$ is inert in $F=F_0[\sqrt{ \epsilon}]$. Denote by $w_0$ the place of $F$ above $v_0$.  Consider $x_{v_0}^\nat=\frac{x_{v_0}}{ \sqrt{\epsilon}}$. Then the characteristic polynomial of $x_{v_0}^\nat$ has coefficients in the base field $F_{v_0}$.

Next we choose a totally real field $F'_0$ with $[F_0':F_0]=n$ and an element $x^\nat\in F_0'$ such that, when setting  $F'=F_0'[\sqrt{ \epsilon}]$ and
$$ x=\sqrt{ \epsilon}\, x^\nat , \quad g=\fkc(x)=-\frac{1-x}{1+x},$$ we have $O_{F_{w_0}}[g]=O_{F_{w_0}}[  g_{v_0}]$ as subrings of $F'\otimes_{F} F_{w_0}$. To achieve this, it suffices to approximate the characteristic polynomial of $x_{v_0}^\nat$  by a polynomial with coefficient in $F_0$, and we may prescribe its local behavior at finitely many places by weak approximation. (Here  the regular semi-simplicity allows us to determines the isomorphism class of the local field $F[g_{v_0}]$ by the characteristic polynomial of $g_{v_0}$.)

With such a choice, we have $g\in F'^1$. For the CM extension $F'/F_0'$, there exists a one-dimensional  $F'/F_0'$-hermitian space $W$ such that $W$ is totally positive definite, and $V:=\Res_{F'_0/F_0}W$, as an $n$-dimensional $F/F_0$-hermitian space, is locally at $v_0$ isometric to $V_{v_0}$. Such hermitian space $W$ exists because we are only imposing local conditions at finitely many places. Then we have an embedding $\xymatrix{\Res_{F_0'/F_0}\RU(W)\to \RU(V)}$ and we will view $g$ as an element of $ \RU(V)(F_0)$. Let $\alpha\in \CA_n(F_0)$ denote its characteristic polynomial.

Now we choose $u\in V$ (and possibly replacing $g$ by an element $v_0$-adically closer to $g_{v_0}$) such that the pair $(g,u)$ is $v_0$-adically close to $(g_{v_0}, u_{v_0})$ and such that
\begin{align}\label{eq:approx orb U}
\Orb((g,u),{\bf 1}_{(\U(V_{v_0})\times V_{v_0})(O_{F_{0,v_0}})})=\Orb((g_{v_0}, u_{v_0}),{\bf 1}_{(\U(V_{v_0})\times V_{v_0})(O_{F_{0,v_0}})}).
\end{align}
This is possible due to the local constancy of orbital integrals near a regular semisimple element. Let $(\gamma,u')\in (S_n\times V_{n}')(F_0)$ be a regular semisimple element matching $(g,u)$. Again by local constancy of orbital integrals we may assume that, possibly replacing $(g,u)$ by an element in $(\U(V)\times V)(F_0)_{\srs}$ that is $v_0$-adically closer to $(g_{v_0}, u_{v_0})$, 
\begin{align}\label{eq:approx orb S}
\Orb((\gamma,u'),{\bf 1}_{(S_n\times V_n')(O_{F_{0,v_0}})})=\Orb((\gamma_{v_0}, u'_{v_0}),{\bf 1}_{(S_n\times V_n')(O_{F_{0,v_0}})}).
\end{align}

Next, we let $S$ be a finite set of of non-archimedean places of $F_0$, such that
\begin{itemize}\item $v_0\notin S$,
\item $S$ contains all places with residue cardinality less than $n$, 
\item for every non-archimedean  $v\notin S\cup\{v_0\}$, the ring $R_{\alpha}$ is locally maximal at $v$, and $V_v$ is a split hermitian space.
\end{itemize}

Choose functions $\Phi=\otimes_v\Phi_v$ and $\Phi'=\otimes_v\Phi_v'$ satisfying the following conditions:
\begin{itemize}
\item for every archimedean $v$, $\Phi_v$ and $\Phi_v'$ are the (partial)  Gaussian test functions (relative to a small neighborhood of $\gamma$ in $S_n(F_{0,v})$),
\item for every non-archimedean $v\in S$, $\Phi_v'$ partially (relative to $\alpha$) transfers to  $\Phi_v\in\CS((\U(V_v)\times V_v)(F_{0,v}))$ and the local orbital integrals do not vanish at $(g,u)$ and $(\gamma,u')$,  cf. Definition \ref{def p tran},
\item for every non-archimedean $v\notin S$ (in particular at $v_0$), noting that the hermitian space $V_v$ is split, choose $\Phi_v={\bf 1}_{(\U(V)\times V)(O_{F_0,v})} $ (w.r.t. a self-dual lattice in $V_v$), and $\Phi'_v={\bf 1}_{(S_{n} \times V')(O_{F_0,v})} $. By enlarging $S$ suitably (while keeping $v_0\notin S$), we may further assume that, for every $v\notin S$, the image of $(g,u)$ in $\CB(F_0)$ lies in the image of the support of $(\U(V)\times V)(O_{F_0,v})$. 
\end{itemize}
By the last condition, for every non-archimedean $v\notin S$ the local orbital integral of $\Phi_v$ does not vanish at   $(g,u)$ (since the function $\Phi_v$ is point-wisely positive on its support). It follows from the special case Proposition \ref{FL DVR} that the same non-vanishing for $\Phi_v'$ for every place $v\notin S\cup\{v_0\}$.

Now, by the induction hypothesis, we are ready to apply Proposition \ref{prop refine} to conclude
$$
\Orb((\gamma,u'),\Phi')=\Orb((g,u),\Phi).
$$ 
By our choices, $\Orb((\gamma,u'),\Phi'_v)=\Orb((g,u),\Phi_v)$ for all $v\neq v_0$, and they do not vanish. It follows that 
$$
\Orb((\gamma,u'),\Phi'_{v_0})=\Orb((g,u),\Phi_{v_0}).
$$
By \eqref{eq:approx orb U} and \eqref{eq:approx orb S}, we have  
$$
\Orb((\gamma_{v_0}, u'_{v_0}),{\bf 1}_{(S_n\times V_n')(O_{F_{0,v_0}})})=\Orb((g_{v_0}, u_{v_0}),{\bf 1}_{(\U(V_{v_0})\times V_{v_0})(O_{F_{0,v_0}})}).
$$

We have assumed that $g_{v_0}\in\U(V_{v_0})$ is regular  semisimple with $\det(1-g_{v_0})\in O_{F_{w_0}}^\times$  and $\pair{u_{v_0},u_{v_0}}\neq 0$. We now remove these assumptions. The condition $\det(1-g_{v_0})\in O_{F_{w_0}}^\times$ is harmless since we may multiply $g_{v_0}$ by a suitable element in $F^1_{w_0}$ (cf. the proof of Prop. \ref{prop AFL L2G}). The set of elements $(g_{v_0}, u_{v_0}) \in (\U(V_{v_0})\times V_{v_0})(F_{0,v_0})_\srs$ with $\pair{u_{v_0},u_{v_0}}\neq 0$ is dense in $ (\U(V_{v_0})\times V_{v_0})(F_{0,v_0})_\rs$.  By local constancy of orbital integrals at regular  semisimple elements, Conjecture \ref{FLconj} part \eqref{FLconj smilie}  holds when $\dim V_0=n$ (over $F_{0,v_0}$ with $q\geq n$). This completes the induction.  
\end{proof}

\section{The comparison for arithmetic intersections}

As a preparation for the proof of AFL conjecture, 
in this section, we compare $\delJ(h,\Phi')$ with the arithmetic intersection number $\Int(\tau,\Phi)$ (cf. \eqref{int g0} in \S\ref{s:int}). 

Let  $V$ be the $n$-dimensional $F/F_0$-hermitian space that we use to define the Shimura variety $\Sh_{K_{\wt G}}\bigl(\wt G, \{h_{\wt G}\}\bigr)$ in \S\ref{ss:S data}. 

As in \S\ref{ss:global tran}, we fix an irreducible $\alpha\in \CA_n(F_0)\subset F[T]_{\deg=n}$ and fix $\gamma\in S_n(\alpha)(F_0)$ (cf. \S\ref{ss: n>1}). Let $\Phi=\otimes_{v<\infty}\Phi_v \in\CS((\U(V)\times V)(\BA_{0,f}))$ be a pure tensor. 
Let  $\Phi'= \otimes_{v}\Phi'_v \in \CS((S_{n} \times
V'_n)(\BA_0))$ be a pure tensor such that 
\begin{itemize}
\item  for every $v\mid\infty$,  $\Phi'_v$ is the partial Gaussian test function, and 
 \item for every non-archimedean $v$, $\Phi_v'$ partially (relative to $\alpha$) transfers to $\Phi_v$.
   \end{itemize}

 \begin{remark}We are now in the ``incoherent" case in the following sense. Due to the signature of such $V$ at the archimedean places, there does not exist any {\em global} $F/F_0$-hermitian space $\wt V$ such that $\Phi'$ transfer to a function in $\CS((\U(\wt V)\times \wt V)(\BA_{0}))$, cf. \cite[\S3.2]{Z12}. 
\end{remark}
 
We now study $\delJ(h,\Phi')$, and we recall $\delJ_v(h,\Phi')$ from \eqref{delJ}.
\begin{lemma}\label{lem:del v spl}
Let $v$ be a place of $F_0$ split in $F$ (necessarily non-archimedean). Then $\delJ_v(h,\Phi')=0$.
\end{lemma}

\begin{proof}This follows from the same argument in \cite[Prop.\ 3.6]{Z12} (also cf. \cite[\S7.2]{RSZ3}). 
\end{proof}

If $v$ is non-split (including the archimedean places), let $V(v)$ be the ``nearby" $F/F_0$-hermitian space  at $v$, cf. Theorem \ref{thm inert} (resp., Theorem \ref{thm infty}) for non-archimedean (resp., archimedean) places. Then, for a non-archimedean $v$, the regular semisimple terms in $\delJ_v(h,\Phi')$  (cf. \eqref{delJ}) is a sum over orbits $(\gamma,u')\in [(S_n(\alpha)\times V')(F_{0})]_\rs $ matching $(\delta,u)\in [(\U(V(v))\times V(v))(F_0)]_\rs$; we will show that the same holds for archimedean places $v$, cf. Lemma \ref{lem I=J infty}. Moreover, we have a Fourier expansion (cf. \eqref{eq:def F coeff})
\begin{align}
\label{dJ v qexp}
\delJ_v(h,\Phi')=\sum_{\xi\in F_0}\delJ_v(\xi,h, \Phi'),
\end{align}
where $\delJ_v(\xi,h, \Phi')$ is the sub-sum, 
\begin{align}
\label{dJ v xi}
\delJ_v(\xi,h,\Phi')=\sum_{(\gamma,u')\in [(S_n(\alpha)\times V'_\xi)(F_0)]\atop u'\neq 0} \del((\gamma,u'),\omega(h_v)\Phi'_v)\cdot  \Orb((\gamma,u'),\omega(h^v)\Phi'^{v}).
\end{align}

\subsection{The archimedean places}
Let $v\mid \infty$. Recall that $V'=V_0\times V_0^\ast$ for $V_0=F_0^n$ carries the tautological quadratic form \eqref{eq: q on V'}
$$\fkq \colon 
\xymatrix@R=0ex{V_0\times V_0^\ast \ar[r] &F_0},
$$and an induced 
quadratic form   \eqref{eq: q' on V'} $$\fkq' \colon
\xymatrix@R=0ex{  V_0\times V_0^\ast \ar[r] &F_0'},
$$
such that for all $u'\in V'$, $
\fkq'(u')=\tr_{F_0'/F_0}\, \fkq'(u'). $
 Set
$$F'_{0,v}:=F_0'\otimes_{ F_0,v}\BR\simeq\prod_{v'\in \Hom(F_0',\BR),v'\mid v} \BR.
$$

\begin{lemma}
\label{lem xi'}
Let $\xi' \in F'_{0,v}$ be an invertible element, and $u'\in V'(F_{0,v})$ with $\fkq'(u')=\xi'$.
\begin{altenumerate}
\renewcommand{\theenumi}{\alph{enumi}}
\item
 $$
\Orb((\gamma,u'),\Phi'_v)=\begin{cases}
e^{-\pi \tr_{F'_{0,v}/F_{0,v}}(\xi')}, &\text{when }\xi' \in F'_{0,v} \text{ is totally positive},\\ 
0,& \text{otherwise}.
\end{cases}
$$\item
Now assume that 
$\xi'$ is not totally positive. 
 Then $\del((\gamma,u'),\Phi'_v)=0$, unless $\xi'$ is negative at exactly one archimedean place $v'$ of $F_0'$ and this place  $v'$ is above $v$, in which case
$$
\del((\gamma,u'),\Phi'_v)= \frac{1}{2} e^{-\pi  \tr_{F'_{0,v}/F_{0,v}}(\xi')}\,\Ei(-2\pi|\xi'|_{v'}).
$$
\end{altenumerate}
\end{lemma} 

 \begin{proof}
 This follows from
Lemma \ref{lem gaussian n=1}, and Lemma \ref{lem Phi2phi}.
 \end{proof}

 \begin{lemma}
\label{lem I=J infty} 

 Let $\xi\in F_0^\times$. Then 
$$
\Int_{v}^{\bK}(\xi,\Phi) =-2     \delJ_v(\xi,\Phi').
$$

\end{lemma} 

 \begin{proof}
 It follows from the previous Lemma \ref{lem xi'}
 that 
\begin{equation*}
    \delJ_v(\xi,\Phi')=\frac{1}{2}
\sum_{} \Ei(-2\pi |\xi'|_{v'})\cdot \Orb\left((\gamma,u'), \Phi'\right),
\end{equation*}
where the sum runs over $(\gamma,u')\in [(S_n(\alpha)\times V'_\xi)(F_{0})]$ such that the refined invariant $\fkq'(u')=\xi'\in F_0'$ is negative at exactly one archimedean place $v'$ of $F_0'$ and this place is above $v$.

 By Corollary \ref{coro ht infty}, 
 \begin{equation*}
	\Int^\bK_{v}(\xi,\Phi) =-  
\sum_{} \Ei(-2\pi |\xi'|_{v'})\cdot \Orb\left((\delta,u), \Phi\right),
\end{equation*} 
where the sum runs over $ (\delta,u)\in [(\U(V(v))(\alpha)\times V(v)_\xi)(F_0)]$ such that the refined invariant $\xi'$ is negative at  exactly one archimedean place $v'$ of $F_0'$ and this place is above $v$.

Therefore, the orbits $(\delta,u)$ in the sum in $\Int^\bK_{v}(\xi,\Phi)$ are bijective to the orbits $(\gamma,u')$  in the sum in $\delJ_v(\xi,\Phi')$.  Now the assertion follows from the fact that $\Phi^{\infty}$ and $\Phi'^{\infty}$ are partial transfers of each other.
 
  \end{proof}

 \subsection{``Holomorphic projection"}

 For the rest of this section, we assume that $F_0=\BQ$. Recall that the difference $\Int^{\bK-\bB}(h,\Phi)$ between the two Green functions is given by \eqref{B-Kv} and \eqref{B-K} (note that this makes sense for any Schwartz function $\Phi^\infty$).
 The following result plays the role of ``holomorphic projection" of the modular generating function on the analytic side.

 \begin{proposition}\label{prop incoh}
 
 Let $F_0=\BQ$. The sum 
 $$
        \delJ_{\rm hol}(h):=      2 \delJ (h,\Phi')+\Int^{\bK-\bB}(h,\Phi),\quad h\in \bH({\BA_0}),
 $$
 lies in $\CA_{\rm hol}(\bH({\BA_0}),K,n)$, where $K$ is the compact open subgroup of $\SL_{2}(\BA_{0,f})$ that acts trivially on both $\Phi$ and $\Phi'$.

 \end{proposition}
 \begin{proof}First of all the function $h\in\bH({\BA_0})\mapsto 2\delJ (h,\Phi')$ belongs to $\CA_{\rm exp}(\bH({\BA_0}),K,n)$. One way to see this is to use the Fourier expansion directly.  Another way is to identify it with a linear combination to $\SL_{2}(\BA_{0})$ of the restriction of (the first derivative at $s=0$ of) a degenerate Siegel--Eisenstein series of parallel weight one on $\SL_{2}(\BA_{F_0'})$, cf. Remark \ref{rem SE}. 
 
 By Corollary \ref{cor:mod int infty}, the second summand $\Int^{\bK-\bB}(\cdot,\Phi)$ also belongs to  $\CA_{\rm exp}(\bH({\BA_0}),K,n)$. Therefore to complete the proof, it suffices to show the holomorphy of the sum $ \delJ_{\rm hol}(h)$ on  the complex upper half plane $\CH$ and at all cusps. Equivalently,  for any $h_f\in \bH({\BA_{0,f}})$, the function $ \delJ^\flat_{{\rm hol}, h_f}$ (associated to  $\delJ_{\rm hol}$ via \eqref{phi2phi flat}) is holomorphic, and holomorphic at the cusp $i\infty$.  Since we can vary $\Phi^\infty$ and $\Phi'^{\infty}$, and by Theorem \ref{thm Weil ST} the Weil representation commutes with (partially relative to $\alpha$) smooth transfer, it suffices to consider the case $h_f=1$ (but allow all matching  $\Phi^\infty$ and $\Phi'^{\infty}$).

We {\em claim} that  the Fourier expansion of $ \delJ^\flat_{{\rm hol}}$  takes the form 
\begin{align}\label{eq:hol qexp}
 \delJ^\flat_{{\rm hol}}(\tau)=\sum_{\xi\in F_{0},\xi\geq 0} A_\xi \,q^\xi,\quad A_\xi\in \BC,
\end{align}
where $A_\xi =0$ unless $\xi$ lies in a (fractional) ideal of $F_0$ (depending on $\Phi',\Phi$). In other words, the non-holomorphic terms all cancel out. The desired holomorphy follows from the claim.

 To show the claim, we use the decomposition \eqref{eqn J' dec} as a sum over places $v$ of $F_0$.  
 
 First, by Lemma \ref{lem:del v spl},  $\delJ_v(h,\Phi')=0$ if $v$ is a split place. 

Next let  $v$ be a non-archimedean non-split place. By \eqref{dJ v qexp} and  \eqref{dJ v xi}, and the fact that $\Phi'_\infty$ is a (partial) Gaussian test function,  we have $\delJ_v(\xi,h, \Phi')=0$ unless $\xi\geq 0$.  We obtain
  \begin{align}\label{eq:delJ v qexp 1}
\delJ_v^\flat(\tau,\Phi')=\sum_{\xi\in F_0,\xi\geq 0}\delJ_v(\xi, \Phi')q^\xi ,
\end{align}
where,  for $\xi\geq 0$,
\begin{align}\label{eq:delJ v xi 1}
\delJ_v(\xi, \Phi')=\sum_{(\gamma,u')\in [(S_n(\alpha)\times V'_\xi)(F_0)]\atop u'\neq 0}  \del((\gamma,u'),\Phi'_v)\cdot \Orb((\gamma,u'),\Phi'^{v,\infty}).
\end{align}
 It follows that $ 2 \delJ ^\flat_v(\tau,\Phi')$ has the desired form of Fourier expansion as in \eqref{eq:hol qexp}.

Finally  let  $v\mid\infty$. We observe that by Lemma \ref{lem I=J infty}, modulo the constant terms, the sum $2\delJ_v(h,\Phi')+\Int_{v}^{\bK-\bB}(h,\Phi)$ is equal to $-\Int_{v}^{\bB}(h,\Phi)$, which has the desired form of Fourier expansion. It remains to consider the constant terms. Note that Lemma \ref{lem xi'} also applies to all $u'\in V'$ with refined invariant $\fkq'(u')=\xi'\in F_0'^\times$ (possibly $\xi=\tr_{F_0'/F_0}\xi'=0$).  Similarly Theorem \ref{thm infty}  also applies to all $u\in V(v)$ with refined invariant $\fkq'(u)=\xi'\in F_0'^\times$ (i.e., $u\neq 0$). Therefore, by the proof of Lemma \ref{lem I=J infty}, the contribution from null-norm ($\xi=0$) non-zero vectors $u\in V(v)$  cancels that from  $u'\in V'$ with $\fkq'(u')=\xi'\neq 0\in F_0'$. It follows that the constant term of $2\delJ_v(h,\Phi')+\Int_{v}^{\bK-\bB}(h,\Phi)$ is the sum of the nilpotent term $2\del(0_\pm,\omega(h)\Phi')$ (from $ 2 \delJ (h,\Phi')$, cf. \eqref{eq: nilp term}) and the only term that has not been cancelled in \eqref{B-Kv}, which is by \eqref{Gr Ku m=0}
 $$
 -\sum_{\delta\in [\U(V(v))(\alpha)(F_0)]}\Orb((\delta,0),\Phi)\log |a|_v,
 $$
 cf. \eqref{orb u0} and \eqref{eqn:cst term}.  By Lemma \ref{nil infty} and Lemma \ref{lem Phi2phi}, the nilpotent term is 
 $$
 2\del(0_\pm,\Phi')=2\Orb((\gamma,0_+),\Phi')\, \log|a|_v +C
 $$for some constant $C$ (depending on $\gamma, \Phi'$).
Since $\Phi$ match $\Phi'$, we claim
$$
\Orb((\gamma,0_+),\Phi')=-\Orb((\gamma,0_-),\Phi')=\frac{1}{2} \sum_{\delta\in [\U(V(v))(\alpha)(F_0)]}\Orb((\delta,0),\Phi).
$$
In fact, this follows from the argument in \cite[(10.4)]{Jac1} (for the quadratic extension $F'/F_0'$). In {\it loc. cit.}, Jacquet proves the analogous identity in the ``coherent" case (here ``coherent" is in the sense of Kudla for one-dimensional hermitian spaces). Since the proof verbatim applies to the current setting, we omit the detail. Therefore the two terms with $\log|a|_v$ cancel, and the sum is a constant independent of $\alpha$. This shows that $2\delJ_v(h,\Phi')+\Int_{v}^{\bK-\bB}(h,\Phi)$ also has the desired form of Fourier expansion when $v\mid\infty$.

The proof is now complete.

 \end{proof}

 \subsection{The comparison}
 Now let  $\CM=\CM_{K_{\wt G}}(\wt G)$ be the moduli stack introduced in Definition \ref{def RSZ glob}. 
 Let $S$ be a finite set of non-archimedean places of $F_0$ such that
\begin{itemize}
\item $S$ contains all places $v\mid\fkd$ and all places with residue cardinality $<\dim V$, and
\item for every non-archimedean $v\notin S$, the hermitian space $V_v$ is split, $\Phi_v={\bf 1}_{(\U(V)\times V)(O_{F_0,v})} $ (w.r.t. a self-dual lattice in $V_v$), and $\Phi'_v={\bf 1}_{(S_{n} \times V'_n)(O_{F_0,v})} $.
\end{itemize}
Now we have the FL for all places $v\notin S$ by Theorem \ref{thm FL}, hence  $\Phi_v$ and  $\Phi_v'$ match for every place $v\nmid S$.
Then in Proposition \ref{prop incoh}, we can assume that the compact open subgroup $K\subset \bH({\BA_{0,f}})$ is the principal congruence subgroup $K(N)$ of level $N$, where the prime factors of $N$ are all contained in $S$.

We have been assuming that the function $\Phi\in\CS((\U(V)\times V)(\BA_{0,f}))$ is valued in $\BQ$. By Proposition \ref{prop incoh},  $\delJ_{\rm hol}^\flat(\cdot,\Phi')$ lies in $\CA_{\rm hol}(\Gamma(N),n)_{\BQ}\otimes_\BQ\BR$ (the Green function takes values in $\BR$). By passing to the quotient $\BR\to \BR_S$ (cf. \eqref{RS}) and then extending coefficients $\BR_S\to\BR_{S,\ov\BQ}$, we obtain an element, still denoted by $\delJ_{\rm hol}^\flat(\cdot,\Phi')$, in $\CA_{\rm hol}(\Gamma(N),n)_{\ov\BQ}\otimes_{\ov\BQ}\BR_{S,\ov\BQ}$. 

By Theorem \ref{thm BHKRY} (cf. \eqref{eq: Int mod}), $\Int(\cdot,\Phi)$ (defined by \eqref{int g0}) also belongs to $\CA_{\rm hol}(\Gamma(N), n)_{\ov\BQ}\otimes_{\ov\BQ}\BR_{S,\ov\BQ}$, hence so does the sum
 $$
    \sE^\flat(\tau)   =2 \delJ_{\rm hol}^\flat(\tau,\Phi')+\Int(\tau,\Phi) ,\quad \tau\in \CH.
 $$
Write the Fourier expansion (at the cusp $i\infty$) as
$$
\sE^\flat(\tau)=\sum_{\xi\in F_0,\, \xi\geq 0}A_\xi \,q^{\xi},\quad A_\xi\in\BR_{S,\ov\BQ}.
$$

 \begin{theorem}\label{thm AFL0}
 Assume that Conjecture \ref{AFLconj} part \eqref{AFL gp} holds for all $p$-adic field $\BQ_p$ with $p\notin S$ and for $S_{n}$. Then
\begin{altenumerate}
\renewcommand{\theenumi}{\alph{enumi}}
 \item $\sE^\flat=0$.
 \item For every non-archimedean place $v\notin S$, and $\xi\in F_0^\times$, we have $$-2\delJ_v(\xi,\Phi')=\Int_{v}(\xi,\Phi)$$
 as an equality in $\BQ\log p_v$ where  $p_v$ denotes the residue characteristic of $F_{0,v}$. Here $\delJ_v(\xi,\Phi')$ (resp. $\Int_{v}(\xi,\Phi)$) is defined by \eqref{eq:delJ v xi 1} (resp. \eqref{intloc}).
\end{altenumerate}
 \end{theorem}

\begin{proof}
The proof of part (a) is analogous to that of Theorem \ref{thm FL 0}. Recall from \eqref{eq: Int=sum v}  and \eqref{eq: Int v=sum xi} that we have a decomposition of the generating function $\Int(\tau,\Phi)$ (excluding the constant term $\Int(0,\Phi)$ defined by \eqref{int g0 xi0}) as a sum of $\Int_{v}(\tau,\Phi)$ over the places $v$ of $F_0$. We have the following equalities,  as formal power series in $\BR_{S,\ov\BQ}[\![q^{1/N}]\!]$ modulo constant terms:
\begin{align*}
    \CE^\flat(\tau)-\Int(0,\Phi)=&  (  2 \delJ ^\flat(\tau,\Phi')+\Int^{\bK-\bB}(\tau,\Phi))+(\Int(\tau,\Phi)-\Int(0,\Phi))
    \\=& \sum_{v\mid \infty}\left(2\delJ_v^\flat(\tau,\Phi')+\Int_{v}^{\bK-\bB}(\tau,\Phi)+\Int_v(\tau,\Phi)\right) +\sum_{v<\infty} \left(2\delJ_v^\flat(\tau,\Phi')+\Int_{v}(\tau,\Phi)\right)\\
    =&\sum_{v\mid \infty}\left(2\delJ_v^\flat(\tau,\Phi')+\Int_{v}^{\bK}(\tau,\Phi)\right) + \sum_{v<\infty} \left(2\delJ_v^\flat(\tau,\Phi')+\Int_{v}(\tau,\Phi)\right)\\
    =& \sum_{v\nmid\infty,\,v \notin S} \left(2\delJ_v^\flat(\tau,\Phi')+\Int_{v}(\tau,\Phi)\right),
\end{align*}
where the last equality (modulo the constant term) follows from Lemma \ref{lem I=J infty}. Here the sum $2\delJ_v^\flat(\tau,\Phi')+\Int_{v}^{\bK-\bB}(\tau,\Phi)$ (resp., $2\delJ_v^\flat(\tau,\Phi')+\Int_{v}^{\bK}(\tau,\Phi)$) both belong to $\BR_{S,\ov\BQ}[\![q^{1/N}]\!]$, even though each summand does not due to the presence of ``non-holomorphic" terms.

Recall from \eqref{eq: Int v=sum xi} that we have the expansion of $\Int_{v}(\tau,\Phi)$ in terms of $\Int_{v}(\xi,\Phi)$ defined by \eqref{intloc}. Also recall from  \eqref{eq:delJ v qexp 1} that we have an expansion of $\delJ_v^\flat(\tau,\Phi')$ in terms of $\delJ_v(\xi,\Phi')$  defined by  \eqref{eq:delJ v xi 1}.

Let $B$ be  the (finite) set of non-archimedean inert places  $v\notin S$ of $F_0$ where $R_v$ is not a maximal order in $F'_v=F'\otimes_{F_0}F_{0,v}$. By the vanishing criterion Lemma \ref{lem density}, to show $\CE^\flat=0$, it remains to show the vanishing of the $\xi$-th Fourier coefficients when $(\xi,B)=1$.

 If $v\notin S$ is split in $F$, the intersection number $\Int_{v}(\tau,\Phi)=0$ vanishes by Corollary \ref{coro split}, and $2\delJ_v^\flat(\tau,\Phi')=0$ by Lemma \ref{lem:del v spl}. 

If $v\notin S$ is inert, then by Theorem \ref{thm inert} and \eqref{eq: Int v=sum xi} we obtain the $q$-expansion of $\Int_{v}(\tau,\Phi)$. Similarly,  \eqref{eq:delJ v qexp 1} and \eqref{eq:delJ v xi 1} give the $q$-expansion of $ 2\delJ_v^\flat(\tau,\Phi')$. There are two cases:
\begin{altenumerate}
\renewcommand{\theenumi}{\alph{enumi}}
 \item[($i$)] \,\, If $v\notin S\cup B$, then $R_v$ is an maximal order and we apply Proposition \ref{AFL DVR} at $v$ to conclude that the $v$-th summand $2\delJ_v^\flat(\tau,\Phi')+\Int_{v}(\tau,\Phi)$ is zero (note that now $\Phi^{(v)}$ and $\Phi'^{(v)}$ match).

\item[($ii$)] \,\, If $v\in B$, then the $v$-th term $2\delJ_v^\flat(\tau,\Phi')+\Int_{v}(\tau,\Phi)$ is a formal power series in $q^{1/N}$ with coefficients in $\BQ\log q_v $ (or its image in $\BR_{S,\ov\BQ}$). By Proposition \ref{prop AFL L2G} , our assumption on Conjecture \ref{AFLconj} part \eqref{AFL gp} (for all $p$-adic field $\BQ_p$ with $p\notin S$ and for $S_{n}$) implies that, for $(\gamma,u')$ matching $(\delta,u)$,
\[
-\del\bigl((\gamma,u'), \mathbf{1}_{(S_{n}\times V'_{n})(O_{F_{0,v}})}\bigr) 
	   = \Int_{v}(\delta,u)\cdot\log q_v,
\]
whenever $\fkq(u)=\fkq(u')=\xi$ is a unit at $v$. In other words, the $\xi$-th Fourier coefficient of  $2\delJ_v^\flat(\tau,\Phi')+\Int_{v}(\tau,\Phi)$ vanishes if $\xi$ is a unit $v$, which holds if $(\xi,B)=1$. 
\end{altenumerate}
Therefore, whenever $(\xi,B)=1$, the $\xi$-th Fourier coefficient of  $\CE^\flat-\Int(0,\Phi)$ (equivalently  $\CE^\flat$) vanishes. By Lemma \ref{lem density} this completes the proof of part (a).

Now we turn to part (b). By $\CE^\flat=0$, taking the $\xi$-th Fourier coefficient (for $\xi>0$) yields 
$$\sum_{v\nmid\infty,v\notin S}(2\delJ_v(\xi,\Phi')+\Int_{v}(\xi,\Phi))=0$$ as an equality in $ \BR_{S,\ov\BQ}$. Note that there are only many nonzero terms in the sum; in fact we have proved the $v$-term vanishes unless $v\in B$. Since both $\delJ_v(\xi,\Phi')$ and $\Int_{v}(\xi,\Phi)$ lie in $\BQ\log p_v$, and   $\{\log p_v\mid  v\in B\}$ are $\ov\BQ$-linearly independent inside $\BR_{S,\ov\BQ}$, the $v$-th term for each $v$ must vanish. This completes the proof of part (b).

\end{proof}

\begin{remark}
A byproduct of Theorem \ref{thm AFL0} is that the constant term of $\CE^\flat$ vanishes. This amounts to an equality relating a certain part of the nilpotent term \eqref{eq: nilp term} to the arithmetic degree of the restriction of the metrized line bundle $\wh{\bf\omega}$ to the derived CM cycle. This may be of some independent interest.
\end{remark}

\begin{corollary}
\label{coro AFL0}
Let $v\notin S$ be inert and assume that Conjecture \ref{AFLconj} part \eqref{AFL gp} holds for all $p$-adic field $\BQ_p$ with $p\notin S$ and for $S_{n}$.  Let $(\delta,u) \in (\U(V(v))(\alpha)\times V(v))(F_0)_\srs$ be an element matching $(\gamma, u')\in (S_{n}(\alpha)\times V')(F_0)_\srs$. If $\fkq(u)\neq 0$,  then $$
-\del\bigl((\gamma,u'), \mathbf{1}_{(S_{n}\times V'_{n})(O_{F_{0,v}})}\bigr) 
	   = \Int_{v}(\delta,u)\cdot\log q_v.
	   $$
\end{corollary}
\begin{proof}
We run the same argument as in the proof of Proposition \ref{prop refine}, where we note that the compactness (modulo $\U(\BV_n)(F_{0,v})$) of the support of the function $\Int_v(\cdot,\cdot)$ holds by Theorem \ref{prop LC}. We then obtain a refinement of the equality in part (ii) of Theorem \ref{thm AFL0}
$$
-\del((\gamma,u'),\Phi'_v)\cdot \Orb\left((\gamma,u'), \Phi'^{v,\infty}\right) =\Int_{v}( \delta,u ) \cdot \Orb\left((\delta,u), \Phi^{v}\right).
$$(We warn the reader that here $\Phi$ does not have the archimedean component.) Here we note that $\Int_v(\xi,\Phi)$ in part (ii) of Theorem \ref{thm AFL0} is given by \eqref{sum inert}. Now the away from $v$ factors on the two sides are equal and can be chosen to be non-zero (e.g., the function $ \Phi^{(v)}$ can be chosen point-wise non-negative with non-empty support containing $(\delta,u)$). 

\end{proof}

\section{The proof of AFL}
\label{s:pf AFL}

Now we return to the set up of  Conjecture \ref{AFLconj} in \S\ref{s:AFL}. 
\begin{theorem} \label{thm AFL}
Conjecture \ref{AFLconj}  holds when $F_0=\BQ_p$ and $p\geq n$.
\end{theorem}
\begin{proof}

The proof is parallel to that of Theorem \ref{thm FL}.  We prove  Conjecture \ref{AFLconj} part \eqref{AFL lie} by induction on $n=\dim \BV_n$. The case $n=1$ is known \cite{Z12}. Assume now that  Conjecture \ref{AFLconj} part \eqref{AFL lie}  holds for $\BV_{n-1}$. Then by Proposition \ref{prop AFL L2G} part \eqref{L2G i},  Conjecture \ref{AFLconj} part \eqref{AFL gp}  holds for $S_{n}$.  We now want to globalize the situation in order to apply Corollary \ref{coro AFL0}. 

We start with the following  local data
\begin{itemize}
\item  a place $v_0$ of $F_0=\BQ$, and an unramified (local) quadratic extension $F_{w_0}/F_{0,v_0}$,
\item  the non-split $F_{v_0}/F_{0,v_0}$-hermitian space $\BV_{n}$ of dimension $n$, 
\item $(g_{v_0}, u_{v_0})\in (\U(\BV_{n})\times \BV_{n})(F_{0,v_0})_{\srs}$,  we further assume that the characteristic polynomial of $g_{v_0}$ has integral coefficients (in $O_{F_{w_0}}$) and $\det(1-g_{v_0})$ is a unit, and $\pair{u_{v_0},u_{v_0}}\neq 0$,
\item  $(\gamma_{v_0}, u'_{v_0})\in (S_{n}\times V')(F_{0,v_0})(F_{0,v_0})_\srs$ matching  $(g_{v_0}, u_{v_0})$.
\end{itemize}

By the proof of Theorem \ref{thm FL}, there exist the following global data
\begin{itemize}
\item an imaginary quadratic field $F/F_0$ such that $F\otimes_{F_0}F_{0,v_0}\simeq F_{w_0}$,
\item a totally real number field $F'_0$, and its quadratic extension $F'=F_0'\otimes_{F_0}F$,
\item
an element $g\in F'^1$ such that $O_{F_{w_0}}[g]=O_{F_{w_0}}[  g_{v_0}]$ as subrings of $F'\otimes_{F} F_{w_0}$,
\item a totally positive definite  $n$-dimensional $F/F_0$-hermitian space $V(v_0)$  that is locally at $v_0$ isometric to $\BV_{n}$, and an embedding $F'^1\incl \U(V(v_0))(F_0)$,
\item $u\in V(v_0)$ such that  the pair $(g,u)$ is $v_0$-adically close to $(g_{v_0}, u_{v_0})$ (in particular $\pair{u,u}\neq 0$).
\end{itemize}
Let $\alpha\in \CA_n(F_0)$ denote the characteristic polynomial of $g$ as an element in $\U(V(v_0))(F_0)$. 

Now we define the Shimura variety and its integral model $\CM$ as in Definition \ref{def RSZ glob} for the nearby hermitian space $V$  of $V(v_0)$ at $v_0$ (that is, non-split at $v_0$, with signature $(n-1,1)$ at $v\mid\infty$, and isomorphic to $V(v_0)$ elsewhere). Let $\fkd$ be a finite set of places as in \S\ref{sss:RSZ} such that $v_0\nmid \fkd$. Let $S$ the set of non-archimedean places such that 
\begin{itemize}\item $v_0\notin S$,
\item $S$ contains all places dividing $\fkd$ and all primes less than $n$,
\item for every non-archimedean  $v\notin S\cup\{v_0\}$, the ring $R_\alpha$ is locally maximal at $v$.
\end{itemize}

Then we proceed as the proof of Theorem \ref{thm FL} to choose $(\gamma,u')\in (S_n\times V_n')(F_0)$ to match $(g,u)$, and choose (partial) Gaussian test functions $\Phi$ and $\Phi'$.

Now we apply Corollary \ref{coro AFL0}  to obtain 
$$
    -\del((\gamma,u'),\Phi'_{v_0})=\Int_{v_0}(g,u)\log q_{v_0}.
$$
Therefore Conjecture \ref{AFLconj} part \eqref{AFL lie} holds when $(g, u)\in (\U(\BV_{n})\times \BV_{n})_{\srs}$.  By the local constancy of the orbital integral, and of the intersection numbers by Theorem \ref{prop LC}, near a strongly regular semisimple $(g,u)$, we conclude that Conjecture \ref{AFLconj} part \eqref{AFL lie} holds when $(g_{v_0}, u_{v_0})\in (\U(\BV_{n})\times \BV_{n})_{\srs}$. 
This complete the induction.

\end{proof}
\appendix
\section{Weil representation commutes with smooth transfer}
\label{s:appA}
We retain the notation in \S\ref{s:FL var}. Let $F/F_0$ be a quadratic extension of local fields (the case $E=F\times F$ could also be allowed but in that case the result below is trivial).
Recall that $V_n=F_0^{n}\times (F_0^{n})^\ast$. We have a bijection of regular semisimple orbits, cf. \S\ref{ss: transfer},
\[
 \xymatrix{\coprod_{V} \bigl[(\U(V)\times V)(F_0) \bigr]_\rs	   \ar[r]^-\sim&  [S_{n}(F_0)\times
V'_{n}]_\rs},
\]
where the disjoint union runs over the set of isometry classes of $F/F_0$-hermitian spaces $V$ of dimension $n$.
The notion of smooth transfer is as in Definition \ref{def st loc} (w.r.t. the transfer factor there).  Here let us focus on one hermitian space $V$ at a time.

The Weil representation (for even dimensional quadratic space) is defined in \S\ref{s:weil}. Here we apply the formula  \eqref{eqn weil} to the second variable in the functions in $ \CS(S_{n} \times
V'_{n})$ and $\CS(\U(V)\times V) $ respectively. To fix the set up, we recall that the structure of $F_0$-bilinear symmetric pairing on $V_{n}'$ is the tautological pairing 
$$\pair{u', u'}=2 \,u_2(u_1),\quad
u'= (u_1,u_2)\in F_0^{n}\times (F_0^{n})^\ast.
$$
and on $V$ the quadratic form is  the induced one, i.e.,  
$$
\pair{u, u}_{F_0}=\tr_{F/F_0}\pair{u,u}_F, \quad u\in V
$$
where $\pair{\cdot,\cdot}_F:V\times V\to F$ is the hermitian pairing ($F$-linear on the first factor and conjugate $F$-linear on the second one).

We now deduce the following result from \cite{Z14} when $F$ is non-archimedean, and \cite{Xue} when $F$ is archimedean.
\begin{theorem}[Weil representation commutes with smooth transfer]\label{thm Weil ST}
 If $\Phi' \in \CS(S_{n} \times
V'_{n})$ matches a function $\Phi \in \CS(\U(V)\times V) $, then $\omega(h)\Phi' $ also matches $\omega(h)\Phi$ for any $h\in \bH(F)$.
\end{theorem}
\begin{remark}
Similar results hold for the partial Fourier transforms on the Lie algebra $\fks_{n} \times
V'_{n}$ and $\fku(V)\times V$. A similar result for the endoscopic transfer can be deduced from a theorem of Waldspurger.  \end{remark}
\begin{proof}
We need to check the assertion for $h$ of the form $\left(\begin{matrix} a& \\
& a^{-1}
\end{matrix}\right), \left(\begin{matrix} 1&b \\
&1
\end{matrix}\right)$ and $\left(\begin{matrix} &1\\
-1&
\end{matrix}\right)$, as in \eqref{eqn weil}.

The assertion for $h=\left(\begin{matrix} 1&b \\
&1
\end{matrix}\right), b\in F$ is trivial.

Now  let $h_a=\left(\begin{matrix} a& \\
& a^{-1}
\end{matrix}\right)$. Then
\[
   \Orb((g,u), \omega(h_a) \Phi ) =\chi_V(a) |a|^{n}  \Orb((g,a u),  \Phi )
   \]
  for all $(g,u)\in (\U(V)\times V)_\rs$. Here 
  $$
  \chi_V(a)=(a,(-1)^{\dim_F V} \det(V)),
  $$
  where $\det(V)$ is the discriminant of $V$ as a quadratic space.
  We claim  
  \begin{align}
  \chi_V(a)=\eta(a)^{\dim_F V}.
\end{align} Since $\det(V_1\oplus V_2)=\det(V_1) \det(V_2) $ (in $F_0^\times/(F_0^\times)^2$) for orthogonal direct sum $V_1\oplus V_2$, it suffices to prove the claim when $\dim_F V=1$. Then there are only two isometry classes and one can check the claim directly.
  
 On the other hand 
   \[ \Orb((\gamma, u'),\omega(h_a) \Phi', s) =\chi_{V_n'}(a) |a|^{n}  \Orb((\gamma, au'), \Phi',s). 
\]
Now $\chi_{V_n'}$ is the trivial character since $V_n'$ is an  orthogonal direct sum of $n$-copies of the hyperbolic $2$-space. We now note that the transfer factor \eqref{SV transfer factor} obeys 
\[
   \omega(\gamma,a u')=\eta(a)^n   \omega(\gamma,u').
\]
This proves the assertion for $\omega(h_a), a\in F^\times$.

Finally, let $h=\left(\begin{matrix} &1\\
-1&
\end{matrix}\right)$. Then
\[
   \Orb((g,u), \omega(h) \Phi ) =\gamma_V  \Orb((g, u),  \wh \Phi )
   \]
   where $\gamma_V$ is the Weil constant. 
   We claim, for our $V$ induced from a hermitian form
   $$
   \gamma_V=\eta(\det(V)_{F/F_0})\, \epsilon(\eta,1/2,\psi)^{\dim_FV}
   $$
   where $\det(V)_{F/F_0}\in F_0^\times /\Nm F^\times$ is the hermitian discriminant of $V$ (as an $F/F_0$-hermitian space). 
   First note that the right hand side is multiplicative with respect to orthogonal direct sum $V_1\oplus V_2$
   $$\det(V_1\oplus V_2)_{F/F_0}=\det(V_1)_{F/F_0} \det(V_2)_{F/F_0}. $$
   Note that, by definition, the Weil constant $\gamma_V$ satisfies
   $$
   \wh{ \psi\circ q}= \gamma_V \,\psi\circ (-q),
   $$ 
   where $\psi\circ q: V\to F_0\to \BC$ (resp., $\psi\circ (-q)$) is the function precomposing $\psi$ with $q$ (resp., $-q$). Here the Fourier transform is understood as applied to distributions.  It follows that it is also multiplicative with respect to orthogonal direct sum $V_1\oplus V_2$: 
    $$\gamma_{V_1\oplus V_2}=\gamma_{V_1}\cdot\gamma_{ V_2} .
    $$ 
Therefore it suffices to show the claim when $\dim_FV=1$. Then one can check the claim directly. In fact it is easy to see that we have $\gamma_{V_a}=\eta(a)^{\dim_FV}\gamma_V$ where $V_a$ denotes the new hermitian space by multiplying the hermitian form by $a\in F_0^\times$. Hence we may just check the case $\det(V)_{F/F_0}\in \Nm F^\times$, which is done in \cite[Lem.\ 1.2]{JL} (where the constant $\lambda_{F/F_0}(\psi)$ in {\it loc. cit.} is the same as $\epsilon(\eta,1/2,\psi)$). 

 On the other hand, the Weil constant $   \gamma_{V_n'}=1$ since $V_n'$ is an  orthogonal direct sum of $n$-copies of the hyperbolic $2$-space.  Hence
   \[ \Orb((\gamma, u'),\omega(h) \Phi', s) =  \Orb((\gamma, u'), \wh \Phi',s). 
\]

Now the desired assertion follows from \cite[Th.\ 4.17]{Z14}  \footnote{Note that in \cite{Z14}, the factor $\eta(\det(V)_{F/F_0})$ is missing.} when $F$ is non-archimedean, and the proof of \cite[Th.\ 9.1]{Xue} when $F$ is archimedean. Note that in \cite{Xue}, $\epsilon(\eta,1/2,\psi)=\sqrt{-1}$ for his choice of the additive character $\psi(x)=e^{2\pi \sqrt{-1} x}, x\in\BR$.
\end{proof}

\section{Grothendieck groups for formal schemes}\label{s:appB}
We collect some facts regarding formal schemes and the Grothendieck group of coherent sheaves, largely following the work by Gillet--Soul\'e \cite{GS87}. No result here is new.

\subsection{Grothendieck groups}\label{ss:G gp}

Let $(X,\CO_X)$ be a noetherian formal scheme \cite[\S10]{EGA1}. Let $Y$ be a closed formal subscheme of $X$ (i.e., closed subscheme of a formal scheme in the terminology in {\it loc. cit.}). Let $\CJ$ be the sheaf of ideals defining $Y$.
A coherent sheaf $\CF$ of $\CO_X$--module is said to be {\em formally supported} on $Y$  if it is annihilated by $\CJ^n$ for some $n\geq 1$. We make this explicit when $(X,\CO_X)$ is an affine formal scheme, say, the formal completion of $\Spec A$ at $\Spec A/I$ for an ideal $I$ of $A$, where $A=\varprojlim _{n}A/I^n$ is $I$-adically complete. Then we may assume that $Y$ is defined by an ideal $ J$ of $A$ (i.e., $\CJ= J^\Delta$, cf. \cite[\S10.10]{EGA1}). Then a coherent sheaf $\CF$ of $\CO_X$--module is formally supported on $Y$ if $M=\Gamma(X,\CF)$  as an $ A$-module (equivalently the sheaf $\wt J$ of $\CO_{\Spec  A}$--module) has support contained in the closed subset $\Spec( A/J)$ of $\Spec  A$.

 Then the definitions in \cite[\S1]{GS87} for noetherian schemes carry over to the setting of noetherian formal schemes. Let  $K'_0(X)$ denote the Grothendieck group of coherent sheaves of $\CO_X$-modules. Let  $K^Y_0(X)$
denote the Grothendieck group of finite complexes of coherent locally free $\CO_X$-modules, acyclic outside $Y$ (i.e., the homology sheaves are supported on $Y$), cf. \cite[\S1.2]{GS87}. Let  $K_0(X)= K^X_0(X)$. The tensor product of (complex of) locally free sheaves  induces the cup product
$$
\xymatrix{\cup\colon K_0^Y(X)\times K_0^Z(X)\ar[r]&K_0^{Y\cap Z}(X) }
$$
by $[\CF_\cdot]\cup [\CG_\cdot]=[\CF_\cdot\otimes \CG_\cdot]$, cf. \cite[\S1.4]{GS87}.

There is a descending filtration on $K_0^Y(X)$ by the subgroups  
\begin{align}\label{Fil K}
\xymatrix{F^i K_0^Y(X)=\cup_{Z\subset Y, {\rm codim}_{X}Z\geq i}\Im(K_0^Z(X)\to K_0^Y(X)).}
\end{align}
The associated graded groups are 
\begin{align}\label{Gr K}
{\rm Gr}^iK_0^Y(X)=F^i K_0^Y(X)/F^{i+1} K_0^Y(X).
\end{align}
Similarly, there is an ascending filtration $F_i K_0'(X)$ on $K_0'(X)$
$$
\xymatrix{F_i K_0'(X)=\cup_{Z\subset X, \dim Z\leq i}\Im(K_0'(Z)\to K'_0(X)).}
$$

From now on we assume that $X$ is regular of pure dimension $d$. Then we have a natural isomorphism 
$$
\xymatrix{K^Y_0(X)\ar[r]^-{\sim}& K'_0(Y)},
$$
and 
$$
\xymatrix{F^{d-i} K^Y_0(X)\ar[r]^-{\sim}&F_{i} K'_0(Y)}.
$$
When $X$ is a scheme, the construction of the Adam operations $\{\psi^k\mid k\in \BZ_{\geq 1}\}$ in \cite{GS87}  induce a decomposition 
$$
K_0^Y(X)_\BQ=\bigoplus_{i\geq 0}K_0^Y(X)_\BQ^i,
$$ 
where $\psi^k$ acts on (the ``weight-$i$" part) $K_0^Y(X)_\BQ^i$ by the scalar $k^i$. Moreover, by \cite[Prop.\ 5.3]{GS87}
$$
F^j K_0^Y(X)_\BQ=\bigoplus_{i\geq j}K_0^Y(X)_\BQ^i,
$$
and for $j_1,j_2\geq 0$, by  \cite[Prop.\ 5.5]{GS87}, the cup product has image
\begin{align}\label{cup Fil}
 F^{j_1} K_0^Y(X)_\BQ\cdot F^{j_2} K_0^Z(X)_\BQ\subset F^{j_1+j_2} K_0^{Y\cap Z}(X)_\BQ.
 \end{align}
 This inclusion is used in \eqref{def LCM fil}. When $X$ is a formal scheme, we expect the same argument to prove \eqref{cup Fil}, and we use this case of \eqref{cup Fil} only in the proof of Proposition \ref{prop C infty}. However, due to the lack of reference, we also indicate a proof of Proposition \ref{prop C infty} without using \eqref{cup Fil}, cf. Remark \ref{rem:no B3}.

Finally, we relax the noetherian hypothesis. For our purpose, we only consider locally noetherian formal schemes $(X,\CO_X)$. It can be written as an increasing union indexed by a poset $I$
$$(X,\CO_X)=\cup_{i\in I}(X_i,\CO_{X_i})$$ of noetherian formal subschemes such that the transition maps $f_{i,i'}:X_i\to X_{i'}$ are open immersions of formal schemes.  We then define
$$
K_0(X)=\varprojlim_{i\in I} K_0(X_i),\quad K'_0(X)=\varprojlim_{i\in I} K'_0(X_i).
$$
If $Y$ is a closed formal subscheme of $X$, setting $Y_i=Y\times_{ X} X_i$ to write $Y$ as the union of $Y_i$'s, we define 
$$
K_0^Y(X)=\varprojlim_{i\in I} K_0^{Y_i}(X_i).
$$
Similarly, we have the filtration $F^iK_0^Y(X)$, and $F_iK_0'(X)$, and they have the same properties as in the noetherian case. All of these K-groups depend only on $X$, rather than the choice of such unions $\cup_{i\in I} X_i$.

Now let $\pi: W\to S= \Spf A$ be a morphism of formal schemes, where $A$ is a complete discrete valuation ring. 
When $\pi$ is  proper \cite[III, 3.4.1]{EGA3} and $W$ is a {\em scheme} (not only a formal scheme), we have a ``degree" map
$$\xymatrix@R=0ex{ 
K'_0(W)\ar[r]& \BZ\\ 
[\CE]\ar@{|->}[r]& \sum_{i\in\BZ}(-1)^{i }\length_{\CO_S}  \RR^i\pi_\ast \CE.}
$$
The assumption on $\pi$ and $W$ implies that  all $ \RR^i\pi_\ast \CE$ are torsion coherent sheaves and hence have finite lengths. It is easy to see that this is independent of the choice of $\CE$ in its equivalence class. Now let $X$ be regular with two closed formal subscheme $Y$ and $Z$. If $\pi: W=Y\cap Z\to S=\Spf A$ is proper and $W$ is a scheme, we obtain a homomorphism
$$
\xymatrix@R=0ex{ K_0^Y(X)\times K_0^Z(X)\ar[r]&\BZ\\
([\CF],[\CG])\ar@{|->}[r]& \chi(X, \CF\Ltimes\CG)  }
$$
where the Euler--Poincar\'e characteristic is defined by
\begin{align}\label{eqn:chi FG}
 \chi(X, \CF\Ltimes\CG)\colon= \sum_{i,j\in\BZ}(-1)^{i +j}\length_{\CO_S}  \RR^i\pi_\ast ({\rm Tor}^{\CO_X}_j(\CF,\CG)).
\end{align}
We also denote
\begin{align}\label{eqn:d jiao}
Y\jiao_X Z\colon= \CO_Y\Ltimes_{\CO_X}\CO_Z\in K_0'(Y\cap Z)\simeq K_0^{Y\cap Z}(X),
\end{align}
and if the ambient formal scheme $X$ is self-evident, we simply write it as $Y\jiao Z$.

\subsection{A few lemmas}
For convenience we record the following results.
\begin{lemma}
\label{lem:cut}
Let $X$ be a locally noetherian formal schemes of the above type.
Let $X=X_1\cup X_2$ be a union of two closed formal subschemes.  Then there is a natural isomorphism \footnote{Here $K'_0(X_1\cap X_2)\to K'_0(X_1)$ is not necessarily injective, so the quotient simply denotes the cokernel. }
$$\xymatrix@R=0ex{
\frac{K'_0(X)}{K'_0(X_1\cap X_2)}\ar[r]^-\sim& \frac{K'_0(X_1)}{K'_0(X_1\cap X_2)} \bigoplus \frac{K_0'(X_2)}{K'_0(X_1\cap X_2)}\\
[\CE]\ar@{|->}[r]& \left([\CE\otimes_{\CO_X}\CO_{X_1}] ,[\CE\otimes_{\CO_X}\CO_{X_2}]\right ) .}
$$
\end{lemma}

\begin{proof}We immediately reduce the question to the case when  $X$ is noetherian, which we assume now. Let $\CI$ and $\CJ$ be the ideal sheaf of $\CO_X$ defining $X_1$ and $X_2$ respectively. Consider the exact sequence of $\CO_X$-modules
$$\xymatrix@R=0ex{
0\ar[r]& \CO_X/(\CI\cap\CJ )\ar[r]&\CO_X/\CI\oplus \CO_X/\CJ \ar[r]& \CO_X/(\CI+\CJ) \ar[r] &0 } ,
$$
Tensoring $\CE$, we obtain an exact sequence
$$\xymatrix@R=0ex{
{\rm Tor}_1^{\CO_X}(\CE, \CO_{X_1\cap X_2} )\ar[r]&\CE\otimes \CO_X/(\CI\cap\CJ) \ar[r]&\CE\otimes\CO_{X_1}\oplus \CE\otimes\CO_{X_2}\ar[r]& \CE\otimes\CO_{X_1\cap X_2} \ar[r] &0 }.
$$
 Since both ${\rm Tor}_1^{\CO_X}(\CE, \CO_{X_1\cap X_2})$ and $ \CE\otimes\CO_{X_1\cap X_2}$ lie in $K_0'(X_1\cap X_2)$, we have  
$$
[\CE\otimes\CO_{X_1}]+[\CE\otimes\CO_{X_2}] = [\CE\otimes \CO_X/(\CI\cap\CJ)] \in \frac{K'_0(X)}{K'_0(X_1\cap X_2)}.
$$
Since $X=X_1\cup X_2$, we have $\CI\cap\CJ=0$ and the proof is complete.
\end{proof}

In the case of  ``proper intersection", the derived tensor product can be simplified:
\begin{lemma}\label{proper int}
Let $X$ be a (locally noetherian) pure finite dimensional formal scheme  of the above type, and $Z_1,Z_2$ two pure dimensional closed formal subschemes on $X$. Assume that the closed immersion $Z_1\to X$  is a regular immersion (e.g., if both $X$ and $Z_1$ are regular), and $Z_2$ is Cohen--Macaulay.
\[
   \xymatrix{   Z_1\cap Z_2 \ar[r] \ar[d] \ar@{}[rd]|*{\square}  &  Z_2\ar[d]\\
 Z_1 \ar[r] & X
.   }
\] 
\begin{altenumerate} 
\item  
If $Z_1\cap Z_2$ has the expected dimension (i.e., ${\rm codim}_{X}Z_1\cap Z_2= {\rm codim}_{X}Z_1+{\rm codim}_{X}Z_2$ at every point of $Z_1\cap Z_2$), then the higher ${\rm Tor}$ sheaves vanish, i.e.,
 $$
  {\rm Tor}_i^{\CO_X}(\CO_{Z_1},\CO_{Z_2})=0,\quad i>0.
  $$ 
  In particular, as elements in $K'_0(Z_1\cap Z_2)$,
  $$
  \CO_{Z_1}\Ltimes\CO_{Z_2}=\CO_{Z_1}\otimes\CO_{Z_2}.
  $$
\item
Let  $Z_1\cap Z_2=Y\cup Y'$ be a union of closed formal subschemes such that $Y$ has the expected dimension.
  Then 
  $$
  {\rm Tor}_i^{\CO_X}(\CO_{Z_1},\CO_{Z_2})|_{Y}\equiv 0,\quad i>0,
  $$ 
as an element in $K'_0(Y)/K'_0(Y\cap Y' )$.
\end{altenumerate} 

\end{lemma}
\begin{proof}This follows from the same argument in the proof of \cite[Prop.\ 8.10]{RSZ1} regarding the vanishing of higher $\rm Tor$ terms.  We prove the first part; the second part is proved similarly by combining Lemma \ref{lem:cut}.

Let $x$ be a point  on $Z_1\cap Z_2$. We need to show that $ ( \CO_{Z_1}\Ltimes\CO_{Z_2})_x$ is represented by $\CO_{Z_1\cap Z_2,x}$.
Let $R$ be the local ring of $x$ on $X$. Since the closed immersion $Z_1\to X$ is a regular immersion, by definition $Z_1$ is defined at $x$ by a regular sequence $f_1,\cdots, f_m$ of $R$.  Then the Koszul complex $K(f_1,\cdots, f_m)$ is a free resolution of the $R$-module $\CO_{Z_1,x}$. It follows that the complex $K(f_1,\cdots, f_m)\otimes _R \CO_{Z_2,x}$ represents $ ( \CO_{Z_1}\Ltimes\CO_{Z_2})_x$.

Now, since $Z_2$ is  Cohen--Macaulay, the dimension hypothesis implies that the images $\ov f_1,\cdots, \ov f_m$ of  $f_1,\cdots, f_m$ in $\CO_{Z_2,x}$  again form a regular sequence which generates the ideal defining $Z_1\cap Z_2$ at $x$ in $Z_2$. Hence $K(\ov f_1,\cdots, \ov f_m)$ is a free resolution of the $\CO_{Z_2,x}$-module  $\CO_{Z_1\cap Z_2,x}$. On the other hand, we have
$$
K(f_1,\cdots, f_m)\otimes _R \CO_{Z_2,x}=K(\ov f_1,\cdots, \ov f_m).
$$
It follows that $ ( \CO_{Z_1}\Ltimes\CO_{Z_2})_x$ is represented by $K(\ov f_1,\cdots, \ov f_m)$, or equivalently by  $\CO_{Z_1\cap Z_2,x}$. This completes the proof.

\end{proof}

\end{document}